\DocumentMetadata{pdfversion=1.7, pdfstandard=A-2b, lang=en-US}
\documentclass{ucbthesis}
\usepackage{rotating}
\usepackage{import}
\usepackage{preamble}

\begin{document}

\title{Comparison of Logarithmic Hochschild Homologies}
\author{Joseph Stahl}
\degreesemester{Summer}
\degreeyear{2026}
\degree{Doctor of Philosophy}
\chair{Martin Olsson}
\othermembers{Professor Sug Woo Shin \\
    Professor Demian Pouzo}
\numberofmembers{3}
\field{Mathematics}

\maketitle
\copyrightpage

\begin{abstract}

We construct a comparison morphism $\HHRog\to\HHOls$ between Rognes' logarithmic Hochschild homology and Olsson's logarithmic Hochschild homology which generalizes the comparison between Gabber's log cotangent complex $\LGab$ and Olsson's log cotangent complex $\LOls$ constructed by Olsson in \cite{Olssoncotangent}. In particular, Olsson's comparison is recovered as the first associated graded piece of our comparison morphism $\HHRog\to\HHOls.$ In the development of this comparison, we also prove a descent result for the cotangent complex generalizing a result of Bhatt, and construct a notion of HKR completed Hochschild homology for a general morphism of derived stacks which is functorial in the morphism and generalizes existing constructions of Hochschild homology when the base is affine.

\end{abstract}

\begin{frontmatter}

\begin{dedication}
\null\vfil
\begin{center}
To Minhyun Kay, with infinite gratitude\\\vspace{12pt}
\end{center}
\vfil\null
\end{dedication}

\tableofcontents
\clearpage

\begin{acknowledgements}
This thesis has benefited greatly from the help of many people, some of whom I will no doubt fail to mention. I would like to thank Harrison Chen, Ben Antieau, Rune Haugseng, and Bertrand To\"en for answering many of the questions I had when I was confused.

Special thanks go to Can Yaylali, for patiently helping me understand the fundamentals of derived algebraic geometry; to Peter Haine, who taught me nearly everything I know about how to work with $\infty$-categories; and to Sam Moore, with whom I shared many enlightening conversations about Hochschild homology.

I would also like to thank my advisor, Martin Olsson, for his invaluable guidance. He introduced me to logarithmic algebraic geometry and originally suggested the topic of comparing the two versions of logarithmic Hochschild homology. I am extremely grateful for all the help and advice he has given me throughout this endeavor, especially as I changed my approach to the problem, got stuck, and explored various other topics.

I also wish to thank my parents, who have given more to me than can be expressed in words. I would not be where I am today without them.

Finally, I must thank my partner, Mary. Through every challenge, her faith in me gave me the determination I needed to push through. I hope to repay the love and support she has given me.
\end{acknowledgements}

\end{frontmatter}

\pagestyle{headings}

\chapter{Introduction}

In recent years, Hochschild homology has become a more frequently used tool in algebraic geometry, especially as techniques develop to include more and more tools from homotopy theory. For instance, the study of integral $p$-adic Hodge theory in \cite{BMS} and further development of prismatic cohomology in \cite{bhatt2022prisms} have shown deep connections between (topological) Hochschild homology and various cohomology theories for schemes and stacks.

This connection extends into the world of logarithmic geometry. In \cite{Rognes_2009}, Rognes develops the notion of (topological) Hochschild homology for prelogarithmic rings and establishes basic results for these new objects. Rognes' theory was further developed by several others, particularly in \cite{BLPO_HKR} where it is shown that one can interpret logarithmic Hochschild homology of a morphism of (pre)logarithmic rings as functions on the derived self-intersection of a variant of the diagonal, similarly to Hochschild homology of ordinary rings or (reasonable) schemes.

More recently, Olsson has also developed a notion of logarithmic Hochschild homology in \cite{Olssonloghh}, by developing a theory of Hochschild homology for a nice morphism from an algebraic space to an algebraic stack, and specializing the the situation where the stack in question is the log stack of \cite{OlssonLogGeomandAlgStacks}. Hablicsek et. al have also developed an approach to logarithmic Hochschild homology via a ``logarithmic derived loop space'' $LX^{\log}$ in \cite{logHHformality}. Their variant of log Hochschild homology agrees with Olsson's under suitable hypotheses (see \cite[Proposition 4.11]{logHHformality}). Notably, both approaches share a close relationship to the interpretation of Hochschild homology as ``(functions on the) derived self-intersection of the diagonal.''

Our main theorem provides a comparison morphism between the two variants of logarithmic Hochschild homology, and as a corollary we can deduce that the two agree under mild conditions on the morphism of log schemes. In particular, we prove the following:
\begin{theorem}[\ref{Main Theorem}]
    Let $f : X\to S$ be a map of Noetherian fs log schemes of finite Krull dimension. Then there exists a morphism
    \[
    c_f : \HHRog(X/S)\to\HHOls(X/S)
    \]
    in $\DF(\ZZ)$ such that on associated graded pieces, $c_f$ induces shifts of wedge powers of the comparison morphism $\eta : \LGab_{X/S}\to \LOls_{X/S}$ constructed in \cite{Olssoncotangent}
    \[
    \gr_{\HKR}^i c_f = \bigwedge^i \eta[i] : \bigwedge^i \LGab_{X/S}[i]\to\bigwedge^i \LOls_{X/S}[i].
    \]
\end{theorem}

We use the language and theory of derived algebraic geometry in a crucial way throughout the thesis, and we begin with a review of this theory, following the development in \cite{toen2006homotopical} and the exposition of \cite{yaylali2022notes} and \cite{AntieauDAGnotes}. We establish some basic results showing that the theory works as expected, and conclude with an examination of the cotangent complex. In particular, we describe in detail how to view the cotangent complex as a functor whose inputs are morphisms of derived stacks $(\cX\to\cS)$ and whose output $\LL_{\cX/\cS}$ is an object of $\QC(\cX).$ We show that in this generality, the cotangent complex satisfies fpqc descent, generalizing a result of Bhatt.

Following our discussion of derived algebraic geometry, we review the classical and derived theories of logarithmic geometry. Most of this section is a review, although we provide a precise universal property of the repletion which we had not seen in the literature previously. We also discuss both variants of the logarithmic cotangent complex, as well as Olsson's comparison between the two found in \cite{Olssoncotangent}.

After reviewing the geometric theory necessary and establishing the descent result for the cotangent complex previously mentioned, the bulk of the work falls to defining the Hochschild homology of a morphism of derived stacks in an appropriate way, so that we can generalize Olsson's construction of logarithmic Hochschild homology to the $\infty$-categorical setting. We define \ref{relativeHH} what we call the ``HKR-completed Hochschild homology'' of a general morphism of derived stacks, and show that it satisfies many of the properties one expects of Hochschild homology, including that it agrees with the Hochschild homology of a map of (animated) rings, admits a complete, descending, $\NN$-indexed filtration whose associated graded pieces are given by shifts of exterior powers of the cotangent complex (when it exists), and for sufficiently nice morphisms may be computed as the derived self-intersection of the diagonal.

To be precise, we define functors
\[\sHH, \HH : \Arr(\dStk)^{\op}\to\QC\]
which realize Hochschild homology of a morphism of derived stacks $f : \cX\to\cS$ functorially in the morphism (as opposed to just the source), such that $\sHH(\cX/\cS)\in\QC(\cX)$ and $\HH(\cX/\cS)\in\QC(\cS).$ We also prove various compatibility results for $\sHH$ and $\HH;$ in particular, we show that $f_*\sHH(\cX/\cS)\simeq\HH(\cX/\cS),$ and that if $\cS = \Spec k$ is affine, that $\HH(\cX/\cS)$ agrees with the usual definition(s) of $\HH(\cX/k)\in\QC(\Spec k).$ We also prove that if the morphism $\cX\to\cS$ is flat, that $\sHH(\cX/\cS)$ and $\HH(\cX/\cS)$ agree with the notions defined in \cite{Olsson_2016} in an appropriate sense.

We show that our constructions of $\sHH$ and $\HH$ are endowed with an HKR filtration (generalizing the usual HKR filtration), with associated graded pieces given by shifts of powers of the cotangent complex. As a result, we generalize the HKR isomorphism for logarithmic Hochschild homology in \cite{Olsson_2016} to non-smooth morphisms.

This filtration is the main tool which allows us to obtain the comparison between Rognes' logarithmic Hochschild homology $\HH^{\Rog}$ and Olsson's logarithmic Hochschild homology $\HH^{\Ols}.$ Following Olsson's comparison between Gabber's logarithmic cotangent complex and his own, we construct a morphism $\HH^{\Rog}\to\HH^{\Ols}$ which yields Olsson's comparison between the cotangent complexes on associated gradeds using the fact that Gabber's cotangent complex is the animation of the functor of logarithmic differentials. Specifically, we construct for sufficiently nice maps of prelog rings $(k,P)\to (A,M)$ a canonical morphism
\[
\HH^{\Rog}((A,M)/(k,P))\to\HH^{\Ols}((A,M)/(k,P))
\]
induced by the universal property of the repletion, and show that for a free polynomial prelog algebra over a free prelog ring, the map is canonically an equivalence, and that the associated graded pieces of the induced map correspond to the comparison morphism between log cotangent complexes constructed in \cite[Section 8]{Olssoncotangent}.

Applying the discussion in \cite[Remark 7.29]{BLPO_HKR}, we deduce that these maps on associated gradeds, while not equivalences in general, do become so after dividing Nisnevich sheafification, and deduce that the morphism $\HH^{\Rog}\to\HH^{\Ols}$ also becomes an equivalence after dividing Nisnevich sheafification. 

We conjecture that there should also be a more robust version of this comparison obtained by geometric considerations. Both logarithmic Hochschild homology theories are closely related to logarithmic modifications of the ``derived self-intersection of the diagonal'' interpretation of ordinary Hochschild homology. However, this would require more theory than we have established currently. We conjecture that for a sufficiently nice map of log schemes $X\to S,$ $\HH^{\Rog}(X/S)$ can be obtained as $R\Gamma(X,\Delta^*\Delta_*\OX)$ for a map $X\to\bfR\cLog_{S},$ where $\bfR\cLog_S$ is a derived version of Olsson's stack of logarithmic structures, and that $\HH^{\Ols}(X/S)$ will be closely related to the classical truncation of $\bfR\cLog_{S}.$ The theory of $\bfR\cLog_S$ has recently been developed in \cite{zhang2026derivedlogarithmicdeformationtheory}, and we hope this further development of this theory will lead to a geometric comparison between the logarithmic Hochschild homology theories.

\section{Conventions}

\begin{enumerate}
    \item We use the theory of $\infty$-categories throughout the work in essential ways. We refer the reader to \cite{HTT} and \cite{HA} for the relevant background. We attempt to work model-independently with regard to $\infty$-categories as much as possible, but when specific models are needed, we follow the conventions of \cite{HTT} unless otherwise stated.
    \item We will refer to classical algebraic objects (e.g., commutative rings, commutative monoids, etc.) as \emph{discrete} or \emph{classical} in order to distinguish them from their animated counterparts. We will not consider topological rings, monoids, etc.. However, for their geometric counterparts, we will use solely the term \emph{classical}; i.e., a classical scheme, classical algebraic space, classical algebraic stack. 
    \item A classical algebraic or Artin stack will mean a stack whose diagonal is representable by algebraic spaces, and which admits a smooth surjection from a scheme. We assume the reader is familiar with the classical theory of algebraic spaces and stacks at the level of \cite{Olsson_2016}, and recommend \cite{Olsson_2016} and \cite{stacks-project} for references.
    \item We will denote a (pre)log scheme $(\underline{X},\cM_{X},\alpha_X)$ simply by $X,$ generally eliding the (pre)log structure, unless there is risk of confusion or we need to emphasize the particular (pre)log structure. For a (pre)log scheme $X,$ the underlying scheme will be denoted $\underline{X}.$
    \item Logarithmic structures will not be assumed to be fine (or saturated) in general. However, when dealing with classical log structures, we will often specialize to this case.
    \item For a discrete commutative ring $k,$ we will let $\Mod_k$ denote its ordinary $1$-category of modules, and we will let $\cMod_k$ denote the $\infty$-category of $k$-modules (meaning $k$-module objects in spectra or equivalently the stabilization of the category $\cMod_k^{\cn}$ of connective or animated $k$-modules).
    \item We choose to use \emph{homological} indexing throughout, except when otherwise stated. This conforms with many sources' approaches to Hochschild homology and to $\infty$-categorical theory in general, although it then differs from the standard algebraic geometry conventions.
    \item Given an object $X$ in a stable $\infty$-category $\cC,$ we let $X[1]$ be the cofiber of the canonical map $X\to 0.$ In particular, there is a canonical pushout square
    \[\begin{tikzcd}
        X\ar[r]\ar[d] &0\ar[d]\\
        0\ar[r] &X[1].
    \end{tikzcd}\]
    \item For any $\infty$-category $\cC,$ the category $\Arr(\cC)$ is defined to be the $\infty$-category $\Fun(\Delta^1,\cC).$ If $\cC$ is an ordinary category, we may use $\Arr(\cC)$ to refer to $\Fun([1],\cC)$ as a $1$-category or the $\infty$-category $\N(\Fun([1],\cC))\simeq\Fun(\Delta^1,\N(\cC)),$ depending on the context.
    \item By definition, all filtrations will be decreasing; see \ref{filtrations}.
\end{enumerate}
\chapter{Derived Algebraic Geometry}

There are several closely related, but distinct, conventions for the foundations of derived algebraic geometry. We begin by reviewing the basics of derived algebraic geometry as established by To\"{e}n-Vezzosi in \cite{toen2004homotopical} and \cite{toen2006homotopical}, and as laid out in more modern language in \cite{yaylali2022notes}. Our exposition follows \cite{yaylali2022notes} closely.

\section{Affine Derived Algebraic Geometry}

Animated rings will play the role in derived algebraic geometry analogous to the role that discrete commutative rings play in classical algebraic geometry. In particular, animated rings will provide the local models for general derived algebraic geometry. As such, we begin by recalling the most important definitions and results about these objects.

\subsection{Animated Rings}

\begin{mydef}
    Let $k$ be a discrete commutative ring. The $\infty$-category $\aRing_k$ of \emph{animated $k$-algebras} is the animation of the category $\CRing_k$ of discrete $k$-algebras. If $k = \ZZ,$ then we simply write $\aRing_k = \aRing.$ If $A$ is an animated $k$-algebra, then we also let $\aRing_A := \left(\aRing_k\right)_{A/}.$
\end{mydef}

\begin{remark}
    The category $\aRing_k$ admits several equivalent descriptions, as described in \ref{AppendixA}. Convenient ones include the $\infty$-category obtained by freely adjoint sifted colimits to $\Poly_k$ (in the $\infty$-categorical sense), the underlying $\infty$-category of the model category of simplicial commutative rings with its standard model structure (see \cite{Quillen_1967a} for a more detailed exposition or \cite{MathewSCR} for an overview), or the full sub-$\infty$-category of $\Fun(\Poly_k^{\op},\Ani)$ spanned by the finite-product preserving functors, where $\Poly_k$ is the category of polynomial $k$-algebras in finitely many variables, and $\Ani$ is the $\infty$-category of anima/spaces.
\end{remark}

Using the fact that $\aRing$ is equivalent to the $\infty$-category associated to the simplicial model category $\sCRing,$ it follows from \cite[5.5.3.11]{HA} and \cite[A.3.7.6]{HA} that $\aRing$ is presentable.

Because $\aRing_k$ is defined to be the animation of a $1$-category, it satisfies a universal property: namely, given a presentable $\infty$-category $\cC,$ the canonical functor $\Poly_k\to\aRing_k$ induces an equivalence of $\infty$-categories
\[
\Fun_{\mathrm{sift}}(\aRing_k,\cC)\to\Fun(\Poly_k,\cC)
\]
where $\Fun_{\mathrm{sift}}$ denotes the full sub-$\infty$-category of $\Fun$ spanned by functors which preserve sifted colimits. As functors between $\infty$-categories can be difficult to construct, this provides us with a very useful method for building functors out of $\aRing_k.$

There are other variants of ``derived commutative ring'' that one could take as the building blocks for derived algebraic geometry, including dg-algebras and connective $E_\infty$-algebras. Despite the fact that our focus will be on animated rings and their geometric counterparts, it will benefit us to make use of some aspects of the theory of $E_\infty$-rings.

For $k$ a discrete commutative ring, let $\CAlgcn_k$ denote the $\infty$-category of connective $E_\infty$-$k$-algebras (see \cite[Chapter 7]{HA} for details). The embedding $\Poly_k\to\CAlgcn_k$ then induces a sifted colimit preserving functor $\theta : \aRing_k\to\CAlgcn_k$ via the universal property of animation. The functor $\theta$ is conservative and admits both a left adjoint $\theta^L$ and a right adjoint $\theta^R,$ see \cite[25.1.2.2]{SAG}. 

\begin{mydef}
    Let $A\in\aRing$ be an animated ring, and let $i\in\ZZ.$ The \emph{$i$th homotopy group of $A$} is defined to be $\pi_i(A) := \pi_i(\theta(A)),$ where $\pi_i : \CAlgcn_{\ZZ}\to\Sp\to\Ab$ are the homotopy group functors.
\end{mydef}

\begin{remark}
    The definition of $\pi_i(A)$ above agrees with the $i$th homotopy group of the underlying anima of $A.$ Additionally, since $\theta(A)$ is connective for any animated ring $A,$ $\pi_i(A)\simeq 0$ for $i < 0.$ Finally, $\pi_*(A)$ has the structure of a graded ring. 
\end{remark}

\begin{mydef}\label{truncated animated rings}
    Let $A\in\aRing_k$ be an animated $k$-algebra, and let $n\geq 0$ be an integer. We say that $A$ is \emph{$n$-truncated} if $\pi_m(A)\simeq 0$ for $m > n,$ and we let $\left(\aRing_k\right)_{\leq n}$ denote the full sub-$\infty$-category of $\aRing_k$ spanned by the $n$-truncated animated $k$-algebras. 
\end{mydef}

\begin{remark}
    By \cite[Remark 25.1.3.3]{SAG}, an animated $k$-algebra $A$ is $n$-truncated if and only if the underlying $E_\infty$-$k$-algebra $\theta(A)$ is $n$-truncated, if and only if $\Map_{\aRing_k}(B,A)$ is $n$-truncated for all $B\in\aRing_k.$
\end{remark}

In fact, if $A$ is an animated ring and $n\geq 0$ is an integer, then the inclusion
\[
\left(\aRing_A\right)_{\leq n}\to\aRing_A
\]
is the right adjoint in an adjunction
\[\begin{tikzcd}
    \tau_{\leq n} : \aRing_A\ar[r,shift left] &\ar[l,shift left]\left(\aRing_A\right)_{\leq n}.
\end{tikzcd}\]
If $n = 0,$ $\left(\aRing_A\right)_{\leq 0}$ is equivalent to the category of discrete commutative $\pi_0 A$ algebras, and we also write $\pi_0$ for $\tau_{\leq 0}.$

\begin{mydef}
    Let $f : A\to B$ be a map of animated $k$-algebras. We say $B$ is \emph{locally of finite presentation over $A$} if it is compact as an animated $A$-algebra; i.e., the functor $\Map_{\aRing_A}(B,-) : \aRing_A\to\Ani$ commutes with filtered colimits.
\end{mydef}

\begin{warning}
    If $f : A\to B$ is locally of finite presentation, then $\pi_0 f : \pi_0 A\to\pi_0 B$ is finitely presented, but the converse is not true, even if $A$ and $B$ are discrete. Instead, $f : A\to B$ is locally of finite presentation if and only if $\pi_0 f : \pi_0 A\to\pi_0 B$ is locally of finite presentation and its cotangent complex is perfect.
\end{warning}

\begin{mydef}
    Let $f : A\to B$ be a map of animated $k$-algebras. We say that $f$ is \emph{flat} (respectively \emph{faithfully flat}, \emph{\'etale}, \emph{smooth}) if the following conditions hold:
    \begin{enumerate}
        \item The morphism $\pi_0 f : \pi_0 A\to\pi_0 B$ of discrete commutative rings is flat (respectively faithfully flat, \'etale, smooth), and
        \item For all $i,$ the canonical morphism $\pi_i A\otimes_{\pi_0 A}\pi_0 B\to\pi_i B$ is an isomorphism.
    \end{enumerate}
\end{mydef}

\subsection{Modules, Animated and General}

The ``forgetful'' functor $\theta$ allows us to define modules over an animated ring in a particularly efficient manner. 

\begin{mydef}
    Let $A\in\aRing$ be an animated ring. The \emph{$\infty$-category of $A$-modules} is defined to be $\cMod_A := \cMod_{\theta(A)}(\Sp);$ i.e., the category of $\theta(A)$-module objects in the $\infty$-category $\Sp$ of spectra.
\end{mydef}

\begin{remark}
    The category $\cMod_A$ can be equivalently described as the stablization of the category of the animation of simplicial commutative $A_\bullet$-modules, where $A_\bullet$ is a presentation of $A$ by a simplicial commutative ring. In fact, $\cMod_A^{\cn}$ is the category of animated $A$-modules. Moreover, if $k$ is a discrete commutative ring, then $\cMod_k\simeq\cD(k)$ is equivalent to the derived $\infty$-category of $k$ as defined in \cite{HA}.
\end{remark}

\begin{mydef}
    Let $M\in\cMod_A^{\cn}.$
    \begin{enumerate}
        \item We say that $M$ is \emph{flat} if $\pi_0 M$ is a flat $\pi_0 A$-module and the natural map $\pi_i A\otimes_{\pi_0 A} \pi_0 M\to\pi_i M$ is an isomorphism of $\pi_0 A$-modules for all $i.$
        \item We say that $M$ is \emph{projective} if for all $N\in\cMod_A^{\cn},$ we have
        \[
        \pi_0\Map_{\cMod_A}(M,N[1])\cong\Hom_{\h\cMod_A}(M,N[1])\cong\Ext^1(M,N)\cong 0.
        \]
    \end{enumerate}
\end{mydef}

By \cite[7.2.2.14]{HA} and \cite[7.2.2.18]{HA}, a projective $A$-module $M$ is automatically flat, and a flat module $M$ is projective if and only if $\pi_0 M$ is a projective $\pi_0 A$-module.

\begin{mydef}
    Let $M$ be an $A$-module. We say $M$ is \emph{finite projective} if it is projective and $\pi_0 M$ is a finitely presented $\pi_0 A$-module.
\end{mydef}

We will make extensive use of the following construction and a number of related variants. Fix a base discrete commutative ring $k.$ There is an $\infty$-category $\MOD_k(\Sp)$ whose objects are pairs $(A,M)$ such that $A\in\CAlgcn_k$ and $M\in\cMod_A(\Sp),$ which comes equipped with a Cartesian fibration $\MOD_k(\Sp)\to\CAlgcn_k,$ given on objects by forgetting the $A$-module $M.$ A morphism $(A,M)\to (B,N)$ in $\MOD_k(\Sp)$ is equivalent to the data of a map $A\to B$ in $\CAlgcn_k$ and a map $M\otimes_A B\to N$ in $\cMod_B(\Sp).$ The precise construction of this Cartesian fibration can be found in \cite[Section 4.5]{HA}. Note that because the categories $\cMod_A(\Sp)$ are presentable for all $A,$ this fibration is both Cartesian and coCartesian.

Let $\MOD_k := \MOD_k(\Sp)\times_{\CAlgcn_k}\aRing_k,$ and let $\MOD_k^{\cn}$ be the full sub-$\infty$-category spanned by those $(A,M)$ such that $M$ is connective. We can realize this $\infty$-category as the animation of a $1$-category, which will be convenient for when we need to construct functors in or out of it.

\begin{proposition}[\protect{\cite[Proposition~25.2.1.2]{SAG}}]\label{SAG 25.2.1.2}
    The $\infty$-category $\MOD_k^\cn$ is the animation of $\mathsf{MOD}_k,$ the ordinary category of pairs $(A,M),$ where $A\in\CRing,$ $M\in\Mod_A,$ and a morphism $(A,M)\to (B,N)$ is given by a pair $(f,\phi),$ where $f : A\to B$ is a morphism of commutative rings, and $\phi : M\otimes_A B\to N$ is a morphism of $B$-modules.
\end{proposition}

\begin{cor}[\protect{\cite[Corollary~25.2.1.3]{SAG}}]
    Let $\mathsf{PolyFree}_k\subseteq\MOD_k^{\cn}$ be the full sub-$\infty$-category spanned by objects $(P,F)$ where $P\in\Poly_k$ and $F$ is a free $P$-module of finite rank. If $\cE$ is an $\infty$-category admitting sifted colimits, then restriction
    \[
    \Fun_{\mathrm{sift}}(\MOD_k^{\cn},\cE)\to\Fun(\mathsf{PolyFree}_k,\cE)
    \]
    is an equivalence of $\infty$-categories.
\end{cor}
\begin{remark}
    The proof of \cite[25.2.1.2]{SAG} has some broken references in the current version of \cite{SAG}. We refer to \cite[2.34]{yaylali2022notes} for indications as to what those references should be. Note that $\mathsf{PolyFree}_k$ is in fact a $1$-category.
\end{remark}

Let $A\in\aRing$ be arbitrary. We also define $\MOD_A := \MOD_{\ZZ}\times_{\aRing}\aRing_A$ and $\MOD_A^{\cn} := \MOD_{\ZZ}^{\cn}\times_{\aRing}\aRing_A.$ In particular, we can identify the objects of $\MOD_A$ (respectively $\MOD_A^{\cn}$) with pairs $(B,M)$ such that $B\in\aRing_A$ and $M\in\cMod_B$ (respectively $M\in\cMod_B^{\cn}$).

\section{Derived Stacks and Geometricity}

We now proceed with the global theory of derived stacks and the analogue of algebraic stacks, which are called geometric stacks.

\begin{mydef}
    The category $\dAff$ of \emph{affine derived schemes} is the opposite of the category of animated rings, $\aRing.$ If $k$ is a fixed animated ring, then we will write $\dAff_k := \left(\aRing_{k/}\right)^{\op}.$ For $A\in\aRing_{k/},$ we denote the corresponding object of $\dAff_k$ by $\Spec A.$
\end{mydef}

The $\infty$-categorical Yoneda embedding allows us to view $\dAff_k$ as a subcategory of $\PSh_{\left(\aRing_k\right)^{\op}} := \Fun(\aRing_k,\Ani),$ where $\Spec A$ is identified with $\Map_{\aRing_k}(A,-).$

\begin{mydef}\label{fpqc and etale topologies on aRing}
    Let $A$ be an animated $k$-algebra.
    \begin{enumerate}
        \item The \emph{fpqc topology} on $\dAff_k$ is defined by the property that a sieve $\cC\subseteq(\dAff_A)_{/\Spec B}\simeq\dAff_B$ is a covering sieve if and only if it contains a finite family $(\Spec B_i\to\Spec B)_{i\in I}$ such that the induced map $B\to \prod_{i\in I} B_i$ is faithfully flat.
        \item The \emph{\'etale topology} on $\dAff_k^{\et}$ (the full subcategory of $\Spec A\to\Spec k$ with $k\to A$ \'etale) is defined by the property that a sieve $\cC\subseteq(\dAff^{\et}_A)_{/\Spec B}\simeq\dAff^{\et}_B$ is a covering sieve if and only if it contains a finite family $(\Spec B_i\to\Spec B)_{i\in I}$ such that the induced map $B\to\prod_{i\in I} B_i$ is faithfully flat (hence \'etale).
    \end{enumerate}
\end{mydef}
\begin{proof}
    A map $f : A\to B$ of animated rings is faithfully flat (respectively faithfully flat and \'etale) if and only if $\theta(f)$ is so, and the collection of all such morphisms satisfies the properties of \cite[A.3.2.1]{SAG} if and only if its image under $\theta$ does, using the fact that $\theta$ admits both a left and right adjoint. The result then follows from \cite[B.6.1.3]{SAG}, \cite[B.6.1.7]{SAG}, and \cite[B.6.2.1]{SAG}.
\end{proof}

\begin{remark}
    For any $B\in\aRing,$ $\Spec B$ is an fpqc sheaf, by \cite[25.1.2.2]{SAG} and \cite[D.6.3.5]{SAG}.
\end{remark}

\begin{mydef}
    We say that a morphism $\Spec B\to\Spec A$ of affine derived schemes is \emph{flat} (respectively, \emph{smooth}, \emph{\'etale}, \emph{locally of finite presentation}) if the underlying morphism $A\to B$ in $\aRing$ is so.
\end{mydef}

Given a discrete ring $A,$ we will denote by $\Spec_{\cl} A$ its corresponding classical scheme: in particular, $\Spec_{\cl} A = \Hom_{\CRing}(A,-) : \CRing\to\Set$ is a set-valued functor, whereas $\Spec A$ is a functor valued in anima even if $A$ is discrete. In particular, $\Map_{\aRing}(A,B)$ need not be discrete for a general $B\in\aRing$ (although we do have $\Map_{\aRing}(B,A)\simeq\Hom_{\CRing}(\pi_0 B,A)$ for any $B\in\aRing,$ $A$ discrete by \ref{truncated animated rings} and the following remark).

With these topologies on $\dAff_k$ in hand, we can now define the geometric objects we will consider.

\begin{mydef}
    Let $k\in\aRing.$ A \emph{derived stack over $k$} is a sheaf on $\dAff_k^{\et}$ valued in $\Ani.$ The $\infty$-category of derived stacks over $k$ will be denoted $\dStk_k,$ and if $k = \ZZ,$ we will write $\dStk := \dStk_{\ZZ}.$
\end{mydef}

Analogously to the classical situation, general derived stacks are too unweildy, and we require derived stacks which are more closely connected to the world of ``animated algebra.'' The derived replacement of a cover is that of effective epimorphisms, which we now define.

\begin{mydef}
    Let $k\in\aRing.$ A map $f : \cX\to\cY$ in $\dStk_k$ is an \emph{effective epimorphism} if the natural map 
    \[
    \colim_{\Delta} \check{C}(\cX/\cY)_{\bullet}\to\cY
    \]
    is an equivalence in $\dStk_k$ after sheafification.
\end{mydef}

We now define the notion of a geometric derived stack, which is the derived analogue of a classical algebraic stack. Unlike in the classical theory, there are varying levels of geometricity, defined inductively.

\begin{mydef}
    Let $\cX$ be a derived stack, and let $f : \cX\to\cY$ be a morphism of derived stacks.
    \begin{enumerate}
        \item \begin{enumerate}
            \item We say $\cX$ is \emph{$(-1)$-geometric} or \emph{affine} if it is equivalent to an affine derived scheme.
            \item We say $f : \cX\to\cY$ is \emph{$(-1)$-geometric} or \emph{affine} if for all maps $\Spec A\to\cY,$ the fiber product $\cX\times_{\cY}\Spec A$ is affine.
            \item A $(-1)$-geometric morphism $f : \cX\to\cY$ of derived stacks is \emph{smooth} (respectively, \emph{\'etale}) if for all maps $\Spec A\to\cY,$ the map $\Spec B\simeq \cX\times_{\cY}\Spec A\to\Spec A$ corresponds to a smooth (respectively \'etale) morphism $A\to B$ of animated rings.
        \end{enumerate}
        \item Let $n\geq 0.$
        \begin{enumerate}
            \item A \emph{(smooth) $n$-atlas} of a $\cX$ is a family $\{\Spec A_i\to\cX\}_{i\in I}$ of morphisms of derived stacks such that
            \begin{enumerate}
                \item every $\Spec A_i\to\cX$ is $(n-1)$-geometric and smooth, and
                \item the induced morphism $\coprod\Spec A_i\to\cX$ is an effective epimorphism.
            \end{enumerate}
            If each morphism $\Spec A_i\to\cX$ is \'etale, we call the $n$-atlas an \emph{\'etale $n$-atlas}.
            \item We say $\cX$ is \emph{$n$-geometric} (respectively, \emph{$n$-DM}) if
            \begin{enumerate}
                \item it has an $n$-atlas (respectively, \'etale $n$-atlas), and
                \item the diagonal $\cX\to\cX\times\cX$ is $(n-1)$-geometric.
            \end{enumerate}
            \item We say $f : \cX\to\cY$ is \emph{$n$-geometric} (respectively, \emph{$n$-DM}) if for all maps $\Spec A\to\cY,$ the base change $\cX\times_{\cY}\Spec A$ is $n$-geometric (respectively, $n$-DM).
        \end{enumerate}
        \item Let $n\geq 0.$
        \begin{enumerate}
            \item An $n$-geometric morphism $f : \cX\to\cY$ is \emph{smooth}, if for all maps $\Spec A\to\cY,$ $\cX\times_{\cY}\Spec A$ has an $n$-atlas given by a family of affine derived schemes $\{\Spec A_i\to\cX\times_{\cY}\Spec A\}$ such that the induced map $A\to A_i$ is smooth for all $i.$
            \item An $n$-DM morphism $f : \cX\to\cY$ is \emph{\'etale} if for all maps $\Spec A\to\cY,$ $\cX\times_{\cY}\Spec A$ has an \'etale $n$-atlas given by a family of affine derived schemes $\{\Spec A_i\to\cX\times_{\cY}\Spec A\}$ such that the induced map $A\to A_i$ is \'etale for all $i.$
            \item We say $f : \cX\to\cY$ is \emph{geometric} (respectively, \emph{DM}) if it is $n$-geometric (respectively, $n$-DM) for some $n\geq -1.$
        \end{enumerate}
    \end{enumerate}
\end{mydef}

The following is a collection of basic properties of geometric stacks and geometric morphisms. Most of these are easily shown by induction or follow immediately from the definitions, but we refer the reader to \cite[Proposition 1.3.3.3]{toen2006homotopical}, \cite[Lemma 4.25]{AG2014}, and \cite[Lemma 4.30]{AG2014} for the nontrivial arguments.

\begin{lemma}\label{geometricity facts}
    Let $f : \cX\to\cY$ and $g : \cY\to\cZ$ be morphisms of derived stacks. Then:
    \begin{enumerate}
        \item If $g\circ f$ is $n$-geometric, so is $\cX\times_{\cZ}\cY\to\cY.$
        \item $f$ is $n$-geometric if and only if $\cX\times_{\cY}\Spec A\to\Spec A$ is $n$-geometric for all $\Spec A\to\cY.$
        \item If $f$ and $g$ are $n$-geometric, then so is $g\circ f.$
        \item If $\cX$ and $\cY$ are $n$-geometric, then $f$ is $n$-geometric.
        \item If $g\circ f$ is $n$-geometric and $g$ is $(n+1)$-geometric, then $f$ is $n$-geometric.
        \item If $\cX$ is $n$-geometric and $\cY\to\cY\times\cY$ is $n$-geometric, then $f$ is $n$-geometric. If moreover $f$ is a smooth effective epimorphism, then $\cY$ is $(n+1)$-geometric.\footnote{Compare this to the result that any map from a classical algebraic space to a classical algebraic stack is representable (by classical algebraic spaces).}
        \item If $\cX$ and $\cY$ are $n$-geometric, then $\cX\to\cX\times_{\cY}\cX$ is $(n-1)$-geometric.
    \end{enumerate}
\end{lemma}

We also take the time to prove a few basic lemmas which allow us to detect geometricity by examining one atlas, instead of considering all possible morphisms from affines.

\begin{lemma}\label{geometricity detected by single atlas}
    Let $f : \cX\to\cY$ be a map of $n$-geometric derived stacks. Then $f$ is $(n-1)$-geometric if and only if there exists an $n$-atlas $\{\Spec A_i\to\cY\}$ such that $\cX_{A_i} := \cX\times_{\cY}\Spec A_i$ is $(n-1)$-geometric for all $i.$
\end{lemma}
\begin{proof}
    The only if direction is trivial. On the other hand, suppose that there exists an $n$-atlas $\{\Spec A_i\to\cY\}_{i\in I}$ such that $\cX_i := \cX\times_{\cY}\Spec A_i$ is $(n-1)$-geometric for all $i.$ Let $\Spec B\to\cY$ be an arbitrary map of derived stacks. We must show that $\cX_B := \cX\times_{\cY}\Spec B$ is $(n-1)$-geometric. Choose an \'etale atlas $\{\Spec B_j\to\Spec B\}_{j\in J}$ such that for all $j,$ there exists $i_j$ with the composition $\Spec B_j\to\Spec B\to\cY$ factoring through $\Spec A_{i_j}\to\cY.$

    Let $\cX_j := \cX\times_{\cY}\Spec B_j,$ and note that we have the following commutative diagram where every square is a pullback
    \[\begin{tikzcd}
        \cX_j\ar[rr]\ar[dd]\ar[rd] &&\Spec B_j\ar[dd]\ar[rd]\\
        & \cX_{i_j}\ar[dd]\ar[rr] &&\Spec A_{i_j}\ar[dd]\\
        \cX_B\ar[rr]\ar[rd] &&\Spec B\ar[rd]\\
        &\cX\ar[rr] &&\cY
    \end{tikzcd}\]
    Thus, since $\cX_j\simeq\cX_{i_j}\times_{A_{i_j}}\Spec B_j,$ we see that $\cX$ is $(n-1)$-geometric.

    So, take a smooth $(n-1)$-atlas $\{\Spec C_{(j,k)}\to\cX_j\}_{k\in K_j}$ of $\cX_j.$ It follows that
    \[\coprod_{j\in J}\{\Spec C_{(j,k)}\to\cX_j\to\cX_B\}_{k\in K_j}\]
    is an $(n-1)$-atlas of $\cX_B,$ so that $\cX_B$ is $(n-1)$-geometric, as desired.
\end{proof}

\begin{cor}
    Let $f : \cX\to\cY$ be a map of $n$-geometric derived stacks. Suppose that there exists an $(n-1)$-geometric stack $Y$ and a smooth effective epimorphism $Y\to\cY$ such that $X := \cX\times_{\cY} Y$ is $(n-1)$-geometric. Then $f$ is $(n-1)$-geometric.
\end{cor}
\begin{proof}
    Choose an atlas $\{\Spec A_i\to Y\}$ of $Y.$ Then $\{\Spec A_i\to Y\to\cY\}$ is an atlas of $\cY$ such that $\cX\times_{\cY}\Spec A_i$ is $(n-1)$-geometric for all $i;$ then the result follows from \ref{geometricity detected by single atlas}.
\end{proof}

\begin{mydef}
    Let $\cP$ be a property of (morphisms of) affine derived schemes that is stable under equivalences\footnote{in the sense that any equivalence has property $\cP$}, pullbacks, compositions, and is smooth-local on the source and target, in the sense that the following are satisfied:
    \begin{enumerate}
        \item If $f : \Spec B\to\Spec A$ is a map of affine derived schemes, and $\{p_i : \Spec B_i\to\Spec B\}$ is any smooth cover (i.e., each $p_i$ is smooth and $\coprod \Spec B_i\to\Spec B$ is an effective epimorphism), then $f$ has property $\cP$ if and only if\footnote{In \cite{toen2006homotopical}, only the if direction is assumed. We assume both directions, which aligns more with classical algebraic geometry, and we do not see any issues in adding the additional condition with any of the developed theory.} each composition $f\circ p_i : \Spec B_i\to\Spec A$ has property $\cP$.
        \item If $f : \Spec B\to\Spec A$ is a map of affine derived schemes, and $\{\Spec A_i\to\Spec A\}$ is a smooth cover, then $f$ has property $\cP$ if and only if all the projections $\Spec B\times_{\Spec A}\Spec A_i\to\Spec A_i$ have property $\cP.$
        \item $\cP$ is stable under smooth postcomposition.\footnote{Note that stability under postcomposition is not explicitly assumed in \cite{toen2006homotopical}, but is assumed in the stacks project: see \cite[\href{https://stacks.math.columbia.edu/tag/06FB}{Tag 06FB}]{stacks-project}.}
    \end{enumerate}
    Such a property is called a \emph{good property.}
\end{mydef}

Examples of good properties include ``locally of finite presentation,'' ``flat,'' and of course, ``smooth.''

We may generalize the notion of a morphism with a good property to geometric morphisms of derived stacks as follows.

\begin{mydef}\label{local properties}
    Let $\cP$ be a good property of morphisms of affine derived schemes. Then a morphism $f : \cX\to\cY$ of derived stacks \emph{has property $\cP$} if it is $n$-geometric (for some $n$) and for an affine $(n-1)$-atlas $U_i$ of the pullback along an affine derived scheme $\Spec B$ the corresponding morphism $U_i\to\Spec B$ has property $\cP.$
\end{mydef}

\begin{remark}
    Note that if $f : \cX\to\cY$ is a map of affine derived schemes and $\cP$ is a good property of morphisms of affine derived schemes, then $f$ has property $\cP$ (in the sense of the above definition) if and only if $f$ has property $\cP$ when $\cP$ is considered as a property of morphisms of affine derived schemes.
\end{remark}

\begin{lemma}\label{refinement through atlas}
    Suppose that $\cX$ is an $n$-geometric derived stack, with given $n$-atlas $\{\Spec A_i\to\cX\}_{i\in I},$ and let $\Spec B\to\cX$ be arbitrary. Then there exists a smooth collection of morphisms $\{\Spec B_j\to\Spec B\}$ such that
    \begin{enumerate}
        \item $\coprod_j\Spec B_j\to\Spec B$ is an effective epimorphism, and 
        \item for all $j,$ there exists $i_j\in I$ and a morphism $\Spec B_j\to\Spec A_{i_j}$ such that
        \[\begin{tikzcd}
            \Spec B_j\ar[r]\ar[d] &\Spec B\ar[d]\\
            \Spec A_{i_j}\ar[r] &\cX
        \end{tikzcd}\]
        commutes.
    \end{enumerate}
\end{lemma}
\begin{proof}
    For all $i,$ let $U_i = \Spec A_i\times_{\cX}\Spec B.$ Since $\{\Spec A_i\to\cX\}_{i\in I}$ is an $n$-atlas of $\cX,$ each $U_i$ is $(n-1)$-geometric. We can then choose an $(n-1)$-atlas $\{\Spec B_{(i,j)}\to U_i\}_{j\in J_i}.$ It follows that we obtain a commutative diagram
    \[\begin{tikzcd}
        \Spec B_{(i,j)}\ar[r]\ar[d] &\Spec B\ar[d]\\
        \Spec A_i\ar[r] &\cX
    \end{tikzcd}\]
    and that $\{\Spec B_{(i,j)}\to\Spec B\}_{i\in I, j\in J_i}$ is a collection of smooth morphisms satisfying the desired properties.
\end{proof}

\begin{lemma}\label{stability under smooth precomposition}
    Let $\cP$ be a good property of affine derived schemes, and suppose that $f : \cX\to\cY$ is an $n$-geometric morphism of derived stacks which has property $\cP.$ If $g : \cX'\to\cX$ is any smooth $n$-geometric morphism of derived stacks, then $f\circ g$ has property $\cP.$
\end{lemma}
\begin{proof}
    We must show that for any $\Spec A\to\cY,$ there exists an $n$-atlas $\{\Spec C_j\to\cX_A'\}$ of $\cX_A' := \cX'\times_{\cY}\Spec A$ such that the compositions $\Spec C_j\to\Spec A$ all have property $\cP.$ As such, we may assume that $\cY\simeq\Spec A$ is affine. Let $\{\Spec B_i\to\cX\}$ be an $n$-atlas such that the compositions $\Spec B_i\to\Spec A$ have property $\cP.$

    Let $\{\Spec C_k\to\cX'\}$ be an $n$-atlas of $\cX',$ and consider the pullback $U_{(i,j)} := \Spec C_j\times_{\cX}\Spec B_i.$ If $\{\Spec C'_{(i,j,k)}\to U_{(i,j)}\}$ is an $n$-atlas of $U_{(i,j)},$ then the compositions $\{\Spec C'_{(i,j,k)}\to\Spec C_j\}$ form a smooth cover as well. If we can then show that $\Spec C'_{(i,j,k)}\to\Spec A$ has property $\cP$ for any tuple $(i,j,k),$ it will follow by locality on the source that $\Spec C_j\to\Spec A$ also has property $\cP.$ However, this follows from the fact that the maps $\Spec C'_{(i,j,k)}\to\Spec B_i$ are all smooth, so that we may find a smooth cover $\{\Spec D_\alpha\to\Spec B_i\}$ which includes the morphism $\Spec C'_{(i,j,k)}\to\Spec B_i,$ and the fact that all $\Spec B_i\to\Spec A$ have property $\cP$ so that all the compositions $\Spec D_\alpha\to\Spec A$ have property $\cP$ as well.
\end{proof}

\begin{cor}\label{P for one atlas implies P for any atlas}
    Let $\cP$ be a good property of affine derived schemes, and suppose that $\cX\to\Spec A$ is a morphism of $n$-geometric derived stacks which has property $\cP.$ Then for any $n$-atlas $\{\Spec B_j\to\cX\},$ the compositions $\Spec B_j\to\Spec A$ have property $\cP.$
\end{cor}

\begin{proposition}[\protect{\cite[Proposition~1.3.6.3]{toen2006homotopical}}]
    Let $\cP$ be a good property of morphisms of affine derived schemes. Then:
    \begin{enumerate}
        \item as a property of morphisms of derived stacks, $\cP$ is stable under equivalences, pullbacks, and compositions.
        \item if $f : \cX\to\cY$ is a morphism of $n$-geometric derived stacks, $f$ has property $\cP$ if and only if there exists an $n$-atlas $\{\Spec A_i\to\cY\}$ such that $\cX\times_{\cY}\Spec A_i\to\Spec A_i$ has property $\cP$ for all $i.$
    \end{enumerate}
\end{proposition}
\begin{proof}
    We provide a proof for the convenience of the reader, as this is claimed in \cite{toen2006homotopical}, but not proven, and we did not manage to find a proof of the second point which does not use our slightly strengthened version of what smooth local on the source-and-target means.
    \begin{enumerate}
        \item First, let $f : \cX\to\cY$ be an equivalence of derived stacks, and let $\Spec A\to\cY$ be arbitrary. The pullback of an equivalence is an equivalence, so we obtain an equivalence $\cX\times_{\cY}\Spec A\simeq \Spec A'\to\Spec A.$ Since $\cP$ is stable under equivalences for affine derived schemes, $\Spec A'\to\Spec A$ has property $\cP,$ and thus, so does $f.$
        
        Next, suppose $f : \cX\to\cY$ is a morphism of derived stacks which has property $\cP,$ and suppose that
        \[\begin{tikzcd}
            \cX'\ar[d]\ar[r,"f'"] &\cY'\ar[d]\\
            \cX\ar[r,"f"] &\cY
        \end{tikzcd}\]
        is a Cartesian diagram of derived stacks. We wish to show that $f'$ has property $\cP.$ This is immediate: given $\Spec A\to\cY',$ $\cX'\times_{\cY'}\Spec A\simeq \cX\times_{\cY}\Spec A,$ so the result follows. Thus a good property is stable under pullbacks.

        Finally, suppose we are given $f : \cX\to\cY$ and $g : \cY\to\cZ$ morphisms of derived stacks such that both $f$ and $g$ have property $\cP.$ We wish to show that $g\circ f$ has property $\cP.$ Let $\Spec A\to\cZ$ be any morphism, and set $\cY_A := \cY\times_{\cZ}\Spec A$ and $\cX_A := \cX\times_{\cZ}\Spec A.$ Since $g$ has property $\cP,$ we may find an atlas $\{\Spec B_i\to\cY_A\}$ such that the compositions $\Spec B_i\to\Spec A$ all have property $\cP.$ If we then set $\cX_{B_i} := \cX_A\times_{\cY_A}\Spec B_i\simeq\cX\times_{\cY}\Spec B_i,$ we may then find an atlas $\{\Spec C_{(i,j)}\to\cX_{B_i}\}$ of $\cX_{B_i}$ such that the compositions $\Spec C_{(i,j)}\to\Spec B_i$ all have property $\cP,$ as $f$ has property $\cP.$ The desired result then follows from the assumption that $\cP$ is stable under compositions of morphisms of affine derived schemes.
        \item Let $f : \cX\to\cY$ be a morphism of $n$-geometric derived stacks (hence $f$ is $n$-geometric as well). The ``only if'' direction is trivial, so suppose conversely that there exists an $n$-atlas $\{\Spec A_i\to\cY\}$ such that $\cX_i := \cX\times_{\cY}\Spec A_i\to\Spec A_i$ has property $\cP$ for all $i.$ Note that it is enough to suppose the existence of an atlas $\{U_{(i,i')}\to\cX_i\}$ such that the compositions $U_{(i,i')}\to\Spec A_i$ all have property $\cP,$ since given any $\Spec B\to\Spec A_i,$ we can simply base change $U_{(i,i')}\times_{\Spec A_i}\Spec B$ to obtain an atlas for $\cX_i\times_{\Spec A_i}\Spec B$ with the desired property.

        We must show that for any $\Spec B\to\cY,$ there exists an $n$-atlas $\{\Spec C_k\to\cX\times_{\cY}\Spec B\}$ such that the compositions $\Spec C_k\to\Spec B$ all have property $\cP.$ Let $U_i' := \Spec A_i\times_{\cY}\Spec B,$ and let $\cX_i' := \cX_i\times_{\Spec A_i} U_i \simeq \cX_i\times_{\cY} \Spec B.$ Then because $\cP$ is stable under pullbacks, $\cX_i'\to U_i'$ has property $\cP,$ as it is the base change of $\cX_i\to\Spec A_i.$ Choose an $n$-atlas $\{\Spec D_{(i,\alpha)}\to U_i\}$ of $U_i,$ and let $\cX'_{(i,\alpha)} := \cX_i'\times_{U_i}\Spec D_{(i,\alpha)}.$ Then once again the maps $\cX'_{(i,\alpha)}\to \Spec D_{(i,\alpha)}$ have property $\cP$ by base change.
        
        Finally, if $\{\Spec C_k\to\cX_B\}$ is any $n$-atlas, consider the collection of morphisms $\{\Spec C_k\otimes_B D_{(i,\alpha)}\to\cX'_{(i,\alpha)}\}.$ This collection forms an $n$-atlas of $\cX'_{(i,\alpha)}$ since the $\Spec C_k$ form an $n$-atlas of $\cX_B,$ \ref{P for one atlas implies P for any atlas} implies that the maps $\Spec C_k\otimes_B D_{(i,\alpha)}\to\Spec D_{(i,\alpha)}$ all have property $\cP.$ A diagram of all the derived stacks would now look like the following:
        \[
        \begin{tikzcd}[column sep=1ex,row sep=1ex]
        {\Spec C_k\otimes_B D_{(i,\alpha)}} \arrow[rrdd, two heads] \arrow[rr, two heads] &  & {\cX'_{(i,\alpha)}} \arrow[rd, two heads] \arrow[rr] &                                                    & {\Spec D_{(i,\alpha)}} \arrow[rd, two heads] &                                       &                    \\
                                                                       &  &                                                      & \cX_i' \arrow[dd] \arrow[rr] \arrow[rd, two heads] &                                              & U_i' \arrow[dd] \arrow[rd, two heads] &                    \\
                                                                       &  & \Spec C_k \arrow[rr, two heads]                      &                                                    & \cX_B \arrow[dd] \arrow[rr]                  &                                       & \Spec B \arrow[dd] \\
                                                                       &  &                                                      & \cX_i \arrow[rd, two heads] \arrow[rr]             &                                              & \Spec A_i \arrow[rd, two heads]       &                    \\
                                                                       &  &                                                      &                                                    & \cX \arrow[rr]                               &                                       & \cY               
        \end{tikzcd}\]
        where the double headed arrows indicate that the map in question participates in an $n$-atlas or smooth cover.
        
        Then the fact that $\cP$ is local on the target for morphisms of affine derived schemes implies that the maps $\Spec C_k\to\Spec B$ all have property $\cP,$ so the desired result follows.
    \end{enumerate}
\end{proof}

It is also useful to be able to detect properties of a map of geometric stacks via a sufficiently nice map between ``covers'' of those stacks.

\begin{mydef}
     Let $f : \cX\to\cY$ be a morphism of $n$-geometric derived stacks, with $n\geq 0.$ A \emph{chart for $f$} is a commutative diagram
     \[\begin{tikzcd}
         X\ar[rr,bend left, "h"]\ar[r,"g"]\ar[rd] &\cX'\ar[r,"f'"]\ar[d] &Y\ar[d,"p"]\\
         &\cX\ar[r,"f"] &\cY
     \end{tikzcd}\]
     such that $X$ and $Y$ are disjoint unions of $(n-1)$-geometric derived stacks\footnote{The ``disjoint union'' part of this condition is an artifact of our definition of $(-1)$-geometric derived stacks, and this is one instance where the conventions of \cite{AG2014} are slightly more convenient. If $n\neq 0,$ a disjoint union of $(n-1)$-geometric derived stacks is again $(n-1)$-geometric. However, if $n = 0,$ a disjoint union of $(-1)$-geometric stacks need not be $(-1)$-geometric, as an arbitrary disjoint union of affines need not be affine.}, the square is Cartesian, and $g$ and $p$ are smooth effective epimorphisms.
\end{mydef}

The following lemma is a standard sort of argument, but we provide the details for the convenience of the reader and to reassure ourselves that properties of morphisms of geometric derived stacks behave as one would expect from experience of classical algebraic geometry.

\begin{lemma}
    Suppose that $\cP$ is a good property of affine derived schemes, and let $f : \cX\to\cY$ be a morphism of $n$-geometric derived stacks, with $n\geq 0.$ Consider diagrams of the form
    \begin{equation}\label{square-chart}\begin{tikzcd}
        U\ar[r,"h"]\ar[d] &V\ar[d]\\
        \cX\ar[r,"f"] &\cY
    \end{tikzcd}\end{equation}
    where $U$ and $V$ are disjoint unions of $(n-1)$-geometric stacks and the vertical morphisms are smooth. Then the following are equivalent:
    \begin{enumerate}
        \item[1.] For all maps $\Spec A\to\cY,$ and for any $n$-atlas $\{\Spec B_i\to \cX\times_{\cY}\Spec A\},$ the compositions $\Spec B_i\to\Spec A$ have property $\cP$ for all $i.$
        \item[1'.] For all maps $\Spec A\to\cY,$ there exists an $n$-atlas $\{\Spec B_i\to \cX\times_{\cY}\Spec A\}$ such that the compositions $\Spec B_i\to\Spec A$ have property $\cP$ for all $i$ (i.e., $f$ has property $\cP$).
        \item[2.] For all $n$-atlases $\{\Spec A_i\to\cY\}$ of $\cY$ and for all families of $n$-atlases $\{\Spec B_{(i,j)}\to\cX\times_{\cY}\Spec A_i\}_{j\in J_i}$ of the $\cX\times_{\cY}\Spec A_i,$ the compositions $\Spec B_{(i,j)}\to\Spec A_i$ have property $\cP$ for all $i, j\in J_i.$
        \item[2'.] For all $n$-atlases $\{\Spec A_i\to\cY\}$ of $\cY$ there exists a family of $n$-atlases $\{\Spec B_{(i,j)}\to\cX\times_{\cY}\Spec A_i\}_{j\in J_i}$ of $\cX\times_{\cY}\Spec A_i,$ such that the compositions $\Spec B_{(i,j)}\to\Spec A_i$ have property $\cP$ for all $i, j\in J_i.$
        \item[2''.] There exists an $n$-atlas $\{\Spec A_i\to\cY\}$ of $\cY$ and for each $i$ an $n$-atlas $\{\Spec B_{(i,j)}\to\cX\times_{\cY}\Spec A_i\}_{j\in J_i}$ of $\cX\times_{\cY}\Spec A_i,$ such that all the compositions $\Spec B_{(i,j)}\to\Spec A_i$ have property $\cP.$
        \item[3.] For all diagrams as in \ref{square-chart} such that in addition $U\to\cX\times_{\cY} V$ is smooth, the morphism $h$ has property $\cP$ (in the sense of 1').
        \item[3'.] For some diagram as in \ref{square-chart} with $U\to\cX$ an effective epimorphism, the morphism $h$ has property $\cP$ (in the sense of 1').
        \item[4.] For all diagrams as in \ref{square-chart} such that in addition $U\to\cX\times_{\cY} V$ is smooth and $U$ and $V$ are disjoint unions of affine derived schemes, the morphism $h$ has property $\cP$ (in the sense of 1').
        \item[4'.] For some diagram as in \ref{square-chart} with $U$ and $V$ disjoint unions of affine derived schemes and $U\to\cX$ an effective epimorphism, the morphism $h$ has property $\cP$ (in the sense of 1').
    \end{enumerate}
\end{lemma}
\begin{proof}
    Trivially, the following implications hold:
    \[\begin{tikzcd}
        (1.) \arrow[d, Rightarrow] \arrow[r, Rightarrow] & (2.) \arrow[d, Rightarrow]  & (4.) \arrow[d, Rightarrow] & (3.) \arrow[d, Rightarrow] \arrow[l, Rightarrow] \\
        (1'.) \arrow[r, Rightarrow]                      & (2'.) \arrow[d, Rightarrow] & (4'.) \arrow[r,Rightarrow] & (3'.)\\
                                                 & (2''.)                      &                                                  &                                                 
    \end{tikzcd}\]
    \begin{enumerate}
        \item ($(3'.)\implies (3.)$) This is analogous to (part of) \cite[\href{https://stacks.math.columbia.edu/tag/06FM}{Tag 06FM}]{stacks-project}. Our proof is mostly identical with only a few adjustments, reproduced here for the convenience of the reader.

        Without loss of generality, we may assume that $V\to\cY$ is an effective epimorphism. Indeed, we can take some $(n-1)$-atlas $\{\Spec A_i\to V\}_{i\in I}$ of $V$ and consider the collection of maps $\{\Spec A_i\to V\to\cY\}_{i\in I}.$ We may then extend the collection of smooth $(n-1)$-geometric maps $\{\Spec A_i\to\cY\}_{i\in I}$ to an $n$-atlas $\{\Spec \tilde{A}_i\to\cY\}_{i\in \tilde{I}}$ of $\cY$ such that $I\subseteq\tilde{I}$ and if $i\in I\subseteq\tilde{I},$ then $\tilde{A}_i = A_i.$ Letting $\tilde{V} = V\coprod \left(\coprod_{i\in \tilde{I}\setminus I}\Spec\tilde{A}_i\right)$ we obtain a smooth morphism $V\to\tilde{V},$ and since $\cP$ is a good property and thus closed under smooth postcomposition, the composition $U\to\tilde{V}$ also satisfies $\cP.$ So going forward, assume that $V\to\cY$ is an effective epimorphism.
        
        Suppose that
        \[\begin{tikzcd}
            U\ar[r,"h"]\ar[d] &V\ar[d]\\
            \cX\ar[r,"f"] &\cY
        \end{tikzcd}\]
        is a diagram as in $(4'.),$ and let
        \[\begin{tikzcd}
            U'\ar[r,"h'"]\ar[d] &V'\ar[d]\\
            \cX\ar[r,"f"] &\cY
        \end{tikzcd}\]
        be a diagram as in $(3.).$ We must show that $h'$ has property $\cP.$ Consider the commutative diagram
        \[\begin{tikzcd}
            U\ar[dd,"h"] &\ar[l] U\times_{\cX} U'\ar[d]\ar[dd, bend left=49, "{(h,h')}"]\ar[r] & U'\ar[dd,"h'"]\\
            &\ar[lu] U\times_{\cY} V'\ar[d]\\
            V &\ar[l] V\times_{\cY} V'\ar[r] &V'
        \end{tikzcd}\]
        and note that all geometric stacks appearing in the diagram are $(n-1)$-geometric by \ref{geometricity facts}. The horizontal arrows are smooth by base change. The diagram 
        \[\begin{tikzcd}
            U\times_{\cX} U'\ar[r]\ar[d] &U'\ar[d]\\
            U\times_{\cY} V'\ar[r] &\cX\times_{\cY} V'
        \end{tikzcd}\]
        is Cartesian, so the left vertical arrow is smooth because $U'\to\cX\times_{\cY} V'$ is by assumption. Because $\cP$ is preserved under base change, it follows that $U\times_{\cY} V'\simeq U\times_{V} (V\times_{\cY} V')\to V\times_{\cY} V'$ has property $\cP,$ and thus $(h,h')$ has $\cP$ since $\cP$ is stable under smooth precomposition by \ref{stability under smooth precomposition}.

        Finally, consider the commutative cube
        \[\begin{tikzcd}
            Z\ar[dd]\ar[rd]\ar[rr] &&W\ar[dd]\ar[rd]\\
            &U\times_{\cX} U'\ar[dd]\ar[rr] && V\times_{\cY} V'\ar[dd]\\
            U_A\ar[rr]\ar[rd] &&\Spec A\ar[rd]\\
            &U'\ar[rr] &&V'
        \end{tikzcd}\]
        in which all faces are pullbacks. All vertical arrows are smooth effective epimorphisms, and since $U\times_{\cX} U'\to V\times_{\cY}V'$ has property $\cP,$ so does $Z\to W.$ 

        Choose $(n-1)$-atlases $\{\Spec B_i\to U_A\}$ and $\{\Spec C_j\to W\}.$ Considering the diagram
        \[\begin{tikzcd}
            \Spec B_i\otimes_A C_j\ar[r]\ar[d] &Z\times_W\Spec C_j\ar[r]\ar[d] &\Spec C_j\ar[d]\\
            \Spec B_i\times_{U_A} Z\ar[r]\ar[d] & Z\ar[r]\ar[d] &W\ar[d]\\
            \Spec B_i\ar[r] &U_A\ar[r] &\Spec A
        \end{tikzcd}\]
        such that all squares are pullbacks, we find that $\Spec B_i\otimes_A C_j\to\Spec C_j$ has property $\cP$ because $Z\times_W \Spec C_j\to\Spec C_j$ does (by base change) and because of \ref{stability under smooth precomposition}. Then because $W\to\Spec A$ is an effective epimorphism and $\{\Spec C_j\to W\}$ is an $(n-1)$-atlas, the collection of compositions $\{\Spec C_j\to W\to\Spec A\}$ form a smooth cover, and similarly for $\{\Spec B_i\otimes_A C_j\to\Spec B_i\times_{U_A}Z\to\Spec B_i\}$ for any $i.$ Thus, because $\cP$ is smooth local on the target for maps of derived affine schemes, it follows that $\Spec B_i\to\Spec A$ has property $\cP$ for any $i,$ which proves that $h'$ has property $\cP,$ as desired.
        
        \item ($(3.)\implies (2.)$) Let $\{\Spec A_i\to\cY\}$ be any $n$-atlas of $\cY,$ and for any $i,$ let $\{\Spec B_{(i,j)}\to \cX_i\}_{j\in J_i}$ be any $n$-atlas of $\cX_i := \cX\times_{\cY}\Spec A_i.$ Set $V = \coprod_i\Spec A_i$ with induced morphism $V\to\cY,$ $U_i = \coprod_{j\in J_i}\Spec B_{(i,j)}$ with induced morphism $U_i\to \cX_i,$ and $U = \coprod_i U_i$ with induced morphism $U\to \cX\times_{\cY} V\simeq\coprod_i\cX_i.$

        By $(3.),$ we then have that $U\to V$ has property $\cP,$ and thus so does any pullback $U_i\to\Spec A_i.$ The result then follows from \ref{P for one atlas implies P for any atlas}.
        \item ($(2''.)\implies (1.)$) Suppose that we are given an $n$-atlas $\{\Spec A_i\to\cY\}$ of $\cY$ and for each $i$ an $n$-atlas $\{\Spec B_{(i,j)}\to\cX_i\}$ of $\cX_i := \cX\times_{\cY}\Spec A_i$ such that all compositions $\Spec B_{(i,j)}\to\Spec A_i$ have property $\cP.$ Additionally, let $\Spec C\to\cY$ be arbitrary, and let $\{\Spec D_k\to\cX_C\}$ be an arbitrary $n$-atlas of $\cX_C := \cX\times_{\cY}\Spec C.$ Consider the commutative diagram
        
        \[\begin{tikzcd}
            \cZ''_{(i,j,k)}\ar[r]\ar[dddd] &\cZ'_{(i,j)}\ar[dd]\ar[rd]\\
            &&\Spec B_{(i,j)}\ar[dd]\\
            &\cZ_i\ar[dd]\ar[rd]\ar[rr] &&\cW_i\ar[dd]\ar[rd]\\
            &&\cX_i\ar[dd]\ar[rr]&&\Spec A_i\ar[dd]\\
            \Spec D_k\ar[r] &\cX_C\ar[rd]\ar[rr] &&\Spec C\ar[rd]\\
            &&\cX\ar[rr] &&\cY
        \end{tikzcd}\]
        where all squares are Cartesian. Then it follows that for any tuple $(i,j),$ the map $\cZ'_{(i,j)}\to\cW_i$ has property $\cP,$ as $\Spec B_{(i,j)}\to\Spec A_i$ was assumed to have $\cP.$ Because $\Spec D_k\to\cX_C$ is smooth, so is $\cZ''_{(i,j,k)}\to\cZ'_{(i,j)},$  and thus $\cZ''_{(i,j,k)}\to\cW_i$ has property $\cP$ as $\cP$ is stable under smooth precomposition by \ref{stability under smooth precomposition}. Arguing as we did at the end of the proof of $(3'.)\implies (3.),$ it follows that $\Spec D_k\to\Spec C$ has property $\cP$ for any $k.$
        \item ($(1.)\implies (3'.)$) Finally, assume $(1.)$ holds, and choose an $n$-atlas $\{\Spec A_i\to\cY\}$ of $\cY$ and for each $i$ an $n$-atlas $\{\Spec B_{(i,j)}\to\cX_i\}_{j\in J_i}$ of $\cX_i := \cX\times_{\cY}\Spec A_i$ such that the composite morphisms $\Spec B_{(i,j)}\to\Spec A_i$ all have property $\cP.$ If we set $U = \coprod_i \coprod_{j\in J_i}\Spec B_{(i,j)}$ and $V = \coprod_i\Spec A_i$ with $h : U\to V$ the induced map. Then $U$ and $V$ are $0$-geometric, and the map $h$ clearly has property $\cP,$ as we can verify $(2''.)$ using the natural $0$-atlas $\{\Spec A_i\to V\}$ and the maps $\Spec B_{(i,j)}\to\Spec A_i$ which have $\cP.$
    \end{enumerate}
\end{proof}

Properties of morphisms of affine derived schemes which satisfy these conditions include ``locally of finite presentation,''  ``flat,'' and ``smooth.'' However, ``\'etale'' does not satisfy the conditions, since it is not smooth-local on the base. We now recall important properties of morphisms of derived stacks which are not defined directly from good properties of affine derived schemes.

\begin{mydef}
    A morphism of derived stacks is \emph{\'etale} if it is DM and \'etale.
\end{mydef}

\begin{mydef}
    A morphism of derived stacks $f : \cX\to\cY$ is
    \begin{enumerate}
        \item an \emph{open immersion} if it is flat (as in \ref{local properties}), locally of finite presentation, and a monomorphism in the sense of \cite{HA} (this is equivalent to $\cX\to\cX\times_{\cY}\cX$ being an equivalence).
        \item a \emph{closed immersion} if it is affine and for any $\Spec B\to\cY,$ the corresponding morphism $\cX\times_{\cY}\Spec B\simeq\Spec C\to\Spec B$ induces a surjection $\pi_0 B\to\pi_0 C$ of discrete rings.
    \end{enumerate}
\end{mydef}

\begin{remark}
    \begin{enumerate}
        \item Open immersions are local on the target in the fpqc topology. One checks that if $A\to B$ is a map of animated rings such that $\Spec B\to\Spec A$ is an open immersion, and $A\to A'$ is faithfully flat, then the pushout $A'\to A'\otimes_A B$ is also flat, locally of finite presentation, and an epimorphism in $\aRing$ (a monomorphism $\Spec A'\otimes_A B\to\Spec A'.$
        \item Closed immersions are local on the target in the fpqc topology. Indeed, if $A\to B$ is a map of animated rings and $A\to A'$ is faithfully flat, then $\pi_0 A\to\pi_0 B$ is surjective if and only if $\pi_0 A'\to\pi_0(B\otimes_A A')$ is so, since by flatness $\pi_0(B\otimes_A A')\cong\pi_0 B\otimes_
        {\pi_0 A}\pi_0 A'.$ Now, the result follows from the classical result.
    \end{enumerate}
\end{remark}

\begin{proposition}
    Suppose $\cX$ and $\cY$ are geometric stacks, and suppose $f : \cX\to\cY$ is affine. Then $f$ is a closed immersion if and only if $\cX\times_{\cY}\Spec A\to\Spec A$ is a closed immersion of affine derived schemes for all $\Spec A\to\cY$ in an atlas of $\cY.$
\end{proposition}
\begin{proof}
    The ``only if'' direction is trivial. For the ``if'' direction, suppose that we are given an $n$-atlas $\{\Spec A_i\to\cY\}$ of $\cY$ such that for all $i,$ the maps $\cX\times_{\cY}\Spec A_i\to\Spec A_i$ are affine and induce surjections $\pi_0 A_i\to\pi_0 B_i,$ where $\cX\times_{\cY}\Spec A_i\simeq\Spec B_i.$ Then $f$ is affine by \ref{geometricity detected by single atlas}, and the previous remark implies the result when combined with \ref{refinement through atlas}.
\end{proof}

\begin{lemma}[\protect{\cite[1.2.9.4]{toen2006homotopical}}]
    Let $A\in\aRing_k,$ and let $f\in\pi_0 A.$ The canonical map $A\to A[f^{-1}]$ induces an open immersion $j : \Spec A[f^{-1}]\to\Spec A.$
\end{lemma}

\begin{remark}
    A proof of the above lemma can also be found as \cite[Lemma~3.20]{yaylali2022notes} in the animated language.
\end{remark}

The following is relatively straightforward to show directly from the definitions. We refer the interested reader to \cite[Lemma~3.32]{yaylali2022notes}.

\begin{lemma}\label{open immersions local on target}
    A map $f : \cX\to\cY$ of derived stacks is an open immersion if and only if for any $\Spec A\to\cY,$ $\cX\times_{\cY}\Spec A\to\Spec A$ is a monomorphism and there is an effective epimorphism $\coprod_{i\in I}\Spec A_i\to\cX\times_{\cY}\Spec A$ such that each $\Spec A_i\to\Spec A$ is an open immersion.
\end{lemma}

We also have the following useful alternate characterization of open immersions relating them to \'etale morphisms.

\begin{proposition}
    Let $f : \cX\to\cY$ be a map of derived stacks. Then $f$ is an open immersion if and only if it is an \'etale monomorphism.
\end{proposition}
\begin{proof}
    If $f$ is an open immersion, it is certainly a monomorphism. To show that $f$ is \'etale, we must show that for all $\Spec A\to\cY,$ $\cX\times_{\cY}\Spec A$ has an \'etale atlas $\{\Spec A_i\to\cX\times_{\cY}\Spec A\}$ such that the induced maps $A\to A_i$ are \'etale for all $i.$ By \ref{open immersions local on target}, there exists an effective epimorphism $\coprod_{i}\Spec A_i\to\cX\times_{\cY}\Spec A$ such that the induced maps $\Spec A_i\to\Spec A$ are open immersions, and in particular, \'etale.

    Conversely, suppose that $f$ is an \'etale monomorphism. To show that $f$ is an open immersion, we must show that it is flat, locally of finite presentation, and a monomorphism. It is a monomorphism by assumption. Let $\Spec A\to\cY$ be an arbitrary map. Then there exists an \'etale $n$-atlas $\{\Spec A_i\to\cX\times_{\cY}\Spec A\}$ such that the induced maps $A\to A_i$ are all \'etale. However, \'etale maps are flat and locally of finite presentation, so the result follows.
\end{proof}

We now consider separation conditions on (morphisms of) derived stacks.

\begin{mydef}
    A derived stack is \emph{separated} if the diagonal is a closed immersion.
\end{mydef}

\begin{mydef}
    A derived stack $\cX$ is \emph{quasi-compact} if there exists an $n$-atlas consisting of a single affine. A morphism $f : \cX\to\cY$ is \emph{quasi-compact} if for all maps $\Spec A\to\cY,$ the base change $\cX\times_{\cY}\Spec A$ is quasi-compact.
\end{mydef}

\begin{remark}
    Notice that with this definition, a quasi-compact stack is necessarily geometric.
\end{remark}

The following lemmas are easy to verify.

\begin{lemma}
    Let $\cX$ be a derived stack. The following are equivalent:
    \begin{enumerate}
        \item $\cX$ is quasi-compact. 
        \item There exists an $n$-atlas of $\cX$ consisting of a finite number of affines.
    \end{enumerate}
\end{lemma} 

\begin{lemma}
    Let $\cX\to\cY$ be a quasi-compact morphism of derived stacks. Then for any morphism $\cZ\to\cY,$ the base change $\cX\times_{\cY}\cZ\to\cZ$ is quasi-compact. 
\end{lemma}

\begin{lemma}
    Let
    \[\begin{tikzcd}
        \cX\ar[r]\ar[d] &\cZ\ar[d] &\ar[l]\cY\ar[d]\\
        \cX'\ar[r] &\cZ' &\ar[l] \cY'
    \end{tikzcd}\]
    be a commutative diagram of derived stacks such that $\cX\to\cX',$ $\cY\to\cY',$ and $\cZ\to\cZ'$ are quasi-compact. Then $\cX\times_{\cZ}\cY\to\cX'\times_{\cZ'}\cY$ is quasi-compact.
\end{lemma}

The derived generalizations of quasi-separated stacks and morphisms are more subtle. 

\begin{mydef}
    Let $\cX$ be a derived stack, and let $f : \cX\to\cY$ be a morphism of derived stacks.
    \begin{enumerate}
        \item We say $\cX$ is \emph{$0$-quasi-separated} if the diagonal $\cX\to\cX\times\cX$ is quasi-compact.
        \item We say $f : \cX\to\cY$ is \emph{$0$-quasi-separated} if for any $\Spec A\to\cY,$ $\cX\times_{\cY}\Spec A$ is $0$-quasi-separated.
        \item Suppose $k > 0.$ We say $\cX$ is \emph{$k$-quasi-separated} if the diagonal morphism $\cX\to\cX\times\cX$ is $(k-1)$-quasi-separated.
        \item Suppose $k > 0.$ We say $f : \cX\to\cY$ of derived stacks is \emph{$k$-quasi-separated} if for any $\Spec A\to\cY,$ $\cX\times_{\cY}\Spec A$ is $k$-quasi-separated.
        \item If $\cX$ is a $n$-geometric derived stack, we say $\cX$ is \emph{quasi-separated} if it is $k$-quasi-separated for all $0\leq k\leq n.$
        \item If $f : \cX\to\cY$ is $n$-geometric, we say $f$ is \emph{quasi-separated} if for any $\Spec A\to\cY,$ $\cX\times_{\cY}\Spec A$ is quasi-separated.
    \end{enumerate}
\end{mydef}

\begin{remark}
    Note that if $f : \cX\to\cY$ is $k$-quasi-separated, then the base change $\cX\times_{\cY}\cZ\to\cZ$ is $k$-quasi-separated for any $\cZ\to\cY.$ Additionally, one can use induction to prove that given a commutative diagram
    \[\begin{tikzcd}
        \cX\ar[r]\ar[d] &\cZ\ar[d] &\ar[l]\cY\ar[d]\\
        \cX'\ar[r] &\cZ' &\ar[l] \cY'
    \end{tikzcd}\]
    of derived stacks such that $\cX\to\cX',$ $\cY\to\cY',$ and $\cZ\to\cZ'$ are $k$-quasi-separated for some $k\leq 0,$ the induced map $\cX\times_{\cZ}\cY\to\cX'\times_{\cZ'}\cY'$ is also $k$-quasi-separated.
\end{remark}

\begin{lemma}\label{qsep iff qsep rel affine}
    Let $\cX\to\Spec A$ be a morphism of geometric derived stacks. Then $\cX$ is quasi-separated if and only if $\cX\to\Spec A$ is quasi-separated. 
\end{lemma}
\begin{proof}
    If $\cX\to\Spec A$ is $k$-quasi-separated, then we have $\cX\times_{\Spec A}\Spec B$ quasi-separated for any $\Spec B\to\Spec A.$ Taking $B = A$ and the identity map shows that $\cX$ is $k$-quasi-separated.

    Conversely, suppose that $\cX$ is $k$-quasi-separated, let $\Spec B\to\Spec A$ be arbitrary, and consider the diagonal map
    \[
    \cX\times_{\Spec A}\Spec B\to(\cX\times_{\Spec A}\Spec B)\times (\cX\times_{\Spec A}\Spec B)\simeq(\cX\times \cX)\times_{\Spec A\otimes A}\Spec B\otimes B.
    \]
    Our previous observations tell us that if $\cX$ is $0$-quasi-separated, then this map will be quasi-compact as required, and that if $\cX$ is $k$-quasi-separated for $k > 0,$ then this map is $(k-1)$-quasi-separated (because $\Spec A$ and $\Spec B$ are automatically $k$-separated for any $k$).
\end{proof}

\begin{proposition}
    Let $f : \cX\to\cY$ be a morphism of derived geometric stacks. Then $f$ is $k$-quasi-separated if and only if $\cX\to\cX\times_{\cY}\cX$ is $(k-1)$-quasi-separated.
\end{proposition}
\begin{proof}
    Suppose that $\cX\to\cX\times_{\cY}\cX$ is $(k-1)$-quasi-separated. Then the base change $\cX\times_{\cY}\Spec A\to(\cX\times_{\cY}\cX)\times_{\cY}\Spec A$ is $(k-1)$-quasi-separated for any $\Spec A\to\cY,$ and this map is equivalent to the diagonal of $\cX\times_{\cY}\Spec A\to\Spec A.$ However, \ref{qsep iff qsep rel affine} then tells us that this means $\cX\times_{\cY}\Spec A$ is $k$-quasi-separated, so that $\cX\to\cY$ is $k$-quasi-separated.

    Conversely, $f$ is quasi-separated if $\cX\times_{\cY}\Spec A$ is $k$-quasi-separated for any $\Spec A\to\cY.$ Then \ref{qsep iff qsep rel affine} implies that this is the case if and only if $\cX\times_{\cY}\Spec A\to\Spec A$ is $k$-quasi-separated. In particular, $f$ is $k$-quasi-separated if and only if 
    \[
    \cX\times_{\cY}\Spec A\to(\cX\times_{\cY}\Spec A)\times_A (\cX\times_{\cY}\Spec A)\simeq (\cX\times_{\cY}\cX)\times_{\cY}\Spec A
    \]
    is $(k-1)$-quasi-separated, if and only if
    \[
    (\cX\times_{\cY}\Spec A)\times_{(\cX\times_{\cY}\cX)\times_{\cY}\Spec A}\Spec B\simeq \cX\times_{\cX\times_{\cY}\cX}\Spec B
    \]
    is $(k-1)$-quasi-separated for any $\Spec B\to(\cX\times_{\cY}\cX)\times_{\cY}\Spec A.$ We must show that for any $\Spec B\to\cX\times_{\cY}\cX,$ $\cX\times_{\cX\times_{\cY}\cX}\Spec B$ is $(k-1)$-quasi-separated. Fixing an aribitrary $\Spec B\to\cX\times_{\cY}\cX$ and letting $\Spec A\to\cY$ be $\Spec B\to\cX\times_{\cY}\cX\to\cY,$ it follows that $\cX\to\cX\times_{\cY}\cX$ is indeed $(k-1)$-quasi-separated, so we are finished.
\end{proof}

\begin{proposition}
    Let $f : \cX\to\cY$ be a morphism of derived stacks. There exists a factorization
    \[\cX\to\cI\to\cY\]
    of $f$ such that $\cX\to\cI$ is an effective epimorphism and $\cI\to\cY$ is a monomorphism. The space of such factorizations is contractible, and we call any such factorization an \emph{epi-mono factorization}, and write $\im(f)$ for any such $\cI,$ and say that $\im(f)$ is the \emph{image of $f.$}
\end{proposition}
\begin{proof}
    This follows from \cite[5.2.8.16]{HTT} and \cite[5.2.8.17]{HTT}.
\end{proof}

We have the following useful fact about the image of a smooth geometric morphism.

\begin{proposition}[\protect{\cite[Lemma~4.40]{AG2014}}]
    Let $f : \cX\to\cY$ be a smooth geometric morphism, and let $\cX\to\cI\to\cY$ be an epi-mono factorization of $f.$ Then $\cI\to\cY$ is an open immersion.
\end{proposition}

\begin{mydef}
    Let $X$ be a derived stack. We say $X$ is a \emph{derived scheme} if it admits an atlas $\{\Spec A_i\to X\}_{i\in I}$ such that each $\Spec A_i\to X$ is an open immersion. In particular, any derived scheme $X$ is $1$-geometric. We let $\dSch$ denote the full subcategory of $\dStk$ consisting of derived schemes.
\end{mydef}

\begin{remark}
    Let $k$ be a discrete ring. Then the adjunction
    \[\begin{tikzcd}
        \pi_0 : \aRing_k\ar[r,shift left] &\ar[l, shift left] \left(\aRing_k\right)_{\leq 0}\simeq\CRing_{k} : i
    \end{tikzcd}\]
    formally induces adjunctions
    \[\begin{tikzcd}
        i_! : \PSh((\CRing_{k})^{\op})\ar[r,shift left] &\ar[l, shift left]\PSh((\aRing_k)^{\op}) : i^*
    \end{tikzcd}\]
    and
    \[\begin{tikzcd}
        i^* : \PSh((\aRing_k)^{\op})\ar[r,shift left] &\ar[l,shift left]\PSh((\CRing_{k})^{\op}) : \pi_0^*,
    \end{tikzcd}\]
    where $i^*$ and $\pi_0^*$ are given by restriction along the respective functor, and $i_!$ is obtained by applying the ``left Kan extension along $i$'' to the functor in question (that is, for $X\in\PSh((\CRing_k)^{\op}),$ $i_! X = \Lan_i X$). Because $i^*$ and $\pi_0^*$ are right adjoints, they preserve \'etale sheaves, and by composing $i_!$ with the \'etale sheafification, we obtain adjunctions
    \[\begin{tikzcd}
        L_{\et}\circ i_! : \Sh_{\et}((\CRing_{k})^{\op})\ar[r,shift left] &\ar[l, shift left]\Sh_{\et}((\aRing_k)^{\op}) : i^*
    \end{tikzcd}\]
    and
    \[\begin{tikzcd}
        i^* : \Sh_{\et}((\aRing_k)^{\op})\ar[r,shift left] &\ar[l,shift left]\Sh_{\et}((\CRing_{k})^{\op}) : \pi_0^*,
    \end{tikzcd}\]
    and we set $t_0 := i^*$ and $\iota := i_!.$ We include a few remarks about the properties of these functors, which can be found in \cite[Section 2.2.4]{toen2006homotopical}, specifically \cite[2.2.4.1]{toen2006homotopical} and \cite[2.2.4.4]{toen2006homotopical}.
    \begin{enumerate}
        \item The functor $\iota$ is fully faithful. This follows from the general properties of Kan extensions (see \cite[Section 4.3.2]{HTT}), but a proof in the model category theoretic setting can be found as \cite[2.2.4.1]{toen2006homotopical}. 
        \item The functor $t_0$ has both a left and right adjoint, and thus preserves both limits and colimits. In particular, $t_0$ preserves epimorphisms.
        \item The functor $t_0$ preserves geometricity. It also preserves the properties of being flat, being \'etale, and being smooth for morphisms between geometric derived stacks.
        \item The functor $\iota$ preserves pullbacks of geometric stacks along a flat morphism, preserves geometricity, and preserves the properties of being flat, being smooth, and being \'etale for morphisms between classical geometric (higher) stacks.
        \item If $X$ is a classical geometric (higher) stack, and $X'\to\iota(X)$ is a flat morphism of geometric derived stacks, then $X'$ is in the image of $\iota.$
    \end{enumerate}
\end{remark}

Many of the concepts we have previously defined for derived (geometric) stacks can be seen to be generalizations of the corresponding notions for classical algebraic stacks. In particular, we have the following, which can be checked from the definitions or by appealing to \cite[2.2.4.4]{toen2006homotopical}.

\begin{proposition}
    Let $f : X\to Y$ be a map of classical algebraic stacks. Then
    \begin{enumerate}
        \item If $X$ is a classical algebraic space (respectively, a classical algebraic stack), then $\iota X$ is $1$-DM (respectively, $2$-geometric). If $X$ is a classical scheme, then $\iota X$ is a derived scheme.
        \item If $X$ is quasi-compact (respectively, qcqs), then $\iota X$ is quasi-compact (respectively, quasi-compact and quasi-separated).
        \item If $f$ is flat (respectively smooth, \'etale, an open immersion, a closed immersion, quasi-compact, or quasi-separated) then so is $\iota f.$
    \end{enumerate}
\end{proposition}

We also have the following global generalization of \cite[5.2.4]{cesnavicius2023purity}.

\begin{lemma}
    Let $\cX$ be a derived geometric stack. Then truncation induces equivalences of categories
    \[t_0 : \Shv^{n-\et}_{/\cX}\to\Shv^{n-\et}_{/t_0\cX}\]
    and
    \[t_0 : \Shv^{\mathrm{open}}_{/\cX}\to\Shv^{\mathrm{open}}_{t_0\cX}\]
    where $\Shv^{n-\et}_{/\cX}$ denotes the full subcategory of derived stacks $\cX'\to\cX$ which are $n$-DM and \'etale over $\cX,$ $\Shv^{n-\et}_{/t_0\cX}$ denotes the category of higher stacks $n$-DM and \'etale over $t_0\cX,$ $\Shv^{\mathrm{open}}_{/\cX}$ the full subcategory of open immersions into $\cX,$ and $\Shv^{\mathrm{open}}_{/t_0\cX}$ the full subcategory of open immersions of higher stacks into $t_0\cX.$
\end{lemma}
\begin{proof}
    This follows from the previous remarks regarding truncation, the equivalence of categories $\dAff_k^{\et}\to\Aff_{\pi_0 k}^{\et},$ and the fact that a morphism $\cX\to\cY$ is an open immersion if and only if it is an \'etale monomorphism.

    Suppose first that $\cX\simeq\Spec A$ is affine. Let $\cY\to\Spec A$ be $n$-DM and \'etale. We proceed by induction on $n.$ If $n = -1,$ $\cY$ is affine and \'etale over $A,$ and we have the correspondence
    \[\dStk^{(-1)-\et}_A\simeq\dAff^{\et}_A\simeq\dAff^{\et}_{\pi_0 A}\simeq\AlgStk^{(-1)-\et}_{\pi_0 A}.\]
    
    Consider the case $n = 0.$ We may choose an \'etale $0$-atlas $\{\Spec B_i\to\cY\}_{i\in I}$ by affine derived schemes \'etale over $\cY,$ and thus these are also affine over $\Spec A.$ Taking truncations, we obtain $\Spec\pi_0(B_i)$ \'etale over $\cY_0 := t_0\cY$ and \'etale over $\Spec\pi_0 B.$ On the other hand, given the data of an \'etale $0$-atlas $\{\Spec B_{0,i}\to\cY_0\}_{i\in I}$ of $\cY_0$ which is $0$-DM and \'etale over $\Spec \pi_0 A,$ we may lift each $\Spec B_{0,i}$ to an affine derived scheme $\Spec B_i$ \'etale over $\Spec A.$
    
    Because $\cY_0$ is $0$-DM, we also find that in particular the diagonal is affine. If $Z_0 = \coprod_i \Spec B_{0,i},$ this implies that the \v{C}ech nerve $\check{C}(Z_0/\cY)_\bullet$ is equivalent to a disjoint union of simplicial classical affine schemes:
    \[
    \check{C}(Z_0/\cY)_n \simeq \coprod_{x\in I^n} S_{0,x,n}\simeq \coprod_{x\in I^n}\Spec C_{0,x}^n
    \]
    It is clear that each $S_{0,x,n}$ is \'etale over $\Spec A,$ so that by the affine correspondence we may lift $S_{0,x,\bullet}$ to a \'etale simplicial scheme $Z_\bullet$ over $\Spec A,$ such that each component is a  disjoint union of \'etale affine schemes $S_{x,n}.$

    Now, $\check{C}(Z_0/\cY)_\bullet$ is a $(-1)$-\'etale Segal groupoid object (in higher stacks over $\pi_0 A),$ and it follows formally from the affine equivalence of categories that $\cZ_\bullet$ is a $(-1)$-\'etale Segal groupoid object in derived stacks over $A.$ Thus the colimit $\cY$ of $\cZ_\bullet$ is a derived $0$-stack over $A$ by \cite[Proposition 1.3.4.2]{toen2006homotopical} and $\cZ_0\to\cY$ provides an \'etale $0$-atlas so that $\cY$ is $0$-DM and \'etale over $A.$ It is a simple check to verify that this construction is inverse to truncation.

    The induction step from $n$ to $n+1$ follows in exactly the same way as the above, as after forming an \'etale $(n+1)$-atlas of $\cY_0,$ we may form the \v{C}ech nerve associated to the cover as before, except each object of the cover will be a disjoint union of $n$-DM higher stacks over $\pi_0 A,$ which we may then lift to $A$ via the inductive hypothesis.
    
    Descending these constructions along all affines (or even just affines in a cover) of the original $\cX$ gives the equivalence when the base is not affine.

    The open case follows from the \'etale case. The truncation of an open immersion is an open immersion by already observed properties of the truncation, so we need to show that if $U\to \cX$ is \'etale such that $t_0 U\to t_0\cX$ is an open immersion (equivalently, an \'etale monomorphism), then $U\to \cX$ is in fact a monomorphism. To that end, consider the epi-mono factorization of $U\to U\times_{\cX} U:$
    \[
    U\to\cI\to U\times_{\cX} U.
    \]
    The truncation of these maps is an epi-mono factorization of
    \[
    t_0 U\to t_0(U\times_{\cX} U)\cong t_0 U\times_{t_0 \cX} t_0 U,
    \]
    and since $t_0 U\to t_0\cX$ is a monomorphism, it follows that the maps
    \[
    t_0 U\to t_0\cI\to t_0 U\times_{t_0\cX} t_0 U
    \]
    are equivalences. We now appeal to the \'etale bijection: we know that $\cI\to U\times_{\cX} U$ is \'etale, and corresponds to an isomorphism $t_0\cI\to t_0(U\times_{\cX} U),$ so that it must in fact be equivalent to the equivalence $U\times_{\cX} U\to U\times_{\cX} U,$ which completes the proof.
\end{proof}

As a consequence, we find that $t_0$ also detects smooth effective epimorphisms in the sense that if $\cX\to \cY$ is a smooth morphism of geometric derived stacks such that the induced morphism $t_0 \cX\to t_0 \cY$ is an effective epimorphism, then $\cX\to \cY$ is an effective epimorphism. To see this, consider the epi-mono factorization $\cX\to U\to \cY$ of $\cX\to \cY.$ The induced factorization $t_0 \cX\to t_0 U\to t_0 \cY$ is also an epi-mono factorization, and since $t_0 \cX\to t_0 \cY$ was assumed to be an effective epimorphism, it follows that $t_0 U\to t_0 \cY$ is an equivalence. Then, by the equivalence of categories between \'etale stacks over $\cY$ and $t_0 \cY,$ it follows that $U\to \cY$ must itself be an equivalence, since $U\to \cY$ is \'etale by \cite[Lemma 4.40]{AG2014}.

One special class of derived stacks which we will consider is the following:
\begin{mydef}
    A derived stack $X$ is a \emph{derived algebraic space} if its truncation $t_0 X$ is a classical algebraic space.
\end{mydef}

The following lemma follows from the fact that the property of being a closed immersion can be checked on classical truncations.
\begin{lemma}
    Let $X$ be a derived algebraic space. Then the diagonal $X\to X\times X$ is separated.
\end{lemma}

\section{Quasi-Coherent Sheaves}

In this section, we will review the basic theory of quasi-coherent sheaves on derived stacks. The main result will be a base change theorem for sufficiently nice derived stacks \ref{Base change for derived algebraic spaces}, which is a generalization of \cite[Proposition 1.4]{toen2012proper}. 

\subsection{Generalities on Sheaves and Hypersheaves}

\begin{mydef}
    Let $\cC$ be a presentable $\infty$-category and consider the fpqc or \'etale topology on $\aRing_k.$ A functor $F : \PSh(\dAff_k)^{\op}\to\cC$ is a \emph{sheaf} (respectively, \emph{hypersheaf}) if for any effective epimorphism $\cX\to\cY$ (respectively, any hypercover $\cX_\bullet\to\cY$), the natural map
    \[
    F(\cY)\to\lim_{n\in\Delta}F(\Cech(\cX/\cY)_n)
    \]
    (respectively,
    \[
    F(\cY)\to\lim_{n\in\Delta}F(\cX_n)\textrm{)}
    \]
    is an equivalence.
\end{mydef}

\begin{proposition}[\protect{\cite[Propositions~A.5,A.6]{Aoki2023-mt}}]
    Let $\cC$ be an $\infty$-site. Suppose that $\cB$ is a basis for $\cC$ and $F$ is a presheaf on $\cB.$ Then $F$ is a (hyper)sheaf on $\cB$ if and only if its right Kan extension is a (hyper)sheaf on $\cC.$
\end{proposition}

\begin{remark}
    The functor informally given by
    \begin{align*}
        \aRing_k&\to\Cat_\infty\\
        A&\mapsto\cMod_A
    \end{align*}
    is a hypersheaf for the fpqc topology, by \cite[D.6.3.3]{SAG}.
\end{remark}

\begin{proposition}[\protect{\cite[Proposition~3.48]{yaylali2022notes}}]
Let $\cC$ be a presentable $\infty$-category, and let $F : \aRing_k\to\cC$ be a (hypercomplete) sheaf with respect to either the fpqc or \'etale topology on $\aRing_k.$ Let $\bfR F$ denote the right Kan extension of $F$ along $\aRing_k\to\PSh(\dAff_k)^{\op}.$ For any diagram $\cX_{-} : K\to\PSh(\dAff_k),$ where $K$ is a simplicial set, any any morphism $\colim_k\cX_k\to\cY$ which becomes an equivalence after (hyper)sheafification, the natural map
\[
\bfR F(\cY)\to\lim_{k\in K}\bfR F(\cX_k)
\]
is an equivalence.
\end{proposition}

\begin{cor}[\protect{\cite[Proposition~3.52]{yaylali2022notes}}]
    Let $\cC$ be a presentable $\infty$-category, and let $F : \aRing_k\to\cC$ be an \'etale sheaf. Let $\bfR F$ denote a right Kan extension of $F$ along the Yoneda embedding $\aRing_k\to\PSh(\dAff_k)^{\op}.$ Then, for any derived scheme $X$ over $k,$ the natural morphism
    \[
    \bfR F(X)\to\lim_{U\to X} F(U)
    \]
    is an equivalence, where the limit is taken over affine open derived subschemes $U\to X.$
\end{cor}

Similarly, if $\cX$ is a geometric (respectively DM) derived stack, we may calculate its value on any \'etale/fpqc sheaf $F$ as
\[
F(\cX)\simeq \lim_{\Spec A\to\cX} F(\Spec A)
\]
where the limit is taken over smooth (respectively \'etale) maps $\Spec A\to\cX.$

\subsection{The Category of Quasi-Coherent Sheaves}

\begin{mydef}
    Let $k$ be an animated ring, and let $\cX$ be a presheaf on $\dAff_k.$ The \emph{$\infty$-category of quasi-coherent modules on $\cX,$} denoted $\QC(\cX),$ is defined to be the value at $\cX$ of the right Kan extension of 
    \[
    \cMod_{(-)} : \aRing\to\Cat_{\infty}
    \]
    along (the opposite of) the Yoneda embedding $\aRing_k\simeq\dAff_k^{\op}\to\PSh(\dAff_k)^{\op}.$ Denote this functor by
    \[
    \QCoh : \PSh(\dAff_k)^{\op}\to\Cat_\infty.
    \]
\end{mydef}

\begin{remark}
    Somewhat more explicitly, this Kan extension is given by
    \[
    \QC(\cX)\simeq\lim_{\Spec A\to\cX}\cMod_A,
    \]
    with transition maps given by (derived) pullback/tensor product. Informally, an object of $\QC(\cX)$ consists of an $A$-module $f^*\cF$ for every $f : \Spec A\to\cX,$ such that given any composition $\Spec B\xrightarrow{g}\Spec A\xrightarrow{f}\cX,$ we have compatible equivalences
    \[
    g^*\left(f^*\cF\right)\xrightarrow{\sim}(g\circ f)^*\cF.
    \]
    Given a map $f : \Spec A\to\cX,$ we denote the canonical functor $\QC(\cX)\to\cMod_A$ by $f^*.$ Note also that in fact, $\cMod_{(-)}$ factors through $\Pr^{L,\st},$ the $\infty$-category of presentable stable $\infty$-categories with left adjoints as morphisms, and as such, $\QC(\cX)$ is also a presentable stable $\infty$-category\footnote{This uses the fact that the transition maps in the limit description of $\QC(\cX)$ preserve finite limits.}, and $f^*$ is a left adjoint, whose right adjoint is denoted $f_*.$
\end{remark}

This construction is a generalization of the derived category of sheaves of $\OX$-modules with quasi-coherent cohomology in the following sense.

\begin{proposition}[\protect{\cite[Proposition~3.57]{yaylali2022notes}}]
    Let $X$ be a classical scheme. Then there is a natural equivalence of $\infty$-categories $\Dqc(X)\simeq\QC(X),$ where $\Dqc(X)$ denotes the derived $\infty$-category of $\OX$-modules with quasi-coherent cohomology.
\end{proposition}

In particular, the homotopy category of $\QC(X)$ is equivalent to the $1$-category $\sfD_{\qc}(X).$ 

The $\infty$-category $\QC(\cX)$ admits a t-structure $(\QC(\cX)_{\geq 0},\QC(\cX)_{\leq 0}),$ determined by the rule that $\cF\in\QC(\cX)_{\geq 0}$ if for any $f : \Spec A\to\cY,$ we have $f^*\cF\in\cMod_A^{\cn}$. This defines a t-structure by \cite[1.2.1.16]{HA}, and the adjoint functor theorem guarantees that the inclusion $\QC(\cX)_{\geq 0}\to\QC(\cX)$ admits a right adjoint, since $\QC(\cX)_{\geq 0}$ is stable under colimits.

When $\cX$ is geometric, we can be more explicit about the t-structure.
\begin{proposition}[\protect{\cite[Proposition~1.5.4]{GR2019}}]
    \begin{enumerate}
        \item Let $\cX$ be a derived geometric stack, and let $\{f_i : \Spec A_i\to\cX\}$ be a smooth atlas of $\cX.$ Then $\cF\in\QC(\cX)$ belongs to $\QC(\cX)_{\geq 0}$ (respectively $\QC(\cX)_{< 0}$) if and only if each $f_i^*\cF$ belongs to $\QC(\Spec A_i)_{\geq 0}\simeq\cMod_{A_i}^{\cn}$ (respectively, $\QC(\Spec A_i)_{< 0}$).
        \item Let $f : \cX\to\cY$ be a flat map of geometric derived stacks. Then $f^* : \QC(\cY)\to\QC(\cX)$ is t-exact.
    \end{enumerate}
\end{proposition}

\subsection{Base Change and the Spectral Stack associated to a Geometric Stack}

We will deduce our base change theorem as a corollary of Lurie's base change theorem for qcqs spectral algebraic spaces \cite[Corollary 3.4.2.2]{SAG}. To achieve this, we must establish a few comparison and compatibility results relating derived stacks and spectral stacks as in \cite{SAG}. The comparisons here are possibly well-known to the experts, but we have not seen many discussions in the literature. For the same comparisons discussed in a different language, we refer to \cite[Section 2]{https://doi.org/10.1112/jlms.12738} and \cite[Section 3]{chough2025formalderivedalgebraicgeometry}.

Recall the forgetful functor $\theta : \aRing\to\CAlgcn$ from animated rings to connective $E_\infty$-algebras. Composing with the Yoneda embedding $j : (\CAlgcn)^{\op}\to\Fun(\CAlgcn,\Ani)$ produces a functor
\[\aRing^{\op}\xrightarrow{\tilde{\theta}}\Fun(\CAlgcn,\Ani)\]
which is given on objects by $A\mapsto \Map_{\CAlgcn}(\theta(A),-).$ Forming the left Kan extension $\Theta$ of $\tilde{\theta}$ in the diagram
\[\begin{tikzcd}
    \aRing^{\op}\ar[d, hook]\ar[r,"\tilde{\theta}"] &\Fun(\CAlgcn,\Ani)\\
    \dStk\ar[ru,"\Theta"']
\end{tikzcd}\]
produces our desired ``underlying spectral stack'' functor $\Theta : \dStk\to\Fun(\CAlgcn,\Ani).$ Note that $\tilde{\theta}$ preserves all finite limits and colimits, as $j$ preserves all limits and colimits by \cite[Proposition 5.1.3.2]{HTT} and $\theta$ preserves all limits and colimits because it is both a right and left adjoint, so that $\Theta$ preserves finite limits by \cite[Proposition 6.1.5.2]{HTT}\footnote{Note that $\Theta$ is the restriction along $\dStk\to\Fun(\aRing,\Ani)$ of the left Kan extension of $\tilde{\theta}$ along the Yoneda embedding $\aRing^{\op}\to\Fun(\aRing,\Ani).$}. In particular, it follows that $\Theta(\cX\times_{\cY}\cX')\simeq\Theta(\cX)\times_{\Theta(\cY)}\Theta(\cX').$

Let $\cQ$ denote a right Kan extension of $\cMod$ along the Yoneda embedding
\[\begin{tikzcd}
    \CAlgcn\ar[r,"\cMod"]\ar[d] &\Pr^{L,\st}\\
    \PSh((\CAlgcn)^{\op})^{\op}\ar[ru,"\cQ = \Ran\cMod"'],
\end{tikzcd}\]
then the universal property of the right Kan extension induces a canonical natural transformation
\[
\rho : \cQ(\Theta(-))\to\QC(-)
\]
of functors $\Fun(\aRing,\Ani)\to\Pr^{L,\st}.$ Our goal is to show that for a geometric derived stack $\cX,$ these functors agree. 

\begin{lemma}
    Let $A\to B$ be a flat morphism of animated rings. Then $\theta(A)\to\theta(B)$ is a flat morphism of connective $E_\infty$-algebras. If moreover $A\to B$ is \'etale, then so is $\theta(A)\to\theta(B).$
\end{lemma}
\begin{proof}
    Flatness and \'etaleness are both conditions on homotopy groups. In particular, $A\to B$ is flat (respectively, \'etale) if $\pi_0 A\to\pi_0 B$ is flat (respectively, \'etale) and $\pi_* A\otimes_{\pi_0 A}\pi_0 B\to\pi_* B$ is an isomorphism of graded rings (the definitions are identical for animated rings and for $E_\infty$-rings). The result then follows from the fact that $\theta$ preserves homotopy groups. 
\end{proof}

In fact, the functor $\Theta$ preserves many nice properties of derived stacks which are desirable. For example, if $f : \cX\to\cY$ is a map of derived DM stacks, then $f$ is quasi-compact (respectively quasi-separated) if and only if $\Theta(f)$ is quasi-compact (respectively quasi-separated). In particular, this applies to derived algebraic spaces, which is the generality in which we will use this result. We refer to \cite[Lemma 3.46]{chough2025formalderivedalgebraicgeometry} for a proof, in a language closer to that of \cite{SAG}.

\begin{cor}
    Suppose $\cX$ is a geometric derived stack. Then $\rho_{\cX} : \cQ(\Theta(\cX))\to\QC(\cX)$ is an equivalence.
\end{cor}

\begin{remark}
    We thank Marius Nielsen for pointing out to us that $\rho_{\cX} : \cQ(\Theta(\cX))\to\QC(\cX)$ is in fact an equivalence for all functors $\cX : \aRing\to\Ani.$ However, we do not need the result in this generality, and so only provide the proof in the geometric case.
\end{remark}

\begin{proof}
    By the previous lemma and \cite[Remark 6.2.3.6]{HTT}, fpqc covers by affines are sent to fpqc covers by affines via $\Theta$. Then we may compute $\cQ(\Theta(\cX))$ and $\QC(\cX)$ via the same limit by choosing a smooth cover of $\cX$ by affine derived schemes, since both $\cQ$ and $\QC$ are fpqc sheaves.
\end{proof}

Going forward, we will denote $\cQ$ by $\QC$ as well, since we will only consider its value on sufficiently nice functors (spectral DM stacks).

The following proposition follows readily.
\begin{proposition}
    The functor $\Theta$ restricts to a functor
    \[
    \Theta : \dStk^{\DM}\to\mathsf{spDM}\footnote{Here, $\mathsf{spDM}$ denotes the $\infty$-category of spectral Deligne-Mumford stacks as in \cite{SAG}, viewed as functors via the embedding of \cite[Section 1.6.4]{SAG}.}.
    \]
    Moreover, $\Theta$ is compatible with quasi-coherent sheaves in the sense that for any morphism of geometric derived stacks $f : \cX\to\cY,$ the square
    \[\begin{tikzcd}
        \QC(\Theta(\cY))\ar[r,"(\Theta f)^*"]\ar[d,"\rho_{\cY}"] &\QC(\Theta(\cX))\ar[d,"\rho_\cX"]\\
        \QC(\cY)\ar[r,"f^*"] &\QC(\cX)
    \end{tikzcd}\]
    is commutative. 
\end{proposition}

We may then use the above comparisons to deduce a base change result. If
\[\begin{tikzcd}
    X'\ar[r,"\overline{g}"]\ar[d,"\overline{f}"] &X\ar[d,"f"]\\
    S'\ar[r,"g"] &S
\end{tikzcd}\]
is a Cartesian diagram of derived geometric stacks such that $\Theta(X'), \Theta(X), \Theta(S')$ and $\Theta(S)$ are qcqs spectral algebraic spaces, then the base change morphism $g^* f_*\to\overline{f}^* \overline{g}_*$ is an equivalence, since the same is true for $(\Theta g)^*(\Theta f)_*\to (\Theta\overline{f})_*(\Theta\overline{g})^*$ by \cite[3.4.2.2]{SAG}\footnote{This uses our earlier observation that $\Theta$ preserves finite limits, so that the corresponding square after applying $\Theta$ is also Cartesian.}, and the previous theorem identifies the $\infty$-categories of quasi-coherent sheaves on a geometric derived stack with the $\infty$-category of quasi-coherent sheaves on the corresponding ``spectral stack.''

\begin{cor}\label{Base change for derived algebraic spaces}
    Let
    \[\begin{tikzcd}
        X'\ar[r,"\overline{g}"]\ar[d,"\overline{f}"] &X\ar[d,"f"]\\
        S'\ar[r,"g"] &S
    \end{tikzcd}\]
    be a Cartesian square of qcqs derived algebraic spaces. Then the base change morphism
    \[
    g^* f_*\to\overline{f}_*\overline{g}^*
    \]
    is an equivalence of functors $\QC(X)\to\QC(S').$
\end{cor}

In particular, we have base change when $S', S,$ and $X$ are all classical algebraic spaces.

\subsection{Fibrations Associated to Quasi-Coherent Sheaves}

In this subsection, we discuss quasi-coherent sheaves (and variants thereof) from a slightly different point of view -- as a Cartesian and coCartesian fibration over $\dStk^{\op}.$ 

Via the straightening-unstraightening correspondence, the functor $\QCoh : \dStk^{\op}\to\Pr^{L,\st}$ corresponds to a Cartesian and coCartesian fibration $\QC\to\dStk^{op}$ (restricting now to only considering derived stacks rather than all presheaves). The objects of $\QC$ can be identified with pairs $(\cX,\cF),$ where $\cX$ is a derived stack and $\cF$ is an object of $\QC(\cX).$ Via this identification, a $1$-morphism $(\cX',\cF')\to (\cX,\cF)$ in $\QC$ is given by a pair $(f,\phi),$ where $f : \cX\to\cX'$ is a map of derived stacks, and $\phi : f^*\cF'\to\cF$ is a morphism in $\QC(\cX).$

It will be useful to review how limits in $\QC$ are calculated. Given a functor $F : I\to\QC,$ write $F(i) = (\cX_i,\cF_i).$ The limit of $F,$ if it exists, is given (as an object) by first calculating the limit $\cX$ of the composition
\[
I\xrightarrow{F}\QC\to\dStk^{\op}.
\]
Then $\cX$ comes equipped with maps $f_i : \cX_i\to\cX$ of derived stacks for any $i\in I.$ The functor $F$ then induces a functor
\begin{align*}
    F_{\cX} : I&\to\QC(\cX)\\
    i&\mapsto (f_i)_*\cF_i.
\end{align*}
The limit of $F$ is then given by
\[\lim_{i} (\cX_i,\cF_i)\simeq (\cX,\lim_{i} (f_i)_*\cF_i),\]
where $\lim_{i} (f_i)_*\cF_i$ is calculated in $\QC(\cX).$

Colimits are computed in a dual way.

\begin{mydef}
    Let $S\in\dStk.$ We define $\QC_S$ to be the pullback
    \[\begin{tikzcd}
        \QC_S\ar[r]\ar[d] &\QC\ar[d]\\
        (\dStk_S)^{\op}\ar[r] &\dStk^{\op},
    \end{tikzcd}\]
    where $(\dStk_S)^{\op}\to\dStk^{\op}$ is the forgetful functor. Then $\QC_S\to(\dStk_S)^{\op}$ is also a Cartesian and coCartesian fibration. Informally, the objects of $\QC_S$ are given by pairs $(\cX\to S,\cF),$ where $\cX\to S$ is a morphism of derived stacks and $\cF\in\QC(\cX),$ and the morphisms in $\QC_S$ are defined analogously to morphisms in $\QC.$ If $S = \Spec k$ is affine, we also write $\QC_k := \QC_S.$ 
\end{mydef}

We can use this perspective to define a general ``global sections'' functor $\QC_S\to\QC(S).$

\begin{mydef}
    \cite[Construction B.9.2]{Calaque2025-hm} provides a commutative square
    \[\begin{tikzcd}
        \QC_S\ar[r]\ar[d] &\QC\ar[d]\\
        (\dStk_S)^{\op}\ar[r] &\dStk^{\op}
    \end{tikzcd}\]
    where the bottom horizontal functor is the constant functor at $S;$ i.e., it factors as
    \[
    (\dStk_S)^{\op}\to\{S\}\to\dStk^{\op}.
    \]
    Thus, the upper horizontal functor factors as $\QC_S\to\QC(S)\to\QC.$ We denote the resulting functor $\QC_S\to\QC(S)$ by $\Gamma_S,$ and refer to $\Gamma_S(\cX\to S,\cF)$ as the \emph{global sections of $\cF$ over $S$}.\footnote{We may also abusively refer to the composition $\QC_S\to\QC(S)\to\QC$ as the global sections functor.}
\end{mydef}

Informally, $\Gamma_S(f : \cX\to S,\cF)$ is given by $(S,f_*\cF).$

Let $\ZZ^{\op}$ denote the $1$-category associated to the opposite of the poset $(\ZZ,\geq 0),$ and let $\ZZ^{\disc}$ denote the integers viewed as a discrete category (i.e., there are no nontrivial morphisms).

Recall that if $\cC$ is a stable $\infty$-category, then the category of filtered objects of $\cC$ is defined to be $\Fun(\ZZ^{\op},\cC),$ and the category of graded objects of $\cC$ is defined to be $\Fun(\ZZ^{\disc},\cC).$ If $\cC$ is stable and presentable, so too are $\Fun(\ZZ^{\op},\cC)$ and $\Fun(\ZZ^{\disc},\cC).$

\begin{mydef}
    Consider the composition of $\QCoh : \dStk^{\op}\to\Pr^{L,\st}$ with the functor
\[
\Fun(\ZZ^{\op},-) : \Pr^{L,\st}\to\Pr^{L,\st}
\]
(respectively, the functor
\[
\Fun(\ZZ^{\mathrm{disc}},-) : \Pr^{L,\st}\to\Pr^{L,\st}).)
\]
The Cartesian and coCartesian fibration associated to this functor is defined to be $\QCF\to\dStk^{\op}$ (respectively, $\QCGR\to\dStk^{\op}$), and the fiber of this fibration over a derived stack $\cX$ is the $\infty$-category of \emph{filtered quasi-coherent sheaves on $\cX$}
\[
\QCF(\cX) := \Fun(\ZZ^{\op},\QC(X))
\]
(respectively the $\infty$-category of \emph{graded quasi-coherent sheaves on $\cX$}
\[
\QCGR(\cX) := \Fun(\ZZ^{\disc},\QC(\cX))).
\]
\end{mydef}

The composition
\[
\Fun(\ZZ^{\op},-)\xrightarrow{\prod \operatorname{res}_{\{n\leq n+1\}^{\op}}}\prod_n\Fun(\{n\leq n+1\}^{\op},-)\xrightarrow{\prod \cofib}\prod_n\Fun(\{n\},-)\simeq\Fun(\ZZ^{\mathrm{disc}},-)
\]
induces for any $\cC\in\Pr^{L,\st}$ the map sending a filtered object of $\cC$ to its associated graded object, and admits a right adjoint which on objects maps a graded object $X^\bullet$ to the filtered object $\triv(X^\bullet),$ which has $\triv(X^\bullet)^n := X^n,$ and has all transition maps given by the zero maps. Then (the dual of) \cite[7.3.2.6]{HA} implies that the associated map of Cartesian and coCartesian fibrations $\gr : \QCF\to\QCGR$ admits a right adjoint as well.

Similarly, the restriction natural transformation $\Fun(\ZZ^{\mathrm{disc}},-)\to\Fun(\{n\},-)$ admits a right adjoint when applied to any $\cC$ (given by forming the right Kan extension along $\{n\}\to\ZZ^{\mathrm{disc}}$), so that the induced map $\QCGR\to\QC$ admits a right adjoint as well. 

\section{The Cotangent Complex}

\subsection{The Cotangent Complex of a Morphism of Affine Derived Schemes}

Classically, the cotangent complex of a morphism $k\to A$ of rings is defined by taking a simplicial resolution of $A$ by polynomial $k$-algebras, forming the K\"{a}hler differential of this simplicial resolution, and then base changing to $A.$ There is a standard choice of resolution
\[\begin{tikzcd}
    \dots\ar[r,shift left=3]\ar[r,shift left]\ar[r,shift right]\ar[r,shift right=3] &k[k[k[A]]]\ar[r,shift left=2]\ar[r]\ar[r,shift right=2] &k[k[A]]\ar[r,shift left]\ar[r,shift right] & k[A]\ar[r] &A
\end{tikzcd}\]
(we do not denote the degeneracy maps for simplicity). If we let $P_\bullet$ denote this resolution, then the cotangent complex is given by $\LL_{A/k} := \Omega^1_{P_\bullet/k}\otimes_{P_\bullet}A.$

From a more modern point of view, the above construction is an explicit description of the animation of $\Omega^1_{-/k} : \CRing_k\to\Mod_k.$ However, defining the cotangent complex as the animation of this functor is not completely satisfying for a few reasons:
\begin{enumerate}
    \item This construction only gives the cotangent complex as a $k$-module, whereas one can see from the explicit description that it admits the structure of an $A$-module.
    \item The cotangent complex is functorial not only in the target ring, but in the morphism of rings itself.
    \item This construction does not make the universal property of the cotangent complex evident.
\end{enumerate}
It will be beneficial to examine how to remedy this deficiency to obtain the full functoriality of the cotangent complex.

First, let us state what the universal property of the cotangent complex $\LL_{A/k}$ as an $A$-module is.

Fix a base discrete commutative ring $k,$ and let $\mathsf{PolyFree}_k$ denote the $1$-category of pairs $(P,F),$ where $P\cong k[x_1,\dots, x_n]$ is a finitely presented polynomial $k$-algebra, and $F\cong P^{\oplus r}$ is a finitely generated free $P$-module, and whose morphisms $(P,F)\to (P',F')$ are given by pairs $(f,\phi),$ where $f : P\to P'$ is a map of $k$-algebras and $\phi : F\otimes_P P'\to F'$ is a map of free $P'$-modules. Consider the functor
\[
\mathsf{PolyFree}_k\to\CRing\to\aRing_k
\]
which maps $(P,F)\mapsto P\oplus F,$ where $P\oplus F$ is the classical square-zero extension of $P$ by $F.$ The universal property of animation then guarantees us a functor
\[
\MOD_k^{\cn}\to\aRing_k
\]
which commutes with sifted colimits.

\begin{mydef}
    Let $(A,M)\in\MOD_k^{\cn}.$ The \emph{square zero extension of $A$ by $M$} is defined to be the image of $(A,M)$ under the animation of the square zero functor $(P,F)\mapsto P\oplus F$ above, and we denote this animated ring by $A\oplus M.$
\end{mydef}

The underlying $k$-module of $A\oplus M$ is equivalent to the direct sum of $A$ and $M$ as $k$-modules, because the forgetful functor $\aRing_k\to\cMod_k$ preserves colimits. Additionally, the canonical maps $0\to M\to 0$ in $\Mod_A$ induce maps of animated rings $A\to A\oplus M\to A$ such that the composition is equivalent to the identity. 

\begin{mydef}
    Let $k\to A$ be a map of animated rings, and let $M$ be a connective $A$-module. The \emph{module of $k$-linear derivations of $A$ into $M$} is given by
    \[
    \Der_k(A,M) := \Map_{(\aRing_k)_{/A}}(A, A\oplus M).
    \]
\end{mydef}
\begin{remark}
    There is a pullback square
    \[\begin{tikzcd}
        \Der_k(A,M)\ar[r]\ar[d] &\Map_{\aRing_k}(A,A\oplus M)\ar[d]\\
        \{\id_A\}\ar[r] &\Map_{\aRing_k}(A,A)
    \end{tikzcd}\]
    in $\cMod_k^{\cn}.$
\end{remark}

We are now ready to define the cotangent complex via its universal property.

\begin{mydef}
    The \emph{cotangent complex} of $A/k$ is a connective $A$-module $\LL_{A/k}$ such that there is a natural equivalence of functors
    \[
    \Map_{\cMod_A^{\cn}}(\LL_{A/k},-)\simeq\Der_k(A,-).
    \]
\end{mydef}

The following is well-known (for example, see \cite[Proposition~25.3.1.5]{HA}).

\begin{proposition}
    Let $k\to A$ be a morphism of animated rings. Then the cotangent complex $\LL_{A/k}$ exists.
\end{proposition}

If we want to encode the cotangent complex in a more functorial manner, there are two options. The first is to, as described above, fix a base ring $k$ and consider the animation of the K\"ahler differentials functor $\Omega^1_{-/k}$ (valued in $\Mod_k$). On the other hand, if we want to capture the functoriality of the cotangent complex while remembering that $L_{A/k}$ has the structure of an $A$-module (and not just a $k$-module), we may also consider the cotangent complex as a functor
\[
\bfL : \Arr(\aRing)\to\MOD
\]
in the following way. (Note that via the canonical embedding $\MOD\to\QC,$ we could also view $\bfL$ as valued in $\QC.$)

Let $\PolyPoly$ denote the $1$-category of maps of commutative rings of the form
\[
\ZZ[x_1,\dots, x_m]\to\ZZ[x_1,\dots, x_m,x_{m+1},\dots, x_n]
\]
where $x_i\mapsto x_i$ for $1\leq i\leq m,$ let $\PolyCAlg$ denote the $1$-category of maps of commutative rings of the form $k\to k[x_1,\dots, x_n]$ (including $k$ as constant polynomials), and let $\PolyAlg$ denote the full sub-$\infty$-category of $\Arr(\aRing)$ spanned by maps of the form $k\to k[x_1,\dots, x_n]$ given by the canonical map $k\to k\otimes_{\ZZ}\ZZ[x_1,\dots, x_n].$

\begin{proposition}\label{Kan extension represents derivations}
    The left Kan extension (or equivalently, the animation) of the functor
    \begin{align*}
    \mathbf{\Omega} : \PolyPoly&\to\MOD\\
    (\ZZ[x_1,\dots, x_m]\to \ZZ[x_1,\dots, x_m,x_{m+1},\dots, x_n])&\mapsto (\ZZ[x_1,\dots, x_n], \Omega^1_{\ZZ[x_1,\dots, x_n]/\ZZ[x_1,\dots, x_m]})
    \end{align*}
    along the inclusion $\PolyPoly\to\Arr(\aRing)$ recovers the cotangent complex as defined above in the sense that for a morphism of animated rings $k\to A,$ $\bfL(k\to A)\simeq (A, \LL_{A/k}).$
\end{proposition}

Before we prove the result, let us prove an important result about about this construction.

\begin{proposition}
    Let $k\to A\to B$ be maps of animated rings. Then there is a cofiber sequence
    \[
    \LL_{A/k}\otimes_A B\to\LL_{B/k}\to\LL_{B/A}
    \]
    in $\Mod_B.$
\end{proposition}
\begin{proof}
    Consider the $1$-category $\Poly^3$ whose objects are composable morphisms of finitely presented polynomial $\ZZ$-algebras of the form
    \[
    \ZZ[x_1,\dots, x_n]\to\ZZ[x_1,\dots,x_n,y_1,\dots, y_m]\to\ZZ[x_1,\dots, x_n,y_1,\dots, y_m,z_1,\dots, z_\ell]
    \]
    where the maps send $x_i\mapsto x_i$ and $y_j\mapsto y_j,$ and whose morphisms are diagrams
    \[\begin{tikzcd}
    P\ar[r]\ar[d] &Q\ar[r]\ar[d] &R\ar[d]\\
    P'\ar[r] &Q'\ar[r] &R',
    \end{tikzcd}\]
    where $P\to Q\to R$ and $P\to Q'\to R'$ are in $\Poly^3.$

    Consider the functor
    \begin{align*}
        \mathbf{\Omega}' : \Poly^3&\to\mathsf{MOD}\\
        (P\to Q\to R)&\mapsto (R,\Omega^1_{Q/P}\otimes_Q R)
    \end{align*}
    and its animation $\bfL' : \Ani(\Poly^3)\to\MOD.$\footnote{Here, $\mathsf{MOD}$ is the $1$-category of pairs $(A,M)$ where $A$ is a commutative ring and $M$ is an $A$-module, and whose animation is $\MOD.$}

    For any map $k\to A\to B\in\Ani(\Poly^3),$ there is a pushout square of the form
    \[\begin{tikzcd}
        (k\to A\to B)\ar[r]\ar[d] &(k\to B\to B)\ar[d]\\
        (A\to A\to B)\ar[r] &(A\to B\to B)
    \end{tikzcd}\]
    (pushout squares of this form evidently exist when $k\to A\to B\in\Poly^3,$ so we may rewrite $(k\to A\to B)$ as a sifted colimit of objects in $\Poly^3,$ and then the sifted colimit of the pushout squares will yield the desired pushout square).

    We claim that applying the animated functor $\bfL'$ to the square gives a pushout square
    \[\begin{tikzcd}
        (B,\bfL_{A/k}\otimes_A B)\ar[r]\ar[d] &(B,\bfL_{B/k}\otimes_B B)\ar[d]\\
        (B,\bfL_{A/A}\otimes_A B)\ar[r] &(B,\bfL_{B/A}\otimes_B B)
    \end{tikzcd}\]
    and since every object of this diagram has animated ring $B,$ it corresponds to the desired pushout square
    \[\begin{tikzcd}
        \bfL_{A/k}\otimes_A B \ar[r]\ar[d] & \bfL_{B/k}\ar[d]\\
        0\ar[r] & \bfL_{B/A}.
    \end{tikzcd}\]
    This follows because the functor $\bfL'$ factors as
    \[
    \Ani(\Poly^3)\xrightarrow{\bfL'_1}\Ani(\mathsf{PolyPolyFree})\xrightarrow{\bfL'_2}\MOD
    \]
    where $\mathsf{PolyPolyFree}$ is the $1$-category whose objects are pairs $(P\to Q,F)$ such that $P\to Q$ is a map of the form $\ZZ[x_1,\dots, x_n]\to\ZZ[x_1,\dots, x_n,y_1,\dots, y_m]$ and $F\cong P^{\oplus r}\in\Mod_P.$ The factorization in question arises as the composition of the animations of the functors
    \[
    (P\to Q\to R)\mapsto (Q\to R,\Omega^1_{Q/P})
    \]
    and
    \[
    (Q\to R, F)\mapsto (R, F\otimes_Q R)
    \]
    (note that the animation of the compositions is the composition of the animations by \cite[Proposition~5.1.5(b)]{cesnavicius2023purity}).

    Finally, for a map $A\to B$ of animated rings, we have a commutative diagram
    \[\begin{tikzcd}
        &\cMod_A\ar[d]\ar[r,"-\otimes_A B"] &\cMod_B\ar[d]\\
        \Ani(\Poly^3)\ar[d,"{(P\to Q\to R)\mapsto (P\to Q)}"']\ar[r,"{\bfL'_1}"] &\Ani(\mathsf{PolyPolyFree})\ar[r,"{\bfL'_2}"]\ar[d,"{(Q\to R,F)\mapsto (Q,F)}"] &\MOD\\
        \Arr(\aRing)\ar[r,"{\bfL}"] &\MOD
    \end{tikzcd}\]
    where the functor $\Mod_A\to\Ani(\mathsf{PolyPolyFree})$ identifies $\Mod_A$ with the full sub-$\infty$-category of $\Ani(\mathsf{PolyPolyFree})$ of objects of the form $(A\to B,M)$ (via $M\mapsto (A\to B,M)$). These functors are readily seen to preserve colimits, which completes the argument.
\end{proof}

We are now ready to prove \ref{Kan extension represents derivations}.

\begin{proof}
    Consider the diagram
    \[\begin{tikzcd}
        \PolyPoly\ar[d]\ar[r,"\mathbf{\Omega}"] &\MOD\\
        \PolyCAlg\ar[d]\\
        \PolyAlg\ar[d]\\
        \Arr(\aRing)\ar[ruuu,"\bfL"']
    \end{tikzcd}\]
    The restriction of $\bfL : \Arr(\aRing)\to\MOD$ to $\PolyAlg$ (respectively $\PolyCAlg$) is the left Kan extension of $\mathbf{\Omega}$ to $\PolyAlg$ (respectively $\PolyCAlg$) by general facts about Kan extensions (see for example \cite[\href{https://kerodon.net/tag/0308}{Tag 0308}]{kerodon}).

    Now let $k\to k\otimes_{\ZZ}\ZZ[y_1,\dots, y_n]\in\PolyCAlg,$ and choose a system $P_\lambda\in\Poly$ of finitely generated polynomial rings over $\ZZ$ such that $\colim_{\lambda}P_{\lambda}\simeq k$ (note that $\colim_{\lambda}(P_{\lambda}\to P_{\lambda}\otimes_{\ZZ}\ZZ[y_1,\dots, y_n])\simeq (k\to k\otimes_{\ZZ}\ZZ[y_1,\dots, y_n])$).
    
    Using the fact that the left Kan extension is the unique colimit preserving functor $\PolyCAlg\to\MOD$ which restricts to $\mathbf{\Omega}$ on $\PolyPoly,$ we compute $\bfL(k\to k\otimes_{\ZZ}\ZZ[y_1,\dots, y_n]):$ we have
    \begin{align*}
        \bfL(k\to k\otimes_{\ZZ}\ZZ[y_1,\dots, y_n])&\simeq\colim_{\lambda}(P_\lambda\otimes_{\ZZ}\ZZ[y_1,\dots, y_n],\Omega^1_{P_\lambda\otimes_{\ZZ}\ZZ[y_1,\dots, y_n]/P_{\lambda}})\\
        &\simeq \colim_{\lambda} (P_{\lambda}\otimes_{\ZZ}\ZZ[y_1,\dots, y_n],P_\lambda\otimes_{\ZZ}\Omega^1_{\ZZ[y_1,\dots, y_n]})\\
        &\simeq \colim_{\lambda} \left((P_{\lambda},0)\coprod(\ZZ[y_1,\dots, y_n],\Omega^1_{\ZZ[y_1,\dots, y_n]})\right)\\
        &\simeq \left(\colim_{\lambda} (P_{\lambda},0)\right)\coprod(\ZZ[y_1,\dots, y_n],\Omega^1_{\ZZ[y_1,\dots, y_n]})\\
        &\simeq (k,0)\coprod(\ZZ[y_1,\dots, y_n],\Omega^1_{\ZZ[y_1,\dots, y_n]})\\
        &\simeq (k\otimes_{\ZZ}\ZZ[y_1,\dots, y_n],\Omega^1_{k\otimes_{\ZZ}\ZZ[y_1,\dots, y_n]/k})
    \end{align*}
    so that $\left.\bfL\right|_{\PolyCAlg}$ agrees with $\mathbf{\Omega}$ on $\PolyCAlg.$

    We can then form the diagram
    \[\begin{tikzcd}
        \PolyCAlg\ar[r]\ar[rd] &(\CRing/\aRing)\ar[r]\ar[d] &\Arr(\aRing)\ar[d]\ar[r] &\MOD\\
        &\CRing\ar[r] &\aRing,
    \end{tikzcd}\]
    where the vertical maps send $(k\to A)\mapsto k$ (i.e, they are the source map). Then \cite[4.3.3.9]{HTT} implies that forming the pullback along $\{k\}\to\CAlg$ produces the diagram
    \[\begin{tikzcd}
        \Poly_{k/}\ar[r]\ar[rd] &\aRing_{k/}\ar[r,"\sim"]\ar[d] &\aRing_{k/}\ar[d]\ar[r] &\MOD\\
        &\{k\}\ar[r,"\id"] &\{k\},
    \end{tikzcd}\]
    and the induced functors in the top row are given by the left Kan extensions of $\left.\mathbf{\Omega}\right|_{\Poly_{k/}}.$

    We now show that this functor satisfies the desired universal property. Let $k\to A$ be a map from a commutative ring to an arbitrary animated ring, and let $M$ be a connective $A$-module.
    
    Note first that
    \[
    \colim_{P\to A}(P,\Omega^1_{P/k})\simeq\colim_{P\to A}(A,\Omega^1_{P/k}\otimes_P A)\simeq (A,\colim_{P\to A}\Omega^1_{P/k}\otimes_P A),
    \]
    where the colimit is taken over finitely presented polynomial $k$-algebras $P$ mapping to $A.$ This is essentially no more than an explicit description of colimits in a coCartesian fibration, see for example the statement and proof of \cite[\href{https://kerodon.net/tag/06B2}{Tag 06B2}]{kerodon}. Intuitively, let $(B,N)\in\MOD.$ A morphism $(P,\Omega^1_{P/k})\to (B,N)$ is equivalent to the data of a map $P\to B$ and a map of $B$-modules $\Omega^1_{P/k}\otimes_P B\to N.$ A map $\colim (P,\Omega^1_{P/k})\to (B,N)$ is equivalent to the data of a compatible family of maps $P\to B$ and $\Omega^1_{P/k}\otimes_P B\to N$ as $P$ ranges over all the maps $P\to A$ in question. However, since $A\simeq\colim P,$ it follows that the maps $P\to B$ factor as $P\to A\to B,$ and thus the maps of modules also may be written $\Omega^1_{P/k}\otimes_P A\otimes_A B\to N.$ Therefore we have
    \[
    \Map_{\MOD}(\colim (P,\Omega^1_{P/k}),(B,N))\simeq\Map_{\MOD}(\colim (A,\Omega^1_{P/k}\otimes_P A),(B,N)),
    \]
    which shows that $\colim_{P\to A}(P,\Omega^1_{P/k})\simeq\colim_{P\to A}(A,\Omega^1_{P/k}\otimes_P A),$ as desired.
    
    Using the above, \cite[Proposition~2.50]{yaylali2022notes} completes the proof that $\LL_{A/k}$ satisfies the desired universal property (when $k\in\CRing$).

    Finally, let $k\to A$ be a general morphism of animated rings. Considering the composition $\ZZ\to k\to A,$ \ref{cotangent complex cofiber sequence} implies that we have a cofiber sequence
    \[
    \LL_{k/\ZZ}\otimes_k A\to \LL_{A/\ZZ}\to\LL_{A/k}
    \]
    of $A$-modules. We now appeal to \cite[Remarks~2.44,2.46]{yaylali2022notes} to complete the argument.
\end{proof}

In an analogous manner, we may also define functors
\begin{align*}
    \bfL^n : \Arr(\aRing)&\to\MOD\\
    (k\to A)&\mapsto (A, \wedge^n_A \LL_{A/k})
\end{align*}
encoding the $A$-module structure on the higher (derived) exterior powers of the cotangent complex for any nonnegative integer $n\geq 0,$ which arise as the animation of the functors
\begin{align*}
    \mathbf{\Omega}^n : \PolyPoly&\to\mathsf{MOD}\\
    (P\to Q)&\mapsto (Q, \Omega^n_{Q/P}).
\end{align*}

\begin{proposition}
    Let
    \[\begin{tikzcd}
        k\ar[r]\ar[d] & A\ar[d]\\
        k'\ar[r] & A'
    \end{tikzcd}\]
    be a commutative diagram of animated rings. Then there is a canonical morphism
    \[A'\otimes_{A} \LL_{A/k}\to \LL_{A'/k'}\]a
    which is an equivalence if $A'\simeq A\otimes_{k} k'.$
\end{proposition}
\begin{proof}
    The existence of the morphism follows from the fact that $\bfL$ is a functor with values in $\MOD,$ since a morphism $(A,M)\to (A',M')\in\MOD$ is given by a map $\phi, A\to A'$ of animated rings and a morphism $\phi^* M = M\otimes_{A,\phi}A'\to A'$ in $\cMod_{A'}.$

    Thus the square
    \[\begin{tikzcd}
        k\ar[r]\ar[d] & A\ar[d]\\
        k'\ar[r] & A'
    \end{tikzcd}\]
    induces a morphism $\bfL(k\to A)\to \bfL(k'\to A')$ which is $(A, \LL_{A/k})\to (A', \LL_{A'/k'}),$ where $\phi : A\to A'$ is the morphism given by the commutative square, so that $\phi^* \LL_{A/k}\simeq A'\otimes_{A} \LL_{A/k}.$

    That this morphism is an equivalence is a consequence of the corresponding fact for $\Omega^1$ and the fact that $\bfL$ is the left Kan extension of $\Omega^1.$
\end{proof}

The cotangent complex provides us with derived ``differential criteria'' for \'etaleness and smoothness of animated rings.

\begin{proposition}[\protect{\cite[Theorem~2.2.2.6]{toen2006homotopical}}]
    Let $f : k\to A$ be a morphism of animated rings.
    \begin{enumerate}
        \item The morphism $f$ is smooth if and only if the $A$-module $\LL_{A/k}$ is finite projective and $\pi_0 A$ is of finite presentation over $\pi_0 k.$
        \item The morphism $f$ is \'etale if and only if $\LL_{A/k}\simeq 0$ and $\pi_0 A$ is of finite presentation over $\pi_0 k.$
    \end{enumerate}
\end{proposition}

A crucial fact about the cotangent complex is that it is an fpqc sheaf. 

\begin{proposition}[\protect{\cite[Corollary~2.7]{bhatt2012completionsderivedrhamcohomology}}]
    Let $k\to A$ be a map of (simplicial) commutative rings. If $A_\bullet$ denotes the \v{C}ech conerve of $k\to A,$ then there is an equivalence
    \[
    \Tot(\wedge^i_k \LL_{A_\bullet/k})\simeq \wedge^i_k \LL_{A/k}
    \]
    in $\cMod_k.$ That is, $\wedge^i_k L_{-/k} : \CRing_k\to\cMod_k$ satisfies fpqc descent.
\end{proposition}

We require a generalization of this statement.

\begin{theorem}\label{cotangent complex is fpqc sheaf on morphisms}
    For all nonnegative integers $n,$ the functors
    \[
    \bfL^n : \Arr(\dAff)^{\op}\to\QC
    \]
    are fpqc sheaves.\footnote{See \cite[Section 7.1]{BLPO_HKR} for a definition of the $\tau$-topology on $\Arr(\cC)$ for an $\infty$-category $\cC$ with a topology $\tau.$}
\end{theorem}
\begin{proof}
    For $n=0,$ the result is trivial.
    
    Let $k\to A$ be a map of animated rings, and consider a diagram
    \[\begin{tikzcd}
        A\ar[r] & B\\
        k\ar[u]\ar[r] & k'\ar[u]
    \end{tikzcd}\]
    such that $A\to B$ and $k\to k'$ are faithfully flat. We need to show that the canonical map
    \[
    (\Spec A, \LL_{A/k})\to \lim_n (\Spec B^n, \LL_{B^n/k^n})
    \]
    is an equivalence, where $B^\bullet = \Cech(A\to B)$ and $k^\bullet = \Cech(k\to k')$ are the \v{C}ech conerves of the indicated maps. Since $\Spec B\to\Spec A$ is an effective epimorphism, $\colim_n \Spec B^n\simeq\Spec A,$ so we must show that the canonical map
    \[
    \LL_{A/k}\to\lim_n \LL_{B^n/k^n}
    \]
    is an equivalence, where all of the objects are considered as $A$-modules.

    Let $A^\bullet = A\otimes_{k} k^\bullet,$ so that we obtain a diagram
    \[\begin{tikzcd}
        A\ar[r] & A^\bullet\ar[r] & B^\bullet\\
        k\ar[u]\ar[r] & k^\bullet\ar[u]\ar[ru]
    \end{tikzcd}\]
    where the square is a pushout. By construction, we have that $A^\bullet = \Cech(A\to A\otimes_k k').$ Now, consider the cofiber sequence obtained by considering the composition of morphisms $k^\bullet\to A^\bullet\to B^\bullet.$ We obtain
    \[
    \LL_{A^\bullet/k^\bullet}\otimes_{A^\bullet} B^\bullet\to \LL_{B^\bullet/k^\bullet}\to \LL_{B^\bullet/A^\bullet},
    \]
    and it follows from base change that $\LL_{A^\bullet/k^\bullet}\otimes_{A^\bullet} B^\bullet\simeq \LL_{A/k}\otimes_{A} B^\bullet.$ Faithfully flat descent then implies that the totalization of this term is $\LL_{A/k},$ so it suffices to prove that the totalization of $\LL_{B^\bullet/A^\bullet}$ vanishes.

    Next, consider the cofiber sequence
    \[
    \LL_{A^\bullet/A}\otimes_{A^\bullet}B^\bullet\to \LL_{B^\bullet/A}\to \LL_{B^\bullet/A^\bullet}
    \]
    obtained from the composition $A\to A^\bullet\to B^\bullet.$ By \cite[Corollary 2.7]{bhatt2012completionsderivedrhamcohomology} we find that the totalization of $\LL_{B^\bullet/A}$ vanishes, so we are reduced to showing $\Tot\left(\LL_{A^\bullet/A}\otimes_{A^\bullet}B^\bullet\right)\simeq 0.$

    For this, note that it will suffice to show that the totalization of 
    $A'\otimes_A \LL_{A^\bullet/A}\otimes_{A^\bullet}B^\bullet$ vanishes, since $A\to A'$ is faithfully flat. Then we have $A'\otimes_A \LL_{A^\bullet/A}\simeq \LL_{\Cech(A'\to A'\otimes_A A')^\bullet/A'},$ but since $A'\to A'\otimes_A A'$ admits a section, there is a homotopy between the identity on $\LL_{\Cech(A'\to A'\otimes_A A')^\bullet/A'}$ and the zero map, which implies that the totalization of this object vanishes. Then, taking the tensor product of this homotopy up to $B^\bullet$ produces a homotopy between the identity on $\LL_{\Cech(A'\to A'\otimes_A A')^\bullet/A'}\otimes_{A^\bullet} B^\bullet$ and the zero map, so the totalization of this object vanishes as well, which completes the proof for $n=1.$

    For higher $n,$ we proceed via induction analogously to the argument of \cite[Theorem 3.1]{BMS}. Suppose that we know $\Tot(\wedge^m_A \LL_{B^\bullet/k^\bullet})\simeq \wedge^m_A \LL_{A/k}$ and $\Tot(\wedge^m_A \LL_{B^\bullet/A^\bullet})\simeq 0$ for all $0\leq m < n.$ Applying $\wedge^n_A$ to the cofiber sequence
    \[
    \LL_{A^\bullet/k^\bullet}\otimes_{A^\bullet} B^\bullet\to \LL_{B^\bullet/k^\bullet}\to \LL_{B^\bullet/A^\bullet}
    \]
    induces a filtration
    \[
    \wedge^n_A\left(\LL_{A^\bullet/k^\bullet}\otimes_{A^\bullet} B^\bullet\right)\simeq  F^0\to F^1\to\dots\to F^{n-1}\to F^n\simeq\wedge^n_A \LL_{B^\bullet/k^\bullet}
    \]
    with $i$th associated graded piece given by
    \[
    F^{i-1}\to F^i\to \left(\wedge^{n-i}_A\left(\LL_{A^\bullet/k^\bullet}\otimes_{A^\bullet} B^\bullet\right)\right)\otimes_A\left(\wedge^i_A \LL_{B^\bullet/A^\bullet}\right)
    \]
    As before, since $\LL_{A^\bullet/k^\bullet}\otimes_{A^\bullet} B^\bullet\simeq \LL_{A/k}\otimes_A B^\bullet,$ Bhatt's result applies, and we are reduced to showing that $\Tot(\wedge^n \LL_{B^\bullet/A^\bullet})\simeq 0.$ To show this, we proceed as we did in the $n=1$ case: use the cofiber sequence
    \[
    \LL_{A^\bullet/A}\otimes_{A^\bullet}B^\bullet\to \LL_{B^\bullet/A}\to \LL_{B^\bullet/A^\bullet}
    \]
    and consider the induced filtration
    \[
    \wedge^n_A\left(\LL_{A^\bullet/A}\otimes_{A^\bullet}B^\bullet\right)\simeq  F^0\to F^1\to\dots\to F^{n-1}\to F^n\simeq\wedge^n_A \LL_{B^\bullet/A}.
    \]
    Again using Bhatt's result that $\Tot\left(\wedge^n_A \LL_{B^\bullet/A}\right)\simeq 0,$ we are reduced to showing that
    \[\Tot\left(\wedge^n_A\left(\LL_{A^\bullet/A}\otimes_{A^\bullet}B^\bullet\right)\right)\simeq 0.\]
    The argument for this proceeds as in the $n=1$ case, which completes the proof.
\end{proof}

\subsection{The Cotangent Complex of a Morphism of Derived Stacks}

\begin{mydef}
    Let $f : \cX\to\cY$ be a morphism of derived stacks, let $x : \Spec A\to\cX$ be a point, and let $M\in\cMod_A^{\cn}.$ Then we define the \emph{module of derivations at $x$} to be
    \[\Der_x(f,M) := \Fib(\cX(A\oplus M)\to \cX(A)\times_{\cY(A)}\cY(A\oplus M)),\]
    where the fiber is taken at the point of $\cX(A)\times_{\cY(A)}\cY(A\oplus M))$ corresponding to $x : \Spec A\to \cX$ and $\Spec A\to\Spec (A\oplus M)\to\cY$ corresponding to the canonical projection $A\oplus M\to A.$
\end{mydef}

\begin{mydef}
    Let $f : \cX\to\cY$ be a morphism of derived stacks. We say that $L\in\QC(\cX)$ is a \emph{cotangent complex for $f$} if $L$ is $(-n)$-connective for some $n\geq 0,$ and for any $x : \Spec A\to\cX,$ there is an equivalence
    \[\Der_x(f,M)\simeq\Map_{\cMod_A^{\cn}}(x^*L, M).\]
\end{mydef}

For the time being, if $L$ is a cotangent complex for a morphism $f : \cX\to\cY$ of derived stacks, we will denote $\LL^{\der}_{\cX/\cY} := L.$ We begin by outlining important properties of the cotangent complex.

\begin{lemma}[\protect{\cite[Lemma 1.4.1.16]{toen2006homotopical}}]\label{cotangent complex properties 1}
    Let $f : \cX\to\cY$ be a morphism of derived stacks over $\Spec k$.
    \begin{enumerate}
        \item If $\cX$ and $\cY$ admit cotangent complexes relative to $\Spec k,$ then $f$ admits a cotangent complex, and for any $x : \Spec A\to\cX,$ there is a cofiber sequence
        \[
        (f\circ x)^*\LL^{\der}_{\cY/\Spec k}\to x^* \LL^{\der}_{\cX/\Spec k}\to x^* \LL^{\der}_{\cX/\cY}.
        \]
        \item If $f$ admits a cotangent complex, then for any morphism of derived stacks $\cZ\to\cY$ and any $x : \Spec A\to\cX\times_{\cY}\cZ,$ we have
        \[
        (x\circ\pr_1)^* \LL^{\der}_{\cX/\cY}\simeq x^* \LL^{\der}_{\cX\times_{\cY}\cZ/\cZ}
        \]
        \item If for any morphism $x : \Spec A\to X$ the projection $\pr : X\times_{Y, f\circ x}\Spec A\to\Spec A$ admits a cotangent complex, then $f$ admits a cotangent complex and further we have
        \[
        x^* \LL^{\der}_{\cX/\cY}\to (x,\id_{\Spec A})^* \LL^{\der}_{\cX\times_{\cY}\Spec A/\Spec A}
        \]
        \item If for any $x : \Spec A\to\cX$ the stack $\cX\times_{\cY,f\circ x}\Spec A$ admits a cotangent complex, then $f$ has a cotangent complex and we have a cofiber sequence
        \[
        \id_{\Spec A}^* \LL^{\der}_{\Spec A/\Spec k}\to (x,\id_{\Spec A})^* \LL^{\der}_{\cX\times_{\cY}\Spec A/\Spec k}\to x^* \LL^{\der}_{\cX/\cY}.
        \]
    \end{enumerate}
\end{lemma}

Note that the cotangent complex of an arbitrary morphism of derived stacks may not exist. However, the fundamental cofiber sequence of cotangent complexes generalizes to compositions of morphisms of derived stacks, and similarly to the affine case, it may be used to show that a cotangent complex exists if other cotangent complexes are known to exist.

\begin{proposition}\label{cotangent complex cofiber sequence}
    Suppose $f : \cX\to\cY$ and $g : \cY\to\cZ$ are morphisms of derived stacks, and suppose that $g$ admits a cotangent complex $\LL^{\der}_{\cY/\cZ}.$ Then $f$ admits a cotangent complex if and only if $g\circ f$ does. Moreover, if all three cotangent complexes exist, then there is a cofiber sequence
    \[
    f^* \LL^{\der}_{\cY/\cZ}\to \LL^{\der}_{\cX/\cZ}\to \LL^{\der}_{\cX/\cY}.
    \]
\end{proposition}
\begin{proof}
    Let $x : \Spec A\to\cX$ be arbitrary and let $M\in\cMod_A^{\cn}.$ The result follows from the fact that there is a pullback diagram
    \[\begin{tikzcd}
        \Der_x(f,M)\ar[r]\ar[d] &\Der_x(g\circ f,M)\ar[d]\\
        \{\triv\}\ar[r] &\Der_{f\circ x}(g,M),
    \end{tikzcd}\]
    where $\triv$ is the trivial derivation together with (the $\infty$-categorical analogue of) \cite[1.2.11.3]{toen2006homotopical}.
\end{proof}

Importantly, cotangent complexes of monomorphisms always exist, and are trivial.

\begin{lemma}
    Let $j : \cX\to\cY$ be a monomorphism of derived stacks over $k.$ Then $j$ admits a cotangent complex and $\LL^{\der}_{\cX/\cY}\simeq 0.$
\end{lemma}
\begin{proof}
    We will show that $0$ satisfies the desired universal property; i.e., that $\Der_x(f,M)\simeq 0$ for any $x : \Spec A\to \cX.$
    
    Consider the canonical map
    \[
    \cX(A\oplus M)\to\cX(A)\times_{\cY(A)}\cY(A\oplus M).
    \]
    We claim that this map is itself a monomorphism. Indeed, consider the diagram
    \[\begin{tikzcd}
        \cX(A)\ar[d]\ar[r] &\cX(A)\ar[d] &\ar[l]\cX(A\oplus M)\ar[d]\\
        \cX(A)\ar[r] &\cY(A) &\ar[l]\cY(A\oplus M)
    \end{tikzcd}\]
    from which the canonical map is induced. Then clearly $\cX(A)\to\cX(A)$ and $\cX(A\oplus M)\to\cY(A\oplus M)$ are both monomorphisms, as is the diagonal of $\cX(A)\to\cY(A).$ Since monomorphisms are preserved under pullbacks, the claim follows.

    This implies that $\Der_x(f,M)$ is either empty or contractible (see \cite[5.5.6.8]{HTT}), but as there is always the trivial derivation, it is nonempty. Since $\Der_x(f,M)$ is contractible for any $x$ and $M,$ it follows that $0$ is a representing object. This completes the proof, as the pullback of $0\in\QC(\cX)$ along any $x : \Spec A\to\cX$ is $0.$
\end{proof}

Most importantly, a geometric morphism of derived stacks (thus any morphism between geometric derived stacks) admits a cotangent complex, and knowledge about the geometricity of said morphism produces information about the connectivity of the cotangent complex. We will not prove this result, but refer to either \cite[2.2.3.3]{toen2006homotopical}\footnote{In \cite{toen2006homotopical}, the proof is split up into many parts, and 2.2.3.3 only constitutes one specific portion of the result}, \cite[4.45]{AG2014}, or \cite[3.76]{yaylali2022notes}.

\begin{theorem}\label{existence of cotangent complex for geometric morphism}
    If $f : \cX\to\cY$ is an $n$-geometric morphism, then $L^{\der}_f$ exists and is $(-n)$-connective. If $f$ is moreover of finite presentation, then $f$ is smooth if and only if $\LL^{\der}_f$ has Tor-amplitude in $[-n-1,0].$
\end{theorem}

For the next result, we require a categorical lemma.

\begin{lemma}\label{cofiber square}
    Suppose that
    \[\begin{tikzcd}
        X\ar[d]\ar[r] & X'\ar[d]\ar[r] & X''\ar[d]\\
        Y\ar[d]\ar[r] & Y'\ar[d]\ar[r] & Y''\ar[d]\\
        Z      \ar[r] & Z'      \ar[r] & Z''
    \end{tikzcd}\]
    is a commutative diagram in a stable $\infty$-category where each row and column is a (co)fiber sequence. Then there is a canonical equivalence
    \[
    \cofib(X'\oplus_X Y\to Y')\simeq (Z'\oplus_{Z} X[1])\oplus X[2].
    \]
\end{lemma}
\begin{proof}
    \begin{align}
        \cofib(X'\oplus_X Y\to Y')&\simeq (0\oplus_0 0)\oplus_{(X'\oplus_X Y)} (Y'\oplus_{Y'}Y') \nonumber\\
        &\simeq Z'\oplus_{(0\oplus_X Y')} Y''\label{eq:1}\\
        &\simeq (Z'\oplus Y'')\oplus_{(Z'\oplus_X Y'')} (Z'\oplus_{Y'} Y'')\label{eq:2}\\
        &\simeq (Z'\oplus Y'')\oplus_{(Z'\oplus Y''\oplus X[1])} (Z'\oplus_{Y'} Y'')\label{eq:3}\\
        &\simeq (Z'\oplus_{Y'}Y'')\oplus X[2]\label{eq:4}\\
        &\simeq (Z'\oplus_{Y} 0)\oplus X[2]\label{eq:5}\\
        &\simeq (Z'\oplus_Z X[1])\oplus X[2]\label{eq:6}.
    \end{align}
    \ref{eq:1} follows from consideration of the commutative diagram
    \[\begin{tikzcd}
        0 & \ar[l] X'\ar[r] & Y'\\
        0\ar[u]\ar[d] & \ar[l] X\ar[u]\ar[d]\ar[r] & Y'\ar[u]\ar[d]\\
        0 & \ar[l] Y\ar[r] & Y'
    \end{tikzcd}\]
    \ref{eq:2} follows from 
    \[\begin{tikzcd}
        Z' & \ar[l] 0\ar[r] & Y''\\
        Z'\ar[u]\ar[d] & \ar[l] X\ar[u]\ar[d]\ar[r] & Y''\ar[u]\ar[d]\\
        Z' & \ar[l] Y'\ar[r] & Y''
    \end{tikzcd}\]
    In \ref{eq:3}, $Z'\oplus_X Y''\simeq Z'\oplus Y''\oplus X[1]$ follows from
    \[\begin{tikzcd}
        X\ar[d]\ar[r] & 0\ar[d]\ar[r] & Y''\ar[d]\\
        0\ar[d]\ar[r] & X[1]\ar[d]\ar[r] & X[1]\oplus Y''\ar[d]\\
        Z'\ar[r] & Z'\oplus X[1]\ar[r] & (Z'\oplus X[1])\oplus_{X[1]} (X[1]\oplus Y'')
    \end{tikzcd}\]
    where every square is coCartesian, together with the equivalence
    \[(Z'\oplus X[1])\oplus_{X[1]} (X[1]\oplus Y'')\simeq Z'\oplus Y''\oplus X[1]\]
    coming from
    \[\begin{tikzcd}
        Z' & \ar[l] 0\ar[r] & Y''\\
        0\ar[u]\ar[d] & \ar[l] 0\ar[u]\ar[d]\ar[r] & 0\ar[u]\ar[d]\\
        X[1] & \ar[l] X[1]\ar[r] & X[1]
    \end{tikzcd}\]
    
    \ref{eq:4} follows from 
    \[\begin{tikzcd}
        Z'\oplus Y'' & \ar[l] Z'\oplus Y''\ar[r] & Z'\oplus_{Y'} Y''\\
        0\ar[u]\ar[d] & \ar[l] 0\ar[u]\ar[d]\ar[r] & 0\ar[u]\ar[d]\\
        0 & \ar[l] X[1]\ar[r] & 0
    \end{tikzcd}\]
    
    \ref{eq:5} follows from considering
    \[\begin{tikzcd}
        Y\ar[d]\ar[r] & Y'\ar[d]\ar[r] & Z'\ar[d]\\
        0\ar[r] & Y''\ar[r] & Z'\oplus_{Y'}Y''
    \end{tikzcd}\]
    and \ref{eq:6} follows from
    \[\begin{tikzcd}
        Y\ar[d]\ar[r] & Z\ar[d]\ar[r] & Z'\ar[d]\\
        0\ar[r] & X[1]\ar[r] & Z'\oplus_{Z} X[1].
    \end{tikzcd}\]
\end{proof}

\begin{proposition}\label{cotangent complex of fiber product}
    Suppose that $\cX\to\cS$ and $\cY\to\cS$ are morphisms of derived stacks which admit cotangent complexes. Then $\cX\times_{\cS}\cY\to\cS$ also admits a cotangent complex, and
    \[
    \LL^{\der}_{\cX\times_{\cS}\cY/\cS}\simeq p_1^* \LL^{\der}_{\cX/\cS}\oplus p_2^* \LL^{\der}_{\cY/\cS}.
    \]
\end{proposition}
\begin{proof}
    Consider the commutative diagram
    \[\begin{tikzcd}
        0\ar[r]\ar[d] & p_1^* \LL^{\der}_{\cX/\cS}\ar[r]\ar[d] & p_1^* \LL^{\der}_{\cX/\cS}\ar[d]\\
        p_2^* \LL^{\der}_{\cY/\cS}\ar[r]\ar[d] & \LL^{\der}_{\cX\times_{\cS}\cY/\cS}\ar[r]\ar[d] & \LL^{\der}_{\cX\times_{\cS}\cY/\cY}\ar[d]\\
        p_2^* \LL^{\der}_{\cY/\cS}\ar[r] & \LL^{\der}_{\cX\times_{\cS}\cY/\cX}\ar[r] & E
    \end{tikzcd}\]
    in which each row and column is a cofiber sequence. In particular, the middle row is the canonical cofiber sequence associated to
    \[
    \cX\times_{\cS}\cY\xrightarrow{p_2}\cY\to\cS
    \]
    and the middle column is the canonical cofiber sequence associated to
    \[
    \cX\times_{\cS}\cY\xrightarrow{p_1}\cX\to\cS.
    \]
    It follows from applying \ref{cotangent complex properties 1} that both $p_2^* \LL^{\der}_{\cY/\cS}\to \LL^{\der}_{\cX\times_{\cS}\cY/\cX}$ and $p_1^* \LL^{\der}_{\cX/\cS}\to \LL^{\der}_{\cX\times_{\cS}\cY/\cY}$ are equivalences, and thus $E\simeq 0.$ We then apply \ref{cofiber square} to deduce the result.
\end{proof}

\begin{proposition}\label{cotangent complex and diagonal}
    Let $f : \cX\to\cY$ be a map of geometric derived stacks. Then there is a canonical equivalence
    \[
    \LL^{\der}_{\cX/\cY}\to\LL^{\der}_{\cX/\cX\times_{\cY}\cX}[-1]
    \]
    where the map $\cX\to\cX\times_{\cY}\cX$ in question is the diagonal.
\end{proposition}
\begin{proof}
    Consider the cofiber sequence
    \[
    \Delta^*\LL^{\der}_{\cX\times_{\cY}\cX/\cY}\to\LL^{\der}_{\cX/\cY}\to\LL^{\der}_{\cX/\cX\times_{\cY}\cX}
    \]
    induced by the composition
    \[
    \cX\xrightarrow{\Delta}\cX\times_{\cY}\cX\to\cY.
    \]
    By \ref{cotangent complex of fiber product}, there is a canonical equivalence
    \[
    \LL^{\der}_{\cX/\cY}\oplus\LL^{\der}_{\cX/\cY}\to\Delta^*\LL_{\cX\times_{\cY}\cX/\cY},
    \]
    and the map
    \[
    \LL^{\der}_{\cX/\cY}\oplus\LL^{\der}_{\cX/\cY}\to\Delta^*\LL_{\cX\times_{\cY}\cX/\cY}\to\LL^{\der}_{\cX/\cY}
    \]
    is the universal codiagonal map. This means we have a cofiber sequence
    \[
    \LL^{\der}_{\cX/\cY}\oplus\LL^{\der}_{\cX/\cY}\to\LL^{\der}_{\cX/\cY}\to\LL^{\der}_{\cX/\cX\times_{\cY}\cX}
    \]
    and a morphism of cofiber sequences
    \[\begin{tikzcd}
        \LL^{\der}_{\cX/\cY}\oplus\LL^{\der}_{\cX/\cY}\ar[d]\ar[r] &\LL^{\der}_{\cX/\cY}\ar[d]\ar[r] &\LL^{\der}_{\cX/\cX\times_{\cY}\cX}\ar[d]\\
        \LL^{\der}_{\cX/\cY}\oplus\LL^{\der}_{\cX/\cY}\ar[r] &\LL^{\der}_{\cX/\cY}\ar[r] &\LL^{\der}_{\cX/\cY}[-1]
    \end{tikzcd}\]
    induced by the identities on the first two components. The third map is therefore also an equivalence.
\end{proof}

We now show that the cotangent complex of a morphism of derived geometric stacks admits another description as a Kan extension.

\begin{construction}
    Consider the cotangent complex of a morphism of affine schemes, viewed as a functor
    \begin{align*}
        \bfL : \Arr(\dAff)^\op&\to\QC\\
        (\Spec A\to\Spec k)&\mapsto (\Spec A,\LL_{A/k}).
    \end{align*}
    Let $\bfL^{\Kan}$ denote a right Kan extension of this functor as in the diagram
    \[\begin{tikzcd}
        \Arr(\dAff)^{\op}\ar[r,"\bfL"]\ar[d]&\QC\\
        \Arr(\dStk)^{\op},\ar[ru,"\bfL^{\Kan} := \Ran\bfL"']
    \end{tikzcd}\]
    and let $\bfL^{\Kan}(\cX\to\cY) = (\cX,\LL^{\Kan}_{\cX/\cY}).$
\end{construction}

\begin{remark}
    Explicitly unwinding how limits/colimits in a coCartesian fibration are computed, this means that for $f : \cX\to\cY\in\Arr(\dStk),$
    \[
    \LL^{\Kan}_{\cX/\cY}\simeq\lim_{\Spec(B)/\Spec(A)\in\dAff_{\cX/\cY}}p^B_* \LL_{B/A},
    \]
    where $\dAff_{\cX/\cY}$ is the full subcategory of $\Arr(\dStk)_{/f}$ spanned by $X\to Y$ (over $f$) with $X$ and $Y$ affine, $p^B : \Spec B\to\cX,$ and the limit is calculated in $\QC(\cX).$ 
\end{remark}

Note that by the universal property of the right Kan extension, there exists an essentially unique comparison morphism $\rho_f : \LL^{\der}_f\to \LL^{\Kan}_f$ for any $f : \cX\to\cY\in\dStk$ admitting a cotangent complex. In particular, there exists a comparison morphism $\rho_f : \LL^{\der}_f\to \LL^{\Kan}_f$ for any morphism of geometric derived stacks by \ref{existence of cotangent complex for geometric morphism}.

\begin{proposition}
    Let $f : \cX\to \cY$ be a morphism of $n$-geometric derived $k$-stacks. Then 
    \[
    \rho_f : \LL^{\der}_{\cX/\cY}\to \LL^{\Kan}_{\cX/\cY}
    \]
    is an equivalence.
\end{proposition}
\begin{proof}
    Consider a Cartesian square of derived stacks
    \[\begin{tikzcd}
        \cX_A\ar[r,"p_A"]\ar[d] &\cX\ar[d,"f"]\\
        \Spec A\ar[r] &\cY.
    \end{tikzcd}\]
    By base change for the cotangent complex, we find that the canonical map
    \[
    (p_A)^* L^{\der}_{\cX/\cY}\to L^{\der}_{\cX_A/A}
    \]
    is an equivalence. Varying the diagram over all $\Spec A\to \cY,$ the collection $\{p_A : \cX_A\to\cX\}$ induces an effective epimorphism to $\cX,$ and thus
    \[
    \lim_{A} (p_A)_* (p_A)^* \LL^{\der}_{\cX/\cY}\simeq \LL^{\der}_{\cX/\cY},
    \]
    so it suffices to prove the result when $\cY\simeq\Spec A$ is affine.

    Fix a ground ring $k.$ \cite[Corollary 2.7]{bhatt2012completionsderivedrhamcohomology} shows that the functor
    \begin{align*}
        \Gamma_k(-,\LL_{-/k}) = \Gamma_k\circ\left.\bfL\right|_{\dAff_k^{\op}} : \dAff_k^{\op}&\to\cMod_k\\
        (\Spec A\to\Spec k)&\mapsto \Gamma_k(\Spec A, \LL_{A/k})
    \end{align*}
    is an fpqc sheaf. However, the same proof also shows that the variant of this functor
    \begin{align*}
        \bfL_{k} = \left.\bfL\right|_{\dAff_k^{\op}} : \dAff_k^{\op}&\to\QC\\
        (\Spec A\to\Spec k)&\mapsto (\Spec A, \LL_{A/k})
    \end{align*}
    is also an fpqc sheaf. In particular, $\bfL_{k}$ factors through $\QC^{\Aff}_k\hookrightarrow\QC$ (here $\QC_k^{\Aff}$ denotes the full sub-$\infty$-category of $\QC_k$ consisting of objects $(\cX\to\Spec k,\cF)$ with $\cX$ affine), and the restriction functor $\cMod_A\to\cMod_k$ is conservative for any map of animated rings $k\to A,$ so that a faithfully flat map of rings $A\to B$ with \v{C}ech conerve $\check{C}(B/A)_\bullet$ satisfies $\lim_n (p_n)_*\LL_{\check{C}(B/A)_n/A}\simeq 0$ in $\cMod_A,$\footnote{Rather than just in $\cMod_k.$} where $p_n : \Spec(\check{C}(B/A)^n)\to\Spec A$ is the structure morphism.

    This implies that the functor
    \begin{align*}
        \bfL^{\Kan}_{k} : \dStk_k^{\op}&\to\QC\\
        (\cX\to\Spec k)&\mapsto (\cX, \LL^{\Kan}_{\cX/k})
    \end{align*}
    is an fpqc sheaf. As a corollary, we find that for a derived geometric stack $\cX,$ $\LL^{\Kan}_{\cX/k}\simeq\lim_{p : \Spec A\to\cX} p_* \LL_{A/k}$ where the limit is taken over maps $p : \Spec A\to\cX$ which are smooth, or even that $\LL^{\Kan}_{\cX/k}\simeq\lim_n (p_n)_* \LL^{\Kan}_{U_n/k},$ where $p : U\to\cX$ is a smooth effective epimorphism and $p_\bullet : U_\bullet\to\cX$ is its \v{C}ech nerve.

    Now, we will show that $\LL^{\der}_{\cX/A}$ can be computed by the same limit. Let $p : U\to \cX$ be a smooth $n$-atlas of $\cX,$\footnote{More precisely, let $\{f_i : U_i\to\cX\}$ be a smooth $n$-atlas of $\cX$ and let $U = \coprod_i U_i$ with induced effective epimorphism $U\to\cX.$} and let $p_\bullet : U_\bullet\to\cX$ be its \v{C}ech nerve.

    Consider the canonical cofiber sequence
    \[
    p_n^* \LL^{\der}_{\cX/A}\to \LL^{\der}_{U_n/A}\to \LL^{\der}_{U_n/\cX}.
    \]
    Pushing forward along $p_n$ and taking the limit, it follows from the fact that $\colim_n U_n\simeq\cX$ that we have a cofiber sequence
    \[
    \LL^{\der}_{\cX/k}\to \lim_n (p_n)_* \LL^{\der}_{U_n/k}\to\lim_n (p_n)_* \LL^{\der}_{U_n/\cX},
    \]
    so it suffices to prove that $\lim_n (p_n)_* \LL^{\der}_{U_n/\cX}\simeq 0.$ 

    However, by further base change along $U\to\cX,$ we reduce to the case when the simplicial derived stack $U_\bullet\to\cX$ admits a section, so that the cosimplicial object $(p_n)_* \LL^{\der}_{U_n/\cX}$ admits a contracting homotopy. It then follows formally that $\lim_n (p_n)_* \LL^{\der}_{U_n/\cX}\simeq 0,$ and we are finished.
\end{proof}

\begin{warning}
    Despite the above proposition, for an arbitrary composition of derived stacks $\cX\xrightarrow{f}\cY\xrightarrow{g}\cZ,$ it may not be the case that
    \[
    f^* \LL^{\Kan}_{\cY/\cZ}\to \LL^{\Kan}_{\cX/\cZ}\to \LL^{\Kan}_{\cX/\cY}
    \]
    forms a cofiber sequence. The source of this problem is that $\LL^{\der}_{\cX/\cY}$ or $\LL^{\der}_{\cY/\cZ}$ may not exist, and in this situation $\LL^{\Kan}$ does not necessarily inherit the standard properties of $\LL^{\der}.$
\end{warning}

Finally, we obtain our desired generalization of \cite[Corollary~2.7]{bhatt2012completionsderivedrhamcohomology}.

\begin{cor}
    The functor
    \[
    \bfL^{\Kan} : \Arr(\dStk)^{\op}\to\QC
    \]
    is an fpqc sheaf.
\end{cor}
\begin{proof}
    This follows from \ref{cotangent complex is fpqc sheaf on morphisms}, together with the fact that the right Kan extension of a sheaf is a sheaf.
\end{proof}

\begin{cor}
    The functor
    \begin{align*}
        \bfL^{\der} : \Arr_{\mathrm{geom}}(\dStk)^{\op}&\to\QC\\
        (\cX\to\cS)&\mapsto (\cX,\LL_{\cX/\cS})
    \end{align*}
    is an fpqc sheaf, where $\Arr_{\mathrm{geom}}(\dStk)$ is the full sub-$\infty$-category of $\Arr(\dStk)$ spanned by the geometric morphisms.
\end{cor}

We end by noting that the constructions above specialize to the usual cotangent complex when the map of stacks $\cX\to\cY$ is given by a map of classical algebraic stacks.

\begin{theorem}
    Suppose that $\cX\to\cY$ is a map of classical algebraic stacks. Then the derived cotangent complex of $\cX\to\cY$ agrees with the classical cotangent complex of $\cX\to\cY,$ in the sense that the ``homotopy category'' functor
    \[h : \QC(\cX)\to h\QC(\cX)\simeq\N(\sfD_{\qc}(\cX))\]
    sends the derived cotangent complex to the classical cotangent complex.
\end{theorem}

\begin{proof}
    Because $\bfL^{\Kan}$ is defined as a right Kan extension, for any $\cX\to\cY,$ we obtain an essentially unique morphism $\delta : \LL^{\Kan}_{\cX/\cY}\to \LL^{\cl}_{\cX/\cY}.$ However, because $\bfL^{\Kan}$ agrees with $\bfL^{\der}$ for geometric stacks, base change and \'etale invariance of the latter implies this morphism is an equivalence after pullback to sufficiently small affine schemes \'etale over $\cX,$ so that the map is an equivalence in $\sfD_{\qc}(\cX).$
\end{proof}
\chapter{Logarithmic Geometry}

\section{Classical Logarithmic Geometry}

The ideas here are due originally to Fontaine and Illusie, and were first formally recorded in \cite{KatoLogStructures}.

\subsection{Commutative Monoids}

One of the foundational objects in logarithmic geometry is that of the commutative monoid. We first summarize the basic definitions and results on commutative monoids which we will need in the rest of this article. For a more thorough development of the theory of commutative monoids, along with proofs of any results we omit, we recommend \cite[Chapter~I]{OgusLectures} (unless otherwise indicated).

\begin{mydef}
    A \emph{monoid} is a triple $(M,\star,e_M),$ where $M$ is a set, $\star : M\times M\to M$ is an associative binary operation on $M,$ and $e_M$ is a two-sided identity element of $M$ for the operation $\star.$ A \emph{homomorphism of monoids} $(M,\star,e_M)\to (M',\star',e_{M'})$ is a function $\theta : M\to M'$ such that $\theta(e_M) = e_{M'}$ and $\theta(m \star m') = \theta(m)\star'\theta(m').$ 
\end{mydef}

We will often denote a monoid $(M,\star,e_M)$ simply by $M$ when the binary operation and unit are understood. For us, monoids will always be commutative unless otherwise stated, and by ``monoid'' we will always mean ``commutative monoid'' unless otherwise stated. The category of commutative monoids and monoid homomorphisms will be denoted by $\CMon.$ We will often write monoids additively and denote the binary operation by $+$ and the unit by $0.$

The category of monoids is pointed, with zero object given by the trivial monoid $M = \{0\}.$ Perhaps the most important monoid is the monoid $\NN = \{0,1,2,\dots\}$ of the natural numbers under addition, and $\NN$ is the free monoid on a single generator -- for any monoid $M$ and any $m\in\NN,$ there is a unique homomorphism of monoids $\NN\to M$ which maps $1\mapsto m.$

The following proposition is easy to prove.
\begin{proposition}
    The forgetful functor
    \begin{align*}
        \CMon&\to\Set\\
        (M,\star, e_M)&\mapsto M
    \end{align*}
    admits a left adjoint, given by
    \begin{align*}
        \Set&\to\CMon\\
        X&\mapsto \NN^X := \Hom_{\Set}(X,\NN)
    \end{align*}
    where the binary operation on $\NN^X$ is given by pointwise addition of functions. We call $\NN^X$ the \emph{free monoid on $X.$}
\end{proposition}

Using free monoids, we may define the notion of a finitely generated monoid: a monoid $M$ is \emph{finitely generated} if there exists a finite set $X$ and a surjection $\NN^X\to M.$

Importantly, the category $\CMon$ is complete and cocomplete, and binary (hence finite) products and coproducts coincide.

Clearly, every abelian group is in particular also a monoid, so we obtain a forgetful functor
\[
U : \Ab\to\CMon.
\]
There are two natural ways we may associate to a monoid $M$ an abelian group.

\begin{mydef}
    Let $M$ be a monoid.
    \begin{enumerate}
        \item The \emph{units of $M,$} denoted $M^\times,$ is the submonoid of $M$ consisting of elements of $M$ which are invertible.
        \item The \emph{group completion of $M,$} denoted $M^{\gp},$ is defined to be
        \[
        M^{\gp} := M\times M/\sim,
        \]
        where $(m,m')\sim (n,n')$ if $m + n' + k = n + m' + k$ for some $k\in M.$
    \end{enumerate}
\end{mydef}

From the definition, it is clear that $M^\times$ is an abelian group, and a standard verification shows that $M^{\gp}$ is indeed an abelian group with group operation given by $[m,m'] + [n,n'] = [m+n,m'+n']$ (the unit is $[0,0]$ and the inverse of $[m,m']$ is given by $[m',m]$). We may think of the class $[m,m']$ as representing $m - m'.$ It is also simple to check that both constructions provide functors $\CMon\to\Ab.$

We can also check that given an abelian group $A$ and a monoid $M,$ any map of monoids $A\to M$ in fact factors uniquely through the inclusion $M^{\times}\to M,$ as the image of any $a\in A$ must be invertible in $M.$

On the other hand, any map $M\to A$ factors uniquely through the map
\begin{align*}
    M&\to M^{\gp}\\
    m&\mapsto [m,0].
\end{align*}
Indeed, let $\theta : M\to A$ be any morphism. We may define $\theta^{\gp}$ to be the map
\begin{align*}
    M^{\gp}&\to A\\
    [m,m']&\mapsto \theta(m) - \theta(m').
\end{align*}
If $[m,m']= [n,n'],$ let $k\in M$ be such that $m + n' + k = n + m' + k.$ We check that
\begin{align*}
    \theta^{\gp}([m,m']) &= \theta(m) - \theta(m')\\
    &= \theta(m) + \theta(n'+ k) - \theta(n'+k) - \theta(m')\\
    &= \theta(m + n' + k) - \theta(n' + k) - \theta(m')\\
    &= \theta(n + m' + k) - \theta(n' + k) - \theta(m')\\
    &= \theta(n) - \theta(n')\\
    &= \theta^{\gp}([n,n']),
\end{align*}
so that $\theta^{\gp}$ is well-defined. It is clear that it is a morphism of groups, and we can verify uniqueness as follows: suppose that $\theta' : M^{\gp}\to A$ is another morphism which factors $\theta : M\to A$ through the given map $M\to M^{\gp}.$ We must have
\[
\theta'([m,m']) = \theta'([m,0]) + \theta'([0,m]) = \theta(m) + \theta'([0,m']),
\]
and since $[0,0] = [m,m],$ we must also have
\[
0 = \theta'([0,0]) = \theta'([m,m]) = \theta(m) + \theta'([0,m]),
\]
and thus $\theta'([0,m]) = -\theta(m).$ This shows that $\theta' = \theta^{\gp},$ as desired.

Going forward, we will write simply $[m]$ for the class $[m,0]$ in $M^{\gp},$ so that a general element of $M^{\gp}$ can be written $[m] - [m']$ for some $m,m'\in M.$

From the above, we deduce the following proposition:

\begin{proposition}
    The forgetful functor $U : \Ab\to\CMon$ admits both a left and right adjoint, given by $M\mapsto M^{\gp}$ and $M\to M^{\times},$ respectively.
    \[\begin{tikzcd}
        \CMon\ar[r, shift left=2,"{(-)^{\gp}}"]\ar[r, shift right=2,"{(-)^\times}"'] &\ar[l]\Ab : U
    \end{tikzcd}\]
    The unit and counit of the adjunction $((-)^{\gp}\dashv U)$ are given by
    \begin{align*}
    M&\to M^{\gp}\\
    m&\mapsto [m,0]
    \end{align*}
    and
    \[
    A\xrightarrow{\id_A} A
    \]
    respectively, and the unit and counit of $(U\dashv (-)^\times)$ are given by
    \[
    A\xrightarrow{\id_A} A^\times = A
    \]
    and
    \begin{align*}
        M^\times&\to M\\
        m&\mapsto m
    \end{align*}
    respectively.
\end{proposition}

Thus, $M\to M^{\gp}$ is initial among maps from $M$ to abelian groups, and $M^\times\to M$ is terminal among maps from abelian groups to $M.$

There is also a close relationship between commutative monoids and commutative rings. The forgetful functor $(-,\,\cdot\,) : \CRing\to\CMon$ which sends a commutative ring to its underlying multiplicative monoid admits a left adjoint given by the \emph{monoid algebra} construction. In particular, let $P$ be a commutative monoid (written additively). Let $\ZZ[P]$ denote the free $\ZZ$-module on the set $P,$ and write $e^p$ for the basis element corresponding to $p\in P.$ We endow $\ZZ[P]$ with the multiplication determined by setting $e^{p}e^{q} := e^{p+q}.$

\begin{remark}
    We may identify $\ZZ[P]$ with the ring of functions $f : P\to\ZZ$ with finite support, with addition given by
    \[
    (f + g)(p) := f(p) + g(p)
    \]
    and multiplication given by
    \[
    (fg)(p) = \sum_{a+b = p}f(a)g(b).
    \]
    The isomorphism is given by sending $e^p$ (for any $p\in P$) to the function which maps $p\mapsto 1$ and $p'\mapsto 0$ for any $p'\neq p.$
\end{remark}

The monoid algebra construction also extends to base rings other than $\ZZ.$ In particular, if we set $k[P] := k\otimes_{\ZZ}\ZZ[P],$ then we also have an adjunction
\[\begin{tikzcd}
    k[-] : \CMon\ar[r,shift left] &\CRing_k\ar[l,shift left] : (-,\,\cdot\,).
\end{tikzcd}\]
Of course, $k[P]$ is also isomorphic to the ring of functions $P\to k$ with finite support, analogous to the case of $\ZZ.$

We often restrict our attention to certain subclasses of monoids, as general monoids can become too unwieldy. The most important for our purposes are the following.

\begin{mydef}
    Let $M$ be a monoid.
    \begin{enumerate}
        \item We say $M$ is \emph{integral} if it is cancellative; i.e., if $m + n = m + n'$ implies $n = n'.$
        \item We say $M$ is \emph{fine} if it is finitely generated and integral.
        \item We say $M$ is \emph{saturated} if given $m\in M^{\gp}$ such that $n\cdot m\in M$ for some $n\in\NN,$ then $m\in M.$
        \item We say $M$ is \emph{fs} if it is both fine and saturated.
    \end{enumerate}
\end{mydef}

Integral monoids play a crucial role in classical logarithmic geometry. However, it is not necessarily the case that if $\theta : P\to M$ and $\psi : P\to N$ are maps of integral monoids, that the pushout $M\oplus_P N$ is itself integral.

\begin{mydef}
    A homomorphism $\theta : P\to M$ of integral monoids is \emph{integral} if it satisfies the following equivalent conditions.
    \begin{enumerate}
        \item Whenever $m_1,m_2\in M$ and $p_1,p_2\in P$ satisfy
        \[
        \theta(p_1) + m_1 = \theta(p_2) + m_2,
        \]
        there exist $m'\in M$ and $p_i'\in P$ such that
        \[
        m_i = \theta(p_i') + m' \quad\textrm{and}\quad p_1 + p_1' = p_2 + p_2'.
        \]
        \item For every homomorphism $P\to P'$ of integral monoids, the pushout $P'\oplus_P M$ is integral.
    \end{enumerate}
\end{mydef}

\begin{example}
    Let $P$ and $Q$ be integral monoids, and let $\theta : P\to P\oplus Q$ be the canonical map. Then $\theta$ is integral. This is because
    \[
    Q'\oplus_{P}(P\oplus Q)\simeq Q'\oplus Q
    \]
    is integral.
\end{example}

\begin{mydef}
    Let $P$ be a monoid. An \emph{$P$-set} is a set $S$ and a map of monoids $\rho : P\to\End_{\Set}(S)$ (note that $\End_{\Set}(S)$ is a monoid under composition). If $S$ is an $P$-set, a \emph{basis} of $S$ is a map $i : T\to S$ of sets such that the map
    \begin{align*}
        P\times T&\to S\\
        (p,t)&\mapsto \rho_p(i(t))
    \end{align*}
    is a bijection. We say that $S$ is a \emph{free} $P$-set if a basis of $S$ exists.
\end{mydef}

In particular, any morphism of monoids $\theta : P\to M$ allows us to consider $M$ as a $P$-set via the action of $P$ induced by $\theta.$

\begin{example}
    Let $P$ and $Q$ be monoids, and consider the canonical map $P\to P\oplus Q.$ Then $P\oplus Q$ is a free $P$-set, with basis $Q\to P\oplus Q.$
\end{example}

Just as with modules over rings, we may speak of when an $M$-set $S$ is flat. The definition may be found as \cite[Definition~I.4.5.2]{OgusLectures}, but we will make use of the following equivalent characterization.

\begin{mydef}
    Let $M$ be a monoid and let $S$ be an $M$-set. We say $S$ is \emph{flat} if $\ZZ[S]$ is a flat $\ZZ[P]$-module.
\end{mydef}

\begin{example}
    If $S$ is a free $M$-set, then $S$ is flat. Indeed, a basis $T\to S$ provides a basis for the $\ZZ[M]$-module $\ZZ[S],$ so that $\ZZ[S]$ is a free $\ZZ[M]$-module.
\end{example} 

\begin{mydef}
    Let $\theta : P\to M$ be a morphism of monoids. We say that $\theta$ is \emph{logarithmic} if the canonical map
    \[
    \theta^{-1}(M^\times)\to M^\times
    \]
    is an equivalence.
\end{mydef}

\begin{mydef}
    Let $\theta : N\to M$ be a map of monoids. We say $\theta$ is \emph{exact} if the diagram
    \[\begin{tikzcd}
        N\ar[d,"\theta"]\ar[r] &N^{\gp}\ar[d,"\theta^{\gp}"]\\
        M\ar[r] &M^{\gp}
    \end{tikzcd}\]
    is Cartesian.
\end{mydef}

\begin{mydef}
    Let $\theta : N\to M$ be a map of monoids. We define the \emph{pre-exactification} $\theta^{\ex}$ of $\theta$ to be the canonical map $N^{\ex}_{\theta} := M\times_{M^{\gp}} N^{\gp}\to M.$ We say that $\theta^{\ex} : N^{\ex}_{\theta}\to M$ is the \emph{exactification} of $\theta$ if $\theta^{\ex}$ is exact.
\end{mydef}

\begin{warning}
    In general, it may not be the case that $N^{\ex}_{\theta}$ is itself exact. However, in the following special cases, the pre-exactification as defined above is indeed exact.
\end{warning}

\begin{lemma}[\protect{\cite[Lemma~3.8]{Rognes_2009}}]
    Let $\theta : N\to M$ be a map of monoids. Then the canonical map $\theta^{\ex} : N^{\ex}_{\theta}\to N^{\gp}$ induces an isomorphism $(N^{\ex}_{\theta})^{\gp}\to N^{\gp}$ if either of the following hold:
    \begin{enumerate}
        \item $M$ is integral, or
        \item $\theta^{\gp}$ is surjective.
    \end{enumerate}
    In particular, in either of these cases, the map $\theta^{\ex} : N^{\ex}_{\theta}\to M$ is exact.
\end{lemma}
\begin{proof}
    In both cases, the map $(N^{\ex}_{\theta})^{\gp}\to N^{\gp}$ is surjective. Indeed, the canonical map $N\to N^{\gp}$ factors through $N^{\ex}_{\theta},$ so that passing to group completions provides a factorization
    \[
    N^{\gp}\to (N^{\ex}_{\theta})^{\gp}\to (N^{\gp})^{\gp}\cong N^{\gp}
    \]
    of the identity map.

    Consider an element $[(m,g)]-[(m',g')]\in (N^{\ex})^{\gp}$ which is in the kernel of $(N^{\ex}_{\theta})^{\gp}\to N^{\gp}.$ It follows immediately that $g = g'$ in $N^{\gp},$ and thus we have $[m] = \theta^{\gp}(g) = [m']$ in $M^{\gp}.$
    
    If $M$ is integral, the canonical map $M\to M^{\gp}$ is an injection, so that $m = m'$ in $M$ and thus $[(m,g)]-[(m',g')] = 0$ in $(N^{\ex})^{\gp}.$

    Suppose on the other hand that $\theta^{\gp}$ is surjective. Then $[m] = [m']$ implies that $m + k = m' + k$ for some $k\in M.$ Choose $h\in N^{\gp}$ such that $\theta^{\gp}(h) = [k].$ Then
    \[
    (m,g) + (k,h) = (m',g') + (k,h)
    \]
    in $N^{\ex},$ which implies that $[(m,g)] - [(m',g')] = 0,$ as desired.

    Finally, note that if $N^{\ex}_{\theta}\to N^{\gp}$ induces an equivalence $(N^{\ex}_{\theta})^{\gp}\to N^{\gp},$ then the map $N^{\ex}_{\theta}\to N^{\gp}$ exhibits $N^{\gp}$ as the group completion of $N^{\ex}_{\theta},$ whence the desired diagram is Cartesian.
\end{proof}

\begin{example}
    In particular, note that the exactification of the relative codiagonal $\sigma : M\oplus_P M\to M$ of a morphism $\theta : P\to M$ of monoids is always exact, since $\sigma$ is surjective. Moreover, we have a canonical isomorphism
    \[
    (M\oplus_P M)^{\ex}_{\sigma} \to M\oplus M^{\gp}/P^{\gp}
    \]
    induced by the shear map
    \begin{align*}
        s : M\oplus_P M&\to M\oplus M^{\gp}/P^{\gp}\\
        [m,m'] &\mapsto (m,[m']),
    \end{align*}
    which we may see from the commutative diagram
    \[\begin{tikzcd}
        M\oplus_P M\arrow[dd,"\sigma"']\arrow[rrr]\arrow[rd,"s"]&&&M^{\gp}\oplus_{P^{\gp}}M^{\gp}\arrow[d,"s^{\gp}"']\arrow[dd,"\sigma^{\gp}",bend left=80]\\
        &M\times M^{\gp}/P^{\gp}\arrow[rr]\arrow[d,"\pr_1"']&&M^{\gp}\oplus M^{\gp}/P^{\gp}\arrow[d,"\pr_1"']\\
        M\arrow[r,"\id"]&M\arrow[rr]&&M^{\gp}
    \end{tikzcd}\]
    in which the unlabeled arrows are the canonical maps from a monoid to its group completion, and in which the lower right square is evidently Cartesian (note that $s^{\gp}$ is an isomorphism as a consequence of the $9$-lemma).
\end{example}

\subsection{Affine Logarithmic Geometry}

In the classical theory of logarithmic geometry, log schemes and their morphisms are modeled by the spectra of prelog rings\footnote{At least, if we restrict our attention to only fine or fs log schemes.}. In this section, we develop the basic theory of prelog rings which will be sufficient for our applications.

\begin{mydef}
    The category $\PreLog$ of \emph{prelogarithmic rings} (or \emph{prelog rings}) is the comma category $(\id_{\CMon}\downarrow (-,\,\cdot\,)).$ Explicitly, the objects of $\PreLog$ are triples $(k,P,\beta),$ where $k\in\CRing,$ $P\in\CMon,$ and $\beta : P\to(k,\,\cdot\,)$ is a morphism of monoids\footnote{The data of $\beta$ is equivalent to the data of a morphism $\overline{\beta} : \ZZ[\beta]\to k$ of rings by adjunction.}, and a morphism of prelog rings $f : (k,P,\beta)\to (A,M,\alpha)$ is a pair $(\underline{f},f^\flat),$ where $\underline{f} : k\to A$ is a map of commutative rings and $f^{\flat} : P\to M$ is a map of monoids such that the diagram
    \[\begin{tikzcd}
        P\ar[r,"f^{\flat}"]\ar[d,"\beta"] &M\ar[d,"\alpha"]\\
        (k,\,\cdot\,) \ar[r,"\underline{f}"] &(A,\,\cdot\,)
    \end{tikzcd}\]
    commutes.
\end{mydef}

Note that the category $\PreLog$ comes equipped with two forgetful functors:
\begin{align*}
    U_{\CMon} : \PreLog&\to\CMon\\
    (k,P,\beta) &\mapsto P,
\end{align*}
and
\begin{align*}
    U_{\CRing} : \PreLog&\to\CRing\\
    (k,P,\beta)&\mapsto k.
\end{align*}
We define the category of \emph{prelogarithmic structures on $k$} (or \emph{prelog structures on $k$}) to be the fiber of $U_{\CRing}$ over $\{k\}\to \CRing;$ explicitly, a prelogarithmic structure on $k$ is a pair $(P,\beta)$ such that $(k,P,\beta)$ is a prelog ring. A morphism $(P,\beta)\to (P',\beta')$ of prelog structures on $k$ is a morphism $f^\flat : P\to P'$ of monoids such that $\beta'\circ f^\flat= \beta.$\footnote{The data of a morphism $(P,\beta)\to (P',\beta')$ of prelog structures on $k$ is equivalent to the data of a morphism of prelog rings $f : (k,P,\beta)\to (k,P',\beta')$ such that the underlying map $\underline{f} : k\to k$ is the identity.}

\begin{remark}
    In \cite[Definition III.1.2.3]{OgusLectures}, what we call prelog rings are instead called log rings -- we use the former terminology to conform more closely with the terminology of \cite{BLPO_HKR}.
\end{remark}

\begin{lemma}\label{initial prelog ring under monoid}
    Let $(A,M)$ be a prelog ring, and let $f^\flat : M\to N$ be a morphism of monoids. Then the category of morphisms of prelog rings $(A,M)\to (A',M')$ under $(A,M)$ equipped with a morphism of monoids $N\to M'$ and a commutative diagram
    \[\begin{tikzcd}
        M\ar[r]\ar[rd] &N\ar[d]\\
        &M'
    \end{tikzcd}\]
    admits an initial object, given by
    \[
    (A,M)\to (A\otimes_{\ZZ[M]}\ZZ[N],N).
    \]
    Thus, any morphism of prelog rings $(A,M)\to (A',M')$ equipped with a given factorization $M\to N\to M'$ of the morphism of monoids canonically admits a factorization through $(A\otimes_{\ZZ[M]}\ZZ[N],N).$
\end{lemma}
\begin{proof}
    A morphism of prelog rings $(A,M)\to (A',M')$ is equivalent to the data of a commutative diagram of rings
    \[\begin{tikzcd}
        \ZZ[M]\ar[d]\ar[r] &\ZZ[M']\ar[d]\\
        A\ar[r] & A'
    \end{tikzcd}\]
    where the top morphism is induced by a morphism of monoids $M\to M',$ by adjunction. The result is then clear.
\end{proof}

One class of prelog rings which is particularly nice is that of free prelog rings and free prelog algebras, which are defined as follows. Let $(k,P,\beta)$ be a prelog ring. Composing the two adjunctions above with the further forgetful functors to $\Set,$ we obtain an adjunction
\[\begin{tikzcd}[row sep=small]
    (X,S)\ar[r,maps to] &(k[X,\NN^S], P\oplus \NN^S,\beta\oplus \id)\\
    F_{(k,P,\beta)} : \Set\times\Set\ar[r,shift left] &\PreLog_{(k,P,\beta)}\ar[l,shift left] : U\\
    (A,M)&\ar[l,maps to](A,M,\alpha).
\end{tikzcd}\]

\begin{remark}
    Our description above omits the structure maps. For the right adjoint, they are irrelevant, but for the left adjoint, the map $(k,P,\beta)\to F_{(k,P,\beta)}(X,S)$ is given by the canonical map $k\to k[X,\NN^S]$ of rings and the canonical inclusion $P\to P\oplus \NN^S.$
\end{remark}

\begin{mydef}\label{free prelog rings/algebras definition}
    We call an object $F_{(k,P)}(X,S)$ in the image of $F_{(k,P)}$ a \emph{free $(k,P)$-algebra.} If $(k,P) = (\ZZ,\{1\})$ is the initial prelog ring, we write $F_{(k,P)} = F.$ We will call a free $(k,P)$-algebra simply a \emph{free prelog ring.}
\end{mydef}

Free prelog rings and free prelog algebras over free prelog rings will be critical to our investigation of the relationship between the two notions of logarithmic Hochschild homology. Let us define $\PolyPreLog$ to be the subcategory of $\Arr(\PreLog)$ consisting of free prelog algebras over free prelog rings; i.e., morphisms of the form $F(X,S)\to F_{F(X,S)}(Y,T).$\footnote{Note that this is slightly different from the category of polynomial prelog rings considered in \cite[Remark~2.11]{BLPO_HKR}. In particular, for us, the base prelog ring is also free.}

\subsection{Global Logarithmic Geometry}

There is also a rich global theory of logarithmic geometry generalizing the affine theory discussed previously. For the most part, we do not need to consider general logarithmic schemes, but for completeness we recall some of the important aspects of the theory. A brief account of the theory can be found in the original source, \cite{KatoLogStructures}, and a more complete account in great generality can be found in \cite{OgusLectures}.

\begin{mydef}
    Let $X$ be a scheme.
    \begin{enumerate}
        \item A \emph{prelogarithmic structure} (or \emph{prelog structure}) on $X$ is a map $\alpha : \cM\to(\OX,\,\cdot\,)$ of sheaves of monoids on the \'etale site of $X.$
        \item A prelogarithmic structure $\alpha : \cM\to (\OX,\,\cdot\,)$ on $X$ is a \emph{logarithmic structure} (or \emph{log structure}) if $\alpha^{-1}(\OX^{\times})\to\OX^{\times}$ is an isomorphism. 
        \item A morphism of prelog structures on $X$ $(\alpha : \cM\to(\OX,\,\cdot\,))\to (\alpha' : \cM'\to (\OX,\,\cdot\,))$ is a homomorphism $f^\flat : \cM\to \cM'$ of sheaves of monoids such that $\alpha'\circ f^\flat = \alpha.$
        \item The category $\PreLogStr_X$ is the category whose objects are prelog structures $\alpha$ on $X$ and whose morphisms are morphisms of prelog structures as defined above.
        \item The category $\LogStr_X$ is the full subcategory of $\PreLogStr_X$ spanned by those prelog structures which are log structures.
    \end{enumerate}
\end{mydef}

As in the affine case, there is an adjunction between log and prelog structures on $X.$
\begin{proposition}[\protect{\cite[Proposition~III.1.1.3]{OgusLectures}}]
    Let $X$ be a scheme.
    \begin{enumerate}
        \item The inclusion functor $\LogStr_X\to\PreLogStr_X$ admits a left adjoint, denoted $(\alpha : \cM\to\OX)\mapsto (\alpha^a : \cM^a_{\alpha}\to\OX)$ (or simply by $(\cM^a\to\OX)$ if $\alpha$ is understood).
        \item The identity $\OX\to\OX$ is the final object of both $\PreLogStr_X$ and $\LogStr_X.$ The initial object of $\PreLogStr_X$ is $\{1\}\to\OX,$ and the initial object of $\LogStr_X$ is $\OX^\times\to\OX.$
        \item The categories $\PreLogStr_X$ and $\LogStr_X$ both admit pushouts, the formation of which is compatible with the forgetful functor $\LogStr_X\to\PreLogStr_X.$
    \end{enumerate}
\end{proposition}

We can also transfer (pre)log structures along a morphism of schemes.

\begin{mydef}
    Let $f : X'\to X$ be a morphism of schemes.
    \begin{enumerate}
        \item Let $\alpha' : \cM'\to\cO_{X'}$ be a prelog structure on $X'.$ Let $f^{\log}_*\alpha' : f^{\log}_*\cM'\to\OX$ be defined by the Cartesian diagram
        \[\begin{tikzcd}
            f^{\log}_*\cM'\ar[d]\ar[r,"{f^{\log}_*\alpha'}"] &\OX\ar[d,"f^\sharp"]\\
            f_*\cM'\ar[r] &f_*\cO_{X'}.
        \end{tikzcd}\]
        \item Let $\alpha : \cM\to\OX$ be a prelog structure on $X.$ Let $f^*_{\pre}\alpha : f^*_{\pre}\cM\to\cO_{X'}$ be the prelog structure on $X'$ given by the composition
        \[
        f^*_{\pre}\cM := f^{-1}\cM\xrightarrow{f^{-1}\alpha}f^{-1}\OX\to\cO_{X'}.
        \]
        \item Let $\alpha : \cM\to\OX$ be a prelog structure on $X.$ Let $f^*_{\log}\alpha : f^*_{\log}\cM\to\cO_{X'}$ be the log structure on $X'$ associated to $f^{-1}_{\pre}\alpha.$
    \end{enumerate}
\end{mydef}

As we saw in the affine case, if $\alpha' : \cM'\to\cO_{X'}$ is a log structure on $X',$ then $f^{\log}_*\alpha' : f^{\log}_*\cM'\to\OX$ is a log structure on $X,$ and we have a commutative diagram of adjunctions
\[\begin{tikzcd}
\PreLogStr_X\arrow[rr,"f^*_{\pre}",shift left]\arrow[dd,"(-)^a"',shift right]&&\PreLogStr_{X'}\arrow[dd,"(-)^a"',shift right]\arrow[ll,"f_*^{\log}",shift left]\\
&&\\
\LogStr_X\arrow[rr,"f^*_{\log}",shift left]\arrow[uu,"U"',shift right]&&\LogStr_{X'}\arrow[uu,"U"',shift right]\arrow[ll,"f^{\log}_*",shift left]
\end{tikzcd}\]
(meaning that the diagram of left adjoints commutes and the diagram of right adjoints commutes).

\begin{mydef}
    The category $\PreLogSch$ of \emph{prelogarithmic schemes} (or \emph{prelog schemes}) is the category whose objects $(X,\cM,\alpha)$ are schemes $X$ equipped with a prelog structure $\alpha : \cM\to\OX$ on $X,$ and whose set of morphisms $(X',\cM',\alpha')\to (X,\cM,\alpha)$ consists of pairs $(f,f^\flat)$ such that $f : X'\to X$ is a morphism of schemes and $f^\flat : \cM\to f_*\cM'$ is a morphism of sheaves of monoids such that the diagram
    \[\begin{tikzcd}
        \cM\ar[d,"\alpha"']\ar[r,"f^\flat"] & f_*\cM'\ar[d,"f_*\alpha'"]\\
        \OX\ar[r,"f^\sharp"] &f_*\cO_{X'}
    \end{tikzcd}\]
    commutes. The category $\LogSch$ of \emph{logarithmic schemes} (or \emph{log schemes}) is the full subcategory of $\PreLogSch$ whose objects are those $(X,\cM,\alpha)$ such that $\alpha : \cM\to\OX$ is a log structure on $X.$
\end{mydef}

\begin{remark}
    The data of the morphism $f^\flat$ is equivalent to any of the following:
    \begin{itemize}
        \item A morphism $f^{-1}\cM\to\cM'$ such that the appropriate diagram commutes,
        \item A morphism $f^*_{\pre}\cM\to\cM'$ of prelog structures on $X',$
        \item A morphism $f^*_{\log}\cM\to\cM'$ of log structures on $X',$ in the case of $\LogSch,$
        \item A morphism $\cM\to f^{\log}_*\cM'$ of prelog structures on $X.$
    \end{itemize}
\end{remark}

Going forward, we will often denote a log scheme simply by $X,$ and write $\underline{X}$ for its underlying scheme. Given a morphism $f : X\to Y$ of (pre)log schemes, we will let $\underline{f} : \underline{X}\to\underline{Y}$ be the underlying morphism of schemes.

There is an evident functor $\PreLogSch\to\Sch$ given by sending a prelog scheme $X$ to its underlying scheme $\underline{X}$ and a morphism $f = (\underline{f},f^\flat)$ of prelog schemes to its underlying morphism of schemes $\underline{f}.$ We then obtain adjunctions:
\[\begin{tikzcd}
    \PreLogSch\ar[r,shift left] &\Sch\ar[l,shift left,"\triv"]
\end{tikzcd}\]
where the left adjoint $\PreLogSch\to\Sch$ is the functor which sends a prelog scheme to its underlying scheme, and the right adjoint $\triv : \Sch\to\PreLogSch$ sends a scheme $X$ to the trivial prelog scheme $(X,\{1\}).$

We also have an adjunction
\[\begin{tikzcd}
    \LogSch\ar[r,shift left] &\Sch\ar[l,shift left,"\triv"]
\end{tikzcd}\]
where the left adjoint is also given by the forgetful functor, and the right adjoint is the functor $X\mapsto (X,\OX^\times).$ Both right adjoints are fully faithful embeddings. If $X$ is an ordinary scheme, $X$ will also denote either the trivial prelog scheme $(X,\{1\})$ or the trivial log scheme $(X,\OX^\times)$ (depending on whether we are working with prelog schemes or log schemes).

\begin{remark}
    The forgetful functors $\PreLogSch\to\Sch$ $\LogSch\to\Sch$ also admit left adjoints, both given by $X\mapsto (X,\OX).$
\end{remark}

Now let us describe how to associate a log scheme to a prelog ring.

\begin{mydef}
    Let $(A,M,\alpha)$ be a prelog ring. We define $\Spec(A,M)$ to be the log scheme whose underlying scheme is $\Spec A$ and whose log structure $\cM_{\Spec(A,M)}$ is the logification of $M_{\Spec A}\to\mathcal{O}_{\Spec A}$ (induced by $\alpha : M\to A$).
\end{mydef}

\begin{mydef}
    Let $f : X\to Y$ be a morphism of prelog schemes. We say that $f$ is \emph{strict} if the canonical morphism
    \[
    \underline{f}^*_{\log}((\cM_Y)^a)\to(\cM_X)^a
    \]
    is an equivalence of log structures on $X.$
\end{mydef}

Intuitively, a map $f : X\to Y$ is strict if it does not encode any logarithmic information besides the log structure $\cM_X.$ We will see that many constructions, when applied to strict morphisms of log schemes, produce the analogous construction for the underlying morphism of schemes.

\begin{remark}
    In classical logarithmic algebraic geometry, one typically only defines strict morphisms of log schemes, in which case the logifications in the above definition are redundant. 
\end{remark}

\begin{proposition}\label{strictness pullback equivalence}
    Let $f : X\to Y$ be a morphism of log schemes. Then $f$ is strict if and only if the diagram
    \[\begin{tikzcd}
        X\ar[d]\ar[r,"f"] &Y\ar[d]\\
        \underline{X}_{\triv}\ar[r,"\underline{f}"] &\underline{Y}_{\triv}
    \end{tikzcd}\]
    is Cartesian. Here, $\underline{X}_{\triv}$ denotes the trivial log structure on $\underline{X}.$
\end{proposition} 

If we allow ourselves to work locally, we can often avoid the sheaf of monoids of a log structure, and instead work with ordinary monoids. This can be accomplished through the use of charts, which we now define.

\begin{mydef}
    Let $X$ be a scheme, and let $\alpha : \cM\to\OX$ be a log structure on $X.$ A \emph{chart} of $\alpha$ is any of the equivalent data:
    \begin{itemize}
        \item A morphism of monoids $\beta : M\to\Gamma(X,\cM)$ such that the associated map $(M_X)^a\to\cM$ is an isomorphism,
        \item A strict morphism of prelog schemes $(X,\cM,\alpha)\to(\Spec\ZZ[M],M_{\Spec\ZZ[M]}),$
        \item A strict morphism of log schemes $(X,\cM,\alpha)\to\Spec(\ZZ[M],M).$
    \end{itemize}

    If $f : (X',\cM',\alpha')\to (X,\cM,\alpha)$ is a morphism of log schemes, a \emph{chart} of $f$ is a tuple $(M,\beta,M',\beta',\theta),$ where:
    \begin{itemize}
        \item $M$ and $M'$ are monoids,
        \item $\beta : M\to\Gamma(X,\cM)$ and $\beta' : M'\to\Gamma(X',\cM')$ are charts of $\alpha$ and $\alpha',$ respectively, and
        \item $\theta : M\to M'$ is a map of monoids such that the diagram
        \[\begin{tikzcd}
            M\ar[r,"\theta"]\ar[d,"\beta"] &M'\ar[d,"\beta'"]\\
            \Gamma(X,\cM)\ar[r] &\Gamma(X',\cM')
        \end{tikzcd}\]
        commutes.
    \end{itemize}
\end{mydef}

\begin{remark}
    We can also view a chart of a morphism $f : X'\to X$ of log schemes as a commutative diagram
    \[\begin{tikzcd}
        X'\ar[r,"f"]\ar[d] &X\ar[d]\\
        \Spec(\ZZ[M'],M')\ar[r] &\Spec(\ZZ[M],M)
    \end{tikzcd}\]
    of log schemes such that the vertical maps are strict and the map $\Spec(\ZZ[M'],M')\to\Spec(\ZZ[M],M)$ is induced by a map $\theta : M\to M'$ of monoids.
\end{remark}

\begin{mydef}
    Let $X$ be a log scheme, and let $(M,\beta : M\to\Gamma(X,\cM))$ be a chart of $X.$ We say that the chart is \emph{coherent} (respectively \emph{integral}, \emph{fine}, \emph{saturated}, \emph{fs}) if $M$ is finitely generated (respectively integral, fine, saturated, fs).
\end{mydef}

The existence of a chart, at least locally, for a log structure is of great importance, so we have the following definition to be able to talk about this situation more concisely.

\begin{mydef}
    A log structure $\alpha : \cM\to\OX$ is \emph{quasi-coherent} if \'etale locally on $X,$ there exists a monoid $M$ and a chart $\beta : M\to\Gamma(X,\cM).$ The log structure is \emph{coherent} if \'etale locally on $X,$ there exists a finitely generated monoid $M$ and a chart $\beta : M\to\Gamma(X,\cM).$
\end{mydef}

Given an arbitrary log scheme $X,$ the question still remains whether its log structure is (quasi)-coherent. This is not the case in general, but if we restrict the class of log structures we consider, the situation becomes much more reasonable. To that end, we make the following definition.

\begin{mydef}
    Let $X = (\underline{X},\cM)$ be a log scheme, and let $x\in X.$ We say that $\cM$ is
    \begin{enumerate}
        \item \emph{integral at $x$} if the stalk $\cM_x$ is integral.
        \item \emph{integral} if it is integral at all $x\in X.$
        \item \emph{fine} if $\cM$ is coherent and integral.
        \item \emph{saturated at $x$} if the stalk $\cM_x$ is saturated.
        \item \emph{saturated} if it is saturated at all $x\in X.$
        \item \emph{fs} if it is fine and saturated.
    \end{enumerate}
    Let $f : X=(\underline{X},\cM)\to Y=(\underline{Y},\cN)$ be a map of log schemes. We say that $f$ is \emph{integral} if the map of monoids $f^\flat_x : (f^*_{\log}\cN_Y)_x\to\cM_x$ is integral for all $x\in X.$
\end{mydef}

We will also abusively say that a log scheme $X$ is itself integral, fine, saturated, or fs if the sheaf of monoids of its log structure is so.

\begin{proposition}[\protect{\cite[Proposition~II.1.1.3]{OgusLectures}}]
    Let $X = (\underline{X},\cM)$ be a log scheme. Then $X$ is integral (respectively, saturated) if and only if $\cM(U)$ is integral (respectively, saturated) for all $U\subseteq X$ open.
\end{proposition}

In the case of fine or fs log schemes, the following proposition guarantees us that charts exist, at least \'etale locally.

\begin{proposition}[\protect{\cite[Corollary~2.3.6]{OgusLectures}}]
    Let $X = (\underline{X},\cM)$ be a log scheme. Then $\cM$ is fine (respectively, fs) if and only if \'etale locally it admits a fine (respectively fs) chart.
\end{proposition}

Under reasonable assumptions, we can also guarantee that a chart of a morphism $f : X\to S$ of log schemes exists (in fact, a chart extending a given chart of $S$).

\begin{proposition}[\protect{\cite[Proposition~2.4.2]{OgusLectures}}]
    Let $f : X\to S$ be a map of log schemes, and let $\beta_S : P\to\Gamma(S,\cM_S)$ be a chart of $\cM_S,$ where $P$ is finitely generated. If $\cM_X$ is coherent, then \'etale locally on $X,$ there exists a commutative diagram
    \[\begin{tikzcd}
        P\ar[d,"\beta_S"]\ar[r,"\theta"] &M\ar[d,"\beta_X"]\\
        \Gamma(S,\cM_S)\ar[r,"f^\flat"] &\Gamma(X,\cM_X)
    \end{tikzcd}\]
    such that the tuple $(M,\beta_X,P,\beta_S,\theta)$ is a chart of $f.$
\end{proposition}

As one would expect, it is also desirable to construct new log schemes from known ones in categorical ways. To that end, we can ask whether limits or colimits of log schemes exist.

\begin{proposition}[\protect{\cite[Proposition~III.2.1.2]{OgusLectures}}]
    The category $\LogSch$ of log schemes admits fiber products, and the forgetful functor $X\mapsto\underline{X}$ commutes with fiber products. Moreover, the fiber product of coherent log schemes is coherent.
\end{proposition}

We omit the proof, but it is worth noting that given morphisms of log schemes $f : X\to Z$ and $g : Y\to Z,$ the log structure on the $X\times_Z Y$ is given by the coproduct of the morphisms of log structures $h^*_{\log}\alpha_Z\to (\pi_{\underline{X}})^*_{\log}\alpha_X$ and $h^*_{\log}\alpha_Z\to (\pi_{\underline{Y}})^*_{\log}\alpha_Y$ (which exists by \cite[Proposition~III.2.1.1]{OgusLectures}), where $h :\underline{X}\times_{\underline{Z}}\underline{Y}\to\underline{Z}$ is the canonical map from the fiber product to $\underline{Z}.$

Of course, since we often want to restrict our attention to certain classes of log schemes which are sufficiently nice, one would ask if the fiber product of such log schemes is again sufficiently nice. This is not generally the case if we form the fiber product in the category of all log schemes, but the subcategories we are interested in do themselves admit fiber products.

\begin{proposition}[\protect{\cite[Corollary~III.2.1.5]{OgusLectures}}]
    Let $\LogSch^{\mathrm{coh}},\LogSch^{\fine},\LogSch^{\fs}$ denote the categories of coherent, fine, and fs log schemes, respectively. Then there are adjunctions
    \[\begin{tikzcd}
        \LogSch^{\fs}\ar[r, shift left] &\ar[l,shift left,"{(-)^{\mathrm{sat}}}"]\LogSch^{\fine}\ar[r,shift left] &\ar[l,shift left,"{(-)^{\mathrm{int}}}"] \LogSch^{\mathrm{coh}}
    \end{tikzcd}\]
    in which the left adjoints are the forgetful functors.
\end{proposition}

\begin{remark}
    Similarly to \ref{initial prelog ring under monoid}, the maps of underlying schemes $X^{\mathrm{sat}}\to X^{\mathrm{int}}$ and $X^{\mathrm{int}}\to X$ induced by the counits of the above adjunctions are not in general homeomorphisms. In particular, we do not have $\cM_{X^{\mathrm{int}}}\cong\cM_{X}^{\mathrm{int}}$ or $\cM_{X^{\mathrm{sat}}}\cong\cM_{X}^{\mathrm{sat}}.$
\end{remark}

\begin{proposition}[\protect{\cite[Corollary~III.2.1.6]{OgusLectures}}]
    The category of fine log schemes admits finite projective limits, as does the category of fs log schemes, given by
    $X\times^{\fine}_Z Y = (X\times_Z Y)^{\mathrm{int}}$ and $X\times^{\fs}_Z Y = (X\times_Y Z)^{\mathrm{sat}}.$ If $X$ and $Y$ are fine log schemes over a fine log scheme $Z,$ then the natural map
    \[
    \underline{X\times^{\mathrm{fine}}_Z Y}\to\underline{X}\times_{\underline{Z}}\underline{Y}
    \]
    is a closed immersion, and the natural map
    \[
    \underline{X\times^{\mathrm{fs}}_Z Y}\to\underline{X}\times_{\underline{Z}}\underline{Y}
    \]
    is a finite morphism.
\end{proposition}

\begin{proposition}[\protect{\cite[Proposition~III.2.5.3]{OgusLectures}}]
    Let $f : X\to Y$ be a morphism of fine log schemes. If $f$ is integral and $Y'\to Y$ is any morphism of fine log schemes, then $X' := X\times_Y Y'$ is integral, and the natural morphism $X'\to Y'$ is integral.
\end{proposition}

\begin{remark}
    If $M$ is a fine (resp. fs) monoid and $(A,M,\alpha)$ is a prelog ring, then $\Spec(A,M)$ is a fine (resp. fs) log scheme.
\end{remark}

With the previous definition, we can distinguish a class of prelog rings and log schemes which will be of particular interest to us in our comparison.

\begin{mydef}
    An \emph{fs semi-log ring} is a prelog ring $(A,M)$ such that $\overline{M} := M/M^\times$ is fs and the naturally induced map
    \[
    M\to\Gamma(\Spec(A,M),\cM_{\Spec(A,M)})
    \]
    is an isomorphism. We denote the full subcategory of $\PreLog$ spanned by fs semi-log rings $(A,M)$ such that $A$ is noetherian and has finite Krull dimension by $\SemiLog.$
\end{mydef}

\begin{proposition}[\protect{\cite[Proposition~7.7]{BLPO_HKR}}]
    The functor $\Spec : \SemiLog^{\op}\to\LogSch^{\fs,\mathrm{noeth}}$ is fully faithful, where $\LogSch^{\fs,\mathrm{noeth}}$ denotes the full subcategory of $\LogSch^{\fs}$ consisting of fs log schemes whose underlying scheme is Noetherian and of finite Krull dimension.
\end{proposition}

\begin{mydef}
    Let $\LogAff$ denote the full subcategory of $\LogSch^{\fs,\mathrm{noeth}}$ consisting of the spectra of fs semi-log rings.
\end{mydef}

\begin{remark}
    Every $X\in\LogSch^{\fs,\mathrm{noeth}}$ admits a Zariski covering $\{U_i\to X\}_{i\in I}$ such that $U_i\in\LogAff$ for all $i$ by \cite[Lemmas~1.3.2,1.3.3]{Tsuji1999-su}.
\end{remark}

\begin{mydef}
    Let $f : X\to Y$ be a morphism of log schemes, and let
    \begin{equation}\label{log thickening}\begin{tikzcd}
            S\ar[r]\ar[d] &T\ar[d]\ar[ld,dashed]\\
            X\ar[r,"f"] &Y
    \end{tikzcd}\end{equation}
    be a commutative diagram of log schemes, where $S\to T$ is a strict closed immersion such that
    \begin{enumerate}
        \item the ideal $\cI$ of $S$ in $T$ is a nil ideal, and
        \item the subgroup $1 + \cI\subseteq\cO_{T}^\times\cong\cM_{T}^\times$ operates freely on $\cM_T.$
    \end{enumerate}
    \begin{enumerate}
        \item We say $f$ is \emph{formally log smooth} if for any diagram of solid arrows \ref{log thickening}, there exists locally on $X$ at least one dotted arrow $T\to X$ making the diagram commute.
        \item We say $f$ is \emph{log smooth} if it is formally log smooth, $\cM_X,$ $\cM_Y$ are coherent, and $\underline{f}$ is locally of finite presentation.
        \item We say $f$ is \emph{formally log unramified} if for any diagram of solid arrows \ref{log thickening}, there exists locally on $X$ at most one dotted arrow $T\to X$ making the diagram commute.
        \item We say $f$ is \emph{log unramified} if it is formally log unramified, $\cM_X,$ $\cM_Y$ are coherent, and $\underline{f}$ is locally of finite presentation.
        \item We say $f$ is \emph{formally log \'etale} if for any diagram of solid arrows \ref{log thickening}, there exists locally on $X$ exactly one dotted arrow $T\to X$ making the diagram commute.
        \item We say $f$ is \emph{log \'etale} if it is formally log \'etale, $\cM_X,$ $\cM_Y$ are coherent, and $\underline{f}$ is locally of finite presentation.
    \end{enumerate}
\end{mydef}

The fundamental example of a smooth (respectively \'etale) morphism of log schemes is a morphism of the form
\[
\Spec(k[M],M)\to\Spec(k[P],P),
\]
induced by a morphism $\theta : P\to M$ of finitely generated integral monoids such that $\ker(\theta^{\gp})$ and the torsion part of $\cok(\theta^{\gp})$ (respectively, $\ker(\theta^{\gp})$ and $\cok(\theta^{\gp})$) are finite groups whose orders are invertible in $k.$

In fact, such morphisms provide a local characterization of smooth (respectively \'etale) morphisms. More precisely, we have the following theorem.

\begin{theorem}[\protect{\cite[Theorem~3.5]{KatoLogStructures}}]
    Let $f : X\to S$ be a morphism of fine log schemes, and let $\beta_S : P\to\Gamma(S,\cM_S)$ be a chart of $\cM_S.$ Then the following are equivalent.
    \begin{enumerate}
        \item $f$ is smooth (respectively \'etale).
        \item \'Etale locally on $X,$ there exists a chart $(M,\beta_X,P,\beta_Y,\theta)$ of $f$ (extending the given chart of $\cM_S$) such that
        \begin{itemize}
            \item The kernel and torsion part of the cokernel (respectively, the kernel and cokernel) of $\theta^{\gp}$ are finite groups whose orders are invertible in $X.$
            \item The induced morphism $\underline{X}\to \underline{Y}\times_{\Spec\ZZ[P]}\Spec\ZZ[M]$ is an \'etale map of schemes.
        \end{itemize}
    \end{enumerate}
\end{theorem}

\subsection{Olsson's Log Stack}

In \cite{OlssonLogGeomandAlgStacks}, Olsson defines an algebraic stack $\cLog_S$ associated to a fine log stack $S$ whose fiber over an $\underline{S}$-scheme $\underline{X}\to\underline{S}$ is the groupoid of fine log structures on $\underline{X}.$ He proves that the association $S\rightsquigarrow\cLog_S$ provides a fully faithful embedding of logarithmic schemes into algebraic stacks, and that a number of properties of morphisms of fine log schemes $f : X\to S$ can be detected via the corresponding morphisms $\cLog_X\to\cLog_S$ or even the classifying morphism $\underline{X}\to\cLog_S.$ In this section, we review the definitions and main results, along with a few of the useful properties that we will make use of in the remainder of this paper.

\begin{mydef}
    Let $S$ be a fine log scheme with underlying scheme $\underline{S}.$ Let
    \[
    \cLog_S\to\Sch_{\underline{S}}
    \]
    be the fibered category whose objects are morphisms are fine log schemes $f : X\to S,$ and whose morphisms from $(X'/S)$ to $(X/S)$ are given by morphisms of fine $S$-log schemes $g : X'\to S$ such that $g^\flat : g^*\cM_X\to\cM_{X'}$ is an isomorphism. The structure map sends $f : X\to S$ to the underlying morphism of schemes $\underline{f} : \underline{X}\to\underline{S},$ and similarly for morphisms.
\end{mydef}

\begin{theorem}[\protect{\cite[Theorem~1.1]{OlssonLogGeomandAlgStacks}}]
    The fibered category $\cLog_S$ is an algebraic stack locally of finite presentation over $\underline{S}.$ Moreover, the diagonal $\Delta : \cLog_S\to\cLog_S\times_{\underline{S}}\cLog_S$ is of finite presentation (in particular, $\Delta$ is qcqs\footnote{In \cite[remark 3.17]{OlssonLogGeomandAlgStacks}, Olsson states that the stack $\cLog_S$ is not quasi-separated. However, this is a matter of differing conventions -- Olsson requires a quasi-separated stack to have \emph{separated} diagonal, rather than simply a qcqs diagonal as in \cite[\href{https://stacks.math.columbia.edu/tag/050D}{Tag 050D}]{stacks-project}.}).
\end{theorem}

\begin{remark}
    If $S$ is an fs log scheme, one may also consider the substack of $\cLog_S$ which classifies fs log schemes over $S.$ In \cite{OlssonLogGeomandAlgStacks}, this substack is denoted by $\cT\mathrm{or}_S,$ and the canonical inclusion $\cT\mathrm{or}_S\to\cLog_S$ is an open immersion. 
\end{remark}

By definition, there is a canonical morphism $\underline{S}\to\cLog_S$ classifying $\id : S\to S.$

\begin{lemma}
    Let $S$ be a fine log scheme. Then the canonical map $S\to\cLog_S$ is an open immersion.
\end{lemma}

In general, a morphism of fine log schemes $f : X\to S$ corresponds to a morphism $f_{\cM_X} : \underline{X}\to\cLog_S.$ We may detect various properties of the morphism $f$ via properties of the morphism $f_{\cM_X}.$

\begin{lemma}[\protect{\cite[Proposition~3.19]{OlssonLogGeomandAlgStacks}}]
    A morphism $f : X\to S$ of fine log schemes is strict if and only if the map $f_{\cM_X} : \underline{X}\to \cLog_S$ factors through the canonical map $\underline{S}\to\cLog_S.$
\end{lemma}

\begin{theorem}[\protect{\cite[Theorem~4.6]{OlssonLogGeomandAlgStacks}}]
    Let $f : X\to S$ be a morphism of fine log schemes, with corresponding morphism $f_{\cM_X} : \underline{X}\to\cLog_S.$
    Then $f$ is log flat (respectively log smooth, log \'etale) if and only if $f_{\cM_X}$ is flat (respectively smooth, \'etale).
\end{theorem}

\begin{remark}
    One can also define the notion of a prelog structure or log structure on an algebraic stack, either analogously to the definition of a (pre)log structure on a scheme, or in the case of a log structure, via a morphism of stacks to $\cLog_{S}.$ In particular, the identity $\cLog_S\to\cLog_S$ provides a canonical log structure on $\cLog_S.$ In terms of logarithmic Hochschild homology, we can identify Olsson's logarithmic Hochschild homology of $f : X\to S$ with $L\Delta^*R\Delta_*\cO_{\underline{X}},$ where
    \[
    \Delta : \underline{X}\to\underline{X}\times_{\cLog_S}\underline{X}
    \]
    is the relative diagonal of $f_{\cM_X} : \underline{X}\to\cLog_{S}.$ One could consider an appropriate ``derived logarithmic stack'' corresponding to $X\times^R_{X\times^R_{\cLog_S} X}X,$ and investigate how the log structure on this stack is related to the logarithmic Hochschild homology, although we do not investigate this here.
\end{remark}

The stack $\cLog_S$ also admits a convenient \'etale cover by relatively simple stacks, which will be useful for us when calculating and comparing log Hochschild homologies. We recall the result here for convenience of the reader. Fix a base scheme $\underline{S}$ and a fine monoid $P.$ Then the scheme $\Spec_{\underline{S}}(\cO_{\underline{S}[P]})$ represents the functor
\begin{align*}
    \Aff_{\underline{S}}&\to\Set\\
    (\Spec A\to\underline{S})&\mapsto \Hom_{\Mon}(P,A)
\end{align*}
and admits an action of the $\underline{S}$-group scheme $\underline{S}[P^{\gp}],$ induced by the ``diagonal'' map $P\to P\oplus P^{\gp}.$ Let $\cS_P:= [\underline{S}[P]/\underline{S}[P^{\gp}]]$ denote the stack quotient by this action. If $\theta : P\to Q$ is a map of commutative monoids, then there is an induced map of algebraic stacks $\cS_Q\to\cS_P.$

\begin{cor}[\protect{\cite[Corollary~5.22]{OlssonLogGeomandAlgStacks}}]
    The fibered category over the category of $S$-log schemes defined by $\cS_P\times_{\cS_Q}\underline{S}$ is naturally equivalent to that defined by the functor which to any morphism of log schemes $f : T\to S$ associates the set of maps $\rho : P\to\Gamma(T,\overline{\cM}_T)$ for which the diagram
    \[\begin{tikzcd}
        Q\ar[r,"\theta"]\ar[d] &P\ar[d,"\rho"]\\
        \Gamma(S,\overline{\cM}_S)\ar[r,"f^\flat"] &\Gamma(T,\overline{\cM}_T)
    \end{tikzcd}\]
    commutes.
\end{cor}

Going forward, suppose $f : X\to S$ is a map of prelog schemes, and that there exists a chart of $f:$
\[
(f,\theta) : (X,M,\beta_M)\to (S,P,\beta_P).
\]
we let $\cS_{(f,\theta)} := \cA_M\times_{\cA_P} \underline{S},$ where $\cA_N := [\Spec\ZZ[N]/\Spec\ZZ[N^{\gp}]].$ Note that the canonical map
\[
\cS_{M}\times_{\cS_P}\underline{S}\to\cA_M\times_{\cA_P}\underline{S}
\]
is an equivalence of stacks, and we will work with either object as is most convenient.

\begin{proposition}[\protect{\cite[Corollary~5.25]{OlssonLogGeomandAlgStacks}}]
    Let $S$ be a fine log scheme, and let $J$ be the set of triples $(U,\beta, \ell),$ where $U\to\underline{S}$ is \'etale, $\beta : Q\to\left.\cM_S\right|_U$ is a chart, and $\ell : Q\to P$ is a morphism of monoids. Then the natural morphism
    \[
    \coprod_{(U,\beta,\ell)\in J} \cS_P\times_{\cS_Q} U\to\cLog_S
    \]
    induced by the maps of \cite[corollary 5.24]{OlssonLogGeomandAlgStacks} is representable, \'etale, and surjective.
\end{proposition}

This proposition will be critical for us, as in some situations, it will allow us to work with the stack $\cS_P\times_{\cS_Q} U$ instead of $\cLog_S.$ In particular, we have the following important lemma.

\begin{lemma}
    Let $(k,P)\to (A,M)$ be a map of fine prelog rings. Then the map $\Spec A\to\cLog_{\Spec(k,P)}$ induced by the canonical map
    \[
    \Spec(A,M)\xrightarrow{f}\Spec(k,P)
    \]
    factors through the representable and \'etale map 
    \[
    \cS_M\times_{\cS_P}\Spec k\to\cLog_{\Spec(k,P)}
    \]
    of \cite[Corollary 5.24]{OlssonLogGeomandAlgStacks}
\end{lemma}
\begin{proof}
    Indeed, this follows by consideration of the diagram
    \[\begin{tikzcd}
        \Spec A\ar[d]\ar[dd,bend right=90]\ar[rd,"\cM_{\Spec(A,M)}"]\ar[rrd, bend left=15,"\underline{f}"]\\
        \cS_M\times_{\cS_P}\Spec k\ar[d]\ar[r] &\cLog_{\Spec(k,P)}\ar[d]\ar[r] &\Spec k\ar[d]\\
        \cS_M\ar[r] &\cLog_{(\cS_P,\cM_{\cS_P})}\ar[r] &\cS_P
    \end{tikzcd}\]
    in which the squares are Cartesian (see the proof of \cite[Corollary 5.24]{OlssonLogGeomandAlgStacks}).
\end{proof}

\subsubsection{Exactification, Revisited}

There is also a connection between Olsson's log stack and the idea of the exactification of a morphism of log schemes. We explain the basic results in this section, but details can be found in \cite[Appendix~A]{Olssonnonabelian}.

Let $f : X\to S$ be a morphism of fine log schemes, and consider the fiber product $X\times^{\fine}_S X$ in the category of fine log schemes, along with the diagonal morphism $\Delta : X\to X\times^{\fine}_S X.$

\begin{proposition}[\protect{\cite[Proposition~A.3]{Olssonnonabelian}}]\label{Martin exactification}
    Let $\sfC$ denote the category whose objects are morphisms $g : T\to X\times^{\fine}_S X$ of fine log algebraic spaces over $S$ such that the composite of $g$ with either projection $X\times^{\fine}_S X\to X$ is strict, and whose morphisms are morphisms over $X\times^{\fine}_S X.$ Then $\sfC$ admits a final object, which is given by $X\times_{\cLog_S} X.$ 
\end{proposition}

By the above proposition, we obtain a factorization of the diagonal
\[
X\xrightarrow{\Delta^{\ex}} X\times_{\cLog_S} X\to X\times^{\fine}_S X,
\]
where $\Delta^{\ex}$ is in fact an exact immersion. Olsson calls $\Delta^{\ex}$ the \emph{exactification} of the diagonal.

We will not make explicit use of this construction or its properties, but we mention it here for the following reason: one challenge in generalizing Hochschild homology to the logarithmic situation is that the diagonal morphism of log schemes $\Delta : X\to X\times_S X$ (or $X\to X\times^{\fine}_S X$ is not exact. However, by replacing $X\times^{\fine}_S X$ with $X\times_{\cLog_S} X,$ we obtain a map which is exact and can replace the log diagonal for many purposes. In particular, if we identify Hochschild homology with $L\Delta^* R\Delta_*\OX,$ then Olsson's log Hochschild homology can be identified $L(\Delta^{\ex})^* R\Delta^{\ex}_*\OX.$ This will be made precise when we define log Hochschild homology.

\section{Derived Logarithmic Geometry}

In this section, we review the basic definitions and constructions of derived logarithmic geometry which we will need in what follows. Additional details can be found in \cite{Sagave_2014} (in the language of model structures) and \cite{BLPO_HKR}.

\subsection{Animated Monoids}

\begin{mydef}
    The $\infty$-category $\aMon$ of \emph{animated monoids} is the animation of the ordinary category $\CMon$ of discrete commutative monoids, or alternatively, the $\infty$-category associated to the standard model structure on $\sCMon,$ the category of simplicial commutative monoids.
\end{mydef}

The $\infty$-category $\aMon$ is closely related to the $\infty$-category of animated abelian groups, or equivalently, $\cMod_{\ZZ}^{\cn}.$ There is an evident fully faithful forgetful functor $\cMod_{\ZZ}^{\cn}\to\aMon$\footnote{For example, by animating the classical forgetful functor, or by model category theoretic methods.}, and we call an animated monoid in the image of this functor \emph{grouplike}. Equivalently, $M$ is grouplike if $\pi_0(M)$ is a group. We will also often identify an animated abelian group with the corresponding grouplike animated monoid.

The forgetful functor $\cMod_{\ZZ}^{\cn}\to\aMon$ also admits both a right and a left adjoint, so there are two ways of canonically associating a grouplike animated monoid to an arbitrary animated monoid. 

\begin{mydef}
    Let $M$ be an animated monoid. The \emph{space of units of $M$} is the animated monoid $M^\times$ given by the following pullback
    \[\begin{tikzcd}
        M^\times\ar[d]\ar[r] &M\ar[d]\\
        \pi_0(M)^\times\ar[r] &\pi_0(M).
    \end{tikzcd}\]
\end{mydef}

\begin{remark}
    What we denote by $M^\times$ is sometimes denoted by $\GL_1(M)$ in other references.
\end{remark}

Note that we have 
\begin{align*}
    \Map_{\cMod_{\ZZ}^{\cn}}(A,M^\times)&\cong\Map_{\cMod_{\ZZ}^{\cn}}(A,\pi_0(M)^\times\times_{\pi_0(M)} M)\\
    &\cong\Map_{\cMod_{\ZZ}^{\cn}}(A,\pi_0(M)^\times)\times_{\Map_{\cMod_{\ZZ}^{\cn}}(A,\pi_0(M))}\Map_{\aMon}(A,M)\\
    &\cong \Map_{\cMod_{\ZZ}^{\cn}}(\pi_0 A,\pi_0(M)^\times)\times_{\Map_{\cMod_{\ZZ}^{\cn}}(\pi_0 A,\pi_0(M))}\Map_{\aMon}(A,M)\\
    &\cong \Map_{\Ab}(\pi_0 A,\pi_0(M)^\times)\times_{\Map_{\CMon}(\pi_0 A,\pi_0(M))}\Map_{\aMon}(A,M)\\
    &\cong \Map_{\aMon}(A,M);
\end{align*}
i.e., $M^\times\to M$ is the terminal map from a grouplike animated monoid to $M.$

\begin{mydef}
    Let $\theta : P\to M$ be a map of animated monoids.
    \begin{enumerate}
        \item We say $\theta$ is \emph{logarithmic} if the canonical map
        \[
            \theta^{-1} M^\times := P\times_{M} M^\times\to M^\times
        \]
        is an equivalence.
        \item Let $P^a_{\theta} := P\oplus_{\theta^{-1}M^\times} M^\times.$ The \emph{logification} $\theta^a : P^a_{\theta}\to M$ is the canonical map determined by the maps $\theta : P\to M$ and $M^\times\to M$ via the universal property of the pushout.
    \end{enumerate}
\end{mydef}

\begin{proposition}[\protect{\cite[Proposition~2.5]{BLPO_HKR}}]
    Let $\theta : P\to M$ be a map of animated monoids. The logification $\theta^a : P^a_\theta\to M$ is the initial logarithmic map under $\theta.$ In other words, if $\aMon_{/M}^{\log}$ denotes the full subcategory of $\aMon_{/M}$ of maps $\theta : P\to M$ which are logarithmic, then logification provides an adjunction
    \[\begin{tikzcd}
        (-)^a : \aMon_{/M}\ar[r,shift left] &\ar[l,shift left]\aMon^{\log}_{/M} : U.
    \end{tikzcd}\]
\end{proposition}

The other adjunction is provided by the group completion functor.

\begin{mydef}
    Let $M$ be an animated monoid. We say that a morphism $\gamma : M\to M'$ realizes $M'$ as a \emph{group completion of $M$} if $M'$ is grouplike and the map $\gamma$ induces an equivalence
    \[
    \Map_{\cMod_{\ZZ}^{\cn}}(M',A)\to \Map_{\aMon}(M,A)
    \]
    for any grouplike animated monoid/animated abelian group $A.$
\end{mydef}

A group completion of any animated monoid $M$ always exists. One way to see this is to apply \cite[Corollary~2.2.3]{mao2024revisitingderivedcrystallinecohomology} to the adjunction $(-)^{\gp}\dashv U,$ where $(-)^{\gp} : \CMon\to\Ab$ is the classical group completion functor and $U : \Ab\to\CMon$ is the forgetful functor. Since $U$ is also a left adjoint, it satisfies the hypotheses and we obtain an adjunction
\[\begin{tikzcd}
    (-)^{\gp} : \aMon\ar[r,shift left] &\ar[l,shift left] \cMod_\ZZ^{\cn} : U
\end{tikzcd}\]
between the animations, and we will abusively denote $\Ani((-)^{\gp})$ simply by $(-)^{\gp}.$

There are various other equivalent ways to construct the group completion. If $\mathrm{B}\dashv\Omega$ is the looping/delooping adjunction on $\aMon,$ which can be modeled explicitly on simplicial commutative monoids as the tensor and copower $\mathrm{B}M_\bullet = M_\bullet\otimes S^1 = M_\bullet\otimes \Delta^1/\partial\Delta^1$ and $\Omega M = M^{S^1},$ then $M\to M^{\gp}$ is equivalent to $M\to\Omega\mathrm{B}M.$ Another way to model $M^{\gp}$ at the level of simplicial commutative monoids is to form the classical group completion degree-wise on a simplicial monoid $M_\bullet.$ We refer to \cite[Section~2.1]{Sagave_2014} for more details.

From the above discussion, we deduce the following proposition.

\begin{proposition}
    The forgetful functor
    \[
    U : \cMod_\ZZ^{\cn}\to\aMon
    \]
    is left adjoint to the space of units functor
    \[
    \left(-\right)^\times : \aMon\to\cMod_\ZZ^{\cn}
    \]
    and right adjoint to the group completion functor
    \[
    (-)^{\gp} : \aMon\to\cMod_\ZZ^{\cn}
    \]
\end{proposition}

As in the classical case, there is also a universal way to construct an animated ring from an animated monoid. Recall that there is an adjunction of $1$-categories
\[\begin{tikzcd}
    \ZZ[-] : \CMon\ar[r,shift left] &\CRing\ar[l, shift left] : (-,\,\cdot\,).
\end{tikzcd}\]
If we replace $\CMon$ and $\CRing$ by $\sCMon$ and $\sCRing$ respectively, we can apply this adjunction degree-wise to obtain a Quillen adjunction
\[\begin{tikzcd}
    \ZZ[-] : \sCMon\ar[r,shift left] &\sCRing\ar[l, shift left] : (-,\,\cdot\,),
\end{tikzcd}\]
(see \cite{Sagave_2014} for details on the model structures) and thus an adjunction
\[\begin{tikzcd}
    \ZZ[-] : \aMon\ar[r,shift left] &\aRing\ar[l, shift left] : (-,\,\cdot\,)
\end{tikzcd}\]
between the associated $\infty$-categories by \cite[Proposition~5.2.4.6]{HTT}.

\subsection{Derived Affine Logarithmic Geometry}

\begin{mydef}
    Consider the composition
    \[\begin{tikzcd}
        \aRing\ar[rrr, bend left, "F"]\ar[r,"{(-,\,\cdot\,)}"'] &\aMon\ar[rr,"{\aMon_{/(-)}}"'] &&\Cat_{\infty},
    \end{tikzcd}\]
    where $\aMon_{/(-)}$ is the functor which on objects is given by
    \begin{align*}
        \aMon_{/(-)} : \aMon&\to\Cat_{\infty}\\
        M&\mapsto\aMon_{/M}.
    \end{align*}
    \begin{enumerate}
        \item The $\infty$-category $\aPreLog$ of \emph{animated prelog rings} is the $\infty$-category associated to the composition $F$ via the $\infty$-categorical Grothendieck construction (or in the language of \cite{HTT}, the ``straightening-unstraightening correspondence'').
        \item Let $A$ be an animated ring. The $\infty$-category $\aPreLogStr_A$ of \emph{animated prelog structures on $A$} is defined to be the category $\aMon_{/(A,\,\cdot\,)}$ of maps of animated monoids $\alpha : M\to (A,\,\cdot\,).$ Equivalently, $\aPreLogStr_A\simeq\aPreLog\times_{\aRing}\{A\}.$
    \end{enumerate}
\end{mydef}

Note that $\aPreLog$ is thus equivalent to the comma category $(\id_{\aMon}\downarrow F),$ since the functor $\aMon_{/(-)}$ is the coCartesian straightening of the target functor $\Arr(\aMon)\to\aMon,$ and therefore the Grothendieck construction of the composition is the pullback of the $\Arr(\aMon)\to\aMon$ along $(-,\,\cdot\,),$ which is by definition the comma category.

There are several other equivalent descriptions of $\aPreLog.$ It is the $\infty$-category associated to either the projective or injective model structure on the category of simplicial prelog rings, c.f. \cite[Propositions~3.2,3.3]{Sagave_2014}, or the animation of the category $\PreLog,$ as explained in \cite[Remark~2.11]{BLPO_HKR}. The equivalence between our definition and these alternative descriptions is discussed in \cite[Remark~2.12]{BLPO_HKR}.

If $(A,M,\alpha)$ is an animated prelog ring, we will also write $\underline{(A,M,\alpha)}$ to denote the underlying animated ring $A.$

\begin{mydef}
    The $\infty$-category of \emph{affine derived prelog schemes} $\dPreLogAff$ is the opposite of the $\infty$-category of animated prelog rings. For an animated prelog ring $(A,M),$ the corresponding object of $\dPreLogAff$ will be denoted by $\Spec(A,M).$ If $X = \Spec(A,M)$ is an affine derived prelog scheme, then $\underline{X}$ will denote the underlying affine derived scheme $\Spec A.$ A \emph{derived prelog structure} on an affine derived scheme $\Spec A$ is a prelog structure on the animated ring $A.$
\end{mydef}

\begin{mydef}
    Let $(A,M,\alpha)$ be an animated prelog ring.
    \begin{enumerate}
        \item $(A,M,\alpha)$ is an \emph{animated log ring} if $\alpha$ is logarithmic.
        \item The $\infty$-category $\aLog$ is the full subcategory of $\aPreLog$ consisting of animated prelog rings which are logarithmic.
        \item The $\infty$-category $\aLogStr_A$ is the pullback $\aLog\times_{\aRing}\{A\},$ or equivalently, the full subcategory of $\aPreLog_A$ consisting of animated prelog structures $\alpha : M\to A$ which are logarithmic.
        \item The \emph{logification} $(A,M,\alpha)^a := (A,M^a_{\alpha},\alpha^a)$ is the animated ring determined by the logification of the structure map $\alpha.$
    \end{enumerate}
\end{mydef}

\begin{warning}
    Note that if one models animated prelog rings by simplicial prelog rings, animated log rings are not then modeled by simplicial objects in the category of log rings; see \cite[Remark~3.8]{Sagave_2014}.
\end{warning}

\begin{remark}
    Because $\aPreLog$ and $\aRing$ admit colimits and the forgetful functor $\aPreLog\to\aRing$ preserves colimits, it follows that $\aPreLogStr_A$ admits colimits as well. 
\end{remark}

As before, we have an adjunction
\[\begin{tikzcd}
    (-)^a : \aPreLogStr_A\ar[r,shift left] &\ar[l,shift left] \aLogStr_A.
\end{tikzcd}\]

In fact, $(k,P,\beta)^a$ is the universal animated log ring over $(k,P,\beta)$ in the following sense.
\begin{proposition}
    Let $U : \aLog\to\aPreLog$ denote the canonical fully faithful forgetful functor. Then there is an adjunction
    \[\begin{tikzcd}
        (-)^a : \aPreLog\ar[r,shift left] &\aLog\ar[l, shift left] : U
    \end{tikzcd}\]
    where the left adjoint is given on objects by the logification construction. 
\end{proposition}
\begin{proof}
    A proof of this result in the language of model structures on simplicial prelog rings can be found in \cite[Section~3.2]{Sagave_2014}. One can also deduce the result formally from the adjunction between animated prelog structures and animated log structures on a given animated ring $A$ together with \ref{adjunction between (pre)log structures for map of animated rings}.
\end{proof}

As in the classical case, given a morphism of (animated) rings/(affine) schemes, there is a canonical way to associate an animated prelog structure on either ring given an animated prelog structure on the other.

\begin{mydef}
    Let $f : k\to A$ be a map of animated rings.
    \begin{enumerate}
        \item If $(k,P,\beta)$ is an animated prelog ring whose underlying animated ring is $k,$ the \emph{inverse image prelog structure on $A$} is the morphism
        \[
        P\xrightarrow{\beta}(k,\,\cdot\,)\xrightarrow{(f,\,\cdot\,)}(A,\,\cdot\,),
        \]
        and we will denote $A$ equipped with the inverse image prelog structure by $(A,f^*P,f^*\beta).$
        \item If $(A,M,\alpha)$ is an animated prelog ring whose underlying animated ring is $A,$ the \emph{direct image prelog structure on $k$} is the map $f_*M\xrightarrow{f_*\alpha}(k,\,\cdot\,)$ given by the top horizontal morphism in the Cartesian diagram
        \[\begin{tikzcd}
            f_*M\ar[r,"f_*\alpha"]\ar[d] &(k,\,\cdot\,)\ar[d,"{(f,\,\cdot\,)}"]\\
            M\ar[r,"\alpha"] &(A,\,\cdot\,)
        \end{tikzcd}\]
    \end{enumerate}
\end{mydef}

The above constructions yield the following adjunctions.
\begin{proposition}[\protect{\cite[Proposition~3.26]{Sagave_2014}}]\label{adjunction between (pre)log structures for map of animated rings}
    Let $f : k\to A$ be a morphism of animated rings. Then $f$ induces an adjunction
    \[\begin{tikzcd}
        f^* : \aPreLogStr_{k}\ar[r,shift left] &\aPreLogStr_{A}\ar[l, shift left] : f_*
    \end{tikzcd}\]
    between the categories of animated prelog structures on $k$ and animated prelog structures on $A$ and an adjunction
    \[\begin{tikzcd}
        f^* : \aLogStr_{k}\ar[r,shift left] &\aLogStr_{A}\ar[l, shift left] : (-)^a\circ f_*
    \end{tikzcd}\]
    between the categories of animated log structures on $k$ and animated log structures on $A.$
\end{proposition}

\begin{remark}
    The previous proposition immediately implies that if $(P,\beta)$ is an animated log structure on $k,$ then $f^*(P,\beta)$ is an animated log structure on $A.$
\end{remark}

\begin{remark}
    Our notation for the direct and inverse image prelog structures is reversed from that of \cite{Sagave_2014}, although the definitions for the terminologies are the same. Our notation is chosen to conform with the geometric inverse and direct images.
\end{remark}

\begin{mydef}
    A map $(f,f^\flat) : (k,P)\to (A,M)$ of animated prelog rings is \emph{strict} if the canonical map $(A,f^*P)^a\to (A,M)^a$ is an equivalence.
\end{mydef}

\begin{remark}
    A map $(k,P)\to (A,M)$ of animated prelog rings is strict if and only if the diagram
    \[\begin{tikzcd}
        (k,\{1\})^a\ar[r]\ar[d] &(A,\{1\})^a\ar[d]\\
        (k,P)^a\ar[r] &(A,M)^a
    \end{tikzcd}\]
    is coCartesian. To see this, observe that the diagram
    \[\begin{tikzcd}
        (k,\{1\})\ar[r]\ar[d] &(A,\{1\})\ar[d]\\
        (k,P)\ar[r] &(A,f^*P)
    \end{tikzcd}\]
    is coCartesian for any animated prelog ring $(k,P)$ and any map of animated rings $k\to A$ (essentially by definition of $f^*P$), and because $(-)^a$ is a left adjoint, 
    \[\begin{tikzcd}
        (k,\{1\})^a\ar[r]\ar[d] &(A,\{1\})^a\ar[d]\\
        (k,P)^a\ar[r] &(A,f^*P)^a
    \end{tikzcd}\]
    is also coCartesian.
    
    We also have a canonical factorization
    \[\begin{tikzcd}
        (k,\{1\})\ar[r]\ar[d] &(A,\{1\})\ar[r]\ar[d] &(A,\{1\})\ar[d]\\
        (k,P)\ar[r] &(A,f^*P)\ar[r] &(A,M)
    \end{tikzcd}\]
    of the original diagram by \ref{adjunction between (pre)log structures for map of animated rings}, and thus a corresponding diagram
    \[\begin{tikzcd}
        (k,\{1\})^a\ar[r]\ar[d] &(A,\{1\})^a\ar[r]\ar[d] &(A,\{1\})^a\ar[d]\\
        (k,P)^a\ar[r] &(A,f^*P)^a\ar[r] &(A,M)^a
    \end{tikzcd}\]
    of the associated animated log rings.
    
    If the diagram 
    \[\begin{tikzcd}
        (k,\{1\})^a\ar[r]\ar[d] &(A,\{1\})^a\ar[d]\\
        (k,P)^a\ar[r] &(A,M)^a
    \end{tikzcd}\]
    is coCartesian, then the associated diagram
    \[\begin{tikzcd}
        (k,\{1\})^a\ar[r]\ar[d] &(A,\{1\})^a\ar[r]\ar[d] &(A,\{1\})^a\ar[d]\\
        (k,P)^a\ar[r] &(A,f^*P)^a\ar[r] &(A,M)^a
    \end{tikzcd}\]
    has left square coCartesian and outer square coCartesian, so that the right square is also coCartesian by the ``pushout lemma.'' In particular, $(A,f^*P)^a\to (A,M)^a$ is an equivalence.
    
    Conversely, if $(A,f^*P)^a\to (A,M)^a$ is an equivalence, then the outer square of
    \[\begin{tikzcd}
        (k,\{1\})^a\ar[r]\ar[d] &(A,\{1\})^a\ar[r]\ar[d] &(A,\{1\})^a\ar[d]\\
        (k,P)^a\ar[r] &(A,f^*P)^a\ar[r] &(A,M)^a
    \end{tikzcd}\]
    is coCartesian. Since the left square is always coCartesian, the other direction of the ``pushout lemma'' implies that the right square is coCartesian.

    The argument above works just as well in the classical situation (affine or global), giving an alternate proof of \ref{strictness pullback equivalence} to the one in \cite{OgusLectures}.
\end{remark}

\begin{lemma}\label{initial animated prelog ring under monoid}
    Let $(A,M)$ be an animated prelog ring, and let $\theta : M\to N$ be a morphism of animated monoids. Then the $\infty$-category of morphisms of animated prelog rings $(A,M)\to (A',M')$ under $(A,M)$ equipped with a morphism of monoids $N\to M'$ and a commutative diagram
    \[\begin{tikzcd}
        M\ar[r]\ar[rd] &N\ar[d]\\
        &M'
    \end{tikzcd}\]
    admits an initial object, given by
    \[
    (A,M)\to (A\otimes_{\ZZ[M]}\ZZ[N],N).
    \]
    Thus, any morphism of animated prelog rings $(A,M)\to (A',M')$ equipped with a given factorization $M\to N\to M'$ of the morphism of animated monoids canonically admits a factorization through $(A\otimes_{\ZZ[M]}\ZZ[N],N).$
\end{lemma}
\begin{proof}
    A morphism of animated prelog rings $(A,M)\to (A',M')$ is equivalent to the data of a commutative diagram of animated rings
    \[\begin{tikzcd}
        \ZZ[M]\ar[d]\ar[r] &\ZZ[M']\ar[d]\\
        A\ar[r] & A'
    \end{tikzcd}\]
    where the top morphism is induced by a morphism of animated monoids $M\to M',$ by adjunction. The proof follows from the universal property of the tensor product.
\end{proof}

We end with a remark on animated prelog structures which factor through $A^\times\to A.$

\begin{lemma}
    Suppose that $\alpha : M\to A$ is an animated prelog structure which factors through the canonical monomorphism $A^\times\to A.$ Then $\alpha^a : M^a_{\alpha}\to A$ is equivalent to $A^\times\to A.$ In particular, if $M$ is grouplike, then $M^a\simeq A^\times.$
\end{lemma}
\begin{proof}
    Clear by construction. Proof in the context of (graded) prelog/log structures on commutative ring spectra can be found as \cite[Lemmas 3.12, 3.13]{Sagave2012LogarithmicSO}.
\end{proof}

\begin{cor}\label{logificiation of direct sum with grouplike prelog structure}
    If $\alpha : M\to A$ and $\alpha' : M'\to A^\times\to A$ are animated prelog structures on $A,$ then the canonical map $M\to M\oplus M'$ determines a strict morphism of animated prelog rings $(A,M,\alpha)\to(A,M\oplus M',\alpha\oplus\alpha').$
\end{cor}
\begin{proof}
    The proof is a simple computation:
    \[
    M^a_{\alpha}\simeq (M\oplus^{\textrm{prelog}}\{1\})^a_{\alpha\oplus i}\simeq M^a_\alpha\oplus^{\log}\{1\}^a_{i}\simeq M^a_\alpha\oplus^{\log} A^\times\simeq M^a_{\alpha}\oplus(M')^a_{\alpha'}\simeq (M\oplus^{\textrm{prelog}} M')^a_{\alpha\oplus\alpha'}.
    \]
\end{proof}

\subsection{Repletions}

The repletion of a map of animated monoids, originally due to Rognes, is a key object in the construction of Rognes' Hochschild homology. In some sense, replete morphisms and repletions of morphisms are meant to be a homotopically meaningful replacement for the concepts of integral or saturated monoids; see \cite[Section~3]{Rognes_2009} for a discussion and proofs of results for discrete monoids, \cite[Section~2.2]{Sagave_2014} for the model-category theoretic setting, or \cite[Section~2.16]{BLPO_HKR} for the animated setting.

Fundamentally, exactness or repleteness of a morphism of monoids is a direct generalization of the definition of exactness from the classical/discrete case. 

\begin{mydef}
    Let $\theta : N\to M$ be a morphism of animated monoids. We say $\theta$ is
    \begin{enumerate}
        \item \emph{exact} if the diagram
        \[\begin{tikzcd}
            N\ar[r]\ar[d,"\theta"]& N^{\gp}\ar[d,"\theta^{\gp}"]\\
            M\ar[r] &M^{\gp}
        \end{tikzcd}\]
        is Cartesian.
        \item \emph{virtually surjective} if $\pi_0(\theta^{\gp}) : \pi_0(N^{\gp})\to\pi_0(M^{\gp})$ is a surjection of classical abelian groups.
        \item \emph{replete} if it is exact and virtually surjective.
    \end{enumerate}
\end{mydef}

We also desire a method by which we can associate to a morphism $\theta : N\to M$ of animated rings a morphism $\theta' : N'\to M$ which is exact in a universal sense. In order to better discuss this process, we make the following definition.

\begin{mydef}
     Let $\theta : N\to M$ be a morphism of animated monoids.
    \begin{enumerate}
        \item The \emph{pre-exactification} of $\theta$ is the canonical morphism
        \[
        M\times_{M^{\gp}} N^{\gp}\to M.
        \]
        \item Suppose we are given a factorization
        \[
        N\xrightarrow{\theta'} N'\xrightarrow{\theta^e} M
        \]
        of $\theta$ such that the following hold:
        \begin{enumerate}
            \item $\theta^e$ is exact.
            \item $(\theta')^{\gp} : N^{\gp}\to (N')^{\gp}$ is an equivalence.
        \end{enumerate}
        In this situation, we call $\theta^{e}$ an \emph{exactification} of $\theta,$ and we call $N'$ an \emph{exactification of $N$ (relative to $\theta$)}.
        \item Suppose we are given a factorization
        \[
        N\xrightarrow{\theta'} N'\xrightarrow{\theta^\rep} M
        \]
        of $\theta$ such that the following hold:
        \begin{enumerate}
            \item $\theta^\rep$ is replete.
            \item $(\theta')^{\gp} : N^{\gp}\to (N')^{\gp}$ is an equivalence.
        \end{enumerate}
        In this situation, we call $\theta^{\rep}$ an \emph{repletion} of $\theta,$ and we call $N'$ an \emph{repletion of $N$ (relative to $\theta$)}.
    \end{enumerate}
\end{mydef}

As in the classical situation, the pre-exactification of a morphism of animated monoids need not itself be exact without additional assumptions. However, if $\theta$ is virtually surjective, then the pre-exactification provides a repletion of $\theta.$ We choose the non-standard ``pre-exactification'' language to emphasize that in general, the map $M\times_{M^{\gp}} N^{\gp}\to M$ may not be exact.

\begin{proposition}[\protect{\cite[Proposition~2.17]{Sagave_2014}}]
    Let $\theta : N\to M$ be a virtually surjective morphism of animated monoids. Then the projection
    \[
    \theta^{\rep} : N^{\rep}_{\theta} := M\times_{M^{\gp}}N^{\gp}\to M
    \]
    provides a repletion of $\theta : N\to M.$
\end{proposition}

In order to obtain our comparison morphism, it will be helpful to have a universal property of the exactification/repletion of a morphism of animated monoids (and later, a morphism of animated prelog rings). Similar results in more generality can be found in \cite{zhang2026derivedlogarithmicdeformationtheory}; however, we do not require the generality of the theory developed therein.\footnote{Note that some results in \cite{zhang2026derivedlogarithmicdeformationtheory} are shown for $\EE_{\infty}$-monoids rather than animated monoids, but many of the arguments there still hold for animated monoids. It is not immediately obvious that the definition of an exact morphism in loc. cit. agrees with the definition stated here or in \cite{Rognes_2009}, \cite{Sagave_2014}, or \cite{BLPO_HKR} in the general situation; however, it is the case that the notions agree in the case of virtually surjective morphisms.}

\begin{lemma}
    Let $\theta : N\to M$ be a virtually surjective morphism of animated monoids. The factorization $N\to N^{\rep}_{\theta}\xrightarrow{\theta^{\rep}} M$ is initial among factorizations $N\to N'\xrightarrow{\theta'} M$ of $\theta$ such that $\theta'$ is replete. 
\end{lemma}
\begin{proof}
    Suppose $N\to N'\xrightarrow{\theta'} M$ is such a factorization. We first show that the space of morphisms of such factorizations from $N\to N^{\rep}_{\theta}\to M$ to $N\to N'\to M$ is contractible.

    Consider the diagram
    \[\begin{tikzcd}
        N\ar[r]\ar[d] &N^{\gp}\ar[d]\\
        N'\ar[r]\ar[d] &(N')^{\gp}\ar[d]\\
        M\ar[r] &M^{\gp}
    \end{tikzcd}\]
    Since the bottom square is a pullback, we note that there is a canonical projection morphism $N^{\rep}_{\theta}\to N',$ since $N^{\rep}_{\theta}$ is equivalent to the pullback $N'\times_{(N')^{\gp}} N^{\gp}$ (using the fact that the bottom square is a pullback and that $N^{\rep}_{\theta}$ is the pullback of the lower morphism and the composition of the two right vertical morphisms in the diagram).

    Now, consider the mapping space $\Map_{\theta}(N^{\rep}_{\theta},N')$ -- in other words, the space of diagrams of the form
    \[\begin{tikzcd}
        N\ar[r]\ar[rr, bend left,"\theta"]\ar[d,"\id_N"] & N^{\rep}_{\theta}\ar[r,"\theta^{\rep}"']\ar[d] &M\ar[d,"\id_M"]\\
        N\ar[r]\ar[rr, bend right,"\theta"'] &N'\ar[r,"\theta'"] &M.
    \end{tikzcd}\]

    By repleteness of $\theta',$ a general morphism of animated monoids from $N^{\rep}_{\theta}$ is equivalent to a diagram of animated monoids of the form
    \[\begin{tikzcd}
        N^{\rep}_{\theta}\ar[r]\ar[d] &(N')^{\gp}\ar[d,"(\theta')^{\gp}"]\\
        M\ar[r] &M^{\gp}.
    \end{tikzcd}\]
    In other words, the universal property of the fiber product guarantees that
    \[
    \Map_{\aMon}(N^{\rep}_{\theta},N')\simeq \Map(N^{\rep}_{\theta}, M)\times_{\Map(N^{\rep}_{\theta},M^{\gp})}\Map((N^{\rep}_{\theta})^{\gp},(N')^{\gp}).
    \]
    Now, using the fact that this is not simply a map $N_\theta^{\rep}\to N',$ but a map of factorizations of $\theta,$  it follows that the map $N_\theta^{\rep}\to M$ must be (equivalent to) the repletion $\theta^{\rep}$ of $\theta.$ Additionally, because the map to $M^{\gp}$ is the composition of $N^{\rep}_{\theta}\to M\to M^{\gp},$ this map is also determined up to contractible choice. We are then left with the determination of the map $N^{\rep}_{\theta}\to (N')^{\gp},$ which by the universal property of the groupification, must factor $N^{\rep}_{\theta}\to \left(N^{\rep}_{\theta}\right)^{\gp}\to (N')^{\gp},$ where the first map is the canonical morphism from $N^{\rep}_{\theta}$ to its groupification. However, this map further fits into the diagram
    \[\begin{tikzcd}
        N\ar[r]\ar[d] &N^{\gp}\ar[d]\\
        N^{\rep}_{\theta}\ar[r]\ar[d] &\left(N^{\rep}_{\theta}\right)^{\gp}\ar[d]\\
        N'\ar[r] &(N')^{\gp}.
    \end{tikzcd}\]
    Since the map $N^{\gp}\to (N^{\rep}_{\theta})^{\gp}$ is an equivalence, it follows that the map $\left(N^{\rep}_{\theta}\right)^{\gp}\to (N')^{\gp}$ is equivalent to the groupification of the given map $N\to N',$ via the equivalence $\left(N^{\rep}_{\theta}\right)^{\gp}\to N^{\gp}.$ Thus, this map is also determined up to contractible choice, so that the space of morphisms is contractible, as desired.
\end{proof}

We may reframe the above lemma as the following adjunction.

\begin{cor}
    Let $M$ be an animated monoid, let $\aMon_{/M}^{\vsur}$ be the full subcategory of $\aMon_{/M}$ spanned by those maps $\theta : N\to M$ which are virtually surjective, and let $\aMon_{/M}^{\rep}$ denote the full subcategory of $\aMon_{/M}$ spanned by $\theta : N\to M$ which are replete. Then there is an adjunction
    \[\begin{tikzcd}
        (-)^{\rep} : \aMon_{/M}^{\vsur}\ar[r,shift left] &\aMon_{/M}^{\rep}\ar[l,shift left] : U
    \end{tikzcd}\]
    where the right adjoint is given by the forgetful functor $\aMon_{/M}^{\rep}\to\aMon_{/M}^{\vsur}$ and the left adjoint is given by the pre-exactification/repletion $(\theta : N\to M)\mapsto (\theta^{\rep} : N^{\rep}_{\theta}\simeq M\times_{M^{\gp}} N^{\gp}\to M).$
\end{cor}

We will mainly be concerned with the repletion in the case of the summation or codiagonal morphism $\sigma : M\oplus_P M\to M$ for some fixed morphism $\theta : P\to M$ of animated monoids. In this case, $\sigma$ admits a canonical section (in fact, two canonical sections), and the following lemma provides us with a useful concrete description of the repletion in this situation.

\begin{proposition}[\protect{\cite[Proposition~2.21]{BLPO_HKR}}]\label{repletion with section identification}
    Let $\theta : N\to M$ be a map of animated monoids, and assume that it has a section $\eta : M\to N.$ Then there is an equivalence
    \[
    N^{\rep}_{\theta}\xrightarrow{\sim} M\oplus\cofib(\eta^{\gp})
    \]
    of animated monoids over and under $M,$ induced by the canonical map
    \[
    \theta^{\rep} : N^{\rep}_{\theta}\to M
    \]
    and the composition
    \[
    N^{\rep}_{\theta}\to N^{\gp}\to\cofib(\eta^{\gp}),
    \]
    where the map $N^{\rep}_{\theta} \simeq M\times_{M^{\gp}} N^{\gp}\to N^{\gp}$ is the projection onto $N^{\gp}$ (which can be identified with the group completion of $N^{\rep}_{\theta}$), and the map $N^{\gp}\to\cofib(\eta^{\gp} : M^{\gp}\to N^{\gp})$ is the canonical map.\footnote{Note that in the category of animated monoids, finite products and finite coproducts coincide.}
\end{proposition}

In the specific situation of the codiagonal morphism $\sigma : M\oplus_P M\to M,$ we may simplify the description of the cofiber $\cofib(M^{\gp}\to \left(M\oplus_P M\right)^{\gp})$ (for either section) using the following lemma.

\begin{lemma}
    Let $\cC$ be a stable $\infty$-category, and let $f : X\to Y$ and $g : X\to Y'$ be two maps in $\cC.$ Then there is a canonical equivalence $\cofib(Y\to Y\oplus_X Y')\simeq \cofib(X\xrightarrow{g} Y').$
\end{lemma}
\begin{proof}
    This follows from the fact that $X\xrightarrow{(f,-g)} Y\oplus Y'\to Y\oplus_{f,X,g} Y'$ is a cofiber sequence, together with the square
    \[\begin{tikzcd}
        0\ar[r]\ar[d] &Y\ar[r,"\id"]\ar[d] & Y\ar[d]\\
        X\ar[r,"{(f,-g)}"]\ar[d] & Y\oplus Y'\ar[r]\ar[d] &Y\oplus_X Y'\ar[d]\\
        X\ar[r] & Y'\ar[r] & Q
    \end{tikzcd}\]
    so that $Q$ is identified with both $\Cof(Y\to Y\oplus_X Y')$ and $\Cof(X\to Y').$
\end{proof}

\begin{cor}
    Let $\theta : P\to M$ be a map of animated monoids. Then there is a canonical equivalence
    \[\left(M\oplus_P M\right)_{\sigma}^{\rep}\simeq M\oplus M^{\gp}/P^{\gp}.\]
\end{cor}

If we identify $\left(M\oplus_P M\right)^{\rep}\simeq M\oplus \left(M\oplus_P M\right)^{\gp}/M^{\gp},$ where $M\to M\oplus_P M$ via the canonical section corresponding to the first component, then the diagram
\[\begin{tikzcd}
    M\oplus_P M\ar[r]\ar[rd] & M\oplus \left(M\oplus_P M\right)^{\gp}/M^{\gp}\ar[r]\ar[d] & \left( M\oplus_P M\right)^{\gp}\ar[d]\\
    &M\ar[r] & M^{\gp}
\end{tikzcd}\]
commutes. Here, the map $M\oplus_P M\to M\oplus \left(M\oplus_P M\right)^{\gp}/M^{\gp}\simeq M\times \left(M\oplus_P M\right)^{\gp}/M^{\gp}$ is induced by summation on the first component and the composition of the groupification followed by the canonical quotient on the second component, the map $M\oplus\left(M\oplus_P M\right)^{\gp}/M^{\gp}\to M$ is the projection onto the first component, and the map $M\oplus \left(M\oplus_P M\right)^{\gp}/M^{\gp}\to (M\oplus_P M)^{\gp}$ is the inverse of the map $(M\oplus_P M)^{\gp}\to M^{\gp}\oplus \left(M\oplus_P M\right)^{\gp}/M^{\gp}$ induced by summation in the first component and the projection in the second. 

On the other hand
\[\begin{tikzcd}
    M\oplus_P M\ar[r]\ar[rd] & M\oplus M^{\gp}/P^{\gp}\ar[r]\ar[d] & \left( M\oplus_P M\right)^{\gp}\ar[d]\\
    &M\ar[r] & M^{\gp}
\end{tikzcd}\]
has $M\oplus_P M\to M\oplus M^{\gp}/P^{\gp}$ induced by summation on the first component and projection onto the second component, $M\oplus M^{\gp}/P^{\gp}\to M$ given by projection onto the first component, and $M\oplus M^{\gp}/ P^{\gp}\to (M\oplus_P M)^{\gp}$ given by groupification composed with the inverse of the map $(M\oplus_P M)^{\gp}\to M^{\gp}\oplus M^{\gp}/P^{\gp}$ which is summation onto the first component and projection onto the second.

Tracing through these identifications, the map $M\oplus M^{\gp}/P^{\gp}\to M\oplus\left(M\oplus_P M\right)^{\gp}/M^{\gp}$ is given by the identity on the $M$ component, and via the diagram
\[\begin{tikzcd}
    P^{\gp}\ar[r]\ar[d] & M^{\gp}\ar[d]\ar[r] &M^{\gp}/P^{\gp}\ar[d]\\
    M^{\gp}\ar[r] &(M\oplus_P M)^{\gp}\ar[r] &(M\oplus_P M)^{\gp}/M^{\gp}
\end{tikzcd}\]
on the second (where the left hand square is obtained by applying groupification to the square defining the pushout).

\begin{mydef}
    Let $(f,f^\flat) : (B,N)\to (A,M)$ be a map of animated prelog rings. We say $(f,f^\flat)$ is \emph{virtually surjective} (respectively, \emph{exact}, \emph{replete}) if the morphism of animated monoids $f^\flat : N\to M$ is virtually surjective (respectively, \emph{exact}, \emph{replete}).
\end{mydef}

We are now ready to state the universal property of the repletion.

\begin{proposition}\label{universal property of repletion}
    Let $(f,f^\flat)  : (B,N)\to (A,M)$  be a virtually surjective map of animated prelog rings. The prelog ring $(B\otimes_{\ZZ[N]}\ZZ[N^{\rep}_{f^\flat}],N^{\rep}_{f^\flat})$ together with the canonical maps
    \[\begin{tikzcd}
        N\ar[r]\ar[d] &N^{\rep}_{f^\flat}\ar[r,"(f^\flat)^{\rep}"]\ar[d] &M\ar[d]\\
        B\ar[r] &B\otimes_{\ZZ[N]}\ZZ[N^{\rep}_{f^\flat}]\ar[r] &A
    \end{tikzcd}\]
    is initial among factorizations $(B,N)\to (B',N')\xrightarrow{(g,g^\flat)} (A,M)$ of $(f,f^\flat)$ such that $g^\flat$ is replete. In other words, let $\aPreLog^{\rep}_{(B,N)//(A,M)}$ denote the sub-$\infty$-category of factorizations $(B,N)\to (B',N')\to (A,M)$ of $(f,f^\flat)$ spanned by those such that $(B',N')\to (A,M)$ is replete. Then the composition $(B,N)\to (B,N)^{\rep}_{(f,f^\flat)}\to (A,M)$ is initial in $\aPreLog^{\rep}_{(B,N)//(A,M)}.$
\end{proposition}
\begin{proof}
    This follows from \ref{initial animated prelog ring under monoid}.
\end{proof}

We can rephrase \ref{universal property of repletion} as an adjunction.

\begin{cor}
    Let $(A,M)$ be an animated prelog ring. Then there is an adjunction
    \[\begin{tikzcd}
        (-)^{\rep} : \aPreLog_{/(A,M)}^{\vsur}\ar[r,shift left] &\aPreLog_{/(A,M)}^{\rep}\ar[l,shift left] : U
    \end{tikzcd}\]
    where the right adjoint is the forgetful functor/inclusion $\aPreLog_{/(A,M)}^{\rep}\to\aPreLog_{/(A,M)}^{\vsur},$ and the right adjoint is the repletion
    \[
    ((B,N)\xrightarrow{(f,f^\flat)}(A,M))\mapsto (B,N)^{\rep}_{(f,f^\flat)} := (B\otimes_{\ZZ[N]}\ZZ[N^{\rep}_{f^\flat}], N^{\rep}_{f^\flat}).
    \]
\end{cor}

We end with a few comments about the relationship between the repletion and strictness, in the case of a map $(B,N)\to (A,M)$ such that the map $N\to M$ admits a section.

\begin{lemma}[\protect{\cite[Corollary~2.23]{BLPO_HKR}}]\label{repletion with section is strict}
    Suppose $(f,f^\flat) : (B,N,\beta)\to (A,M,\alpha)$ is a morphism of animated prelog rings such that the map $f^\flat : N\to M$ admits a section $\eta.$ Then the repletion $(f,f^\flat)^{\rep} : (B,N,\beta)^{\rep}_{(f,f^\flat)}\to (A,M,\alpha)$ is strict.
\end{lemma}
\begin{proof}
    This follows from \ref{repletion with section identification}. In particular, the repletion can be identified with
    \[
    (B\otimes_{\ZZ[N]}\ZZ[M\oplus\cofib(\eta^{\gp})],M\oplus\cofib(\eta^{\gp}))\to (A,M)
    \]
    One can check directly from this description that logification induces an equivalence
    \[
    (f^{\rep})^*(M\oplus\cofib(\eta^{\gp}))^a_{(f^{\rep})^*\beta^{\rep}} \to M^a,
    \]
    or alternatively, one can appeal to \ref{logificiation of direct sum with grouplike prelog structure}.
\end{proof}

Note that only the map $f^\flat$ of monoids is required to admit a section in the previous lemma. In particular, it is not necessary that the map $(f,f^\flat)$ of prelog rings in the previous lemma itself admits a section. However, if the map $(f,f^\flat)$ does admit a section, we can say more.

\begin{proposition}
    Let $(f,f^\flat) : (B,N)\to (A,M)$ be a map of animated prelog rings which admits a section $(g,\eta) : (A,M)\to (B,N).$  Then the composition $(A,M)\to (B,N)\to (B,N)^{\rep}_{(f,f^\flat)}$ is strict.
\end{proposition}
\begin{proof}
    By \ref{repletion with section identification}, we have
    \[
    N^{\rep}_{f^\flat}\simeq M\oplus\cofib(\eta^{\gp}).
    \]
    Via this identification, we may view the map of monoids underlying $(A,M)\to (B,N)^{\rep}_{(f,f^\flat)}$ as being the canonical map $M\to M\oplus\cofib(\eta^{\gp}).$ The result then follows from \ref{logificiation of direct sum with grouplike prelog structure}.
\end{proof}

In particular, this applies to the multiplication map $(A,M)\otimes_{(k,P)}(A,M)\to (A,M).$ 

\begin{cor}
    Let $(k,P)\to (A,M)$ be a morphism of animated prelog rings. Then both of the canonical morphisms $(A,M)\to (A,M)\otimes^{\rep}_{(k,P)}(A,M)$ are strict.
\end{cor}

\begin{remark}
    A theory of general global derived log schemes or stacks has not yet been thoroughly investigated, although there are candidates for a replacement of Olsson's log stack, see \cite{pridham2013derived} and \cite{mao2024revisitingderivedcrystallinecohomology} for more details. It would be interesting to see whether the fiber product $X\times_{\cL}X$ satisfies a universal property similar to \ref{Martin exactification} for either variant of derived log stack $\cL.$ If there was a derived log stack $\cL$ satisfying such a universal property, this would yield another comparison map $X\times^{\rep}_S X\to X\times_\cL X$ by the above corollary, and one could check what this map produces on Hochschild homology.
\end{remark}

\section{The Logarithmic Cotangent Complexes}

There are two distinct flavors of logarithmic cotangent complexes, and it will be important for us to understand the relationship between them in order to compare the two notions of logarithmic Hochschild homologies which we consider here.

It would be reasonable to ask why there is more than one theory of logarithmic cotangent complex. This can be explained by the following observation, due to W. Bauer.
\begin{claim}[\protect{\cite[Section~7]{Olssoncotangent}}]\label{no theory of log cotangent exists satisfying}
    There does not exist a theory of cotangent complexes for log schemes
    \[
    (X\to S)\mapsto \LL^{\log}_{X/S}\in\sfD_{\qc}(X)
    \]
    satisfying the following properties:
    \begin{enumerate}
        \item If $X\to S$ is strict, then the diagram
        \[\begin{tikzcd}
            X\ar[d]\ar[r] &\underline{X}\ar[d]\\
            S\ar[r] &\underline{S}
        \end{tikzcd}\]
        induces an equivalence $\LL_{\underline{X}/\underline{S}}\to\LL^{\log}_{X/S}$ between the ordinary cotangent complex of the underlying morphism of schemes and the log cotangent complex.
        \item If $X\to S$ is log smooth, then $\LL^{\log}_{X/S}\simeq\Omega^1_{X/S}[0].$
        \item For every commutative diagram
        \[\begin{tikzcd}
            X'\ar[r,"a"]\ar[d,"f'"] &X\ar[d,"f"]\\
            S'\ar[r,"b"] &S
        \end{tikzcd}\]
        there is a natural map $La^*\LL^{\log}_{X/S}\to\LL^{\log}_{X'/S'},$ which is identified with the corresponding map of ordinary cotangent complexes when $f$ and $f'$ are strict.
        \item For every composite $X\xrightarrow{f} Y\xrightarrow{g} Z,$ there is a distinguished triangle
        \[
        Lf^*\LL^{\log}_{Y/Z}\to\LL^{\log}_{X/Z}\to\LL^{\log}_{X/Y}\to Lf^*\LL^{\log}_{Y/Z}[1].
        \]
    \end{enumerate}
\end{claim}

The two theories of log cotangent complex must each then compromise on one of these properties. Olsson's log cotangent complex sacrifices the existence of a distinguished triangle for all compositions of fine log schemes, and on the other hand, Gabber's log cotangent complex compromises by relaxing the condition that $\LGab_{X/S}$ be represented by $\Omega^1_{X/S}[0]$ for $X/S$ log smooth.

\subsection{Logarithmic Differentials}

The theory of logarithmic geometry has a notion of differentials, similar to the ordinary theory of algebraic geometry. Just as for schemes, the theory of differentials is closely related to \'etale and smooth morphisms, and is also closely tied to the logarithmic cotangent complex(es), which play a key role in our study of logarithmic Hochschild homology. In this section, we recall the definition of the module of log differentials of a map of (pre)log rings or schemes, and the relation to the logarithmic versions of \'etale and smooth morphisms. We begin with a brief overview of this theory. 

\begin{mydef}
    Let $(f,f^\flat) : (k,P,\beta)\to (A,M,\alpha)$ be a morphism of prelog rings, and let $E$ be an $A$-module. A \emph{log derivation of $(f,f^\flat)$ with values in $E$} is a pair $(D,\log D),$ where $D : A\to E$ is a morphism of $k$-modules, and $\log D : M\to E$ is a morphism of monoids, such that:
    \begin{enumerate}
        \item $D$ is a derivation of $f : k\to A$ valued in $E,$
        \item $D(\alpha(m)) = \alpha(m)\log D(m)$ for all $m\in M;$ and
        \item $\log D(f^\flat(p)) = 0$ for all $p\in P.$
    \end{enumerate}
\end{mydef}

The set of log derivations of $(f,f^\flat)$ with values in $E$ will be denoted $\Der_{(f,f^\flat)}(E).$ Notice that $\Der_{(f,f^\flat)}(E)$ is equivalent to the pullback
\[\begin{tikzcd}
    \Der_{(f,f^\flat)}(E)\ar[r]\ar[d] &\Hom_{\PreLog_{(k,P)/}}((A,M),(A\oplus E,M\oplus E))\ar[d,"\varpi_J\circ -"]\\
    \{\id_{(A,M)}\}\ar[r] &\Hom_{\PreLog_{(k,P)/}}((A,M),(A,M))
\end{tikzcd}\]
where $(A\oplus E,M\oplus E)$ is the prelog ring whose underlying ring is the square zero extension $A\oplus E$ of $A$ by $E,$ and whose underlying monoid is $M\oplus(1 + E),$ with structure map
\begin{align*}
    M\oplus (1 + E)&\to A\oplus E\\
    (m,1 + x)&\mapsto (\alpha(m),x).
\end{align*}
The map $\varpi_J : (A\oplus E,M\oplus E)\to (A,E)$ is the canonical projection. The $(k,P)$-algebra structure on $(A\oplus E,M\oplus E)$ is induced by the $(k,P)$-algebra structure on $(A,M)$ and the canonical section $(A,M)\to (A\oplus E,M\oplus E)$ of $\varpi_J.$

\begin{proposition}[\protect{\cite[Proposition~IV.1.1.2]{OgusLectures}}]
    Let $(f,f^\flat) : (k,P,\beta)\to (A,M,\alpha)$ be a morphism of prelog rings. The functor
    \[
    \Der_{(f,f^\flat)} : \Mod_{A}\to\Set
    \]
    is representable by a universal derivation
    \[
    (d,\log d) : (A,M)\to\Omega^1_{(f,f^\flat)} = \Omega^1_{(A,M)/(k,P)}
    \]
    whose target is called the \emph{module of log differentials of $(\phi,\phi^\flat).$}
\end{proposition}

As is the case for the usual module of K\"ahler differentials, there are many equivalent presentations of $\Omega^1_{(f,f^\flat)}.$ One possibility is
\[
\Omega^1_{(A,M)/(k,P)} = \left(\Omega^1_{A/k}\oplus (A\otimes M^{\gp}/P^{\gp}\right)/R,
\]
where $R$ is the submodule of $\left(\Omega^1_{A/k}\oplus (A\otimes M^{\gp}/P^{\gp}\right)$ generated by elements of the form $(d\alpha(m),-\alpha(m)\otimes\pi(m),$ for $m\in M,$ where $\pi : M\to M^{\gp}/P^{\gp}$ is the canonical map. In this case, we have
\begin{align*}
    d : A&\to\left(\Omega^1_{A/k}\oplus (A\otimes M^{\gp}/P^{\gp}\right)/R\\
    a&\mapsto [da,0]
\end{align*}
and
\begin{align*}
    \log d : M&\to \left(\Omega^1_{A/k}\oplus (A\otimes M^{\gp}/P^{\gp}\right)/R\\
    m&\mapsto [0,1\otimes\pi(m)].
\end{align*}

The module of log differentials enjoys many properties directly analogous to the usual module of K\"ahler differentials. We record the most important ones here.

First of all, formation of $\Omega^1_{(f,f^\flat)}$ is functorial in the morphism $(f,f^\flat).$
\begin{proposition}[\protect{\cite[Proposition~IV.1.1.3]{OgusLectures}}]\label{log differentials functoriality}
    Let 
    \[\begin{tikzcd}
        (A,M,\alpha)\ar[r,"{(g,g^{\flat})}"] & (A',M',\alpha')\\
        (k,P,\beta)\ar[r]\ar[u,"{(f,f^\flat)}"] &(k',P',\beta')\ar[u,"{(f',(f')^{\flat})}"]
    \end{tikzcd}\]
    be a commutative diagram of prelog rings. Then there is a canonical homomorphism
    \[
    A'\otimes_A\Omega^1_{(f,f^\flat)}\to\Omega^1_{(f',(f')^\flat)},
    \]
    compatible with the universal derivations. This map is an isomorphism if the square is coCartesian, or if $(k,P,\beta)\to (k',P',\beta')$ is an isomorphism, $M\to M'$ is an isomorphism, and $A\to A'$ is a localization.
\end{proposition}

Let us explain how the map in question is constructed. Let $E$ be an $A'$-module, which we may consider as an $A$-module via the map $g : A\to A'.$ Then we obtain a map
\begin{align*}
    \Der_{(f',(f')^\flat)}(E)&\to\Der_{(f,f^\flat)}(E)\\
    (\partial,\log\partial)&\mapsto (\partial\circ\rho,\log\partial\circ \rho^{\flat}).
\end{align*}
If we apply this with $E = \Omega^1_{(f',(f')^\flat)},$ we obtain
\[
\Hom_{A'}(\Omega^1_{(f',(f')^\flat)},\Omega^1_{(f',(f')^\flat)})\to\Hom_{A}(\Omega^1_{(f,f^{\flat})},\Omega^1_{(f',(f')^\flat)})
\]
and thus
\[
\Der_{(f',(f')^\flat)}(\Omega^1_{(f',(f')^\flat)})\to \Der_{(f,f^{\flat})}(\Omega^1_{((f',(f')^\flat)}).
\]
The image of $\id : \Omega^1_{(f',(f')^\flat)}\to\Omega^1_{(f',(f')^\flat)}$ produces a canonical map $\Omega^1_{(f,f^\flat)}\to\Omega^1_{(f',(f')^\flat)}$ of $A$-modules, and thus by adjunction
\[
A'\otimes_A\Omega^1_{(f,f^{\flat})}\to\Omega^1_{(f',(f')^\flat)}.
\]

\begin{proposition}\label{fundamental exact sequence for log differentials}
    Let $(k,P)\to (A,M)\xrightarrow{(f,f^{\flat})} (B,N)$ be morphisms of prelog rings. Then there is an exact sequence
    \[
    B\otimes_{A}\Omega^1_{(A,M)/(k,P)}\to \Omega^1_{(B,N)/(k,P)}\to\Omega^1_{(B,N)/(A,M)}\to 0
    \]
    of $B$-modules, where the maps are induced by the squares in the diagram
    \[\begin{tikzcd}
        (A,M)\ar[r] &(B,N)\ar[r] &(B,N)\\
        (k,P)\ar[u]\ar[r] &(k,P)\ar[u]\ar[r] &(A,M)\ar[u]
    \end{tikzcd}\]
    via \ref{log differentials functoriality}.
\end{proposition}
\begin{proof}
    This follows from the observation that
    \[
    \Der_{(B,N)/(A,M)}(E)\cong\Ker\left(\Der_{(B,N)/(k,P)}(E)\to\Der_{(A,M)/(k,P)}(E)\right),
    \]
    where the map 
    \[
    \Der_{(B,N)/(k,P)}(E)\to\Der_{(A,M)/(k,P)}(E)
    \]
    sends a derivation $(\partial,\log\partial) : (B,N)\to E$ to the derivation determined by
    \[
    A\to B\xrightarrow{\partial} E
    \]
    and
    \[
    M\to N\xrightarrow{\log\partial} E.
    \]
    Indeed, if $(\partial,\log\partial)$ is in the kernel, then $\partial(f(a)) = 0$ for all $a\in A,$ and $\log\partial(f^{\flat}(m)) = 0$ for all $m\in M,$ which immediately implies that $(\partial,\log\partial)$ in fact defines a log derivation of $(f,f^{\flat}).$
\end{proof}

Using the previous results, we can deduce the following.

\begin{cor}
    Suppose that
    \[\begin{tikzcd}
        (A,M,\alpha)\ar[r,"{(h,h^{\flat})}"] & (A',M',\alpha')\\
        (k,P,\beta)\ar[r,"{(g,g^{\flat})}"]\ar[u,"{(f,f^{\flat})}"] &(k',P',\beta')\ar[u,"{(f',(f')^{\flat})}"']
    \end{tikzcd}\]
    is a coCartesian diagram of prelog rings. Then there is a canonical isomorphism
    \[
    (A'\otimes_A \Omega^1_{(f,f^{\flat})})\oplus (A'\otimes_{k'}\Omega^1_{(g,g^{\flat})})\to \Omega^1_{(A',M')/(k,P)}
    \]
    induced by the maps $A'\otimes_A \Omega^1_{(f,f^{\flat})}\to \Omega^1_{(A',M')/(k,P)}$ and $A'\otimes_{k'}\Omega^1_{(k',P')/(k,P)}\to \Omega^1_{(A',M')/(A,M)}$ of \ref{fundamental exact sequence for log differentials}.
\end{cor}
\begin{proof}
    Consider the diagram
    \[\begin{tikzcd}
        &A'\otimes_A\Omega^1_{(A,M)/(k,P)}\ar[d]\ar[rd]\\
        A'\otimes_{k'}\Omega^1_{(k',P')/(k,P)}\ar[r]\ar[rd] &\Omega^1_{(A',M')/(k,P)}\ar[d]\ar[r] & \Omega^1_{(A',M')/(k',P')}\ar[r] &0\\
        &\Omega^1_{(A',M')/(A,M)}\ar[d]\\
        &0
    \end{tikzcd}\]
    where the diagonal maps are the canonical isomorphisms of \ref{log differentials functoriality} and the vertical and horizontal morphisms are the short exact sequences of \ref{fundamental exact sequence for log differentials}.

    It follows from the construction of the morphisms that the diagram commutes. Thus, the maps $A'\otimes_A\Omega^1_{(A,M)/(k,P)}\to\Omega^1_{(A',M')/(k,P)}$ and $A'\otimes_{k'}\Omega^1_{(k',P')/(k,P)}\to\Omega^1_{(A',M')/(k,P)}$ are injective, and the splitting lemma implies the result.
\end{proof}

\begin{proposition}[\protect{\cite[Proposition~IV.1.1.4]{OgusLectures}}]
    Let $f^{\flat} : P\to M$ be a morphism of monoids, and let $f : k[P]\to k[M]$ be the induced map of monoid algebras. If $k[P]$ and $k[M]$ are endowed with the obvious prelog structures $P\to k[P]$ and $M\to k[M],$ then we have a natural isomorphism
    \[
    \Omega^1_{(k[M],M)/(k[P],P)}\cong k[M]\otimes M^{\gp}/P^{\gp}.
    \]
\end{proposition}

\begin{example}
    Let us calculate the log differentials of the map
    \[
    (f,f^{\flat}) : (k_0[P],P)\to (A_0[M],M)
    \]
    where $k_0[P]$ and $A_0[M]$ are given the obvious prelog structures, and $f$ is induced by the map $f^{\flat} : P\to M$ of monoids and a map $f_0 : k_0\to A_0$ of rings.

    Observe that there is a pushout diagram
    \[\begin{tikzcd}
        (k_0[P],P)\ar[r]\ar[d] &(k_0[M],M)\ar[d]\\
        (A_0[P],P)\ar[r] &(A_0[M],M).
    \end{tikzcd}\]

    Then we have
    \begin{align*}
        \Omega^1_{(A_0[M],M)/(k_0[P],P)} &\cong \left(A_0[M]\otimes_{A_0[P]}\Omega^1_{(A_0[P],P)/(k_0[P],P)}\right)\oplus\left(A_0[M]\otimes_{k_0[M]}\Omega^1_{(k_0[M],M)/(k_0[P],P)}\right).
    \end{align*}
    Since $(k_0[P],P)\to (A_0[P],P)$ is strict, we have
    \[
    \Omega^1_{(A_0[P],P)/(k_0[P],P)} \cong \Omega^1_{A_0[P]/k_0[P]}\cong k_0[P]\otimes_{k_0}\Omega^1_{A_0/k_0}.
    \]
    Moreover, \cite[Proposition~1.1.4]{OgusLectures} implies that
    \[
    \Omega^1_{(k_0[M],M)/(k_0[P],P)}\cong k_0[M]\otimes M^{\gp}/P^{\gp}.
    \]
    Putting these together, we have
    \begin{align*}
        \Omega^1_{(A_0[M],M)/(k_0[P],P)} &\cong \left(A_0[M]\otimes_{A_0[P]}\Omega^1_{(A_0[P],P)/(k_0[P],P)}\right)\oplus\left(A_0[M]\otimes_{k_0[M]}\Omega^1_{(k_0[M],M)/(k_0[P],P)}\right)\\
        &\cong \left(k_0[M]\otimes_{k_0[P]}\otimes k_0[P]\otimes_{k_0}\Omega^1_{A_0/k_0}\right)\oplus\left(A_0[M]\otimes_{k_0[M]}k_0[M]\otimes M^{\gp}/P^{\gp}\right) \\
        &\cong \left(k_0[M]\otimes_{k_0}\Omega^1_{A_0/k_0}\right)\oplus\left(A_0[M]\otimes_{k_0[M]}k_0[M]\otimes M^{\gp}/P^{\gp}\right)\\
        &\cong \left(k_0[M]\otimes_{k_0}\Omega^1_{A_0/k_0}\right)\oplus\left(A_0[M]\otimes M^{\gp}/P^{\gp}\right).
    \end{align*}

    In particular, for a map of the form $(k_0[X,\NN^S],\NN^S)\to (k_0[X,Y,\NN^{S}\oplus\NN^{T}],\NN^{S}\oplus\NN^{T}),$ we have
    \[
    \Omega^1_{(k_0[X,Y,\NN^{S}\oplus\NN^{T}],\NN^{S}\oplus\NN^{T})/(k_0[X,\NN^S],\NN^S)}\simeq k_0[X,Y,\NN^S\oplus\NN^T]^{\oplus Y}\oplus k_0[X,Y,\NN^S\oplus\NN^T]^{\oplus T}.
    \]
\end{example}

For a morphism $f : X\to Y$ of prelog schemes, we may also define the notion of a log derivation of $f$ with values in a sheaf $\cE$ of $\OX$-modules analogously to the definition above for prelog rings. As in the ``affine'' case, there is also a sheaf $\Omega^1_{X/Y}$ and a universal log derivation valued in $\Omega^1_{X/Y}$ which represents the functor sending a sheaf $\cE$ of $\OX$-modules to the set of log derivations of $f$ valued in $\cE.$ Similarly to the affine situation, we have
\[
\Omega^1_{X/Y}\cong \left(\Omega^1_{\underline{X}/\underline{Y}}\oplus(\OX\otimes\cM_X^{\gp})\right)/\cR
\]
where $\cR$ is the sub $\OX$-module generated by sections of the form $(d\alpha_X(m),-\alpha_X(m)\otimes m)$ for $m\in\cM_X,$ and by $(0,1\otimes f^\flat(n))$ for $n\in f^{-1}\cM_Y,$ with the maps $d : \OX\to\Omega^1_{X/Y}$ and $\log d : \cM_X\to\Omega^1_{X/Y}$ defined analogously to the affine situation.

Additionally, the global version is a direct generalization of the affine version.

\begin{proposition}[\protect{\cite[Corollary~IV.1.2.6]{OgusLectures}}]
    Let $(f,f^{\flat}) : (k,P)\to (A,M)$ be a map of prelog rings and let $\Spec(A,M)\to\Spec(k,P)$ be the corresponding morphism of log schemes. Then $\Omega^1_{\Spec(A,M)/\Spec(k,P)}$ is the quasi-coherent sheaf associated to the $A$-module $\Omega^1_{(f,f^{\flat})}.$
\end{proposition}

If $f : X\to Y$ is a map between reasonably well-behaved log schemes, then the sheaf $\Omega^1_{X/Y}$ is also reasonably well-behaved, as made precise by the following proposition.

\begin{proposition}[\protect{\cite[Corollary~IV.1.2.8]{OgusLectures}}]
    If $f : X\to Y$ is a morphism of coherent log schemes, $\Omega^1_{X/Y}$ is quasi-coherent, and it is of finite type (respectively of finite presentation) if $\underline{f}$ is of finite type (respectively of finite presentation).
\end{proposition}

We also have the following direct generalization of \ref{log differentials functoriality}.

\begin{proposition}[\protect{\cite[Proposition~IV.1.2.15]{OgusLectures}}]
    Consider a commutative diagram
    \[\begin{tikzcd}
        X'\ar[d,"f'"]\ar[r,"g"] & X\ar[d,"f"]\\
        Y'\ar[r,"h"] &Y
    \end{tikzcd}\]
    of prelog schemes. Then there is a unique homomorphism
    \[
    g^*\Omega^1_{X/Y}\to\Omega^1_{X'/Y'}
    \]
    associated to the diagram which maps $1\otimes da\mapsto dg^\sharp(a)$ for $a\in g^{-1}\OX,$ and $1\otimes d\log m\mapsto d\log g^\flat(m)$ for every section $m\in g^{-1}\cM_X.$ The map is an isomorphism in the following cases:
    \begin{enumerate}
        \item $f'$ is the morphism of log schemes associated to the morphism $f$ of prelog schemes;
        \item The diagram is cartesian in the category of prelog schemes.\footnote{For this condition, the category $\PreLogSch$ may be replaced by any of $\LogSch,\LogSch^{\mathrm{coh}},\LogSch^{\fine},$ or $\LogSch^{\fs}.$}
    \end{enumerate}
    In either case, there is a commutative diagram of isomorphisms
    \[\begin{tikzcd}
        f'^*\Omega^1_{Y'/Y}\oplus g^*\Omega^1_{X/Y}\ar[rr]\ar[rd] &&\Omega^1_{X'/X}\oplus\Omega^1_{X'/Y'}\ar[ld]\\
        &\Omega^1_{X'/Y}.
    \end{tikzcd}\]
\end{proposition}

There is also a canonical exact sequence of sheaves of logarithmic differentials associated to a composition of (pre)log schemes.

\begin{proposition}[\protect{\cite[Proposition~2.3.1]{OgusLectures}}]
    Let $f : X\to Y$ and $g : Y\to Z$ be morphisms of (pre)log schemes. Then there is an exact sequence of sheaves of $\OX$-modules
    \[
    f^*\Omega^1_{Y/Z}\to\Omega^1_{X/Z}\to\Omega^1_{X/Y}\to 0
    \]
    induced by the functoriality morphisms associated to the diagram
    \[\begin{tikzcd}
        X\ar[r]\ar[d,"f"] &X\ar[r,"f"]\ar[d,"g\circ f"] &Y\ar[d,"g"]\\
        Y\ar[r,"g"] &Z\ar[r] &Z.
    \end{tikzcd}\]
\end{proposition}

\begin{remark}
    There are also differential criteria for log smoothness and log \'etaleness which parallel the differential criteria for classical smoothness and \'etaleness. We do not require these results, and refer the interested reader to \cite[Section~IV.3.2]{OgusLectures}.
\end{remark}


\subsection{Olsson's Logarithmic Cotangent Complex}

We begin by recalling the definition and basic properties of the logarithmic cotangent complex defined by Olsson in \cite{Olssoncotangent}. The definition of Olsson's logarithmic cotangent complex is simple to state, now that we have the stack of log structures.

\begin{mydef}
    Let $f : X\to S$ be a morphism of fine log schemes, with corresponding morphism $f_{\cM_X} : \underline{X}\to\cLog_S.$ The \emph{Olsson logarithmic cotangent complex} of $f$ is defined to be
    \[
    \LOls_{X/S} := \LL_{\underline{X}/\cLog_S}\in\QC(\underline{X}).
    \]
\end{mydef}

\begin{remark}
    We note that the Olsson logarithmic cotangent complex was originally defined as a projective system $(\dots\to\LL^{\Ols,\geq -n-1}_{X/S}\to\LL^{\Ols,\geq -n}_{X/S}\to\dots\to\LL^{\Ols,\geq 0}_{X/S})$ where each $\LL^{\Ols,\geq -n}_{X/S}$ is an essentially constant ind-object in $\sfD^{[-n,0]}(\OX).$ However, this was a technical limitation -- later developments would allow one to define the cotangent complex of a morphism of stacks as an object in an appropriate derived category rather than as a projective system (see \cite{laszlo2006operations}). Moreover, with modern $\infty$-categorical methods, we may enhance the cotangent complex of a map $f : \cX\to\cY$ to live in $\QC(\cX)$ and as we have already seen, this enhanced variant satisfies $\infty$-categorical properties analogous to the ones satisfied by its classical counterpart. One can recover the classical variants by either passing to the homotopy category, or by first doing so and then forming appropriate truncations.
\end{remark}

By \ref{cotangent complex and diagonal}, it follows that $\LOls_{X/S}\simeq\LL_{\underline{X}/\underline{X}\times^R_{\cLog_{S}}\underline{X}}[-1],$ generalizing \cite[Remark~3.4]{Olssoncotangent} to the non-log flat case. 

\begin{theorem}[\protect{\cite{Olssoncotangent}}]
    Let $f : X\to S$ be a morphism of fine log schemes. The logarithmic cotangent complex $\LOls_{X/S}$ satisfies the following properties:
    \begin{enumerate}
        \item If $X$ and $S$ are endowed with the trivial log structures, there is a canonical equivalence $\LL_{\underline{X}/\underline{S}}\simeq\LOls_{X/S}.$
        \item If $f$ is log smooth, then $\LOls_{X/S}$ is represented by $\Omega^1_{X/S}[0].$
        \item If 
        \[\begin{tikzcd}
            X'\ar[r,"a"]\ar[d,"f'"] &X\ar[d,"f"]\\
            S'\ar[r,"b"] &S
        \end{tikzcd}\]
        is a commutative diagram of fine log schemes, then there is a natural map
        \begin{equation}\label{LOls functoriality 1}
        La^*\LOls_{X/S}\to\LOls_{X'/S'}.
        \end{equation}
        If moreover the diagram is Cartesian and $f$ is log flat, then \ref{LOls functoriality 1} is an equivalence, and if in addition the composite $X'\to S'\to S$ satisfies condition (T) (\ref{condition T}) below, then the map
        \[
        Lf'^*\LOls_{S'/S}\oplus La^*\LOls_{X/S}\to\LOls_{X'/S}
        \]
        is an equivalence.
        \item Given a composite
        \[
        X\xrightarrow{f} Y\xrightarrow{g} Z
        \]
        satisfying condition \ref{condition T} below, there is a natural map
        \[
        \LOls_{X/Y}\to Lf^*\LOls_{Y/Z}[1]
        \]
        such that
        \[
        Lf^*\LOls_{Y/Z}\to\LOls_{X/Z}\to\LOls_{X/Y}\to Lf^*\LOls_{Y/Z}[1]
        \]
        is a cofiber sequence.
    \end{enumerate}
\end{theorem}

We obtain the following invariance result for strict morphisms as an immediate corollary.

\begin{cor}
    If $f$ is strict, then the canonical morphism
    \[
    \LOls_{\underline{X}/\underline{S}}\to\LOls_{X/S}
    \]
    induced by \ref{LOls functoriality 1} via the diagram
    \[\begin{tikzcd}
        X\ar[r]\ar[d,"f"] &\underline{X}\ar[d,"\underline{f}"]\\
        S\ar[r] &\underline{S}
    \end{tikzcd}\]
    is an equivalence. In particular, there is an equivalence $\LL_{\underline{X}/\underline{S}}\simeq\LOls_{X/S},$ where $\LL_{\underline{X}/\underline{S}}$ denotes the usual cotangent complex.
\end{cor}
\begin{proof}
    This follows from the fact that $X\to S$ is strict if and only if the diagram
    \[\begin{tikzcd}
        X\ar[r]\ar[d,"f"] &\underline{X}\ar[d,"\underline{f}"]\\
        S\ar[r] &\underline{S}
    \end{tikzcd}\]
    is Cartesian.
\end{proof}

The desirable property which $\LOls$ fails to satisfy is the existence of a distinguished triangle for all compositions $X\to Y\to Z$ of fine log schemes. This is where condition (T) comes in.

\begin{mydef}\label{condition T}
    Let $X\xrightarrow{f} Y\xrightarrow{g} Z$ be a composition of fine log schemes. We say that the composition satisfies \emph{condition (T)} if there exists a family of of commutative diagrams
    \[\begin{tikzcd}
        X\ar[d,"f"] &\ar[l] X\times_Z Z_i\ar[d] & \ar[l] X\times_Y Y_i\ar[d] & \ar[l,"\pi_{X_i}"']\ar[ld,"f_i"] X_i\\
        Y\ar[d,"g"] &\ar[l] Y\times_Z Z_i\ar[d] & \ar[l,"\pi_{Y_i}"']\ar[dl,"g_i"] Y_i\\
        Z&\ar[l,"\pi_{Z_i}"] Z_i
    \end{tikzcd}\]
    such that
    \begin{enumerate}
        \item The schemes $\underline{X_i},\underline{Y_i},$ and $\underline{Z_i}$ are all affine.
        \item The maps $\pi_{X_i},\pi_{Y_i},$ and $\pi_{Z_i}$ are all strict, and their underlying morphisms of schemes are flat and locally of finite presentation.
        \item The induced map $\coprod_i \underline{X_i}\to\underline{X}$ is surjective.
        \item There exist charts
        \begin{align*}
            &\beta_{X_i} : Q_{X_i}\to\cM_{X_i}, &\beta_{Y_i} : Q_{Y_i}\to\cM_{Y_i}, &&\beta_{Z_i} : Q_{Z_i}\to\cM_{Z_i}
        \end{align*}
        and injective maps
        \begin{align*}
            &\theta_{1,i} : Q_{Z_i}\to Q_{Y_i}, &\theta_{2,i} : Q_{Y_i}\to Q_{X_i}
        \end{align*}
        such that the diagrams
        \[\begin{tikzcd}
            Q_{X_i}\ar[r,"\beta_{X_i}"] &\cM_{X_i} &Q_{Y_i}\ar[r,"\beta_{Y_i}"] &\cM_{Y}\\
            Q_{Y_i}\ar[u,"{\theta_{2,i}}"]\ar[r,"{\beta_{Y_i}}"] &f_i^*\cM_{Y_i}\ar[u,"{f_i^\flat}"] &Q_{Z_i}\ar[u,"{\theta_{1,i}}"]\ar[r,"\beta_{Z_i}"] &g_i^*\cM_{Z_i}\ar[u, "{g_i^\flat}"]
        \end{tikzcd}\]
        commute and
        \[
        \Tor^j_{\cO_{Z_i}\otimes_{\ZZ[Q_{Z_i}]}\ZZ[Q_{Y_i}]}(\cO_{Z_i}\otimes_{\ZZ[Q_{Z_i}]}\ZZ[Q_{X_i}],\cO_{Y_i}[G_i]) = 0
        \]
        for all $j > 0,$ where $G_i = \cok(Q_{Z_i}^{\gp}\to Q_{Y_i}^{\gp})$ and $\cO_{Y_i}[G_i]$ is viewed as a $\cO_{Z_i}\otimes_{\ZZ[Q_{Z_i}]}\ZZ[Q_{Y_i}]$-algebra via the map
        \begin{align*}
            \cO_{Z_i}\otimes_{\ZZ[Q_{Z_i}]}\ZZ[Q_{Y_i}]&\to\cO_{Y_i}[G_i]\\
            t\otimes e^q&\mapsto g^*(t)\beta_{Q_{Y_i}}(q)\cdot [q],
        \end{align*}
        where $[q]$ denotes the image of $q$ in $G.$
    \end{enumerate}
\end{mydef}

Condition (T) for $X\xrightarrow{f} Y\xrightarrow{g} Z$ is rather involved, but it is automatically satisfied if
\begin{itemize}
    \item $g$ is log flat, (\cite[3.14]{Olssoncotangent}) or
    \item $f$ is an integral morphism (\cite[3.15]{Olssoncotangent}).
\end{itemize}

\subsection{Gabber's Logarithmic Cotangent Complex}

The second variant of log cotangent complex is due to Gabber. The construction of Gabber's log cotangent complex directly parallels the construction of the usual cotangent complex: in short, it is the nonabelian derived functor of the functor of logarithmic differentials
\[
((A,M)/(k,P))\mapsto \Omega^1_{(A,M)/(k,P)}.
\]

\begin{remark}
    One can also define Gabber's log cotangent complex in the generality of (pre)log ringed structures in a topos (with enough points), and it is this generality which Olsson uses to talk about the value of $\LGab_{X/S}$ when $X$ and $S$ are log algebraic spaces (which is important for the construction of the comparison morphism $\LGab\to\LOls$). However, similarly to the situation with the usual cotangent complex, we may also define $\LGab_{X/S}$ for log algebraic spaces via sheafification/descent from prelog rings, and we will adopt this perspective.
\end{remark}

For additional information, we refer the reader to \cite[Section 8]{Olssoncotangent}, \cite[Section 4.2]{Sagave_2014}, and \cite[Section 3.3]{BLPO_HKR}.

Gabber's cotangent complex also represents derived log derivations in the following sense.

\begin{mydef}
    Let $(A,M,\alpha)$ be an animated prelog ring, and let $E\in\cMod_A^{\cn}.$
    \begin{enumerate}
        \item The \emph{derived log square zero extension of $(A,M,\alpha)$ by $E$} is the animated prelog ring whose underlying animated ring is the derived square zero extension $A\oplus E$ of $A$ by $E,$ and whose animated prelog structure is given by
        \[
            (s\circ \alpha)\oplus\pi_E : M\oplus\fib((A\oplus E)^\times\xrightarrow{\pi_E} A^\times) \to A\oplus E,
        \]
        where $\pi_E : A\oplus E\to A$ is the canonical projection and $s : A\to A\oplus E$ is its canonical section. Observe that there is a map $\varpi_E : (A\oplus E,M\oplus E)\to (A,M)$ which has a canonical section $(A,M)\to (A\oplus E,M\oplus E).$
        \item Suppose that $(k,P)\to (A,M)$ is a map of animated prelog rings, and that $E\in\cMod_A^{\cn}.$ The \emph{space of derived log derivations of $(A,M)$ over $(k,P)$ valued in $E$} is defined to be the pullback
        \[\begin{tikzcd}
            \Der_{(A,M)/(k,P)}(E)\ar[r]\ar[d] &\Map_{\aPreLog_{(k,P)/}}((A,M),(A\oplus E,M\oplus E))\ar[d,"\varpi_E\circ -"]\\
            \{\id_{(A,M)}\}\ar[r] &\Map_{\aPreLog_{(k,P)/}}((A,M),(A,M))
        \end{tikzcd}\]
    \end{enumerate}
\end{mydef}

Notice that $\Der_{(A,M)/(k,P)}(E)$ reduces to the classical notion of log derivation when $(k,P)\to (A,M)$ is a map of discrete prelog rings, and $E$ is a discrete $A$-module. 

\begin{mydef}
    Let $(f,f^\flat) : (k,P)\to (A,M)$ be a map of prelog rings. The \emph{Gabber log cotangent complex} is a connective $A$-module $\LGab_{(f,f^{\flat})} =\LGab_{(A,M)/(k,P)}$ such that there is a natural equivalence of functors
    \[
    \Map_{\Mod_A^{\cn}}(\LGab_{(A,M)/(k,P)}, -)\simeq\Der_{(A,M)/(k,P)}(-)
    \]
\end{mydef}

\begin{proposition}
    Let $(k,P)\to (A,M)$ be a map of prelog rings. Then the Gabber log cotangent complex $\LGab_{(A,M)/(k,P)}$ exists.
\end{proposition}

Ignoring functoriality, the object $\LGab_{(A,M)/(k,P)}$ can be constructed similarly to the case of animated rings; namely, if $(k,P)$ is a discrete prelog ring, $\LGab_{(A,M)/(k,P)}\simeq\Omega^1_{(A_\bullet,M_\bullet)/(k,P)}\otimes_{A_\bullet} A,$ where $(A_\bullet,M_\bullet)$ is a simplicial resolution of $(A,M)$ by free $(k,P)$-prelog algebras. One can then define $\LGab_{(A,M)/(k,P)}$ for a general morphism of animated prelog rings as the cofiber of the canonical morphism $\LGab_{(k,P)/(\ZZ,\{1\})}\otimes_k A\to\LGab_{(A,M)/(\ZZ,\{1\})},$ for example.

If we do want to encode all relevant functoriality for Gabber's log cotangent complex, we may also define it in the following way.

\begin{proposition}
    Let
    \[
    \bfL^{\Gab} : \Arr(\aPreLog)\to\MOD
    \]
    be the animation of the functor
    \begin{align*}
        \mathbf{\Omega}^{\log} : \PolyPreLog&\to\mathsf{MOD}\\
        (F(X,S)\to F_{F(X,S)}(Y,T))&\mapsto \left(F_{F(X,S)}(Y,T),\Omega^1_{F_{F(X,S)}(Y,T)/F(X,S)}\right),
    \end{align*}
    or equivalently, the left Kan extension of the composition
    \[
    \PolyPreLog\xrightarrow{\mathbf{\Omega}^{\log}}\mathsf{MOD}\to\MOD
    \]
    along the inclusion $\PolyPreLog\to\Arr(\aPreLog),$ where $F(X,S)$ denotes the free prelog ring on the sets $X$ and $S$ and $F_{F(X,S)}(Y,T)$ denotes the free $F(X,S)$-prelog algebra on the sets $Y$ and $T,$ defined in \ref{free prelog rings/algebras definition}. Then $\bfL^{\Gab}$ is given on objects by
    \[
    ((k,P)\to (A,M))\mapsto \left(A,\LGab_{(A,M)/(k,P)}\right).
    \]
\end{proposition}

One proves this result in the same way the analogous result is proved for the cotangent complex of animated rings. In fact, there are a number of properties that follow formally from the above definition and the fact that the object $\LGab_{(A,M)/(k,P)}$ corepresents derived log derivations.

\begin{proposition}\label{Gabber cotangent cofiber sequence}
    Let $(k,P)\to (A,M)\to (B,N)$ be a composition of animated prelog rings. Then the diagram
    \[\begin{tikzcd}
        (A,M)\ar[r] &(B,N)\ar[r] &(B,N)\\
        (k,P)\ar[r]\ar[u] &(k,P)\ar[r]\ar[u] &(A,M)\ar[u]
    \end{tikzcd}\]
    induces morphisms
    \[
    (A,\LGab_{(A,M)/(k,P)})\to (B,\LGab_{(B,N)/(k,P)})\to (B,\LGab_{(B,N)/(A,M)})
    \]
    in $\MOD$ via $\bfL^{\Gab},$ which determine a canonical cofiber sequence
    \[
    \LGab_{(A,M)/(k,P)}\otimes_A B\to \LGab_{(B,N)/(k,P)}\to \LGab_{(B,N)/(A,M)}
    \]
    in $\Mod_B.$
\end{proposition}

\begin{proposition}\label{cocartesian square LGab induces equivalence}
    Let
    \[\begin{tikzcd}
        (A,M)\ar[r] &(A',M')\\
        (k,P)\ar[r]\ar[u] &(k',P')\ar[u]
    \end{tikzcd}\]
    be a coCartesian diagram of animated prelog rings. Then the morphism
    \[
    (A,\LGab_{(A,M)/(k,P)})\to (A',\LGab_{(A',M')/(k',P')})
    \]
    obtained by applying $\bfL^{\Gab}$ induces an equivalence
    \[
    \LGab_{(A,M)/(k,P)}\otimes_A A'\to\LGab_{(A',M')/(k',P')}.
    \]
\end{proposition}

One convenient property of Gabber's log cotangent complex is that it is insensitive to logification. More precisely, the following holds.

\begin{lemma}[\protect{\cite[Lemma~3.11]{BLPO_HKR}}]\label{logification invariance of Gabber log cotangent complex}
    Let $(k,P)\to (A,M)$ be a morphism of animated prelog rings. Logification induces equivalences
    \[
    \LGab_{(A,M)/(k,P)}\xrightarrow{\sim}\LGab_{(A,M)^a/(k,P)}\xrightarrow{\sim}\LGab_{(A,M)^a/(k,P)^a}
    \]
    in $\Mod_A.$
\end{lemma}

From \ref{cocartesian square LGab induces equivalence} and \ref{logification invariance of Gabber log cotangent complex}, we can also deduce that Gabber's log cotangent complex reduces to the ordinary cotangent complex for strict morphisms of animated prelog rings.

\begin{cor}
    Let $f : (k,P)\to (A,M)$ be a strict morphism of animated prelog rings. Then $\LGab_{(A,M)/(k,P)}\simeq\LL_{A/k}.$
\end{cor}
\begin{proof}
    A proof in the case when $(k,P)\to (A,M)$ is a map of discrete prelog rings is given in \cite[Corollary~3.12]{BLPO_HKR}. In the general case, consider the diagram
    \[\begin{tikzcd}
        (A,\{1\})\ar[r] &(A,\{1\})^a\ar[r] &(A,M)^a\\
        (k,\{1\})\ar[r]\ar[u] &(k,\{1\})^a\ar[r]\ar[u] &(k,P)^a\ar[u]
    \end{tikzcd}\]
    where the left pair of horizontal maps are the logifications, and the right square is coCartesian. Then we obtain a composition
    \[
    \LGab_{(A,\{1\})/(k,\{1\})}\to\LGab_{(A,\{1\})^a/(k,\{1\})^a}\to\LGab_{(A,M)^a/(k,P)^a}
    \]
    in which both maps are equivalences.

    Then $\LGab_{(A,\{1\})/(k,\{1\})}$ is easily seen to be equivalent to $\LL_{A/k},$ as given a system of maps $P_{\lambda}\to Q_{\lambda}\in\PolyPoly$ such that $\colim_{\lambda}(P_{\lambda}\to Q_{\lambda})\simeq (k\to A),$ the corresponding system $(P_{\lambda},\{1\})\to (Q_{\lambda},\{1\})$ has colimit $(k,\{1\})\to (A,\{1\}),$ and we have canonical isomorphisms $\Omega^1_{Q_{\lambda}/P_{\lambda}}\cong \Omega^1_{(Q_{\lambda},\{1\})/(P_{\lambda},\{1\})}.$ Thus, the colimit computing $\LGab_{(A,\{1\})/(k,\{1\})}$ canonically computes $\LL_{A/k}$ as well.

    Finally, there is also a canonical equivalence $\LGab_{(A,M)/(k,P)}\to\LGab_{(A,M)^a/(k,P)^a}$ induced by logification, and composing the constructed equivalence $\LL_{A/k}\to \LGab_{(A,M)^a/(k,P)^a}$ with the inverse of the canonical equivalence $\LGab_{(A,M)/(k,P)}\to\LGab_{(A,M)^a/(k,P)^a}$ gives the desired equivalence
    \[
    \LL_{A/k}\to \LGab_{(A,M)^a/(k,P)^a}\to \LGab_{(A,M)/(k,P)}.
    \]
\end{proof}

As mentioned above, no theory of logarithmic cotangent complex is able to satisfy all of the desirable properties listed in \ref{no theory of log cotangent exists satisfying}. The property that Gabber's cotangent complex compromises on is the second: it is not necessarily the case that for a log smooth morphism $f : X\to S,$ Gabber's log cotangent complex is represented by the corresponding module of logarithmic differential forms, and in particular, it need not be the case that if $f : X\to S$ is log \'etale, that $\LGab_f\simeq 0.$ This prompts the following definition.

\begin{mydef}
    Let $(f,f^{\flat}) : (k,P)\to (A,M)$ be a map of animated prelog rings. We say that
    \begin{enumerate}
        \item $(f,f^{\flat})$ is \emph{derived formally log \'etale} if $\LGab_{(A,M)/(k,P)}\simeq 0.$
        \item $(f,f^{\flat})$ is \emph{derived log \'etale} if it is derived formally log \'etale and the underlying map $k\to A$ of animated rings is locally of finite presentation.
        \item $(f,f^{\flat})$ is \emph{derived formally log smooth} if $\LGab_{(A,M)/(k,P)}$ is a projective animated $A$-module.
        \item $(f,f^{\flat})$ is \emph{derived log smooth} if it is derived formally log smooth and the underlying map $k\to A$ of animated rings is locally of finite presentation.
    \end{enumerate}
\end{mydef}

The notions of derived log \'etale-ness and derived log smooth-ness are stronger than their classical counterparts. In particular, we have the following.

\begin{proposition}[\protect{\cite[Theorems~5.6,6.4]{Sagave_2014}}]
    Let $(f,f^{\flat}) : (A,M)\to (B,N)$ be a map of animated prelog rings. Then
    \begin{enumerate}
        \item If $(f,f^{\flat})$ is derived log \'etale and $\pi_0 f^{\flat}$ is of finite presentation, then the induced morphism $(\pi_0 A,\pi_0 M)\to (\pi_0 B,\pi_0 N)$ is a log-\'etale morphism of discrete prelog rings.
        \item If $(f,f^{\flat})$ is derived log smooth, then the induced morphism $(\pi_0 A,\pi_0 M)\to (\pi_0 B,\pi_0 N)$ is a log-\'etale morphism of discrete prelog rings.
    \end{enumerate}
\end{proposition}

Although a log smooth or log \'etale map of discrete prelog rings need not be derived log smooth or derived log \'etale, we do have the following partial converse.

\begin{proposition}[\protect{\cite[Proposition~4.6]{BLPO_HKR}}]\label{integral + log smooth/et implies derived log smooth/et}
    Let $(f,f^{\flat}) : (A_0,M_0)\to (B_0,N_0)$ be an integral map of discrete fine prelog rings.
    \begin{enumerate}
        \item If $(f,f^{\flat})$ is log smooth, then it is derived log smooth.
        \item If $(f,f^{\flat})$ is log \'etale, then it is derived log \'etale.
    \end{enumerate}
\end{proposition}

We end this section by discussing a descent result for $\LGab.$

Consider the functor
\[
R\Gamma(-,\bfL^{\Gab}) : \Arr(\LogAff)^{\op}\to\QC\to\cMod_{\ZZ}
\]
given by composing the Gabber cotangent complex functor with the global sections functor.

\begin{proposition}[\protect{\cite[Proposition~7.20]{BLPO_HKR}}]
    Consider the presheaf given by
    \begin{align*}
        R\Gamma(-,\bfL^{\Gab}) : \Arr(\LogAff)^{\op}&\to\cMod_{\ZZ}\\
        (X\to S)&\mapsto R\Gamma(X,\LGab_{X/S}).
    \end{align*}
    Then $R\Gamma(-,\bfL^{\Gab})$ satisfies integral log \'etale descent.
\end{proposition}

\begin{remark}
    Although \cite[Proposition 7.20]{BLPO_HKR} is stated for the strict Nisnevich global sections, the proof works just as well for the usual global sections.
\end{remark}

\subsection{Comparison of Logarithmic Cotangent Complexes}

In this section, we review the comparison morphism $\LGab_{X/S}\to\LOls_{X/S}$ constructed in \cite{Olssoncotangent}. 

Consider a morphism $f : X\to S$ of fine log schemes, and let $f_{\cM_X} : \underline{X}\to\cLog_S$ be the associated morphism to the log stack. Choose a smooth surjection $u : U\to\cLog_S$ and let $U_\bullet\to\cLog_S$ be the $0$-coskeleton of $u.$ We can then form the pullback
\[\begin{tikzcd}
    V_\bullet\ar[r]\ar[d] &\underline{X}\ar[d,"{f_{\cM_X}}"]\\
    U_\bullet\ar[r] &\cLog_S.
\end{tikzcd}\]
From the definition of $\cLog_S,$ we obtain a simplicial log scheme $(U_\bullet,\cM_{U_\bullet})$ with underlying simplicial scheme $U_\bullet.$ If we then endow $V_\bullet$ with the simplicial log structure $\cM_{V_\bullet}$ given by the pullback of the log structure $\cM_{X}$ on $X,$ then the map $(V_\bullet,\cM_{V_\bullet})\to (U_\bullet,\cM_{U_\bullet})$ is strict. To see this, notice that the log structure on $V_\bullet$ is determined by the composition $V_\bullet\to\underline{X}\to\cLog_X,$ where the map $\underline{X}\to\cLog_X$ is the canonical map classifying the log structure on $X.$ Because $V_\bullet$ is defined as the pullback $U_\bullet\times_{\cLog_S} X,$ the map $V_\bullet\to\cLog_S$ also factors as
\[
V_\bullet\to U_\bullet\to\cLog_{(U_\bullet,\cM_{U_\bullet})}\to\cLog_S
\]
which implies that the map $(V_\bullet,\cM_{V_\bullet})\to (U_\bullet,\cM_{U_\bullet})$ is strict.

We thus have a commutative diagram
\[\begin{tikzcd}
    (V_\bullet,\cM_{V_\bullet})\ar[d,"\pi"]\ar[r] &X\ar[d,"f"]\\
    (U_\bullet,\cM_{U_\bullet})\ar[r] &S
\end{tikzcd}\]
and functoriality of Gabber's cotangent complex gives a composition
\[
\pi^*\LGab_{X/S}\to\LGab_{(V_\bullet,\cM_{V_\bullet})/(U_\bullet,\cM_{U_\bullet})}\simeq\LL_{V_\bullet/U_\bullet}.
\]
However, $\LL_{V_\bullet/U_\bullet}\in\sfD_{\qc}(\cO_{V_\bullet})$ corresponds to $\LOls_{X/S}\in\sfD_{\qc}(\cO_{X})$ via the equivalence $\sfD_{\qc}(\OX)\to \sfD_{\qc}(\cO_{V_\bullet})$ given in \cite{OlssonSheaves}. This gives the desired comparison morphism.

It turns out that the two cotangent complexes agree in general in low degrees.

\begin{theorem}[\protect{\cite[Theorem~8.32]{Olssoncotangent}}]
    The cone of $\LGab_{X/S}\to\LOls_{X/S}$ is cohomologically concentrated in degrees $\leq -3.$
\end{theorem}

Moreover, if $f : X\to S$ is sufficiently nice, the two cotangent complexes actually agree.
\begin{proposition}
    Let $f : X\to S$ be a map of fine log schemes. Suppose that either:
    \begin{enumerate}
        \item $S$ is log flat over $(\Spec\ZZ,\cO_{\Spec\ZZ}^\times),$ or
        \item $f$ is integral.
    \end{enumerate}
    Then the map $\LGab_{X/S}\to\LOls_{X/S}$ is an equivalence.
\end{proposition}

\begin{remark}
    We may also construct the comparison morphism $\LGab\to\LOls$ as follows: \cite[Proposition~3.17]{BLPO_HKR} implies that we have an equivalence $\LGab_{X/S}\simeq\LL_{\underline{X}/\underline{X\times^{\rep}_S X}}[-1].$ Since we similarly have an equivalence $\LOls_{X/S}\simeq\LL_{\underline{X}/\underline{X}\times_{\cLog_S}\underline{X}}[-1],$ it then suffices to functorially construct a map
    \[
    \underline{X}\times^R_{\cLog_S}\underline{X}\to \underline{X\times^{\rep}_S X}.
    \]
    This is essentially the approach we take in our comparison of logarithmic Hochschild homologies: for a log flat map of prelog rings $(f,f^{\flat}) : (k,P)\to (A,M)$ with corresponding morphism of log schemes $X\to S,$ we will functorially construct a map
    \[
    \underline{X}\times_{\cS_{(f,f^{\flat})}}\underline{X}\to\underline{X\times^{\rep}_S X}.
    \]
    Then this map induces morphisms
    \[
    \underline{X}\times^R_{\underline{X}\times_{\cLog_S}\underline{X}}\underline{X}\leftarrow \underline{X}\times^R_{\underline{X}\times_{\cS_{(f,f^{\flat})}}\underline{X}}\underline{X}\to\underline{X}\times^R_{\underline{X\times^{\rep}_S X}}\underline{X}.
    \]
    The left morphism is an equivalence on Hochschild homology because
    \[
    \underline{X}\times_{\cS_{(f,f^{\flat})}}\underline{X}\to\underline{X}\times_{\cLog_S}\underline{X}
    \]
    is \'etale, the the definitions of both versions of logarithmic Hochschild homology then give our map.
\end{remark}
\chapter{Hochschild Homology}

We now turn to Hochschild homology. We will review the classical theory briefly, then turn to the derived setting, which generalizes the classical theory.\footnote{At least, when the morphism in question is flat.} We recommend \cite{WeibelHomological} for a brief overview of the classical situation, \cite{Loday_2010a} for a very thorough classical reference, which also discusses Hochschild homology in more general settings, and \cite{MorrowHHnotes} for a more modern account.

\section{Classical Hochschild Homology}

We begin with a review of classical Hochschild homology.

\subsection{Basic Construction and Results}

Let $k\to A$ be a map of classical commutative rings, and let $M$ be an $A$-bimodule. In the following, let $\otimes$ denote the (underived) tensor product over $k.$

\begin{mydef}[Cyclic Bar Complex]
    $\Bar^\cyc(A/k,M)_n := M\otimes A^{\otimes n}$
    \begin{align*}
        \partial_i (m\otimes a_1\otimes\dots\otimes a_n) &= \begin{cases}
        ma_1\otimes a_2\otimes\dots\otimes a_n, &i=0,\\
        m\otimes a_1\otimes\dots\otimes a_i a_{i+1}\otimes\dots\otimes a_n, &0 < i < n,\\
        (-1)^{|a_n|(|m| + |a_1| + \dots + |a_{n-1}|)}a_n m\otimes a_1\otimes\dots\otimes a_{n-1}, &i=n.
    \end{cases}\\
    \sigma_i (m\otimes a_1\otimes\dots\otimes a_n) &= m \otimes\dots\otimes a_i\otimes 1\otimes a_{i+1}\otimes\dots\otimes a_n.
    \end{align*}
\end{mydef}

For $M = A,$ there is a cyclic structure on $\Bar^\cyc(A/k,A)_\bullet = B^\cyc(A/k)_\bullet$ defined by $t(a_0\otimes\cdots\otimes a_n) := (-1)^{|a_n|(|a_0| + \dots + |a_{n-1}|)}a_n\otimes a_0\otimes\dots\otimes a_{n-1}.$

\begin{remark}
    Hochschild homology makes sense more generally than this as well. If $A$ is simply an associative unital (but not necessarily commutative) $k$-algebra, the above definition works just as well, and we can talk about the cyclic bar complex (and thus Hochschild homology) in this situation as well. One can also define Hochschild homology when $A$ is non-unital, but the definition is slightly different. We will not make use of these generalizations, and will mainly develop the theory of Hochschild homology for a map $k\to A$ of commutative (or animated) rings, and in the situation when $M = A.$
\end{remark}

\begin{mydef}
    Let $k\to A$ be an arbitrary map of rings, and let $M$ be an $A$-bimodule. Let $C^{\Hoch}_*(A/k,M)$ denote the unnormalized chain complex associated to $B^{\cyc}(A/k,M)_{\bullet};$ i.e., the chain complex which in degree $n$ is given by $B^{\cyc}(A/k,M)_n,$ and whose differentials are given by alternating sums of the face maps. More precisely, the differential
    \[
    b : C^{\Hoch}(A/k,M)_n\to C^{\Hoch}(A/k,M)_{n-1}
    \]
    is $b = \sum_{i = 0}^n (-1)^i\partial_i.$\footnote{Note that we use homological indexing here.} We call this chain complex the \emph{Hochschild complex}. If $M = A,$ we write $C^{\Hoch}(A/k)_* := C^{\Hoch}(A/k,A)_*.$ We write $\HH^{\cl}(A/k,M)$ (respectively, $\HH^{\cl}(A/k)$) to denote the image of $C^{\Hoch}(A/k,M)_*$ (respectively $C^{\Hoch}(A/k)_*$) in $\sfD(A).$

    We also define the \emph{(classical) Hochschild homology groups} to be given by
    \[
    \HH^{\cl}_n(A/k,M) := \H_n(C^{\Hoch}(A/k,M)_*) = \H_n(\HH^{\cl}(A/k,M)).
    \]
\end{mydef}

Let $A$ be an associative, unital $k$-algebra. Recall that the bar complex $C^{\Bar}(A/k)_*$ is given by the unnormalized chain complex associated to the simplicial $A$-module $\Bar(A/k)_\bullet$ with
\begin{align*}
    \Bar(A/k)_n &= A^{\otimes n+2}\\
    \partial_i(a_0\otimes\dots\otimes a_{n+1}) &= a_0\otimes\dots\otimes a_i a_{i+1}\otimes\dots\otimes a_{n+1}\\
    \sigma_i(a_0\otimes\dots\otimes a_{n+1}) &= a_0\otimes\dots\otimes a_i\otimes 1\otimes a_{i+1}\otimes\dots\otimes a_{n+1}.
\end{align*}
For the chain complex $C^{\Bar}(A/k)_*,$ the differential $b' : C^{\Bar}(A/k)_n\to C^{\Bar}(A/k)_{n-1}$ is given by
\[
b' = \sum_{i = 0}^{n-1}(-1)^i d_i.
\]

We can define an augmentation $\Bar(A/k)_\bullet\to A[0]$ given by sending
\[
\Bar(A/k)_0 = A\otimes A\to A
\]
by the multiplication map. Let $A^e := A\otimes A^{\op}$ be the enveloping algebra of $A.$ If we give $\Bar(A/k)_\bullet$ the structure of a simplicial $A^e$-module by letting $a\otimes b\in A^e$ act on $\Bar(A/k)_n = A^{\otimes n+1}$ via $(a\otimes b)\star c := a c b,$ then $\Bar(A/k)_\bullet\to A[0]$ is a resolution of $A$ as an $A^e$-module. Indeed, if we let
\begin{align*}
    s : \Bar(A/k)_n&\to\Bar(A/k)_{n+1}\\
    a_0\otimes\dots\otimes a_{n+1}&\mapsto 1\otimes a_0\otimes\dots\otimes a_{n+1}\\
\end{align*}
(sometimes called the \emph{extra degeneracy}), then
\begin{align*}
    \partial_i s(a_0\otimes\dots\otimes a_{n+1})&= \partial_i(1\otimes a_0\otimes\dots\otimes a_{n+1})\\
    &= 1\otimes a_0\otimes\dots\otimes a_{i-1} a_i\otimes\dots a_{n+1})\\
    &= s(a_0\otimes\dots\otimes a_{i-1} a_i\otimes\dots a_{n+1})\\
    &= s\partial_{i-1}(a_0\otimes\dots\otimes a_{n+1})
\end{align*}
and
\begin{align*}
    \partial_0 s(a_0\otimes\dots\otimes a_{n+1}) &= \partial_0(1\otimes a_0\otimes\dots\otimes a_{n+1})\\
    &= a_0\otimes\dots\otimes a_{n+1}.
\end{align*}
From these identities, it follows that
\begin{align*}
    b's + sb' &= \left(\sum_{i = 0}^{n}(-1)^i\partial_i\right) s + s\left(\sum_{i = 0}^{n-1}(-1)^i\partial_i\right)\\
    &= \left(\sum_{i = 0}^{n}(-1)^i\partial_i\right) s + \left(\sum_{i = 0}^{n-1}(-1)^i\partial_{i+1}\right)s\\
    &= \left(\sum_{i = 0}^{n}(-1)^i\partial_i\right) s - \left(\sum_{i = 1}^{n}(-1)^i\partial_{i}\right)s\\
    &= \partial_0 s\\
    &= \id
\end{align*}
so that $s$ is a contracting homotopy for $\Bar(A/k)_\bullet,$ and thus that the complex $C^{\Bar}(A/k)_*$ is acyclic.

The above gives us another way to represent and calculate Hochschild homology. In particular, let $A$ be projective over $k$ (for example, suppose $k\to A$ is flat). Then the bar complex provides a projective resolution of $A$ by $A^e$-modules, so that given an $A$-bimodule $M,$ we have
\[
M[0]\otimes^L_{A^e} A[0]\underset{\mathrm{q.iso}}{\simeq} M\otimes_{A^e}C^{\Bar}(A/k)_*,
\]
and one can check from the definitions that $M\otimes_{A^e} C^{\Bar}(A/k)_*\cong C^{\Hoch}(A/k,M)_*,$ which means that
\[
M\otimes^L_{A^e} A\simeq\HH^{\cl}(A/k,M)
\]
in the derived category $\sfD(A)$ (or $\sfD(k)$). Thus, we find that
\[
\HH^{\cl}_n(A/k,M)\cong\Tor^{A^e}_{n}(M,A).
\]
In particular, we obtain in the commutative case an equivalence
\[
\HH^{\cl}(A/k)\simeq A\otimes^L_{A\otimes A} A
\]
in $\sfD(A)$ (or $\sfD(k)$).

\begin{remark}
    We can also define the Hochschild homology of a sufficiently nice category (for example, a $k$-linear dg-category). When appropriately defined, there is a canonical equivalence between our $\HH^{\cl}(A/k)$ and the Hochschild homology of $\sfD(A)$ (or a suitable enhancement of this category). We do not discuss this here, but refer the interested reader to \cite{MCCARTHY1994251} and \cite{Keller_2021} for more details. 
\end{remark}

\subsection{The Classical HKR Theorem for Affines}

In low degrees, the Hochschild homology groups have familiar interpretations.

\begin{enumerate}
    \item In general, $\HH^{\cl}_{0}(A/k,M)\simeq M/\{am - ma\mid a\in A, m\in M\},$ which is readily seen from the definition of the differential on $\HH^{\cl}(A/k,M).$ In particular, if $A$ is commutative and $M = A,$ then $\HH^{\cl}_0(A/k)\simeq A.$
    \item If $k\to A$ is an arbitrary map of commutative rings, then the map
    \begin{align*}
        \alpha : A\otimes A&\to\Omega^1_{A/k}\\
        a\otimes b&\mapsto a\,db
    \end{align*}
    induces an isomorphism $\HH_1^{\cl}(A/k)\simeq\Omega^1_{A/k}.$ Indeed, if we view $\Omega^1_{A/k}$ as the $A$-module generated by $da,$ where $a\in A,$ modulo the submodule generated by 
    \begin{enumerate}
        \item $d(a_1a_2) - a_1\,da_2 - a_2\,da_1$ for $a_1, a_2\in A,$
        \item $d(a_1 + a_2) - da_1 - da_2$ for $a_1, a_2\in A,$ and
        \item $dc$ for $c\in k,$
    \end{enumerate}
    then the map $\alpha$ is clearly surjective, and $\ker\alpha$ contains the image of of
    \[
    b : C^{\Hoch}(A/k)_2\to C^{\Hoch}(A/k)_1,
    \]
    which is generated by elements of the form $ab\otimes c - a\otimes bc + ca\otimes b,$ where $a,b,c\in A,$ since
    \[
    \alpha(ab\otimes c - a\otimes bc + ca\otimes b) = ab\,dc - a\,d(bc) + ca\,db,
    \]
    which is nothing more than the Leibniz relation $d(bc) - b\,dc - c\,d b$ multiplied by $a.$ Thus, we obtain a well defined map $\HH_1^{\cl}(A/k)\to\Omega^1_{A/k}.$

    The inverse $\epsilon_1$ of this map is given by sending $a\,db\mapsto [a\otimes b].$ This is well-defined, since
    \[
    \epsilon_1(d(ab)) = [1\otimes ab] = [a\otimes b + b\otimes a] = \epsilon_1(a\,db + b\,da).
    \]
    We also check that
    \[
    \epsilon_1(d1) = [1\otimes 1] = [1\otimes 1 + 1\otimes 1] = [2(1\otimes 1)],
    \]
    which implies that $[1\otimes 1] = 0$ in $\HH_1^{\cl}(A/k)$ (and thus that $\epsilon_1(dc) = [1\otimes c] = 0$ for any $c\in k$). It is clear from the construction that $\epsilon_1$ and $\alpha$ are inverses, so that $\HH_1^{\cl}(A/k)\simeq\Omega^1_{A/k}$ as claimed.
\end{enumerate}

An important observation we will need later is that $\epsilon_1$ is actually a natural transformation. More precisely, let $\PolyPoly$ denote the category whose objects are finitely generated polynomial extensions of finitely generated polynomial $\ZZ$-algebras (i.e., maps of commutative rings of the form $\ZZ[x_1,\dots, x_n]\to\ZZ[x_1,\dots, x_n, y_1, \dots, y_{n'}]$) and whose morphisms are commutative squares
\[\begin{tikzcd}
    \ZZ[x_1,\dots, x_n, y_1,\dots, y_{n'}]\ar[r] &\ZZ[x'_1,\dots, x'_m, y'_1,\dots, y'_{m'}]\\
    \ZZ[x_1,\dots, x_n]\ar[u]\ar[r] &\ZZ[x'_1,\dots, x'_n]\ar[u]
\end{tikzcd}\]
where the vertical maps are the obvious ones (but there is no requirement on the horizontal maps) and let $\mathsf{PolyFree}$ denote the category of pairs $(\ZZ[x_1,\dots, x_n], \ZZ[x_1,\dots, x_n]^{\oplus r})$ of finitely generated polynomial $\ZZ$-algebras together with a finitely generated free module over that algebra, and whose morphisms from $(P,P^{\oplus r})\to (Q,Q^{\oplus r'})$ are pairs $(f, \phi)$ where $f : P\to Q$ is a map of commutative rings, and $\phi : P^{\oplus r}\otimes_P Q\to Q^{\oplus r'}$ is a map of $Q$-modules. Then if
\[\begin{tikzcd}
    Q\ar[r] & Q'\\
    P\ar[u]\ar[r] &P'\ar[u]
\end{tikzcd}\]
is a morphism between $P\to Q$ and $P'\to Q'$ in $\PolyPoly,$ one checks from the definition of $\epsilon_1$ that
\[\begin{tikzcd}
    (Q,\Omega^1_{Q/P})\ar[r,"\epsilon_1"]\ar[d] & (Q,\HH^{\cl}_1(Q/P))\ar[d]\\
    (Q',\Omega^1_{Q'/P'})\ar[r,"\epsilon_1"] &(Q',\HH^{\cl}_1(Q'/P'))
\end{tikzcd}\]
commutes, where the horizontal maps are given by $\epsilon_1$ as defined previously and the vertical maps are induced by functoriality of the constructions of K\"ahler differentials and Hochschild homology.

\begin{remark}
    One can also deduce that $\epsilon_1$ is functorial by noting that $\HH^{\cl}_*(A/k)$ has the structure of a graded commutative differential algebra, and so the universal property of $\Omega^*_{A/k}$ and functoriality of $\Omega^*$ produces a map of functors $\Omega^*\to\HH^{\cl}_*.$
\end{remark}

One of the most important results about Hochschild homology is the ``Hochschild-Kostant-Rosenberg theorem,'' which generalizes the two computations above for $\HH^{\cl}_0$ and $\HH^{\cl}_1$ when the morphism $k\to A$ is smooth.

\begin{theorem}[Hochschild-Kostant-Rosenberg Theorem, \protect{\cite{HKRoriginal}}]\label{Classical HKR}
    Let $k\to A$ be a smooth morphism of commutative rings. Then the isomorphism
    \[
    \epsilon_1 : \Omega^1_{A/k}\to\HH^{\cl}_1(A/k)
    \]
    induces an isomorphism of commutative differential graded algebras
    \[
    \epsilon : \Omega^*_{A/k}\to\HH^{\cl}_*(A/k).
    \]
\end{theorem}

\begin{remark}
    The maps $\epsilon_n : \Omega^n_{A/k}\to\HH^{\cl}_n(A/k)$ are called the \emph{antisymmetrization maps} and are explicitly given by
    \[
    \epsilon_n(a\mathrm{d}b_1\wedge\dots\wedge\mathrm{d}b_n) = \sum_{\sigma\in S_n} (-1)^{\operatorname{sign}\sigma} [a\otimes b_{\sigma(1)}\otimes\dots\otimes b_{\sigma(n)}].
    \]
\end{remark}

In order to prove \ref{Classical HKR}, we require a lemma.

\begin{lemma}\label{etale localization}
    Suppose that $A$ is a flat $k$-algebra, and that $A\to A'$ is \'etale. Then the canonical map
    \[
    \HH^{\cl}_n(A/k)\otimes_A A'\to\HH^{\cl}_n(A'/k)
    \]
    is an equivalence for all $n\geq 0.$
\end{lemma}

We learned the following slick proof from the excellent set of notes \cite{MorrowHHnotes}.

\begin{proof}
    Since $A\to A'$ is flat, it suffices to show that $\HH^{\cl}(A/k)\otimes_A A'\simeq\HH^{\cl}(A/k)\otimes^L_A A'\to \HH^{\cl}(A'/k)$ is an equivalence at the level of the derived category.
    
    We now appeal to the Hochschild-Kostant-Rosenberg filtration, \ref{HKR for CRing}. Because the HKR filtration is complete, it suffices to show that the morphism induces an equivalence on all associated graded pieces. On the graded pieces, we have (shifts of) the canonical maps
    \[
    \LL^i_{A/k}\otimes^L_A A'\to \LL^i_{A'/k}
    \]
    which are the $i$th exterior powers of the canonical map
    \[
    \LL_{A/k}\otimes^L_A A'\to\LL_{A'/k}
    \]
    on the degree $1$ graded piece. However, this map is part of the cofiber sequence
    \[
    \LL_{A/k}\otimes^L_A A'\to \LL_{A'/k}\to\LL_{A'/A}
    \]
    of \ref{cotangent complex cofiber sequence}, and as $A\to A'$ is \'etale, it follows that $\LL_{A'/A}\simeq 0,$ which implies that the map $\LL_{A/k}\otimes^L_A A'\to\LL_{A'/k}$ is an equivalence, as desired.
\end{proof}

The following argument is illustrative of how many properties of Hochschild homology can be proven -- use the completeness of the HKR filtration to reduce to a statement about exterior powers of the cotangent complex. At this stage, the claim can often be reduced to the statement on just the first associated graded, since in many cases the higher degree maps are induced by taking exterior powers. Of course, one can prove the result without using the HKR filtration. For example \cite{WG91} uses a clever manipulation of the Tor description of $\HH^{\cl}_n(A/k)$ to prove the lemma.

Interestingly, although \ref{etale localization} is used to prove \ref{Classical HKR} which is in turn used to prove \ref{HKR for CRing} and the proof above relies on \ref{HKR for CRing}, we can avoid circularity, as \cite{MorrowHHnotes} explains. Specifically, the proof of \ref{HKR for CRing} only uses \ref{Classical HKR} when the morphism in question is of the form $k\to k[x_1,\dots, x_n],$ but \ref{etale localization} is only used in the proof of \ref{Classical HKR} to extend from the polynomial case to the case of an arbitrary smooth algebra.

We are now prepared to prove \ref{Classical HKR}.

\begin{proof}
    Since $\Omega^*_{A/k} = \bigwedge^*_A\Omega^1_{A/k},$ we must show that the canonical map
    \[
    \bigwedge^*_A\HH_1^{\cl}\to \HH_*^{\cl}(A/k)
    \]
    is an isomorphism, or equivalently, that
    \[
    \bigwedge^*_A\Tor^{A^e}_1(A,A)\to\Tor^{A^e}_*(A,A)
    \]
    is an isomorphism.

    Suppose first that $A$ is a polynomial $k$-algebra. Then the kernel of $A^e\to A$ is generated by a regular sequence, and it is well known that in this situation\footnote{See for example, \cite[Example 8.7.13]{WeibelHomological}.}, the canonical map 
    \[
    \bigwedge^*_A\Tor^{A^e}_1(A,A)\to\Tor^{A^e}_*(A,A)
    \]
    is an isomorphism.

    Next, suppose that $A$ is an \'etale extension of a polynomial $k$-algebra $P$:
    \[
    k\to P\xrightarrow{\et}A.
    \]
    
    By \ref{etale localization}, we then have 
    \[
    \Omega^n_{A/k}\simeq\Omega^n_{P/k}\otimes_{P} A\simeq\HH^{\cl}_n(P/k)\otimes_P A\simeq\HH^{\cl}_{n}(A/k),
    \]
    and one readily checks that the canonical isomorphisms involved preserve the antisymmetrization maps.

    Finally, in the general case, we may reduce to the case of $A$ being an \'etale extension of a polynomial $k$-algebra $P,$ as any smooth morphism $X\to S$ Zariski locally factors through affine $n$-space over $S.$ Here we use the fact that if $f\in A,$ then $\HH^{\cl}(A/k)\otimes_A A[1/f]\simeq\HH^{\cl}(A[1/f]/k),$ again using \ref{etale localization}.
\end{proof}

When $P\to Q$ is an object of $\PolyPoly$ (and in fact for any smooth map $k\to A$ of commutative rings), one can also check directly from the definitions that as before, the diagram
\[\begin{tikzcd}
    (Q,\Omega^n_{Q/P})\ar[r,"\epsilon_n"]\ar[d] & (Q,\HH^{\cl}_n(Q/P))\ar[d]\\
    (Q',\Omega^n_{Q'/P'})\ar[r,"\epsilon_n"] &(Q',\HH^{\cl}_n(Q'/P'))
\end{tikzcd}\]
commutes.

\subsection{Classical Hochschild Homology of Schemes and Stacks}

We can also define the Hochschild homology of a scheme $X$ over a base ring $k.$ The key result which allows us to do so is \ref{etale localization}. The original reference for this is \cite{WG91}. Below, we write $C^{\Hoch}(A/k)^*$ for the cohomologically indexed Hochschild complex; i.e., $C^{\Hoch}(A/k)^n := C^{\Hoch}(A/k)_{-n}.$

\begin{theorem}[\protect{\cite[Corollary~0.4]{WG91}}]
    Let $X$ be a $k$-scheme, and let $\sHH^{\cl}_n$ be the Zariski sheafification of the presheaf $U\mapsto \HH^{\cl}_n(\Gamma(U,\OX)/k)$ on the big Zariski site of $X.$ Then
    \begin{enumerate}
        \item $\sHH_n^{\cl}$ is a quasicoherent sheaf on $X;$
        \item $\sHH_n^{\cl}$ is also a sheaf for the \'etale topology on $X;$
        \item $\H^i_{\et}(X,\sHH_n^{\cl})\cong\H^i_{\Zar}(X;\sHH_n^{\cl})$ for all $i;$ and
        \item If $X = \Spec A$ is affine, then $\H^0(X,\sHH_n^{\cl})\cong\HH_n^{\cl}(A/k)$ and $\H^i(X,\sHH_n^{\cl}) = 0$ for $i\neq 0.$
    \end{enumerate}
\end{theorem}

We may also define a sheafy Hochschild complex on a scheme $X$ in a similar way, by sheafification of the cyclic bar construction on affines. More precisely, suppose that $\cC^{\Hoch,*}_{X/k}$ is the sheafification of the presheaf
\[
U\mapsto C^{\Hoch}(\Gamma(U,\OX)/k)^* := \HH^{\cl}(\Gamma(U,\OX)/k)_{-*}
\]
which sends an open $U\subseteq X$ to the Hochschild chain complex of $\Gamma(U,\OX)$ relative to $k.$ Then we may define
\[
\HH^{\cl}_n(X/k) := \bfH^{-n}_{\Zar}(X,\cC^{\Hoch,*}_{X/k}).
\]
While the components $\cC^{\Hoch,n}_{X/k}$ are not in general quasi-coherent, the complex itself has quasi-coherent cohomology, and we have the following.

\begin{theorem}[\protect{\cite[Theorems~4.1,4.2]{WG91}}]
    If $X = \Spec A,$ then
    \[
    \HH^{\cl}_n(X/k)\cong\HH^{\cl}_n(A/k)
    \]
    for all $n.$ For a general $k$-scheme $X,$ there is a fourth quadrant spectral sequence with
    \[
    E^{p,q}_2 := \H^p_{\Zar}(X,\sHH^{\cl}_{-q})\cong\H^p_{\et}(X,\sHH^{\cl}_{-q})
    \]
    which approaches $\HH^{\cl}_{-p-q}(X),$ and converges if the spectral sequence is bounded; i.e., for any $n$ there are only finitely many nonzero $\H^p(X,\sHH^{\cl}_{q})$ with $-p + q = n.$
\end{theorem}

In \cite{Olssonloghh}, Olsson generalizes this construction to give Hochschild homology sheaves and groups of morphisms $X\to\cS,$ where $\cS$ is an algebraic stack rather than just a base affine scheme, and $X$ is permitted to be an algebraic space. Let us recall this construction and the most important results. We now return to homological indexing.

The formation of the Hochschild complex $C^{\Hoch}(A/k)_*$ is functorial in the morphism $k\to A$: for a commutative square
\[\begin{tikzcd}
    A\ar[r] &A'\\
    k\ar[r]\ar[u] &k'\ar[u]
\end{tikzcd}\]
of commutative rings, there is a natural map
\[
C^{\Hoch}(A/k)_*\otimes_A A'\to C^{\Hoch}(A'/k')_*.
\]
of chain complexes of $A'$-modules.

Let $X$ be an algebraic space over a ground ring $k.$ As in the case of schemes, let $\cC^{\Hoch}_{X/k,*}$ denote the sheaf of $\OX$-modules on the \'etale site of $X$ which is the sheafification of the presheaf that maps $(\Spec A\xrightarrow{\et} X)\mapsto C^{\Hoch}(A/k)_*.$ We call $\cC^{\Hoch}_{X/k,*}$ the \emph{sheafy Hochschild complex of $X/k$,} and we let $\sHH^{\cl}_{X/k}$ be the image of $\cC^{\Hoch}_{X/k,*}$ in $\sfD(\OX).$

The functoriality of $C^{\Hoch}(-)_*$ implies that $\cC^{\Hoch}_{-,*}$ is also functorial in the morphism $X\to k$ from an algebraic space to a ground ring. In particular, given a commutative square
\[\begin{tikzcd}
    X'\ar[r,"f"]\ar[d] &X\ar[d]\\
    \Spec k'\ar[r] &\Spec k
\end{tikzcd}\]
where $X$ and $X'$ are algebraic spaces, there is a map of $\cO_{X'}$-modules
\[
f^*\cC^{\Hoch}_{X/k,*}\to\cC^{\Hoch}_{X'/k',*}.
\]

If $X\to\cS$ is a flat map from a quasi-compact algebraic space to an algebraic stack with quasi-compact diagonal, Olsson then defines $\cC^{\Hoch}_{X/\cS,*}$ by descent. Specifically, choose a smooth affine hypercover $S_\bullet\to\cS$ with $S_n = \Spec k_n,$ and set $X_\bullet = S_\bullet\times_{\cS} X.$ Then any map $a : [m]\to [n]\in\Delta$ induces a corresponding diagram
\[\begin{tikzcd}
    X_n\ar[r,"x_a"]\ar[d] &X_m\ar[d]\\
    \Spec k_n\ar[r,"S_a"] &\Spec k_m
\end{tikzcd}\]
and thus a map
\[
x_a^*\cC^{\Hoch}_{X_m/k_m,*}\to\cC^{\Hoch}_{X_n/k_n,*}.
\]
The above gives the data of a sheaf $\cC_{X_\bullet/k_\bullet,*}$ of complexes of $\cO_{X_\bullet}$-modules in the topos of sheaves on $X_{\bullet,\et},$ and we let $\sHH_{X_\bullet/S_\bullet}$ be its image in $\sfD(X_{\bullet,\et},\cO_{X_\bullet})$ (see \cite[Section 12.4]{Champs} for details on $X_{\bullet,\et}$ and the corresponding topos and derived category).

\begin{proposition}[\protect{\cite{Olssonloghh}}]\label{Olsson stacky HH}
    The object $\sHH^{\cl}_{X_\bullet/S_\bullet}$ is isomorphic to the pullback of a unique object $\sHH^{\cl}_{X/\cS}\in\sfD_{\qc}(X,\OX),$ and $\sHH^{\cl}_{X/\cS}$ is independent of the choice of hypercover $S_\bullet\to\cS.$
\end{proposition}

\begin{remark}
    Olsson also provides a construction of $\sHH^{\cl}_{X/\cS}$ which avoids the use of hypercovers entirely, see \cite[Section 6]{Olssonloghh}. However, cohomological descent is more closely connected to our approach via Kan extensions, which is why we describe the above construction.
\end{remark}

\begin{mydef}
    Let $X\to\cS$ be a map from a quasi-compact algebraic space $X$ to an algebraic stack $\cS$ with quasi-compact diagonal. Then the \emph{$n$th Hochschild homology group} of $X/\cS$ is defined to be
    \[
    \HH^{\cl}_n(X/\cS) := \H_{n}(X,\sHH^{\cl}_{X/\cS}),
    \]
    and the \emph{$n$th Hochschild homology sheaf} of $X/\cS$ is defined to be
    \[
    \sHH^{\cl}_{X/\cS,n} := \cH_{n}(\sHH^{\cl}_{X/\cS}).
    \]
\end{mydef}

\begin{remark}
    By construction, these agree with the notions defined in \cite{WG91} when $X$ is a scheme and $\cS$ is an affine scheme.
\end{remark}

As one might expect, there is a close connection between $\sHH_{X/\cS}$ and the diagonal map $\Delta : X\to\cS.$ In particular, we have the following result.

\begin{theorem}[\protect{\cite[Theorem~4.2]{Olssonloghh}}]
    Let $X\to\cS$ be a map from a quasi-compact algebraic space to an algebraic stack with quasi-compact diagonal. Then there is a natural isomorphism
    \[
    L\Delta^*\Delta_*\OX\simeq\OX\otimes^L_{\Delta^{-1}\cO_{X\times_{\cS}X}}\OX\simeq\sHH^{\cl}_{X/\cS}
    \]
    in $\sfD(\OX).$
\end{theorem}

Olsson also proves the following variant of the HKR theorem for $\sHH_{X/\cS}.$ One can construct the antisymmetrization maps $\epsilon_n : \Omega^n_{A/k}\to\HH^{\cl}_n(A/k)$ by first noting that $\HH^{\cl}_*(A/k)$ has the structure of a graded commutative differential graded algebra (see \cite[2.3.3, 4.2.7]{Loday_2010a}), and hence by the universal property of $\Omega^\bullet_{A/k},$ there is a unique morphism of cdgas $\epsilon : \Omega^\bullet_{A/k}\to\HH^{\cl}_*(A/k)$ extending the canonical isomorphism $\epsilon_0.$ Sheafification then provides $\sHH^{\cl}_{X/\cS,*}$ with the structure of a cdga as well, and the isomorphism $\epsilon_0 : \Omega^0_{X/\cS}\to\sHH^{\cl}_{X/\cS,0}$ thus extends to a map
\[
\epsilon : \Omega^\bullet_{X/\cS}\to\sHH^{\cl}_{X/\cS,*}
\]
of sheaves of cdgas. Olsson then proves the following.

\begin{theorem}[\protect{\cite[Proposition~7.3]{Olssonloghh}}]\label{sheafy classical HKR}
    The map $\epsilon_1 : \Omega^1_{X/\cS}\to\sHH^{\cl}_{X/\cS,1}$ is an isomorphism. If $X\to\cS$ is smooth, the maps $\epsilon_n : \Omega^n_{X/\cS}\to\sHH^{\cl}_{X/\cS,n}$ are isomorphisms for all $n.$
\end{theorem}

\begin{remark}
    In the above proposition, the map $\epsilon$ is obtained by viewing the sheaves $\sHH^{\cl}_{X/\cS,n}$ as sheaves on a site $\mathscr{C}_{X/\cS},$ which Olsson defines in \cite[Section 6]{Olssonloghh}. We do not go into the details of the construction of $\scrC_{X/\cS}$ here, but note that there is an equivalence of categories between $\QCoh(X)$ (the classical version) and a certain category of quasi-coherent sheaves on $\scrC_{X/\cS},$ so that we may view the above proposition as a statement about quasi-coherent sheaves on $X$ in the usual sense. However, this construction has the benefit of being obviously functorial in the morphism $X\to\cS.$
\end{remark}

\section{Derived Hochschild Homology}

The definition of Hochschild homology in the previous section is unfortunately not the ``correct'' notion in the situation when $X$ is not flat over the ground ring (or stack). As such, we need to ``derive'' the entire construction to obtain something homotopically meaningful.

We will begin with the definition of Hochschild homology of a morphism of an animated ring over a fixed (animated) ring $k,$ and then extend to the Hochschild homology of a derived stack over a fixed (animated) ring $k.$ We will then define two notions of the Hochschild homology of a morphism of derived stacks, and show its compatibility with the standard definitions in various circumstances. 

\subsection{Basic Construction and Results}

\subsubsection{Hochschild Homology of Animated Rings}

If $k\to A$ is a map of (ordinary) commutative rings, then $\HH^{\cl}(A/k)\simeq A\otimes^L_{A\otimes_k A} A,$ as seen in the previous section. To remedy the issue described above, the fix is actually rather simple: rather than taking the tensor product relative to $A\otimes_k A,$ we instead take the tensor product relative to $A\otimes^L_k A,$ considered as an animated ring (or simplicial commutative ring if a more explicit model is required).

Indeed, taking this as our definition of derived Hochschild homology (henceforth simply Hochschild homology) has many pleasant consequences, extending the original definition to the derived setting in such a way that many of the important results and properties still hold true. Additionally, with this definition there is no reason to restrict $k$ and $A$ to be a map of ordinary commutative rings -- the construction works just as well when $k\to A$ is a map of arbitrary animated rings. In this section we will lay out the details of this definition, and the basic properties it enjoys. The material in this section is well-known, but we provide a relatively self-contained exposition for the benefit of the reader. We learned much of the following from \cite{MorrowHHnotes} and \cite{AntieauDAGnotes}.

\begin{mydef}
    Let $k\to A$ be a morphism of animated rings. Then $\HH(A/k) := A\otimes_{A\otimes_k A} A.$\footnote{As we are now working with animated rings, the tensor products here are implicitly derived.}
\end{mydef}

\begin{remark}
    If $k\to A$ is any map of commutative rings, the canonical map of animated rings $A\otimes^{L}_k A\to A\otimes^{\cl}_k A$ induces a natural map
    \[
    \HH(A/k)\to A\otimes^{L}_{A\otimes^{\cl}_k A} A\simeq\HH^{\cl}(A/k).\]
    If $k\to A$ is flat, then this map is an equivalence, but in general it is not. For example, if $k = \ZZ$ and $A = \FFp,$ then $\HH^{\cl}(A/k)\simeq \FFp[0].$ In contrast, $\HH_*(\FFp/\ZZ)\simeq\FFp\langle x\rangle,$ where $\FFp\langle x\rangle$ is a divided power algebra on a single generator $x$ in (homological) degree $2$ (see for example \cite{HHZFp}).
\end{remark}

Hochschild homology $\HH(A/k)$ admits an action of $S^1,$ although we have not mentioned this previously. This can be seen classically by recognizing that the cyclic bar construction yields not only a simplicial object, but a cyclic object (in the case when $M = A$). Another perspective on $S^1$-actions is the following: let $X\in\Ani$ be arbitrary, and let $x$ be a point of $X$ (i.e., a map $\ast\to X,$ where $\ast$ is terminal in $\Ani$). Then we obtain a functor
\[
x^* : \Fun(X,\aRing_k)\to\Fun(\ast,\aRing_k)\simeq\aRing_k
\]
which commutes with limits. The $\infty$-categorical adjoint functor theorem then implies that $x^*$ admits a left adjoint, which we denote
\[
x_! : \aRing_k\to\Fun(X,\aRing_k).
\]
We call $x_!(A)$ the \emph{free animated $k$-algebra on $A$ with an action of $\Omega_x X.$}

In light of the above perspective, we can define an $S^1$ action on an animated ring $A$ to be a functor $F : \mathrm{B}S^1\to\aRing_k$ such that $F$ maps the (essentially) unique object of $\mathrm{B}S^1$ to $A.$ We then have the following.

\begin{proposition}
    $\HH(-/k) : \aRing_k\to\Fun(\mathrm{B}S^1,\aRing_k)$ is left adjoint to the forgetful functor; i.e., $\HH(A/k)$ is the free animated $k$-algebra on $A$ with an action of $\Omega_x\mathrm{B}S^1\simeq S^1.$
\end{proposition}
\begin{proof}
    By the general theory of Kan extensions, the value of the left adjoint
    \[
    x_! : \aRing_k\simeq\Fun(\ast,\aRing_k)\to\Fun(\mathrm{B}S^1,\aRing_k)
    \]
    is given by the left Kan extension of the functor $\ast\to\aRing_k$ which is constant at $A$ along $\ast\to\mathrm{B}S^1.$ We therefore have
    \begin{align*}
        x_!(A)&\simeq\colim_{\ast\downarrow\mathrm{B}S^1} A\\
        &\simeq \colim_{\Omega_x\mathrm{B}S^1} A\\
        &\simeq \colim_{S^1} A.
    \end{align*}
    Writing $S^1$ as the pushout
    \[\begin{tikzcd}
        S^0\ar[d]\ar[r] &\ast\ar[d]\\
        \ast\ar[r] &S^1
    \end{tikzcd}\]
    and recalling that $S^0\simeq\ast\sqcup\ast$ then implies that $\colim_{S^1} A$ is equivalent to the pushout
    \[\begin{tikzcd}
        A\otimes_k A\ar[r]\ar[d] &A\ar[d]\\
        A\ar[r] &A\otimes_{A\otimes_k A} A
    \end{tikzcd}\]
    which is our original definition of $\HH(A/k).$
\end{proof}

Because $\HH(-/k)$ is a left adjoint, we obtain the following corollary.

\begin{cor}\label{HH(-/k) preserves colimits}
    $\HH(-/k) : \aRing_k\to\Fun(\mathrm{B}S^1,\aRing_k)$ preserves colimits.
\end{cor}

\begin{remark}
    This can of course also be deduced from the original definition (without the $S^1$-action), using formal properties of the tensor product $\otimes$ in $\aRing.$
\end{remark}

\begin{proposition}
    $\HH(-/k) : \aRing_k\to\Fun(\mathrm{B}S^1,\cMod_k)$ preserves sifted colimits.
\end{proposition}
\begin{proof}
    Follows from \ref{HH(-/k) preserves colimits} and the fact that $\aRing_k\to\cMod_k$ preserves sifted colimits.
\end{proof}

Hochschild homology satisfies a number of convenient properties, some of which we now state.

\begin{proposition}
    Let $k\to A$ be a map of animated rings, and let $B$ be another animated ring. Then the following identities hold:
    \begin{enumerate}
        \item\label{K\"unneth}
        \[\HH(A\otimes_k B/k)\simeq\HH(A/k)\otimes_k\HH(B/k).\]
        \item\label{Base change for affine HH}
        \[\HH(A/k)\otimes_k B\simeq\HH(A\otimes_k B/B).\]
        \item\label{Collapse reduction} Suppose that $B$ is an $A$-algebra. Then
        \[\HH(B/k)\otimes_{\HH(A/k)} A\simeq\HH(B/A).\]
    \end{enumerate}
\end{proposition}
\begin{proof}
    These properties are easily proven by direct manipulation of the definition.  In the following, let the unadorned tensor product symbol $\otimes$ denote the tensor product $\otimes_k$ over $k.$
    \begin{enumerate}
        \item 
        \begin{align*}
            (A\otimes_{A\otimes A} A)\otimes(B\otimes_{B\otimes B} B)&\simeq(A\otimes B)\otimes_{(A\otimes A)\otimes (B\otimes B)} (A\otimes B)\\
            &\simeq (A\otimes B)\otimes_{(A\otimes B)\otimes (A\otimes B)} (A\otimes B)
        \end{align*}
        \item
        \begin{align*}
            (A\otimes_{A\otimes A} A)\otimes B&\simeq (A\otimes_{A\otimes A} A)\otimes (B\otimes_B B)\\
            &\simeq (A\otimes B)\otimes_{(A\otimes A)\otimes B} (A\otimes B)\\
            &\simeq (A\otimes B)\otimes_{(A\otimes B)\otimes_B (A\otimes B)} (A\otimes B)
        \end{align*}
        \item 
        \begin{align*}
            (B\otimes_{B\otimes B} B)\otimes_{A\otimes_{A\otimes A} A} A&\simeq (B\otimes_{B\otimes B} B)\otimes_{A\otimes_{A\otimes A} A} (A\otimes_A A)\\
            &\simeq (B\otimes_A A)\otimes_{(B\otimes B)_{(A\otimes A)} A} (B\otimes_A A)\\
            &\simeq B\otimes_{(B\otimes_A A)\otimes_{k\otimes A} (B\otimes_A A)} B\\
            &\simeq B\otimes_{B\otimes_A B} B.
        \end{align*}
    \end{enumerate}
\end{proof}

Equally important but significantly less trivial is the following.

\begin{theorem}[\protect{\cite[Theorem~3.1]{BMS}}]
    Let $k$ be a commutative ring. The functor $\HH(-/k) : \aRing_k\to\cMod_k$ is a sheaf for the fpqc topology.
\end{theorem}

In the same way, we can deduce that $\HH(-/k) : \aRing_k\to\cMod_k$ is a sheaf for the fpqc topology when $k$ is an arbitrary animated ring.

The above theorem allows us to make the following definition for the Hochschild homology of a derived stack, which we will generalize in the following sections.

\begin{mydef}\label{Original definition of HH(X/k)}
    Let $\cX$ be a derived stack over a fixed base animated ring $k.$ The \emph{Hochschild homology of $\cX$ relative to $k,$} denoted $\HH(\cX/k),$ is defined as the right Kan extension of the Hochschild homology functor $\HH(-/k) : \dAff^{\op}_k\to\cMod_k$ from (the opposite of) the category of derived affine schemes over $k$ to the category of derived stacks over $k.$
\end{mydef}

Explicitly, this means that the Hochschild homology of $\cX$ relative to $k$ is given by the formula
\[
    \HH(\cX/k)\simeq \lim_{\Spec A\to\cX\in\dAff^{\op}_{/\cX}}\HH(A/k).
\]
It has the convenient property that the induced HKR filtration is complete.

\begin{remark}
    There are alternate definitions for the Hochschild homology of a derived stack which do not necessarily agree with the above in all cases, but we do not require these here. See for example \cite[Section 2]{antieau2017remark} for one alternative definition, based on the Hochschild homology of a stable presentable $k$-linear category. Note that other definitions of Hochschild homology do not necessarily have the property that the HKR filtration is necessarily complete.
\end{remark}

\subsection{The HKR Filtration and Theorem}

The derived version of the HKR theorem is no longer a statement about the individual Hochschild homology groups, but instead guarantees us a filtration with nice pieces. The idea is that for $A$ a polynomial $k$-algebra, the Postnikov filtration on $\HH^{\cl}(A/k)\simeq\HH(A/k)$ has associated graded pieces given by (shifts of) the exterior powers of the K\"ahler differentials by the classical HKR theorem. We then derive (left Kan extend) the Whitehead filtration beyond the polynomial case, to obtain a filtration on $\HH(A/k)$ for a general animated $k$-algebra $A,$ and the associated graded pieces will be given by the left Kan extension of the exterior powers of the K\"ahler differentials, which are the derived exterior powers of the cotangent complex. The precise statement is the following.

\begin{theorem}[Hochschild-Kostant-Rosenberg Theorem, \protect{\cite[Prop.~IV.4.1]{nikolaus2018topological}}]\label{HKR for CRing}
    Let $k\to A$ be a map of commutative unital rings. Then there exists a complete descending $\NN$-indexed filtration $\Fil_{\HKR}^\bullet$ on $\HH(A/k)$ with
    \[\gr^i_{\HKR}(\HH(A/k))\simeq\LL^i_{A/k}[i] := \left(L\bigwedge_A^i \LL_{A/k}\right)[i].\]
\end{theorem}

\begin{proof}
    Let $k\to A_\bullet\to A$ be a resolution of $A$ in $\sCRing_{k/}$ by polynomial $k$-algebras. Then $\HH(A/k)$ is computed by taking the totalization/diagonal of the bisimplicial object $B^{\cyc}(A_\bullet/k)_{\bullet},$ which we visualize by the following diagram (note that $B^{\cyc}(A_j/k)_i$ is in homological bidegree $(i,j)$).
    \[
    \begin{tikzcd}
        \vdots\arrow[d,shift right=4]\arrow[d,shift right=2]\arrow[d,shift left=4]\arrow[d,shift left=2]\arrow[d]&\vdots\arrow[d,shift right=4]\arrow[d,shift right=2]\arrow[d,shift left=4]\arrow[d,shift left=2]\arrow[d]&\vdots\arrow[d,shift right=4]\arrow[d,shift right=2]\arrow[d,shift left=4]\arrow[d,shift left=2]\arrow[d]&\vdots\arrow[d,shift left=4]\arrow[d,shift left=2]\arrow[d,shift right=4]\arrow[d,shift right=2]\arrow[d]&\\
        \HH(A_3/k)\arrow[d,shift right=3]\arrow[d,shift right]\arrow[d,shift left]\arrow[d,shift left=3]&B^{\cyc}(A_3/k)_0\arrow[d,shift right=3]\arrow[d,shift right]\arrow[d,shift left]\arrow[d,shift left=3]&B^{\cyc}(A_3/k)_1\arrow[d,shift right=3]\arrow[d,shift right]\arrow[d,shift left]\arrow[d,shift left=3]\arrow[l,shift left]\arrow[l,shift right]&B^{\cyc}(A_3/k)_2\arrow[d,shift right=3]\arrow[d,shift right]\arrow[d,shift left]\arrow[d,shift left=3]\arrow[l,shift right=2]\arrow[l,shift left=2]\arrow[l]&\dots\arrow[l,shift left=3]\arrow[l,shift left]\arrow[l,shift right]\arrow[l,shift right=3]\\
        \HH(A_2/k)\arrow[d,shift right=2]\arrow[d,shift left=2]\arrow[d]&B^{\cyc}(A_2/k)_0\arrow[d,shift right=2]\arrow[d,shift left=2]\arrow[d]&B^{\cyc}(A_2/k)_1\arrow[d,shift right=2]\arrow[d,shift left=2]\arrow[d]\arrow[l,shift left]\arrow[l,shift right]&B^{\cyc}(A_2/k)_2\arrow[d,shift right=2]\arrow[d,shift left=2]\arrow[d]\arrow[l,shift right=2]\arrow[l,shift left=2]\arrow[l]&\dots\arrow[l,shift left=3]\arrow[l,shift left]\arrow[l,shift right]\arrow[l,shift right=3]\\
        \HH(A_1/k)\arrow[d,shift right]\arrow[d,shift left]&B^{\cyc}(A_1/k)_0\arrow[d,shift right]\arrow[d,shift left]&B^{\cyc}(A_1/k)_1\arrow[d,shift right]\arrow[d,shift left]\arrow[l,shift right]\arrow[l,shift left]&B^{\cyc}(A_1/k)_2\arrow[d,shift right]\arrow[d,shift left]\arrow[l,shift right=2]\arrow[l,shift left=2]\arrow[l]&\dots\arrow[l,shift right=3]\arrow[l,shift right]\arrow[l,shift left]\arrow[l,shift left=3]\\
        \HH(A_0/k)&B^{\cyc}(A_0/k)_0&B^{\cyc}(A_0/k)_1\arrow[l,shift right]\arrow[l,shift left]&B^{\cyc}(A_0/k)_2\arrow[l,shift right=2]\arrow[l,shift left=2]\arrow[l]&\dots\arrow[l,shift right=3]\arrow[l,shift right]\arrow[l,shift left]\arrow[l,shift left=3]
    \end{tikzcd}\]
    Each row of the above diagram computes $\HH(A_i/k).$ Since each $A_i$ is a polynomial $k$-algebra, the classical Hochschild-Kostant-Rosenberg theorem implies that the filtration given by the canonical truncations $\Fil_{\HKR}^\bullet\HH(A_i/k) = \tau_{\geq \bullet}\HH(A_i/k)$ has associated gradeds given by
    \[\gr_{\HKR}^n\HH(A_i/k)\simeq\H_{n}(\HH(A_i/k))[n]\simeq\Omega^n_{A_i/k}[n].\]
    Thus, the filtration $\Fil_{\HKR}^\bullet$ on the bisimplicial object given by truncating in the horizontal direction has associated gradeds given by
    \[\gr_{\HKR}^n\HH(A_\bullet/k)\simeq\H_{n}(\HH(A_\bullet/k))[n]\simeq\Omega^n_{A_\bullet/k}[n].\]
    The desired result follows, as
    \[\Omega^n_{A_\bullet/k}\otimes^{L}_{A_\bullet}A\simeq\LL^n_{A/k},\]
    where the last equivalence is given by the Dold-Kan correspondence.

    Noting that the $n$th step in the filtration on the bisimplicial diagram above induced from the row-wise Whitehead filtration is $n$-connective in each row, it follows that the filtration is complete (modeling the totalization/geometric realization by the diagonal).
\end{proof}

The above theorem generalizes to the case of $k\to A$ an arbitrary map of animated rings, and the proof is similar, except we realize the map $k\to A$ as a map $k_\bullet\to A_\bullet$ of simplicial commutative rings, and then resolve $A_\bullet$ by a bisimplicial commutative $k$-algebra which is polynomial over the respective $k_n$ in each degree. In particular, we obtain the following.

\begin{theorem}
    Let $k\to A$ be a map of animated rings. There is a natural complete decreasing $\NN$-indexed filtration $\FilHH{*}(A/k)$ on $\HH(A/k)$ with
    \[
    \gr_{\HKR}^n\HH(A/k)\simeq\LL^n_{A/k}[n] := \bigwedge^n\LL_{A/k}[n].
    \]
\end{theorem}

\begin{remark}
    In fact, even more is true regarding the HKR filtration. It is a multiplicative filtration, and it is $S^1$-equivariant, and the associated graded pieces have trivial $S^1$-action. For a proof which details the multiplicative structure and the $S^1$-action, see \cite[Theorem~6.3]{AntieauDAGnotes} or \cite[Theorem~6.2.6]{raksit2020hochschild} and \cite{FiltrationsandcohomologyI} for a more involved discussion which discusses the universal property of $\FilHH{*}$ (that is, a universal property of Hochschild homology equipped with its HKR filtration).
\end{remark}

\subsection{Relative Derived Hochschild Homology}

In the previous sections, we have only defined the Hochschild homology of a derived stack relative to a fixed ground animated ring. However, it will be critical for us to consider the Hochschild homology of a morphism of derived stacks where the base is not assumed to be affine or even fixed. In this section, we construct the relative Hochschild homology of a morphism of derived stacks and show that it extends the usual definitions of Hochschild homology.

Consider the functors
\begin{align*}
    \sHoch : \Arr(\dAff)^{\op}&\to\QC\\
    (\Spec A\to\Spec k)&\mapsto (\Spec A,\HH(A/k))
\end{align*}
and
\begin{align*}
    \Hoch : \Arr(\dAff)^{\op}&\to\QC\\
    (\Spec A\to\Spec k)&\mapsto (\Spec k,\HH(A/k)),
\end{align*}
where in the first case, we view $\HH(A/k)$ as an $A$-module and in the second we view $\HH(A/k)$ as a $k$-module.

These assignments on objects can be seen to be functorial by using the functoriality of the diagonal morphism along with functoriality of limits/fiber products. Once one obtains an endofunctor $\Arr(\dAff)\to\Arr(\dAff)$ sending
\[
(\Spec A\to \Spec k)\mapsto (\Spec A\otimes_{A\otimes_k A} A\to\Spec A)
\]
(respectively, $(\Spec A\to \Spec k)\mapsto (\Spec A\otimes_{A\otimes_k A} A\to\Spec k)$), then one can map $\Arr(\dAff)\to\QC$ by using the structure morphism to provide $A\otimes_{A\otimes_k A} A$ with the structure of an $A$-module (respectively, $k$-module).

Explicitly, $\Arr(\aRing)$ is equivalent to the full sub-$\infty$-category of $\Fun(\Delta^1\times\Delta^1,\aRing)$ spanned by those diagrams which
\begin{enumerate}
    \item send the edges $(0,0)\to(1,0)$ and $(0,0)\to (0,1)$ to equivalent morphisms; and
    \item exhibit the image of $(1,1)$ as the colimit of the rest of the diagram.
\end{enumerate}
The desired equivalence follows from the fact that all diagrams in $\aRing$ of this form above admit colimits, together with the universality of colimits. Similar reasoning shows that this is further equivalent to the $\infty$-category of commutative diagrams in $\aRing$ of the form
\[\begin{tikzcd}
    k\ar[r]\ar[d] &A\ar[d]\ar[rd,"\id_A"]\\
    A\ar[r]\ar[rd,"\id_A"'] &A\otimes_k A\ar[r]\ar[d] & A\ar[d,"\iota_2"]\\
    &A\ar[r,"\iota_1"] & A\otimes_{A\otimes_k A} A.
\end{tikzcd}\]

Then restriction to the composition $\iota_1 : A\to A\otimes_{A\otimes_k A}A$ (respectively, $k\to A\otimes_{A\otimes_k A}A$) provides a functor $\Arr(\aRing)\to\Arr(\aRing)$ which on objects maps
\[
(k\to A)\mapsto (A\to A\otimes_{A\otimes_k A}A)
\]
(respectively, $(k\to A)\mapsto (k\to A\otimes_{A\otimes_k A}A)$).

Now, we want to construct the functor $\Arr(\aRing)\to\QC$ which will heuristically forget the animated ring structure on the target, and treat it as a module over the source.

Animating the functor $\Arr(\CRing)\to\mathsf{MOD}$ which associates to a ring homomorphism $(f : A\to B)$ the pair $(A,B)$ where $B$ has $A$-module structure via $f$ produces a functor
\[
\Arr(\aRing)\to\MOD^\cn\hookrightarrow\QC
\]
using \cite[25.2.1.2]{SAG} and \ref{Animation of arrow category} (recall that $\MOD^\cn$ is the full sub-$\infty$-category of $\QC$ spanned by pairs $(\cX,\cF)$ such that $\cX$ is an affine derived scheme $\Spec A,$ and $\cF$ is a connective $A$-module).  

Combining the above gives the desired functor $\Arr(\aRing)\to\QC$ which maps
\[
(k\to A)\mapsto (\Spec A,A\otimes_{A\otimes_k A} A)
\]
(or $(\Spec k,A\otimes_{A\otimes_k A} A)$).

We are now ready to make our definition of relative Hochschild homology of a morphism of derived stacks.
\begin{mydef}[Relative Hochschild homology]\label{relativeHH}
    The \emph{relative Hochschild functors}
    \[
    \sHoch,\Hoch : \Arr(\dStk)^{\op}\to\QC
    \]
    are defined to be the right Kan extensions of $\sHoch,\Hoch : \Arr(\dAff)^{\op}\to\QC$ along the inclusion $\Arr(\dAff)^{\op}\hookrightarrow\Arr(\dStk)^{\op}.$

    As the objects of $\QC$ may be identified with pairs $(\cX,\cF),$ where $\cX\in\dStk$ and $\cF\in\QC(\cX),$ we will write
    \[
    \sHoch(f : \cX\to\cS) = (\cX,\sHH(\cX/\cS))\quad\textrm{ and }\quad\Hoch(\cX\to\cS) = (\cS,\HH(\cX/\cS)).
    \]
    We will call $\sHH(\cX/\cS)$ (respectively, $\HH(\cX/\cS)$) the \emph{Hochschild homology sheaf of $f : \cX\to\cS$} (respectively, the \emph{Hochschild homology module of $f : \cX\to\cS$}).
\end{mydef}

We begin with a result clarifying the relationship between the two Hochschild functors.

\begin{proposition}\label{pushforward compatibility of HH}
    Let $f : \cX\to\cS$ be a morphism of derived stacks. Then the canonical morphism $(\cS,\HH(\cX/\cS))\to (\cX,\sHH(\cX/\cS))$ induces an equivalence
    \[
    \HH(\cX/\cS)\to f_*\sHH(\cX/\cS).
    \]
\end{proposition}
\begin{proof}
    Writing
    \[
    (\cX,\sHH(\cX/\cS))\simeq\lim (\Spec A, A\otimes_{A\otimes_k A}A)
    \]
    and
    \[
    (\cS,\HH(\cX/\cS))\simeq\lim (\Spec k, A\otimes_{A\otimes_k A} A),
    \]
    where both limits are over maps from $(\Spec A\to\Spec k)$ to $(\cX\to\cS)$ in $\Arr(\dStk)^{\op},$ we have an induced map $(\Spec k,A\otimes_{A\otimes_k A} A)\to(\Spec A, A\otimes_{A\otimes_k A} A)$ for any such $\Spec A\to\Spec k.$ Taking the limit produces a morphism $(\cS,\HH(\cX/\cS))\to (\cX,\sHH(\cX/\cS))$ which gives the morphism $\HH(\cX/\cS)\to f_*\sHH(\cX/\cS)$ in question.

    The morphism $\HH(\cX/\cS)\to f_*\sHH(\cX/\cS)$ is then the limit of the corresponding morphisms $A\otimes_{A\otimes_k A}A\to A\otimes_{A\otimes_k A} A.$ As all of these morphisms are equivalences, so is the limit.
\end{proof}

\begin{remark}
    One can also formulate this compatibility using Cartesian transport for the fibration $\QC\to\dStk^{\op}.$ More precisely, the statement would be that the right Kan extension along $\Arr(\dAff)^{\op}\to\Arr(\dStk)^{\op}$ of the composition
    \[
    \Arr(\dAff)^{\op}\to\QC\times_{\dStk^{\op}}\Arr(\dStk^{\op})\to\QC
    \]
    where the first functor is induced by $\sHoch$ and the inclusion and the second functor is Cartesian transport is canonically identified with the composition of the right Kan extension of $\Arr(\dAff)^{\op}\to\QC\times_{\dStk^{\op}}\Arr(\dStk^{\op})$ composed with Cartesian transport.
\end{remark}

The above definitions do indeed extend the usual notions of Hochschild homology. By properties of Kan extensions along fully faithful functors, we see that the Hochschild homology sheaf/module of a map of affine schemes is given by its original definition (thus our abuse of notation in letting $\sHoch$ and $\Hoch$ denote both the functors on affines and arbitrary stacks is justified); but moreover, our definition agrees with (one of) the standard definitions of the Hochschild homology of a stack over a fixed ground ring.

\begin{proposition}
    Let $\cX\to\Spec k$ be a morphism of derived stacks. Then the relative Hochschild homology $\HH(\cX/\Spec k)$ is equivalent to $\HH(\cX/k)$ as defined in \ref{Original definition of HH(X/k)}. 
\end{proposition}
\begin{proof}
    Consider the diagram
    \[\begin{tikzcd}
        \Arr(\dAff)^{\op}\ar[r]\ar[rd,"t^{\op}"'] &\left(\Arr(\dStk)_{/\dAff}\right)^{\op}\ar[r]\ar[d,"t^{\op}"]\arrow[dr, phantom, "\usebox\pullback" , very near start, color=black]&\Arr(\dStk)^{\op}\ar[r,"\Hoch"]\ar[d,"t^{\op}"] & \QC\ar[ld,"p"]\\
        & \dAff^{\op}\ar[r] & \dStk^{\op}.
    \end{tikzcd}\]
    Taking the pullback of this diagram along $\{\Spec k\}\to\dAff^{\op}\to\dStk^{\op}$ yields
    \[\begin{tikzcd}
        \dAff_k^{\op}\ar[r]\ar[rd] &\dStk_k^{\op}\ar[r,"\sim"]\ar[d]\arrow[dr, phantom, "\usebox\pullback" , very near start, color=black]&\dStk_k^{\op}\ar[r]\ar[d] & \Mod_k\ar[ld]\\
        & \{\Spec k\}\ar[r] & \{\Spec k\}
    \end{tikzcd}\]
    using the fact that $\dAff\subseteq\dStk$ is a replete subcategory. The result then follows by (the dual of) \cite[lemma 4.3.3.9]{HTT}: since $t^{\op}$ is Cartesian, the right Kan extension of the pullback of $\Hoch$ along $\{\Spec k\}\to\dAff^{\op}$ is the pullback of the right Kan extension.
\end{proof}

We conclude this section by noting that the Hochschild functors just defined are in fact fpqc sheaves.

\begin{theorem}
    $\Hoch,\sHoch : \Arr(\dStk)^{\op}\to\QC$ are sheaves for the fpqc topology.
\end{theorem}

We delay the proof for now, since we will discuss an improved variant (that these functors may be promoted to sheaves for the canonical topology valued in $\QCF$) in the next section.

\subsection{HKR-Filtered Derived Hochschild Homology}

As we have already seen, Hochschild homology of a morphism of animated rings comes equipped with a canonical, functorial Hochschild-Kostant-Rosenberg filtration, which is induced by the Postnikov filtration for a smooth morphism of polynomial $\ZZ$-algebras. By mimicking \ref{relativeHH} and replacing $\QC$ with $\QCF,$ we can view the Hochschild functors
\[
\sHoch,\Hoch : \Arr(\dAff)^{\op}\to\QC
\]
as functors instead valued in $\QCF,$ the source of the coCartesian fibration classifying the functor $\dStk^{\op}\to\Cat_\infty$ sending $X\mapsto \QCF(X) := \Fun(\ZZ^{\op},\QC(X)).$ To distinguish these functors, we will denote the variants valued in $\QCF$ by $\fsHoch$ and $\fHoch.$ We will denote the resulting filtered object(s)/filtration on $\sHH$ (respectively, $\HH$) by $\sFilHH{*}$ (respectively, $\FilHH{*}$).

Composing these filtered Hochschild functors with the ``underlying object'' functor produces the previously defined Hochschild functors. This follows because the Hochschild functors $\Arr(\dAff)^{\op}\to\QCF$ in fact factor through the full sub-$\infty$-category $\QCF_{\NN}$ spanned by those objects whose filtrations are $\NN$-indexed, in the sense that $\Fil^i X\to\Fil^{i-1} X$ is equivalent to the identity for $i\leq 0.$ Since the subcategory of such filtrations is closed under limits, the right Kan extension to $\QCF$ in fact factors through $\QCF_{\NN}$ to tell us that that the Hochschild filtrations are $\NN$-indexed. Moreover, the underlying object functor is then simply given by evaluation at $0,$ and this is a right adjoint, hence it preserves Kan extensions, so that the underlying object of the Kan extended HKR filtration will be the Kan extension of the functor which constructs the HKR-filtered Hochschild homology and then takes the underlying object. One could alternatively (and equivalently) use $\QCF_{\NN}$ in place of $\QCF$ everywhere.

\begin{proposition}\label{Existence of HKR filtration}
    Let $f : \cX\to\cS$ be a map of derived stacks. Then $\sHH(\cX/\cS)$ and $\HH(\cX/\cS)$ admit complete descending $\NN$-indexed filtrations; i.e., $\sFilHH{*}(\cX/\cS)$ and $\FilHH{*}(\cX/\cS)$ are complete.
\end{proposition}
\begin{proof}
    The previous discussion implies that the filtrations are $\NN$-indexed. Given a filtered object $\Fil^\bullet X,$ completeness means that
    \[
    \lim_{i\to\infty}\Fil^i X := \lim_{\ZZ^{\op}}\Fil^i X\simeq 0.
    \]
    Thus, completeness of the HKR filtration follows from the fact that right Kan extensions commute with limits.
\end{proof}

\begin{remark}
    The filtered Hochschild homology we define here might be called ``HKR-completed Hochschild homology.'' As mentioned previously, there are multiple possible definitions of Hochschild homology, all of which agree for maps of commutative rings (or animated rings), but which may differ once extended to stacks, and in particular, it can be the case that the HKR filtration on another variant of the Hochschild homology may not be complete. It would be interesting to examine an ``uncompleted Hochschild homology'' for a morphism $f : \cX\to\cS$ of derived stacks, either as a relative Hochschild homology for the stable $\infty$-categories $\QC(\cX)$ and $\QC(\cS),$ or as $\Delta^*\Delta_*\OX.$ We will show that under some hypotheses, the latter agrees with our definition, but we do not pursue this question further.
\end{remark}

As a consequence of the above proposition, we find that when we have a morphism of filtered objects between Hochschild homology and another complete filtered object, to check whether the morphism is an equivalence it suffices to check if the maps on associated graded pieces is an equivalence (this follows from \cite[Lemma 2.15]{GwilliamPavlovEnhancing}\footnote{Note that what we call $\Fil\cC$ is called $\operatorname{Seq}\cC$ in \cite{GwilliamPavlovEnhancing}, and what we call $\cFil\cC$ is called $\operatorname{Fil}\cC$ in \cite{GwilliamPavlovEnhancing}.}).

The following proposition will prove invaluable when proving properties of Hochschild homology.

\begin{proposition}\label{Description of associated graded for HKR filtration}
    Suppose that $f : \cX\to\cS$ is a map of derived stacks admitting a cotangent complex (e.g., $f$ geometric). The associated graded sheaves of the HKR filtrations on the Hochschild homology of $\cX/\cS$ are given by
    \begin{align*}
        \gr_{\HKR}^i\sHH(\cX/\cS)&\simeq \left(\bigwedge^i \LL_{\cX/\cS}\right)[i]\\
        \gr_{\HKR}^i\HH(\cX/\cS)&\simeq f_*\left(\bigwedge^i \LL_{\cX/\cS}\right)[i]
    \end{align*}
\end{proposition}
\begin{proof}
    The result is true for affines by construction. Then the result follows by commuting the cofiber (which is a finite colimit) with the limit.
\end{proof}

We actually need to use a slightly stronger version of the above HKR result when proving various properties of Hochschild homology. In particular, we require the following:

\begin{proposition}
    For any $n\geq 0,$ the composition
    \[
    \Arr(\dStk)^{\op}\xrightarrow{\sFilHH{}(-)} \QCF \xrightarrow{\gr}\QCGR\xrightarrow{\ev_n}\QC\xrightarrow{[-n]}\QC
    \]
    agrees with the functor $\bigwedge^n\bfL^{\Kan} : \Arr(\dStk)^{\op}\to\QC.$
\end{proposition}
\begin{proof}
    Consider the functors restricted to $\PolyPoly.$ By the classical HKR theorem, they agree on objects, and by naturality of the antisymmetrization maps, they agree as functors. Forming the left Kan extension along the inclusion $\PolyPoly\to\Arr(\dAff)^{\op}$ gives the functor $\bigwedge^n\bfL^{\Kan}$ by definition.
    
    On the other hand, because $[-n]\circ\ev_n\circ\gr$ is a left adjoint (by virtue of each individual component being a left adjoint), it follows that $[-n]\circ\ev_n\circ\gr$ preserves left Kan extensions, so that the left Kan extension of the composition $[-n]\circ\ev_n\circ\gr\circ\sFilHH{}(-)$ agrees with $[-n]\circ\ev_n\circ\gr$ composed with the left Kan extension of $\sFilHH{}(-).$

    The result then follows as $\Arr(\dAff)^{\op}$ is a basis for $\Arr(\dStk)^{\op}$ and these functors are fpqc sheaves.
\end{proof}

It follows that given a commutative square of derived geometric stacks
\[\begin{tikzcd}
    \cX'\ar[r,"f"]\ar[d] &\cX\ar[d]\\
    \cS'\ar[r] &\cS
\end{tikzcd}\]
the morphism $f^*\sFilHH{}(\cX/\cS)\to\sFilHH{}(\cX'/\cS')$ on Hochschild homology induces on associated gradeds (shifts of) the maps $f^*\bigwedge^n\LL_{\cX/\cS}\to\bigwedge^n\LL_{\cX'/\cS'}$ arising from the same square. In particular, since the map $f^*\bigwedge^n\LL_{\cX/\cS}\to\bigwedge^n\LL_{\cX'/\cS'}$ is the $n$th exterior power of the map $f^*\LL_{\cX/\cS}\to\LL_{\cX'/\cS'},$ it suffices to check that the map on first associated gradeds is an equivalence, and this can be adapted to various situations. One such situation is when we compare $\HH^{\Ols}$ and $\HH^{\Rog},$ although this is not a direct corollary of the above.

\begin{lemma}[\protect{\cite[Proposition~6.3.8]{bhatt2022absoluteprismaticcohomology}}]
    Suppose that $A$ is a commutative ring, and suppose that the absolute cotangent complex $\LL_{A/\ZZ}$ is a flat $A$-module (regarded as a complex concentrated in homological degree zero). Then the HKR filtration on $\HH(A/\ZZ)$ coincides with the Whitehead filtration; i.e., for every integer $n,$ the map $\FilHH{n}(A/\ZZ)\to\HH(A/\ZZ)$ is an isomorphism on homology in degrees $\geq n$ (and the homology groups of $\FilHH{n}(A)$ vanish in degrees $\geq -n$).
\end{lemma}

We note that \cite[proposition 6.3.8]{bhatt2022absoluteprismaticcohomology} extends to the following global variant, with no modification in the proof required.

\begin{theorem}\label{Flat cotangent complex implies postnikov filtration}
    Suppose that $\cX\to\cS$ is map of derived stacks admitting a cotangent complex, and suppose that the relative cotangent complex $L_{\cX/\cS}$ is a flat $\cO_{\cX}$-module (regarded as a complex concentrated in cohomological degree zero, e.g., if $f$ is a smooth morphism from a classical algebraic space to a classical algebraic stack). Then the HKR filtration on $\HH(\cX/\cS)$ coincides with the Whitehead filtration; i.e., for every integer $n,$ the map $\FilHH{n}(\cX/\cS)\to\HH(\cX/\cS)$ is an isomorphism on homology in degrees $\geq n$ (and the cohomology groups of $\FilHH{n}(\cX/\cS)$ vanish in degrees $\geq -n$).
\end{theorem}

\begin{cor}\label{Smooth HKR filtration is Whitehead}
    Suppose $\cX\to\cS$ is a smooth map of derived classical stacks. Then the HKR filtration on $\sHH(\cX/\cS)$ is the Whitehead filtration.
\end{cor}

With the above results in hand, we are ready to prove various important properties of the relative Hochschild homology.

\begin{theorem}\label{HH base change results}
    Suppose that
    \[\begin{tikzcd}
        \cX'\ar[r,"\overline{g}"]\ar[d,"\overline{f}"] &\cX\ar[d,"f"]\\
        \cS'\ar[r,"g"] &\cS
    \end{tikzcd}\]
    is a commutative square of geometric derived stacks. Then:
    \begin{enumerate}
        \item If the square is Cartesian, the induced morphisms
        \[
        \overline{g}^*\sHH(\cX/\cS)\to\sHH(\cX'/\cS')
        \]
        and
        \[
        g^*\HH(\cX/\cS)\to\HH(\cX'/\cS')
        \]
        (arising from $\sHoch(\cX/\cS)\to\sHoch(\cX'/\cS')$ and $\Hoch(\cX/\cS)\to\Hoch(\cX'/\cS'),$ respectively) are equivalences.
        \item If $g$ is \'etale, then the induced morphisms
        \[
        \sHH(\cX'/\cS)\to\sHH(\cX'/\cS')
        \]
        and
        \[
        \HH(\cX'/\cS)\to\HH(\cX'/\cS')
        \]
        are equivalences.
        \item If $\overline{g}$ is \'etale, then the induced morphisms
        \[
        \overline{g}^*\sHH(\cX/\cS)\to\sHH(\cX'/\cS)
        \]
        and
        \[
        \HH(\cX/\cS)\to\HH(\cX'/\cS)
        \]
        are equivalences.
    \end{enumerate}
\end{theorem}
\begin{proof}
    The three arguments are all of the same flavor: we examine the induced morphisms on associated graded pieces, show that we have equivalences there, and appeal to completeness of the HKR filtration.
    \begin{enumerate}
        \item On associated graded pieces, we obtain (shifts of) the canonical equivalence
        \[
        \overline{g}^*\bigwedge^n \LL_{\cX/\cS}\simeq\bigwedge^n \overline{g}^* \LL_{\cX/\cS}\to\bigwedge^n \LL_{\cX'/\cS'}.
        \]
        The result then follows by completeness of the HKR filtration.
        \item On associated graded pieces, this becomes the morphism induced by applying $\bigwedge^n$ to the canonical map $\LL_{\cX'/\cS}\to \LL_{\cX'/\cS'}$ (and shifting). We have an identification
        \[
        \fib\left(\bigwedge^n \LL_{\cX'/\cS}\to\bigwedge^n \LL_{\cX'/\cS'}\right)\simeq f^* \LL_{\cS'/\cS}\otimes\bigwedge^n \LL_{\cX'/\cS},
        \]
        and since $\cS'\to\cS$ is \'etale, $L_{\cS'/\cS}$ vanishes. Alternatively, we can realize that the canonical map on the first associated graded piece
        \[
        \LL_{\cX'/\cS}\to \LL_{\cX'/\cS'}
        \]
        is an equivalence, as the fiber is $f^*\LL_{S'/S},$ which vanishes since $S'\to S$ is \'etale.
        \item Considering the composition $\cX'\to\cX\to\cS$ produces the cofiber sequence
        \[
        \overline{g}^* \LL_{\cX/\cS}\to \LL_{\cX'/\cS}\to \LL_{\cX'/\cX},
        \]
        and as in the previous results, $\LL_{\cX'/\cX}$ vanishes by the \'etale assumption. In the same way, we obtain the desired equivalences on all associated graded pieces (being exterior powers and shifts of this equivalence), and completeness of the HKR filtration completes the proof.
    \end{enumerate}
\end{proof}

Consider the functor
\[
D : \Arr(\dStk)^{\op}\to\QC
\]
which is ``given by'' $D(\cX/\cS) = (\cX,\Delta^*\Delta_*\cO_{\cX/\cS}).$\footnote{One can define this functor precisely using coCartesian transport for $\QC.$} Because $D$ and $\sHoch$ agree on $\Arr(\dAff)^{\op},$ the universal property of the right Kan extension guarantees a unique (up to contractible choice) map of functors $\eta : D\to\sHH,$ and thus a unique map $\eta_{\cX/\cS} : \Delta^*\Delta_*\cO_{\cX/\cS}\to\sHH(\cX/\cS)$ (which in fact exhibits $\sHoch$ as the sheafification of $D$).

\begin{proposition}\label{relative HH is self intersection of diagonal for relative derived scheme}
    Let $\cX\to\cS$ be a relative derived scheme (i.e., a map of derived stacks such that for any map $\Spec A\to\cS,$ $\cX\times_{\cS}\Spec A$ is a derived scheme). Then $\Delta^*\Delta_*\cO_{\cX/\cS}$ and $\sHH(\cX/\cS)$ are canonically equivalent via $\eta_{\cX/\cS}.$ In particular, if $\cX$ is a derived scheme and $\cS$ is a derived algebraic stack with affine diagonal, this holds.
\end{proposition}
\begin{proof}
    To show that $\eta_{\cX/\cS}$ is an equivalence when $\cX\to\cS$ is a relative derived scheme, it is enough to show that for any Cartesian diagram
    \[\begin{tikzcd}
        \cX'\ar[r,"p"]\ar[d] &\cX\ar[d,"f"]\\
        \Spec A\ar[r] & \cS
    \end{tikzcd}\]
    of derived stacks, the map $\eta_{\cX/\cS}$ induces an equivalence
    \[
    p^*\eta_{\cX/\cS} : p^*\Delta^*\Delta_*\cO_{\cX/\cS}\to p^*\sHH(\cX/\cS).
    \]
    A standard calculation shows that $p^*\Delta^*\Delta_*\cO_{\cX/\cS}\simeq\Delta^*\Delta_*\cO_{\cX'/A},$ and by \ref{HH base change results} we have $p^*\sHH(\cX/\cS)\simeq\sHH(\cX'/A).$ Moreover, $p^*\eta_{\cX/\cS}\simeq\eta_{\cX'/A}$ (by naturality), so we reduce to the situation when $\cS = \Spec A$ is affine and $\cX = X$ is a derived scheme.

    The result then follows from the facts that the assignment $U\mapsto \cL_{S} U$ is a Zariski cosheaf on $X,$ that $U\mapsto (U,\sHH(U/S))$ is a Zariski sheaf, and that for $U = \Spec B\to \Spec A,$ the canonical maps 
    \[
    (p_1)_*\cO_{\cL_{S} U}\to\Delta_*\Delta^*\cO_{U/S}\to\sHH(U/S)
    \]
    are equivalences. Alternatively, for $u : U\to X$ an open immersion, we have $u^*\Delta^*\Delta_*\cO_{X/S}\to \Delta^*\Delta_*\cO_{U/S}$ and $u^*\sHH(X/S)\to\sHH(U/S)$ via the canonical maps, so naturality of $\eta$ implies the result.
\end{proof}

\begin{remark}
    We expect that the previous proposition holds in greater generality. In particular, Olsson's result relating classical sheafy Hochschild homology to the pushforward and (derived) pullback of the structure sheaf along the diagonal holds without the assumption that the map is a relative derived scheme. However, we will only apply this result for morphisms $\cX\to\cS$ where $\cX$ is an affine derived scheme and $\cS$ has affine diagonal, so that $\cX\to\cS$ is indeed a relative derived scheme.
\end{remark}

We also have a stronger relative version of the earlier descent results.

\begin{theorem}[\protect{\cite[Theorem~6.2]{FiltrationsandcohomologyII}}]\label{Canonical descent}
    The functors
    \[
    \HH(-) : \Arr(\aRing)\to \aRing
    \]
    and
    \[
    \FilHH{*} : \Arr(\aRing)\to\Fil(\aRing)
    \]
    satisfy descent for the canonical topology on $\Arr(\aRing).$
\end{theorem}

\begin{remark}
    A cover $\{(A/k)\to (A_i/k_i)\}_i$ in the canonical topology on $\Fun(\Delta^1,\aRing)$ in the above reference is a collection of maps such that $\{A\to A_i\}$ is a canonical cover of $A$ in $\aRing.$ In particular, note that there is no condition on the maps $k\to k_i,$ so the result cited is stronger than what one obtains using our definition of the canonical topology.
\end{remark}

We do not need the full power of the above; descent for the fpqc topology will suffice. We suspect that the following proposition can be formally deduced for the canonical topology from \ref{Canonical descent}, but we can also obtain the fpqc case as a corollary of \ref{cotangent complex is fpqc sheaf on morphisms}.

\begin{proposition}
    The functors
    \[
    \sFilHH{*} : \Arr(\aRing)\to\QCF
    \]
    and
    \[
    \FilHH{*} : \Arr(\aRing)\to\QCF
    \]
    are sheaves for the fpqc topology on $\Arr(\aRing).$
\end{proposition}
\begin{proof}
    The proof is essentially identical to \cite[Corollary 3.3]{BMS}. In particular, \ref{cotangent complex is fpqc sheaf on morphisms} and induction will imply that the quotients $\sHH(-)/\cF\operatorname{il}^n_{\HKR}$ are fpqc sheaves for all $n,$ and by completeness, taking the limit produces $\sHH(-)$ which is a sheaf. The same proof works for $\FilHH{*}.$
\end{proof}

\section{Logarithmic Hochschild Homology}

In this section, we will recall the definitions of logarithmic Hochschild homology due to Olsson and Rognes, and culminate with our theorem comparing the two.

\subsection{After Olsson}

In \cite{Olssonloghh}, Olsson makes the following definition.

\begin{mydef}
    Let $f : X\to S$ be a log flat morphism of fine log schemes, and consider the associated morphism $\underline{X}\to\cLog_S.$ The \emph{Olsson logarithmic Hochschild homology} of $f$ is given by
    \[
    \sHH^{\Ols,\cl}(X/S) := \sHH^{\cl}(\underline{X}/\cLog_S),
    \]
    where $\sHH^{\cl}$ is the construction of \ref{Olsson stacky HH}.
\end{mydef}

Using this definition, together with \ref{sheafy classical HKR}, we obtain the following logarithmic version of the HKR theorem.

\begin{theorem}
    Let $f : X\to S$ be a log flat map of fine log schemes. Then there are isomorphisms
    \[
    \epsilon_n : \Omega^n_{X/S}\to\sHH^{\Ols,\cl}_{X/\cS,n}
    \]
    are isomorphisms for all $n.$
\end{theorem}

With our definition of relative Hochschild homology for maps of derived stacks, it is simple to extend Olsson's logarithmic Hochschild homology to the non-flat case.

\begin{mydef}\label{HHOls definition}
    Let $f : X\to S$ be a morphism of fine log schemes. Then we define the \emph{Olsson log Hochschild homology} and \emph{Olsson log Hochschild homology sheaves} as
    \begin{align*}
        \HHOls(X/S) &:= R\Gamma(\cLog_S,\HH(\underline{X}/\cLog_S)),\\
        \sHHOls(X/S) &:= \sHH(\underline{X}/\cLog_S).
    \end{align*}
\end{mydef}

Our definition of $\HHOls(X/S)$ differs slightly from the expected definition so that we can view $\HHOls(X/S)$ as an object of $\cD(\ZZ)$ (another option would be to view it as an object of $\QC(\underline{S}),$ but working in $\cD(\ZZ)$ will be more convenient for our later comparisons).

Using our previously established results for the Hochschild homology of derived stacks along with results in \cite{OlssonLogGeomandAlgStacks}, we can immediately deduce a number of nice results for the Olsson-log Hochschild homologies.

First of all, there is a logarithmic version of the HKR theorem, which is an immediate corollary of our previously established HKR theorem and the definition \ref{HHOls definition}.

\begin{theorem}\label{Log HKR for HHOls}
    Suppose that $f : X\to S$ is a map of fine log schemes, with associated map $f_{\cM_X} : \underline{X}\to\cLog_S.$ Then $\HHOls(X/S)$ and $\sHHOls(X/S)$ admit complete, descending, $\NN$-indexed filtrations, with associated graded pieces given by
    \[
    \gr_{\HKR}^i\HHOls(X/S)\simeq R\Gamma\left(\underline{X},\bigwedge^i \LOls_{X/S}\right)[i]
    \]
    and
    \[
    \gr_{\HKR}^i\sHHOls(X/S)\simeq\left(\bigwedge^i \LOls_{X/S}\right)[i].
    \]
\end{theorem}

\begin{lemma}\label{integral log etale invariance on the source for HHOls}
    Let $(k,P)\to (A,M)\to (B,N)$ be morphisms of fs semilog rings such that the induced map $f : \Spec(B,N)\to\Spec(A,M)$ is integral log-\'etale. Then
    \[
    f^*\sHHOls(\Spec(A,M)/\Spec(k,P))\to\sHHOls(\Spec(B,N)/\Spec(k,P))
    \]
    and
    \[
    \HHOls(\Spec(A,M)/\Spec(k,P))\otimes_A B\to\HHOls(\Spec(B,N)/\Spec(k,P))
    \]
    are equivalences.
\end{lemma}
\begin{proof}
    The proof is the same in both cases, the only difference being whether the objects are considered as $B$-modules or as their underlying $k$-modules.

    By \cite[example 3.15]{Olssoncotangent}, the map on the first associated graded piece sits in a distinguished triangle/cofiber sequence
    \[
    \LOls_{\Spec(A,M)/\Spec(k,P)}\otimes_A B\to\LOls_{\Spec(B,N)/\Spec(k,P)}\to\LOls_{\Spec(B,N)/\Spec(A,M)}\xrightarrow{+1}.
    \]
    As $\Spec(B,N)\to\Spec(A,M)$ is log-\'etale, it follows that $\LOls_{\Spec(B,N)/\Spec(A,M )}\simeq 0.$ This proves that the morphism induces an equivalence on the first associated graded piece, and thus that the morphism induces an equivalence on all associated graded pieces, so that the morphism is an equivalence by completeness of the HKR filtration.
\end{proof}

The proof of the following lemma is similar.
\begin{lemma}\label{integral log etale invariance on the base for HHOls}
    Let $(k,P)\to (A,M)\to (B,N)$ be morphisms of fs semilog rings such that the induced map $\Spec(A,M)\to\Spec(k,P)$ is integral log-\'etale. Then
    \[
    \sHHOls(\Spec(B,N)/\Spec(k,P))\to\sHHOls(\Spec(B,N)/\Spec(A,M))
    \]
    and
    \[
    \HHOls(\Spec(B,N)/\Spec(k,P))\to\HHOls(\Spec(B,N)/\Spec(A,M))
    \]
    are equivalences.
\end{lemma}

\begin{proposition}
    The functors
    \[
    \HHOls,\LOls : \Arr(\LogAff)^{\op}\to\DF(\ZZ)
    \]
    are integral log-\'etale sheaves.
\end{proposition}
\begin{proof}
    The proof of this result for $\HHOls$ is identical to the proof of \cite[Proposition 7.11]{BLPO_HKR}, using \ref{integral log etale invariance on the base for HHOls} in place of \cite[Lemma 7.9]{BLPO_HKR} and \ref{integral log etale invariance on the source for HHOls} in place of \cite[Lemma 7.10]{BLPO_HKR}.

    The proof for $\LOls$ is similar. Consider an integral log \'etale cover of $X\to S$ determined by the diagram
    \[\begin{tikzcd}
        X'\ar[d]\ar[r,"\underline{g}"] &X\ar[d]\\
        S'\ar[r] &S.
    \end{tikzcd}\]
    This corresponds to a diagram
    \[\begin{tikzcd}
        \underline{X'}\ar[d]\ar[r,"\underline{g}"] &\underline{X}\ar[d]\\
        \cLog_{S'}\ar[r] &\cLog_S
    \end{tikzcd}\]
    in which $\underline{X'}\to\underline{X}$ is fpqc by \cite[Theorem~IV.4.3.5]{OgusLectures}, and $\cLog_{S'}\to\cLog_S$ is \'etale by \cite[Theorem~4.6]{OlssonLogGeomandAlgStacks}. We are then reduced to considering the canonical map
    \[
    \LL_{X/\cLog_S}\to\lim_{i}\LL_{\Cech(X'\to X)_i/\cLog_S},
    \]
    which is an equivalence by fpqc descent for the ordinary cotangent complex.
\end{proof}

\subsection{After Rognes}

The original source is \cite{Rognes_2009}. However, our exposition follows \cite{BLPO_HKR}, which uses the same language as this paper. See \cite[5.5]{BLPO_HKR} for a proof that the two constructions agree.

\begin{mydef}\label{BLPO HHRog definition}
    Let $(f,f^\flat) : (k,P)\to (A,M)$ be a morphism of animated prelog rings. The \emph{Rognes-logarithmic Hochschild homology of $(f,f^\flat)$} is defined to be the animated ring
    \[
    \HHRog(f,f^\flat) := A\otimes_{\underline{((A,M)\otimes_(k,P) (A,M))^{\rep}_{\sigma}}} A,
    \]
    where $\sigma : M\oplus_P M\to M$ is the codiagonal morphism of $P\to M.$ We will also write $\HHRog((A,M)/(k,P)),$ if the morphism $(f,f^\flat)$ is understood.
\end{mydef}

\begin{remark}
    From our previous description of the repletion, it follows that for $(f,f^\flat) : (k,P)\to (A,M),$ we have
    \[
    \HHRog(f,f^\flat)\simeq A\otimes_{(A\otimes_k A)\otimes_{\ZZ[M\oplus_P M]}\ZZ[(M\oplus_P M)^{\rep}_{\sigma}]}A.
    \]
\end{remark}

One of the main results of \cite{BLPO_HKR} provides an HKR theorem for the Rognes-log Hochschild homology.

\begin{theorem}[\protect{\cite[Theorem~5.15]{BLPO_HKR}}]\label{HHRog HKR}
    Let $(k,P)\to (A,M)$ be a map of animated prelog rings. Then the Rognes-log Hochschild homology $\HHRog((A,M)/(k,P))$ admits a complete, descending filtration $\FilHHRog{*}((A,M)/(k,P))$ with graded pieces
    \[
    \gr^n\FilHHRog{*}((A,M)/(k,P))\simeq\left(\bigwedge^n_A\LGab_{(A,M)/(k,P)}\right)[n].
    \]
\end{theorem}

\begin{lemma}
    Let $(k,P)$ be an animated prelog ring, and let $f : k\to A$ be a morphism of animated rings. There is a canonical equivalence of filtered $\ZZ$-modules
    \[
    \FilHHRog{*}((A,f^*P)/(k,P))\simeq\FilHH{*}(A/k).
    \]
\end{lemma}
\begin{proof}
    The equivalence between the underlying objects is obtained from the equivalence
    \[
    \HHRog((A,f^*P)/(k,P))\simeq A\otimes_{(A\otimes_k A)\otimes_{\ZZ[f^*P\oplus_P f^*P]}\ZZ[(f^*P\oplus_P f^*P)^{\rep}_\sigma]}A,
    \]
    noting that $f^*P\oplus_P f^*P\simeq P\oplus_P P\simeq P,$ and that under this equivalence, $\sigma : f^*P\oplus_P f^*P\to f^*P$ is the identity, so that the repletion is just $f^*P.$

    Compatibility of the filtrations comes from the observation that the subcategory of $(\aPreLogPoly_{(k,P)})_{/(A,f^*P)}$ spanned by maps of animated prelog rings of the form
    \[
    (P\to k)\otimes_{(\{1\}\to k)} (\{1\}\to k[X])\simeq (P\to k[X])\to(f^*P\to A)
    \]
    (where $X$ is some finite set) is final in $(\aPreLogPoly_{(k,P)})_{/(A,f^*P)}.$
\end{proof}

\begin{proposition}\label{HHR logification invariance}
    Let $(f,f^\flat) : (k,P)\to (A,M)$ be a morphism of animated prelog rings. Then the diagram
    \[\begin{tikzcd}
        (A,M)\ar[r,"a_{(A,M)}"] &(A,M)^a\ar[r] & (A,M)^a\\
        (k,P)\ar[u,"{(f,f^\flat)}"]\ar[r] &(k,P)\ar[u,"a_{(A,M)}\circ {(f,f^\flat)}"]\ar[r,"a_{(k,P)}"] &(k,P)^a\ar[u,"{(f,f^\flat)^a}"'],
    \end{tikzcd}\]
    where $a_{(k,P)}$ denotes the canonical logification map (and similarly for $a_{(A,M)}$) and the unlabeled horizontal maps are given by the relevant identity morphism induces filtered equivalences
    \begin{align*}
        \FilHHRog{*}((A,M)/(k,P))&\xrightarrow{\sim}\FilHHRog{*}((A,M)^a/(k,P))\\
        &\xrightarrow{\sim}\FilHHRog{*}((A,M)^a/(k,P)^a).
    \end{align*}
\end{proposition}
\begin{proof}
    On the associated graded pieces, we obtain the equivalences
    \[
    \LGab_{(A,M)/(k,P)}\to\LGab_{(A,M)^a/(k,P)}\to\LGab_{(A,M)^a/(k,P)^a}
    \]
    of \ref{logification invariance of Gabber log cotangent complex} (and their shifted exterior powers). Because the HKR filtration on $\HHRog$ is complete, this implies that the desired maps are equivalences as well.
\end{proof}

As an application of the above lemmas, we obtain the following corollary.

\begin{cor}
    Let $(k,P)\to (A,M)$ be a strict morphism of animated prelog rings. Then there is a canonical equivalence
    \[
    \FilHHRog{*}((A,M)/(k,P))\simeq\FilHH{*}(A/k).
    \]
\end{cor}

We also have \'etale base change and \'etale invariance results for $\HHRog.$

\begin{lemma}[\protect{\cite[Lemma~7.9]{BLPO_HKR}}]
    Let $(k,P)\xrightarrow{(f,f^\flat)} (A,M)\to (B,N)$ be morphisms of fs semilog rings. If $(f,f^\flat)$ is integral log \'etale, then the canonical map
    \[
    \HHRog((B,N)/(k,P))\xrightarrow{\sim}\HHRog((B,N)/(A,M))
    \]
    is an equivalence.
\end{lemma}

\begin{lemma}[\protect{\cite[Lemma~7.10]{BLPO_HKR}}]
    Let $(k,P)\to (A,M)\xrightarrow{(g,g^{\flat})} (B,N)$ be morphisms of fs semilog rings. If $(g,g^\flat)$ is integral log \'etale, then the canonical map
    \[
    B\otimes_A\HHRog((A,M)/(k,P))\to\HHRog((B,N)/(k,P))
    \]
    is an equivalence.
\end{lemma}

As a corollary of the above base change and invariance results, the authors of \cite{BLPO_HKR} conclude the following descent results for $\HHRog$ and $\LGab.$

\begin{proposition}[\protect{\cite[Proposition~7.11]{BLPO_HKR}}]
    The presheaf $\HHRog : \Arr(\LogAff)\to\cMod_{\ZZ}$ satisfies integral log \'etale descent.
\end{proposition}

\begin{proposition}[\protect{\cite[Proposition~7.11]{BLPO_HKR}}]
    The presheaf $\LGab : \Arr(\LogAff)\to\cMod_{\ZZ}$ satisfies integral log-\'etale descent.
\end{proposition}

\subsection{Comparison}

We will now construct our comparison morphism between $\FilHHRog{*}$ and $\FilHHOls{*}.$

First, note that since $\FilHHRog{*}$ and $\FilHHOls{*}$ are integral log-\'etale sheaves, in order to construct a map $\FilHHRog{*}\to\FilHHOls{*}$ in $\Shv(\Arr(\LogSch),\DF(\ZZ)),$ it suffices to construct a map between their restrictions to $\Arr(\LogAff)\subseteq\Arr(\LogSch).$ Additionally, because $\FilHHRog{*}$ is by definition left Kan extended from the subcategory of ``polynomial extensions of polynomial pre-log rings'' -- i.e., maps of the form
\[
(\ZZ[X, \NN^{S}],\NN^{S})\to (\ZZ[X, Y,\NN^S\oplus\NN^T],\NN^{S}\oplus\NN^{T})
\]
with their obvious log structures (where $X,Y,S,$ and $T$ are finite sets\footnote{$\ZZ[X]\simeq \ZZ[x_1,\dots, x_{\#X}]$ is the free $\ZZ$-algebra on the set $X$ and $\ZZ[X,\NN^S]\simeq\ZZ[x_1,\dots, x_{\#X},s_1,\dots, s_{\#S}]$ is the monoid algebra over $\ZZ[X]$ generated by $\NN^S.$}) -- it suffices to construct a map between their restrictions to the category of maps of pre-log rings, via the logification map to $\Arr(\LogAff).$ Here, we use \ref{HHR logification invariance} to ensure that we can compute and compare $\HHRog$ using the prelog rings themselves, since this will make the computations simpler.

We have seen that if $(f,\theta) : (k,P)\to (A,M)$ is a map of fine prelog rings, the induced map $\Spec(M\to A)\to\cLog_{\Spec(P\to k)}$ factors through the \'etale map $\cS_{(f,\theta)}\to\cLog_{\Spec(P\to k)},$ and because $\cS_{(f,\theta)}\to\cLog_{\Spec(P\to k)}$ is \'etale, the map
\[
\FilHH{*}(\Spec A/\cLog_{\Spec(P\to k)})\to \FilHH{*}(\Spec A/\cS_{(f,\theta)})
\]
is an equivalence. Moreover, this factorization is functorial: given a commutative square of fine prelog rings
\[\begin{tikzcd}
    (A,M)\ar[r] &(A',M')\\
    (k,P)\ar[u,"{(f,\theta)}"]\ar[r] &(k',P')\ar[u,"{(f',\theta')}"],
\end{tikzcd}\]
there is an induced morphism $\cS_{(f',\theta')}\to\cS_{(f,\theta)}$ such that
\[\begin{tikzcd}
    \Spec A'\ar[r]\ar[d] &\cS_{(f',\theta')}\ar[r]\ar[d] &\cLog_{\Spec(k',P')}\ar[d]\\
    \Spec A\ar[r] &\cS_{(f,\theta)}\ar[r] &\cLog_{\Spec(k,P)}
\end{tikzcd}\]
commutes.

So it suffices to construct a map
\[
\FilHHRog{*}((A,M)/(k,P))\to\FilHH{*}(\Spec A/\cS_{(f,\theta)})
\]
which is functorial in the morphism $(k,P)\to (A,M).$ 

\begin{lemma}
    For a map $\Spec(A,M)\to\Spec(k,P)$ of the form 
    \[
    \Spec(\ZZ[X,Y,\NN^S\oplus\NN^T],\NN^S\oplus\NN^T)\to\Spec(\ZZ[X,\NN^S],\NN^S),
    \]
    we have
    \[
    \FilHHOls{*}(\Spec(A,M)/\Spec(k,P))\simeq\tau_{\geq *}\HH(\Spec A/\cS_{(f,\theta)}).
    \]
    I.e., the HKR filtration on Olsson log Hochschild homology of a polynomial extension of a free prelog ring is the Whitehead filtration.
\end{lemma}
\begin{proof}
    Note that the canonical map $\Spec A\to\cS_{(f,\theta)}$ is smooth (e.g., by \cite[corollary 5.31]{OlssonLogGeomandAlgStacks}). Thus, the result follows by combining \ref{HH base change results} and \ref{Flat cotangent complex implies postnikov filtration}.
\end{proof}

We are left with the task of calculating $\HH(\Spec A/\cS_{(f,\theta)})$ in the situation above, and to compare this to $\HHRog((A,M)/(k,P)).$ Because the (pre)log rings/schemes involved are so simple, we can actually calculate the (log) Hochschild homologies directly.

We begin with $\HHRog((A,M)/(k,P)).$ Let us consider first the case when the morphism of prelog rings is of the form
\[\begin{tikzcd}
    P\ar[r,"\theta"]\ar[d] & M\ar[d]\\
    k_0[P]\ar[r,"f"] &A_0[M],
\end{tikzcd}\]
where the log structures are given by the canonical maps from the monoid to its corresponding monoid ring, and the map $f : k_0[P]\to A_0[M]$ is induced by a map $f_0 : k_0\to A_0$ of animated rings together with $\theta.$

In this case, the repletion can be identified with $(A_0\otimes_{k_0} A_0)[(M\oplus_P M)^{\rep}_{\sigma}]$:
\begin{align*}
    \underline{(A,M)\otimes^{\rep}_{(k,P)}(A,M)}&\simeq (A_0[M]\otimes_{k_0[P]}A_0[M])\otimes_{\ZZ[M\oplus_P M]}\ZZ[(M\oplus_P M)^{\rep}_{\sigma}]\\
    &\simeq \left((A_0\otimes_{k_0} A_0)\otimes_{\ZZ}\ZZ[M\oplus_P M]\right)\otimes_{\ZZ[M\oplus_P M]}\ZZ[(M\oplus_P M)^{\rep}_{\sigma}]\\
    &\simeq (A_0\otimes_{k_0} A_0)\otimes_{\ZZ}\ZZ[(M\oplus_P M)^{\rep}_{\sigma}]\\
    &\simeq (A_0\otimes_{k_0} A_0)[(M\oplus_P M)^{\rep}_{\sigma}].
\end{align*}
Let $M_1 = M_2 = M.$ Let
\[\begin{tikzcd}
    P\ar[r,"\theta"]\ar[d,"\theta"] &M_2\ar[d,"\iota_2"]\\
    M_1\ar[r,"\iota_1"] &M_1\oplus_P M_2
\end{tikzcd}\]
be the coCartesian square of animated monoids defining $M_1\oplus_P M_2 = M\oplus_P M$ together with the canonical morphisms $\iota_1,\iota_2 : M\to M\oplus_P M.$ By \ref{repletion with section identification}, we obtain functorial equivalences
\[
M\oplus\cofib(\iota_1)\simeq (M_1\oplus_P M_2)^{\rep}_{\sigma}\simeq M\oplus\cofib(\iota_2),
\]
and either of these give us an identification
\[
(A,M)\otimes^{\rep}_{(k,P)}(A,M)\simeq (A_0\otimes_{k_0} A_0)[(M\oplus_P M)^{\rep}_{\sigma}]\simeq (A_0\otimes_{k_0} A_0)[M\oplus\cofib(\iota_i)].
\]
Suppose moreover $P\to M$ is given by the canonical morphism $P\to P\oplus Q$ for some animated monoid $Q,$ and let $P_1 = P_2 = P$ and $Q_1 = Q_2 = Q,$ so that we may write $M_1 = P_1\oplus Q_1$ and $M_2 = P_2\oplus Q_2.$ If we fix $i\in \{1,2\}$ and let $i'$ denote the other element, we then have
\[
(A,M)\otimes^{\rep}_{(k,P)}(A,M)\simeq (A_0\otimes_{k_0} A_0)[(P\oplus Q)\oplus\cofib(\iota_i)]\simeq (A_0\otimes_{k_0} A_0)[(P\oplus Q)\oplus Q_{i'}].
\]

Next, we calculate $\HH(\Spec A/\cS_{(f,\theta)}).$ Because $\cS_{(f,\theta)}$ has affine diagonal,
\[
\HH(\Spec A/\cS_{(f,\theta)})\simeq R\Gamma(\Spec A,L\Delta^*R\Delta_*\cO_{\Spec A})
\]
by \ref{relative HH is self intersection of diagonal for relative derived scheme}, where
\[
\Delta : \Spec A\to\Spec A\times_{\cS_{(f,\theta)}}\Spec A
\]
is the diagonal map. In fact, since $\cS_{(f,\theta)}$ has affine diagonal, the fiber product
\[
\Spec A\times_{\cS_{(f,\theta)}}\Spec A\simeq\Spec B
\]
is affine, and we can describe it relatively explicitly. Note that since $\Spec A\to\cS_{(f,\theta)}$ is smooth, the ordinary fiber product of stacks will agree with the derived fiber product.

In general, we have that
\[
\Spec A\times_{\cS_{(f,\theta)}}\Spec A\simeq\lim\left(\begin{tikzcd}
    \Spec A\ar[r]\ar[d] &\cS_{M}\ar[d] &\Spec A\ar[l]\ar[d]\\
    \Spec A\ar[r] &\cS_P &\Spec A\ar[l]\\
    \Spec A\ar[r]\ar[u] &\Spec k\ar[u] &\Spec A\ar[l]\ar[u]
\end{tikzcd}\right)
\]

The morphisms $\Spec A\to\cS_M$ and $\Spec A\to\cS_P$ factor as
\[
\Spec A\xrightarrow{\Spec k[\alpha]}\Spec k[M]\xrightarrow{\pi_M}\cS_{M}
\]
and
\[
\Spec A\xrightarrow{f}\Spec k\xrightarrow{\Spec k[\beta]}\Spec k[P]\xrightarrow{\pi_P}\cS_P
\]
respectively, where $f : \Spec A\to\Spec k$ is the map induced by $k\to A,$ $k[\beta] : k[P]\to k$ and $k[\alpha] : k[M]\to A$ are the maps corresponding to the prelog structures $\beta : P\to k$ and $\alpha : M\to A$ under the adjunction
\[\begin{tikzcd}
\CMon\arrow[rr,"{k[-]}",shift left]&&\CRing_k,\arrow[ll,"{(-,\,\,\cdot\,\,)}",shift left]
\end{tikzcd}\]
and $\pi_M$ and $\pi_P$ are the canonical quotient maps. For the sake of brevity, let $b = \Spec k[\beta]$ and $a = \Spec k[\alpha].$ Similarly, let $t : \Spec k[M]\to\Spec k[P]$ be the map corresponding to the map $\theta : P\to M$ of monoids.

Thus, we may realize $\Spec A\times_{\cS_{(f,\theta)}}\Spec A$ as the limit of the diagram
\begin{equation}\label{diagram 2}\begin{tikzcd}
    \Spec A\arrow[d,"\id_A"']\arrow[r,"a"]&{\Spec k[M]}\arrow[r,"\pi_M"]\arrow[d,"t"]&\cS_M\arrow[d]&{\Spec k[M]}\arrow[l,"\pi_M"']\arrow[d,"t"]&\Spec A\arrow[l,"a"']\arrow[d,"\id_A"]\\
    \Spec A\arrow[r,"b\circ f"']&{\Spec k[P]}\arrow[r,"\pi_P"']&\cS_P&{\Spec k[P]}\arrow[l,"\pi_P"]&\Spec A\arrow[l,"b\circ f"]\\
    \Spec A\arrow[u,"\id_A"]\arrow[r,"f"']&\Spec k\arrow[u,"b"]\arrow[r,"\id_k"']&\Spec k\arrow[u]&\Spec k\arrow[u,"b"']\arrow[l,"\id_k"]&\Spec A\arrow[u,"\id_A"']\arrow[l,"f"]
\end{tikzcd}\end{equation}
Recall that we can identify $\Spec k[M]\times_{\cS_M}\Spec k[M]$ with $\Spec k[M]\times_{\Spec k}\Spec k[M^{\gp}]$ where the projection maps then correspond to
\[\begin{tikzcd}
    \Spec k[M]\times_{\Spec k}\Spec k[M^{\gp}]\ar[d,"\pr_M"]\ar[r,"\mathrm{act}_M"] &\Spec k[M]\ar[d]\\
    \Spec k[M]\ar[r] &\cS_M
\end{tikzcd}\]
with $\mathrm{act}_M : \Spec k[M]\times_{\Spec k}\Spec k[M^{\gp}]\to\Spec k[M]$ is the action map and $\pr_M : \Spec k[M]\times_{\Spec k}\Spec k[M^{\gp}]\to\Spec k[M]$ is projection onto the first factor (and of course, $\Spec k[P]\times_{\cS_P}\Spec k[P]\simeq\Spec k[P]\times_{\Spec k}\Spec k[P^{\gp}]$ in a similar way).

Then we may replace diagram \ref{diagram 2} by
\[\begin{tikzcd}
    \Spec A\arrow[d,"\id_A"']\arrow[r,"a"]&{\Spec k[M]}\arrow[d,"t"]&&{\Spec k[M\oplus M^{\gp}]}\arrow[ll,"\pr_M"']\arrow[rr,"\mathrm{act}_M"]\arrow[d,"t\times t^{\gp}"]&&{\Spec k[M]}\arrow[d,"t"]&\Spec A\arrow[l,"a"']\arrow[d,"\id_A"]\\
    \Spec A\arrow[r,"b\circ f"']&{\Spec k[P]}&&{\Spec k[P\oplus P^{\gp}]}\arrow[ll,"\pr_P"]\arrow[rr,"\mathrm{act}_P"']&&{\Spec k[P]}&\Spec A\arrow[l,"b\circ f"]\\
    \Spec A\arrow[u,"\id_A"]\arrow[r,"f"']&\Spec k\arrow[u,"b"]&&\Spec k\arrow[u]\arrow[ll,"\id_k"]\arrow[rr,"\id_k"']&&\Spec k\arrow[u,"b"']&\Spec A\arrow[u,"\id_A"']\arrow[l,"f"]
\end{tikzcd}\]
as the limit of this diagram will be equivalent to $\Spec A\times_{\cS_{(f,\theta)}}\Spec A$ as well. Since all the schemes in this diagram are affine, we may pass to the coordinate rings and take the colimit of the diagram
\[\begin{tikzcd}
    A&{k[M]}\arrow[rr]\arrow[l,"{k[\alpha]}"']&&{k[M\oplus M^{\gp}]}&&{k[M]}\arrow[r,"{k[\alpha]}"]\arrow[ll,"{k[\delta_M]}"']&A\\
    A\arrow[u,"\id_A"]\arrow[d,"\id_A"']&{k[P]}\arrow[d,"{k[\beta]}"]\arrow[u,"{k[\theta]}"]\arrow[rr]\arrow[l,"{f[\beta]}"]&&{k[P\oplus P^{\gp}]}\arrow[u,"{k[\theta\oplus\theta^{\gp}]}"]\arrow[d,"{k[\beta]\otimes k[1]}"']&&{k[P]}\arrow[d,"{k[\beta]}"']\arrow[u,"{k[\theta]}"]\arrow[ll,"{k[\delta_P]}"]\arrow[r,"{f[\beta]}"']&A\arrow[u,"\id_A"']\arrow[d,"\id_A"]\\
    A&k\arrow[l,"f"]\arrow[rr,"\id_k"']&&k&&k\arrow[r,"f"']\arrow[ll,"\id_k"]&A
\end{tikzcd}\]
where $\delta_M : M\to M\oplus M^{\gp}$ is the diagonal map (and similarly for $\delta_P$), and $f[\beta] = f\circ k[\beta].$

We may further simplify this diagram to
\[\begin{tikzcd}
    {A[M^{\gp}]}&{k[M]}\arrow[r,"{k[\alpha]}"]\arrow[l]&A\\
    {A[P^{\gp}]}\arrow[u]\arrow[d]&{k[P]}\arrow[d,"{k[\beta]}"']\arrow[u,"{k[\theta]}"]\arrow[l]\arrow[r,"{f[\beta]}"']&A\arrow[u,"\id_A"']\arrow[d,"\id_A"]\\
    A&k\arrow[r,"f"']\arrow[l,"f"]&A.
\end{tikzcd}\]
It will be helpful to spell out the unlabeled maps in this new diagram. They are given by:
\begin{align*}
    k[M]&\to A[M^{\gp}]\\
    e_m&\mapsto \alpha(m) e_{[m]}\\
    \\
    k[P]&\to A[P^{\gp}]\\
    e_p&\mapsto \alpha(\theta(p)) e_{[\theta(p)]}\\
    \\
    A[P^{\gp}]&\to A[M^{\gp}]\\
    e_{g}&\mapsto e_{\theta^{\gp}(g)}\\
    \\
    A[P^{\gp}]&\to A\\
    e_g&\mapsto 1,
\end{align*}
and one readily verifies that all the squares are indeed commutative using the identities
$f\circ\beta = \alpha\circ\theta$ and $\theta^{\gp}\circ [-] = [-]\circ\theta$ (where $[-]$ denotes taking the image of an element in a monoid in its groupification).

Putting everything together, we find that $\Spec A\times_{\cS_{(f,\theta)}}\Spec A\simeq\Spec B_{(f,\theta)},$ where
\[
B_{(f,\theta)} := (A[M^{\gp}]\otimes_{A[P^{\gp}]} A)\otimes_{k[M]\otimes_{k[P]} k} A\cong \left(A\otimes_k (k[M^{\gp}]\otimes_{k[P^{\gp}]} k)\right)\otimes_{k[M]\otimes_{k[P]} k} A.
\]
Note specifically that $k[M]\otimes_{k[P]} k$ acts on $A$ and $A[M^{\gp}]\otimes_{A[P^{\gp}]}A$ via
\begin{align*}
    k[M]\otimes_{k[P]} k&\to A\otimes_k (k[M^{\gp}]\otimes_{k[P^{\gp}]} k)\\
    be_m\otimes b'&\mapsto \alpha(m)\otimes(b e_{[m]}\otimes b')
\end{align*}
and
\begin{align*}
    k[M]\otimes_{k[P]} k&\to A\\
    be_m\otimes b'&\mapsto f(bb')\alpha(m).
\end{align*}

We now endow $B_{(f,\theta)}$ with a prelog structure $\gamma_{(f,\theta)} : N_{(f,\theta)}\to B_{(f,\theta)}$ which is functorial in the morphism of prelog rings $(f,\theta) : (k,P)\to (A,M),$ and with appropriately defined morphisms, will provide a factorization
\[
(A,M)\otimes_{(k,P)} (A,M)\to (B_{(f,\theta)},N_{(f,\theta)})\to (A,M).
\]
of the multiplication-summation map $(\mu,\sigma) : (A,M)\otimes_{(k,P)} (A,M)\to (A,M).$ Set $N_{(f,\theta)} := M\oplus M^{\gp}/P^{\gp},$ and define
\begin{align*}
    \gamma_{(f,\theta)} : N_{(f,\theta)}&\to \left(A\otimes_k (k[M^{\gp}]\otimes_{k[P^{\gp}]} k)\right)\otimes_{k[M]\otimes_{k[P]} k} A\\
    (m,[g])&\mapsto (\alpha(m)\otimes(e_g\otimes 1))\otimes 1.
\end{align*}

\begin{lemma}
    The map
    \[
    \gamma_{(f,\theta)} : N_{(f,\theta)}\to B_{(f,\theta)}
    \]
    is a well-defined morphism of monoids.
\end{lemma}
\begin{proof}
    The only question is whether the map is well defined. To that end, let $[g] = [g']$ in $M^{\gp}/P^{\gp},$ so that $g - g'$ is in the image of $\theta^{\gp} : P^{\gp}\to M^{\gp}$ -- write $g - g' = \theta^{\gp}(h)$ for some $h\in P^{\gp}.$ Then we have
    \begin{align*}
        (\alpha(m)\otimes(e_{[g']}\otimes 1))\otimes 1 &= (\alpha(m)\otimes(e_{[g+\theta^{\gp}(h)]}\otimes 1))\otimes 1\\
        &= (\alpha(m)\otimes(e_{[g]}e_{[\theta^{\gp}(h)]}\otimes 1))\otimes 1\\
        &= (\alpha(m)\otimes(e_{[g]}\otimes 1))\otimes 1,
    \end{align*}
    where the last equality follows from the maps $k[P^{\gp}]\to k[M^{\gp}]$ and $k[P^{\gp}]\to k$ sending $e_h$ to $e_{[\theta^{\gp}(h)]}$ and $1$ respectively.
\end{proof}

\begin{remark}
    We do not know the precise relation between the prelog structure $\gamma_{(f,\theta)} : N_{(f,\theta)}\to B_{(f,\theta)}$ and the prelog structure given by forming the fiber product $\Spec(A,M)\times_{(\cS_{(f,\theta)},\cM_{\cS_{(f,\theta)}})}\Spec(A,M)$ in the category of prelog stacks, where $\cM_{\cS_{(f,\theta)}}$ is the canonical log structure on $\cS_{(f,\theta)}$ induced by the map $\cS_{(f,\theta)}\to\cLog_{\Spec(k,P)}.$
\end{remark}

Geometrically, $\Spec A\times_{\cS_{(f,\theta)}}\Spec A$ participates in a factorization of the diagonal of $\Spec A$ relative to $\Spec k:$
\[
\Spec A\to\Spec A\times_{\cS_{(f,\theta)}}\Spec A\to\Spec A\times_{\Spec k}\Spec A,
\]
where the first map is the relative diagonal and the second map is induced by the structure map $\cS_{(f,\theta)}\to\Spec k.$ Explicitly, the diagonal map $\Delta : \Spec A\to\Spec A\times_{\cS_{(f,\theta)}}\Spec A$ corresponds to the ring morphism
\begin{align*}
    \left(A\otimes_k (k[M^{\gp}]\otimes_{k[P^{\gp}]} k)\right)\otimes_{k[M]\otimes_{k[P]} k} A&\to A\\
    (a\otimes(be_g\otimes c))\otimes a'&\mapsto a f(bc) a',
\end{align*}
and the map  $\Spec A\times_{\cS_{(f,\theta)}}\Spec A\to\Spec A\otimes_k A$ corresponds to the ring morphism
\begin{align*}
    \psi_{(f,\theta)} : A\otimes_k A&\to \left(A\otimes_k (k[M^{\gp}]\otimes_{k[P^{\gp}]} k)\right)\otimes_{k[M]\otimes_{k[P]} k} A\\
    a\otimes a'&\mapsto (a\otimes(1\otimes 1))\otimes a'.
\end{align*}

With the prelog structure $N_{(f,\theta)}\to B_{(f,\theta)}$ defined above, both of these morphisms upgrade to morphisms of prelog rings
\[
(B_{(f,\theta)},N_{(f,\theta)})\to (A,M)
\]
and
\[
(A,M)\otimes_{(k,P)} (A,M)\to (B_{(f,\theta)}, N_{(f,\theta)}),
\]
if we define
$N_{(f,\theta)}\to M$ to be the projection and 
\begin{align*}
    \rho_{(f,\theta)} : M\oplus_P M&\to M\oplus M^{\gp}/P^{\gp}\\
    [m,m']&\mapsto (m+m',[m']).
\end{align*}

We check that 
\begin{align*}
    \psi_{(f,\theta)}\circ (\alpha\otimes\alpha)([m,m']) &= \psi_{(f,\theta)}(\alpha(m)\otimes\alpha(m'))\\
    &= (\alpha(m)\otimes (1\otimes1))\otimes\alpha(m')
\end{align*}
and
\begin{align*}
    \gamma_{(f,\theta)}\circ\rho_{(f,\theta)}([m,m']) &= \gamma_{(f,\theta)}(m+m',[m'])\\
    &= (\alpha(m+m')\otimes(e_{[m']}\otimes1))\otimes1\\
    &= (\alpha(m)\alpha(m') \otimes(e_{[m']}\otimes 1))\otimes 1
\end{align*}

These are indeed equal, as we can verify by checking the diagram of structure maps:
\[\begin{tikzcd}
    k[M]\otimes_{k[P]} k\ar[rr]\ar[dd] &&A\otimes_k (k[M^{\gp}]\otimes_{k[P^{\gp}]} k)\\
    &e_{[m']}\otimes 1\ar[r, maps to]\ar[d, maps to] &\alpha(m')\otimes (e_{[m']}\otimes 1)\\
    A &\alpha(m')
\end{tikzcd}\]

Thus, we have a factorization
\[\begin{tikzcd}
    (A\otimes^{L}_k A, M\oplus_P M)\ar[r] &(B_{(f,\theta)},N_{(f,\theta)})\ar[r] &(A,M)
\end{tikzcd}\]
and it is clear that the second map is replete (when considered as a map of animated prelog rings). It follows from the universal property of $(A,M)\otimes^{\rep}_{(k,P)}(A,M)$ that there exists a unique (up to contractible choice) map $(A,M)\otimes^{\rep}_{(k,P)}(A,M)\to (B_{(f,\theta)},N_{(f,\theta)})$ of animated prelog rings making the diagram
\[\begin{tikzcd}
    (A,M)\otimes^L_{(k,P)}(A,M)\ar[r]\ar[d] &(A,M)\otimes^{\rep}_{(k,P)}(A,M)\ar[r]\ar[d] &(A,M)\ar[d,"\id"]\\
    (A,M)\otimes^{\cl}_{(k,P)}(A,M)\ar[r] &(B_{(f,\theta)},N_{(f,\theta)})\ar[r] &(A,M)
\end{tikzcd}\]
commute, where the leftmost and rightmost vertical maps are the $0$-truncation maps.

It follows immediately that there is a canonical map
\[
\HHRog((A,M)/(k,P))\to \HHOls((A,M)/(k,P))
\]
(note that $A\otimes^L_{B_{(f,\theta)}} A\simeq\HHOls((A,M)/(k,P))$), functorially in the morphism $(f,\theta)$ of prelog rings. Sheafification implies that we obtain a map
\[
\HHRog\to\HHOls
\]
of functors on $\Arr(\LogAff).$

In fact, this is sufficient to obtain the desired comparison map
\[
c : \FilHHRog{*}\to\FilHHOls{*},
\]
as it suffices to define $\left. c\right|_{\PolyPreLog}$ by universal property of the left Kan extension defining $\FilHHRog{}.$ In this situation, both $\FilHHRog{}(f)$ and $\FilHHOls{*}$ have the Whitehead filtration, which is functorial, and so the map $\HHRog\to\HHOls$ naturally upgrades to
\[
c : \FilHHRog{*}\to\FilHHOls{*}
\]

Before we state the main result, let us examine the map
\[
(A,M)\otimes^{\rep}_{(k,P)}(A,M)\to (B_{(f,\theta)},N_{(f,\theta)})
\]
in more detail in the case where $(f,\theta) : (k,P)\to (A,M)$ is of the form $(\ZZ[X,P],P)\to (\ZZ[X,Y,P\oplus Q],P\oplus Q),$ where $X$ and $Y$ are finite sets and $P$ and $Q$ are finitely generated free monoids.

We have already seen that
\[
(A,M)\otimes^{\rep}_{(k,P)}(A,M)\simeq (A_0\otimes_{k_0} A_0)[(M\oplus_P M)^{\rep}_{\sigma}]\simeq (A_0\otimes_{k_0} A_0)[M\oplus\cofib(\iota_i)].
\]
When we take $(k,P) = (\ZZ[X,P],P)$ and $(A,M) = (\ZZ[X,Y,P\oplus Q],P\oplus Q),$ this becomes
\[
\underline{(A,M)\otimes^{\rep}_{(k,P)}(A,M)}\simeq (\ZZ[X,Y]\otimes_{\ZZ[X]} \ZZ[X,Y])[M\oplus Q^{\gp}].
\]
In this case, the derived constructions in the repletion simplify to the classical constructions due to flatness/freeness of the rings and monoids. Then we must in fact have a commutative diagram
\[\begin{tikzcd}
    (A,M)\otimes^{\cl}_{(k,P)}(A,M)\ar[r]\ar[d,"\id"] &(A,M)\otimes^{\rep}_{(k,P)}(A,M)\ar[r]\ar[d] &(A,M)\ar[d,"\id"]\\
    (A,M)\otimes^{\cl}_{(k,P)}(A,M)\ar[r] &(B_{(f,\theta)},N_{(f,\theta)})\ar[r] &(A,M)
\end{tikzcd}\]
which implies that the canonical map $(A,M)\otimes^{\cl}_{(k,P)}(A,M)\to (A,M)$ factors as
\[
(A,M)\otimes^{\cl}_{(k,P)}(A,M)\to (A,M)\otimes^{\rep}_{(k,P)}(A,M)\to (B_{(f,\theta)},N_{(f,\theta)})\to(A,M).
\]

\[\begin{tikzcd}
    M\oplus_P M\ar[r]\ar[d] &\ZZ[X,Y,P\oplus Q]\otimes_{\ZZ[X,P]} \ZZ[X,Y,P\oplus Q]\ar[d]\\
    M\oplus M^{\gp}/P^{\gp}\ar[r]\ar[d] &(\ZZ[X,Y]\otimes_{\ZZ[X]} \ZZ[X,Y])[M\oplus Q^{\gp}]\ar[d]\\
    M\oplus M^{\gp}/P^{\gp}\ar[r]\ar[d] &B_{(f,\theta)}\ar[d]\\
    M\ar[r] &A
\end{tikzcd}\]
We have isomorphisms
\begin{align*}
    (P\oplus Q)\oplus_P (P\oplus Q)&\to P\oplus Q\oplus Q\\
    [(p_1,q_1),(p_2,q_2)]&\mapsto (p_1+p_2,q_1,q_2)
\end{align*}
\begin{align*}
    (P\oplus Q)\oplus (P\oplus Q)^{\gp}/P^{\gp}&\to P\oplus Q\oplus Q^{\gp}\\
    ((p_1,q_1),[p_2,q_2])&\mapsto (p_1,q_1,[q_2])
\end{align*}
\[
\ZZ[X,Y,P\oplus Q]\otimes_{\ZZ[X,P]}\ZZ[X,Y,P\oplus Q]\cong\ZZ[X,Y,Y,P\oplus Q\oplus Q]
\]
\[
(\ZZ[X,Y]\otimes_{\ZZ[X]} \ZZ[X,Y])[M\oplus Q^{\gp}]\cong \ZZ[X,Y,Y,P\oplus Q\oplus Q^{\gp}]
\]

Examining the individual components of $B_{(f,\theta)},$ we have
\begin{align*}
    A\otimes_k ( k[M^{\gp}]\otimes_{k[P^{\gp}]} k)&\cong \ZZ[X,Y,P\oplus Q]\otimes_{\ZZ[X,P]}(\ZZ[X,P,P^{\gp}\oplus Q^{\gp}]\otimes_{\ZZ[X,P,P^{\gp}]}\ZZ[X,P])\\
    &\cong \ZZ[X,Y,P\oplus Q]\otimes_{\ZZ[X,P]}(\ZZ[Q^{\gp}]\otimes \ZZ[X,P,P^{\gp}]\otimes_{\ZZ[X,P,P^{\gp}]}\ZZ[X,P])\\
    &\cong \ZZ[X,Y,P\oplus Q]\otimes_{\ZZ[X,P]} (\ZZ[Q^{\gp}]\otimes\ZZ[X,P])\\
    &\cong \ZZ[Y,Q]\otimes\ZZ[X,P]\otimes_{\ZZ[X,P]} (\ZZ[Q^{\gp}]\otimes\ZZ[X,P])\\
    &\cong \ZZ[Y,Q]\otimes\ZZ[Q^{\gp}]\otimes\ZZ[X,P]\\
    &\cong\ZZ[X,Y,P\oplus Q\oplus Q^{\gp}],
\end{align*}

and

\begin{align*}
    k[M]\otimes_{k[P]} k&\cong \ZZ[X,P,P\oplus Q]\otimes_{\ZZ[X,P,P]}\ZZ[X,P]\\
    &\cong \ZZ[Q]\otimes\ZZ[X,P,P]\otimes_{\ZZ[X,P,P]}\ZZ[X,P]\\
    &\cong\ZZ[X,P\oplus Q]
\end{align*}

Thus, we have

\begin{align*}
    B_{(f,\theta)} &\cong \ZZ[X,Y,P\oplus Q\oplus Q^{\gp}]\otimes_{\ZZ[X,P\oplus Q]}\ZZ[X,Y,P\oplus Q]\\
    &\cong \ZZ[X,Y,P\oplus Q\oplus Q^{\gp}]\otimes_{\ZZ[X,P\oplus Q]}\ZZ[X,P\oplus Q]\otimes\ZZ[Y]\\
    &\cong\ZZ[X,Y,Y,P\oplus Q\oplus Q^{\gp}]\\
    &\cong\left(\ZZ[X,Y]\otimes_{\ZZ[X]}\ZZ[X,Y]\right)[M\oplus Q^{\gp}].
\end{align*}

In particular, $\underline{(A,M)\otimes^{\rep}_{(k,P)}(A,M)}$ and $B_{(f,\theta)}$ are abstractly isomorphic in this situation. In fact, a tedious but not complicated tracing of the isomorphisms shows that the diagram
\[\begin{tikzcd}
    (A,M)\otimes^{\rep}_{(k,P)}(A,M)\ar[r]\ar[d] &\left((\ZZ[X,Y]\otimes_{\ZZ[X]}\ZZ[X,Y])[M\oplus Q^{\gp}],M\oplus Q^{\gp}\right)\ar[d,"\id"]\\
    (B_{(f,\theta)},N_{(f,\theta)})\ar[r] &\left((\ZZ[X,Y]\otimes_{\ZZ[X]}\ZZ[X,Y])[M\oplus Q^{\gp}],M\oplus Q^{\gp}\right)
\end{tikzcd}\]
commutes, where the horizontal arrows are the equivalences previously described.

We are now ready to prove the main theorem.


\begin{theorem}\label{Main Theorem}
    Let $f : X\to S$ be a map of Noetherian fs log schemes of finite Krull dimension. There exists a morphism
    \[
    c_f : \FilHHRog{*}(f)\to\FilHHOls{*}(f)
    \]
    in $\DF(\ZZ),$ functorial in $f,$ such that the induced morphisms on associated graded pieces are
    \[
    \gr^i c_f\simeq R\Gamma(X,\wedge^i\eta_f[i]) : R\Gamma(X,\left(\bigwedge^i\LGab_{f}\right)[i])\to R\Gamma(X,\left(\bigwedge^i\LOls_{f}\right)[i]),
    \]
    where $\eta_f : \LGab_f\to\LOls_f$ is the comparison morphism of \cite[8.31]{Olssoncotangent}.
\end{theorem}

\begin{proof}
    We have shown the existence of a comparison morphism
    \[
    c_f : \FilHHRog{*}(f)\to\FilHHOls{*}(f)
    \]
    in our work above: it is simply the canonical map induced by the universal property of the left Kan extension defining $\FilHHRog{*},$ extended to the global case via sheafification.

    It follows that the induced morphism $\gr^1 c_f : \LGab_f\to\LOls_f$ is also induced by the universal property of the left Kan extension and sheafification. We must check that the map defined in \cite[8.31]{Olssoncotangent} is equivalent to this map.

    As $\LGab$ is left Kan extended from $\PolyPreLog,$ we only need to verify the relevant map on morphisms of the form
    \[
    f : (\ZZ[X,P],P)\to (\ZZ[X,Y,P\oplus Q],P\oplus Q),
    \]
    where $X,Y$ are finite sets, and $P$ and $Q$ are finitely generated free monoids (and their associated logifications). In fact, as such a map is equivalent to the base change of a map of the form
    \[
    (\ZZ,\{1\})\to(\ZZ[Y,Q],Q)
    \]
    along the map $(\ZZ,\{1\})\to(\ZZ[X,P],P),$ we may reduce to the case when the base is the initial prelog ring and the other prelog ring $(\ZZ[Y,Q],Q)$ is free.
    
    Following the proof of \cite[Theorem~8.32]{Olssoncotangent}, we see that the map $\LGab_f\to\LOls_{f^a}$ is induced by the diagram
    \[\begin{tikzcd}
        \LL_{\ZZ[Y,Q]}\ar[d] &\ar[l] \ZZ[Y, Q]\otimes_{\ZZ[Q]}\LL_{\ZZ[Q]}\ar[r]\ar[d] & \ZZ[Y,Q]\otimes (Q^{\gp})\ar[d]\\
        \LL_{\ZZ[Y, Q]} &\ar[l] \ZZ[Y,Q]\otimes_{\ZZ[Q]}\LL_{\ZZ[Q]}\ar[r] &\ZZ[Y,Q]\otimes (Q^{\gp})
    \end{tikzcd}\]
    as $\LGab_f$ is the homotopy pushout of the first row and $\LOls_f$ is the homotopy pushout of the second. These maps are all identities, and so the map $\LGab_f\to\LOls_{f^a}$ is the identity (or canonical equivalence)
    \[
    \Omega^1_f\to\Omega^1_{f^a}
    \]
    when $f$ is of the form
    \[
    f : (\ZZ[X,P],P)\to (\ZZ[X,Y,P\oplus Q],P\oplus Q),
    \]
    as desired (note that $(\ZZ,\{1\})\to (\ZZ[Y,Q],Q)$ is integral log smooth, thus derived log smooth by \ref{integral + log smooth/et implies derived log smooth/et}, and so both cotangent complexes reduce to the log differentials).

    Finally, note that the map on the $i$th graded piece is given by the $i$th exterior power of the map on the first graded piece.
\end{proof}

\begin{cor}
    Let $f : X\to S$ be a map of Noetherian fs log schemes of finite Krull dimension. Then the comparison map
    \[
    c_f : \FilHHRog{*}(f)\to\FilHHOls{*}(f)
    \]
    is an equivalence if $\eta_f : \LGab_f\to\LOls_f$ is an equivalence. In particular, $c_f$ is an equivalence if $f$ is integral or if $S$ is log flat over $\Spec\ZZ$ with the trivial log structure.
\end{cor}
\begin{proof}
    This follows from \cite[Corollary 8.34]{Olssoncotangent}, combined with \ref{HHRog HKR} and the fact that the HKR filtration is complete.
\end{proof}

We conclude with the observation that the above implies that the dividing Nisnevich sheafification of the two log Hochschild homologies agree.

\begin{cor}
    Let $f : X\to S$ be a map of Noetherian fs log schemes of finite Krull dimension. The comparison morphism $c_f$ of \ref{Main Theorem} becomes an equivalence after dividing Nisnevich sheafification.
\end{cor}

\begin{proof}
    This follows immediately from \ref{Main Theorem} and the discussion in \cite[remark 7.29]{BLPO_HKR}.
\end{proof}

\printbibliography

\appendix
\chapter{Animations}

\label{AppendixA}

The theories of derived algebraic geometry and derived logarithmic geometry are most concisely stated in the language of animated rings and animated monoids. Since our paper uses animation of various categories liberally, we recall here the basic definitions and main results that we make use of. The original reference for these ideas (in this language) is \cite[Section~5]{cesnavicius2023purity}, but other useful references include \cite{yaylali2022notes} and \cite{mao2024revisitingderivedcrystallinecohomology}.

\begin{mydef}
    Let $\sfC$ be a cocomplete $1$-category, and let $X\in\sfC.$
    \begin{enumerate}
        \item We say that $X$ is \emph{of finite presentation} or \emph{compact} if the functor $\Hom_{\sfC}(X,-)$ commutes with filtered colimits, and we let $\sfC^{\fp}$ denote the full subcategory of compact objects.
        \item We say that $X$ is \emph{strongly of finite presentation} or \emph{compact projective} if $\Hom_{\sfC}(X,-)$ commutes with $1$-sifted colimits,\footnote{Recall that a small $1$-category $\sfD$ is \emph{$1$-sifted} if $\sfD$-indexed colimits commute with finite products in $\Set,$ and that a functor $F : \sfC\to\sfC'$ whose domain has finite colimits commutes with $1$-sifted colimits if and only if it commutes with filtered colimits and reflexive coequalizers.} and we let $\sfC^{\sfp}$ denote the full subcategory of $\sfC$ spanned by compact projective objects.
    \end{enumerate}
\end{mydef}

\begin{example}
    The following categories are generated under colimits by $\sfC^{\sfp}:$
    \begin{enumerate}
        \item The category $\Set$ of sets -- in this case, $\Set^{\sfp}$ consists of finite sets;
        \item The category $\Ab$ of abelian groups -- in this case, $\Ab^{\sfp}$ consists of finitely generated free abelian groups;
        \item The category $\CMon$ of commutative monoids -- in this case, $\CMon^{\sfp}$ consists of retracts of finitely generated free monoids;
        \item The category $\CRing$ of commutative rings -- in this case, $\CRing^{\sfp}$ consists of retracts of finite type polynomial $\ZZ$-algebras.
    \end{enumerate}
\end{example}

\begin{mydef}
    Let $\sfC$ be a cocomplete $1$-category which is generated under colimits by its full subcategory $\sfC^{\sfp}$ of compact projective objects. We say that a fully faithful functor $i : \sfC\to\cC$ from $\sfC$ to an $\infty$-category $\cC$ exhibits $\cC$ as an \emph{animation of $\sfC$} if the functor $i$ induces an equivalence
    \[
    \Fun_{\mathrm{sift}}(\cC,\cA)\xrightarrow{\sim}\Fun(\sfC^{\sfp},\cA)
    \]
    for any $\infty$-category $\cA$ with sifted colimits, where $\Fun_{\mathrm{sift}}$ denotes the full subcategory of functors that commute with sifted colimits (equivalently, functors that commute with filtered colimits and geometric realizations). We will call any such $\cC$ an animation of $\sfC^{\sfp},$ and denote any such $\cC$ by $\Ani(\sfC).$
\end{mydef}

In other words, $\Ani(\sfC)$ is freely generated by $\sfC^{\sfp}$ under sifted colimits (equivalently, filtered colimits and geometric realizations).

\begin{remark}
    If $\sfC$ is a (small) cocomplete $1$-category generated under colimits by $\sfC^{\sfp},$ an animation of $\sfC$ always exists: one presentation of $\Ani(C)$ is given by the Yoneda embedding
    \[
    i : \sfC^{\sfp}\to\Fun_{\Pi}((\sfC^{\sfp})^{\op},\Ani),
    \]
    where $\Ani$ denotes the $\infty$-category of $\infty$-groupoids (aka spaces aka anima), and $\Fun_{\Pi}$ denotes the full sub-$\infty$-category of $\Fun$ spanned by functors which preserve finite products.

    If one prefers to work with model categories, one can also model the animation of $\sfC$ as the underlying $\infty$-category of the model category $\Fun(\Delta^{\op},\sfC)$ of simplicial objects of $\sfC,$ where the model structure is the \emph{projective model structure}. 
\end{remark}

Animations also come with a canonical notion of truncations.

\begin{mydef}
    Let $\sfC$ be a cocomplete $1$-category which is generated under colimits by $\sfC^{\sfp}.$ We have seen that we may identify
    \[
    \Ani(\sfC)\simeq\Fun_{\Pi}((\sfC^{\sfp})^{\op},\Ani).
    \]
    For any integer $n,$ given any functor $F : (\sfC^{\sfp})^{\op}\to\Ani,$ we may compose with the truncation $\tau_{\leq n} : \Ani\to\Ani,$ which induces a functor
    \[
    \tau_{\leq n} : \Ani(\sfC)\to\Ani(\sfC).
    \]
    We call this functor the \emph{$n$-truncation functor} on $\Ani(\sfC).$ 
\end{mydef}

The $n$-truncation functor on $\Ani(\sfC)$ is left adjoint to the inclusion of the full sub-$\infty$-category $\Ani(\sfC)_{\leq n}$ spanned by the objects of $\Ani(\cC)$ which are $n$-truncated in the sense of \cite[5.5.6.1]{HTT}. In particular, $\sfC\to\Ani(\sfC)$ is identified with $\Ani(\sfC)_{\leq 0},$ and we have an adjunction
\[\begin{tikzcd}
    \tau_{\leq 0} = \pi_0 : \Ani(\cC)\ar[r,shift left] &\ar[l,shift left] \Ani(\sfC)_{\leq 0}\simeq\sfC
\end{tikzcd}\]

Suppose that $F : \sfC\to\sfD$ is a functor between cocomplete $1$-categories which preserves $1$-sifted colimits. Then $F$ induces an essentially unique functor
\[
\Ani(F) : \Ani(\sfC)\to\Ani(\sfD)
\]
which is called the \emph{animation} of $F$ or the \emph{nonabelian left derived functor} of $F$ which preserves sifted colimits and restricts to $\left.F\right|_{\sfC^{\sfp}} : \sfC^{\sfp}\to\sfD\to\Ani(\sfD).$

The animation of a functor also interacts well with truncations, in the sense that for $F : \sfC\to\sfD$ a functor between cocomplete $1$-categories which preserves $1$-sifted colimits, we have
\[
\pi_0\circ\Ani(F) = F\circ\pi_0.
\]

\begin{proposition}[\protect{\cite[Proposition~5.1.5]{cesnavicius2023purity}}]
    Let $F : \sfC\to\sfD$ and $G : \sfD\to\sfE$ be functors between cocomplete $1$-categories which are generated under colimits by their subcategories of compact projective objects which preserve $1$-sifted colimits.
    \begin{enumerate}
        \item There is a natural transformation $\Ani(G)\circ\Ani(F)\to\Ani(G\circ F).$
        \item If either $F(\sfC^{\sfp})\subseteq\Ind(\sfD^{\sfp})$ in $\sfD$ or $\Ani(G)(F(\sfC^{\sfp}))\subseteq\sfE$ in $\Ani(\sfE),$ then the natural transformation
        \[
        \Ani(G)\circ\Ani(F)\to\Ani(G\circ F)
        \]
        is an equivalence.
    \end{enumerate}
\end{proposition}

\begin{proposition}[\protect{\cite[Corollary~2.3]{mao2024revisitingderivedcrystallinecohomology}}]
    Let
    \[\begin{tikzcd}
        F : \sfC\ar[r,shift left] &\ar[l, shift left] \sfD : G
    \end{tikzcd}\]
    be an adjoint pair of functors between $1$-categories such that
    \begin{enumerate}
        \item $\sfD$ admits filtered colimits and reflexive coequalizers and $G$ preserves filtered colimits and reflexive coequalizers;
        \item $\sfC$ is projectively generated;
        \item $G$ is conservative.
    \end{enumerate}
    Then $\sfD$ is projectively generated, and animation produces an adjunction
    \[\begin{tikzcd}
        \Ani(F) : \Ani(\sfC)\ar[r, shift left] &\ar[l,shift left]\Ani(\sfD) : \Ani(G)
    \end{tikzcd}\]
    of $\infty$-categories. The functor $\Ani(G)$ is conservative, preserves sifted colimits, and the canonical map
    \[
    \tau_{\leq 0}\circ\Ani(G)\to G\circ\tau_{\leq 0}
    \]
    is an equivalence. If $G$ preserves small colimits, then so does $\Ani(G).$
\end{proposition}

\begin{example}
    If $k\to A$ is a map of discrete commutative rings, then we may animate the adjunction
    \[\begin{tikzcd}
        -\otimes_k A : \Mod_k\ar[r,shift left] &\ar[l, shift left] \Mod_A : U
    \end{tikzcd}\]
    to obtain the adjunction
    \[\begin{tikzcd}
        -\otimes^L_k A : \cMod_k^{\cn}\ar[r,shift left] &\ar[l, shift left] \cMod_A^{\cn} : U
    \end{tikzcd}\]
\end{example}

\begin{proposition}[\protect{\cite[Corollaries~2.13,2.14]{mao2024revisitingderivedcrystallinecohomology}}]\label{Animation of arrow category}
    Let $\sfC$ be a projectively generated $1$-category. Then there are canonical equivalences
    \begin{align*}
        \Ani(\Fun((\Delta^1)^{\op},\sfC))&\to \Fun((\Delta^1)^{\op},\Ani(\sfC)),\\
        \Ani(\Fun(\Delta^1,\sfC))&\to \Fun(\Delta^1,\Ani(\sfC)),\\
        \Ani(\Fun((\ZZ,\geq),\sfC))&\to\Fun((\ZZ,\geq),\Ani(\sfC)),\\
        \Ani(\Fun((\NN,\geq),\sfC))&\to\Fun((\NN,\geq),\Ani(\sfC)),\\
        \Ani(\Fun(\ZZ,\sfC))&\to\Fun(\ZZ,\Ani(\sfC)),\\
        \Ani(\Fun(\NN,\sfC))&\to\Fun(\NN,\Ani(\sfC)),\\
        \Ani(\Fun(\{0,1\},\sfC))&\to\Fun(\{0,1\},\Ani(\sfC)),\\
        \Ani(\sfC_{X/})&\to \Ani(\sfC)_{X/}
    \end{align*}
    of $\infty$-categories.
\end{proposition}

The proof of this result also describes a generating set of $\Fun(\Delta^1,\sfC)$ explicitly: if $S\subseteq\sfC^{\sfp}$ is a set of compact projective generators for $\sfC,$ then
\[
\{1\to X\mid X\in S\}\cup\{X\xrightarrow{\id_X} X\mid X\in S\}
\]
is a set of compact projective generators for $\Fun(\Delta^1,\sfC).$
\chapter{Filtrations}\label{filtrations}

\label{AppendixB}

We provide a brief review of the terminology and basic results regarding filtrations in the $\infty$-categorical setting. The reader who is not familiar with the $\infty$-categorical setting should note that asking for a filtration to be a sequence of nested subobjects is too restrictive in this setting.

We let $\ZZ$ (respectively $\NN$) denote the poset of integers (respectively natural numbers) under $\geq,$ and regard this as a category whose objects are integers $n$ and whose set of morphisms $\Hom(n,m)$ is a single element if $m\geq n$ and empty otherwise. We let $\ZZ^{\disc}$ (respectively $\NN^{\disc}$) denote the category whose objects are the integers (respectively natural numbers) and whose only morphisms are identities.

\begin{mydef}
    Let $\cC$ be an $\infty$-category.
    \begin{enumerate}
        \item The \emph{$\infty$-category of filtered objects of $\cC$} is the category $\Fil(\cC) := \Fun(\N(\ZZ^{\op}),\cC).$ We will denote an object of $\Fil(\cC)$ by $F^\bullet$ or $\Fil^\bullet X.$
        \item The \emph{$\infty$-category of $\NN$-filtered objects of $\cC$} is the category $\Fil_{\NN}(\cC) := \Fun(\N(\NN^{\op}),\cC).$
        \item A filtered object $\Fil^\bullet X\in\Fil(\cC)$ (or $\Fil_{\NN}(\cC)$) is \emph{complete} if $\lim\Fil^\bullet X\simeq 0.$ We will denote by $\cFil(\cC)$ the full sub-$\infty$-category of $\Fil(\cC)$ spanned by the complete filtered objects.
    \end{enumerate}
\end{mydef}

\begin{remark}
    Note that our filtrations are decreasing by definition.
\end{remark}

Suppose that $\cC$ is presentable. We have a functor
\[
\colim : \Fil(\cC)\to\cC
\]
which sends a filtered object to its \emph{underlying object}. 

Let $\cC$ be an $\infty$-category admitting cofibers and countable coproducts (for example, a stable $\infty$-category). Precomposition with $i_n : \Delta^{\{n+1,n\}}\to\ZZ^{\op}$ induces functors
\[
i_n^* : \Fil(\cC)\to\Fun(\Delta^1,\cC).
\]

\begin{mydef}
    Let $\cC$ be an $\infty$-category admitting cofibers and countable coproducts.
    \begin{enumerate}
        \item The \emph{$n$-th graded functor} $\gr^n$ is the composite
        \[\gr^n : \Fil(\cC)\xrightarrow{i_n^*}\Fun(\Delta^1,\cC)\xrightarrow{\cofib}\cC.\]
        That is, $\gr^n(F^\bullet) := \cofib(F^{n+1}\to F^n).$
        \item The \emph{associated graded functor} $\gr$ is the functor
        \[\gr : \Fil(\cC)\to\Fun(\N(\ZZ^{\disc}),\cC)\simeq\prod_{n\in\ZZ}\cC\]
        induced by all the $i_n^*.$
    \end{enumerate}
\end{mydef}

If $\cC$ is pointed, then $\gr$ is left adjoint to the functor $\Fun(\N(\ZZ^{\disc}),\cC)$ sending $(X_n)_n\mapsto (\dots\xrightarrow{0} X_{n+1}\xrightarrow{0} X_n\xrightarrow{0}\dots).$

\begin{example}
    Suppose that $A$ is an animated ring. Then the \emph{filtered derived category of $A$} is $\DF(A) := \Fil(\cMod_A).$
\end{example}

\begin{mydef}
    Let $\cC$ be a $\infty$-category admitting $\ZZ$-indexed limits, and let $F^\bullet\in\Fil(\cC).$ We say $F^\bullet$ is \emph{complete} if $\lim_n F^n\simeq 0.$ Let $\cFil(\cC)$ denote the full sub-$\infty$-category of $\Fil(\cC)$ spanned by the complete filtered objects. 
\end{mydef}

The inclusion $\cFil(\cC)\to\Fil(\cC)$ admits a left adjoint, given by sending $F^\bullet$ to the filtered object $\hat{F}^\bullet$ with $\hat{F}^n :=\cofib(F^{\infty}\to F^n),$ where $F^{\infty} := \lim_n F^n.$

\begin{lemma}[\protect{\cite[Lemma~2.15]{GwilliamPavlovEnhancing}}]\label{gradedequivalencecompleteequivalence}
    Let $f : X^\bullet\to Y^\bullet\in\Fil(\cC)$ be a map of filtered objects in $\cC.$ Then the induced map $\hat{f} : \widehat{X}^\bullet\to\widehat{Y}^\bullet$ is an equivalence if and only if $\gr^*(f) : \gr^* X^\bullet\to \gr^* Y^\bullet$ is an equivalence in $\Gr(\cC).$
\end{lemma}

We will also make use of the Whitehead filtration arising from a stable $\infty$-category with a t-structure. For the definition of a t-structure on a stable $\infty$-category, we refer to \cite[1.2.1]{HA}

\begin{mydef}
    Suppose that $(\cC,\cC_{\geq0},\cC_{\leq0})$ is a stable $\infty$-category equipped with a t-structure.
    \begin{enumerate}
        \item The \emph{diagonal t-structure} on $\Fil(\cC)$ is defined by declaring that $F^\bullet X\in\Fil(\cC)$ is connective if $F^n X\in\cC_{\geq n}$ for all $n.$
        \item Let $X\in\cC.$ The \emph{Whitehead filtration} on $X$ is given by the $\tau_{\geq 0}$ truncation of the constant filtered object of $X$ in the diagonal t-structure on $\Fil\cC.$ We will denote this object by $\tau_{\geq *}X.$
    \end{enumerate}
\end{mydef}

\begin{example}
    Suppose that $A$ is an animated ring. The $\infty$-category $\cMod_A$ has a natural t-structure $(\cMod_A,\cMod_A^{\cn},\cMod_A^{\mathrm{ccn}}),$ where $\cMod_A^{\cn}$ is the full sub-$\infty$-category of connective objects, and $\cMod_A^{\mathrm{ccn}}$ is the full sub-$\infty$-category of coconnective objects. This t-structure is induced by the t-structure on spectra (see \cite[1.4.3.6]{HA}). For $M\in\cMod_A,$ the Whitehead filtration on $M$ is given by the system of canonical maps
    \[
    \dots\leftarrow\tau_{\geq-2} M\leftarrow\tau_{\geq -1}M\leftarrow\tau_{\geq0}M\leftarrow \tau_{\geq 1} M\leftarrow\tau_{\geq 2}M\leftarrow\dots
    \]
    of truncations. If $A$ is a discrete ring and we choose a chain complex representing $M,$ the Whitehead filtration on $M$ is given by the system of (homological) canonical truncations of $M,$ see \cite[\href{https://stacks.math.columbia.edu/tag/0118}{Tag 0118}]{stacks-project}.
\end{example}

\begin{remark}
    What we call the Whitehead filtration is also often called the Postnikov filtration in the literature, especially when discussing Hochschild homology.
\end{remark}

\end{document}